%% file: lqg-cle-weighted.tex
\documentclass[a4paper, 11pt, twoside]{amsart}
\usepackage[a4paper, top=3cm, bottom=3cm, right=2.5cm, left=2.5cm]{geometry}

\usepackage[english]{babel}
\usepackage[utf8]{inputenc}
\usepackage[T1]{fontenc}
\usepackage{microtype}
\usepackage[style=british]{csquotes}

\usepackage{amsmath, amsthm, amssymb, bm, amsfonts}
\usepackage{enumerate}

\usepackage{calc}
\usepackage{graphicx}
\usepackage{paralist}
\usepackage{comment}
\usepackage[dvipsnames]{xcolor}
\definecolor{myviolet}{RGB}{160, 32, 240}
 
\usepackage{lmodern}

\graphicspath{{img/}}

\usepackage[normalem]{ulem}

\usepackage[%
	unicode,
	colorlinks=true,
	linkcolor=blue,
	urlcolor=blue,
    linktocpage=true,
	citecolor=blue]{hyperref}
\hypersetup{%
	pdfpagemode = UseNone,
	pdfstartview = FitH,
	pdfdisplaydoctitle = true,
	pdfborder = {0 0 0}
}

\emergencystretch=\maxdimen
\makeatletter
\renewcommand\subsection{\@startsection{subsection}{2}%
  \z@{-.5\linespacing\@plus-.7\linespacing}{.5\linespacing}%
  {\normalfont\scshape}}
\makeatother

\newcommand\blfootnote[1]{%
  \begingroup
  \renewcommand\thefootnote{}\footnote{#1}%
  \addtocounter{footnote}{-1}%
  \endgroup
}

\newtheorem{thm}{Theorem}[section]
\newtheorem{lemma}[thm]{Lemma}
\newtheorem{prop}[thm]{Proposition}
\newtheorem{cor}[thm]{Corollary}
\newtheorem{conj}[thm]{Conjecture}

\newtheorem{defn}[thm]{Definition}

\theoremstyle{remark}
\newtheorem{remark}[thm]{Remark}

\numberwithin{equation}{section}

\def\P{\mathbb{P}}
\def\E{\mathbb{E}}
\def\R{\mathbb{R}}
\def\H{\mathbb{H}}
\def\C{\mathbb{C}}
\def\N{\mathbb{N}}
\def\Z{\mathbb{Z}}
\def\D{\mathbb{D}}
\def\Q{\mathbb{Q}}

\def\cQ{\mathcal{Q}}
\def\cL{\mathcal{L}}
\def\cI{\mathcal{I}}
\def\cF{\mathcal{F}}
\def\cD{\mathcal{D}}
\def\cA{\mathcal{A}}

\newcommand{\z}{{\bm{z}}}
\newcommand{\w}{{\bm{w}}}
\newcommand{\sigmab}{\bm{\sigma}}

\newcommand{\op}[1]{\operatorname{#1}}

\newcommand{\M}{\mathrm{M}}

\newcommand{\ML}{\mathrm{ML}}
\newcommand{\MC}{\mathrm{MC}}

\def\SLE{\operatorname{SLE}}
\def\CLE{\operatorname{CLE}}
\def\Var{\operatorname{Var}}

\newcommand{\ol}{\overline}
\newcommand{\wt}{\widetilde}
\newcommand{\wh}{\widehat}
\newcommand{\frk}{\mathfrak}

\begin{document}

\title[LQG weighted by CLE nesting statistics]{Liouville quantum gravity weighted by conformal loop ensemble nesting statistics}

\author{Nina Holden and Matthis Lehmkuehler}
\date{August 29, 2024}
\begin{abstract} 
	We study Liouville quantum gravity (LQG) surfaces whose law has been reweighted according to nesting statistics for a conformal loop ensemble (CLE) relative to $n\in \mathbb{N}_0$ marked points $z_1,\dots,z_n$. The idea is to consider a reweighting by $\prod_{B\subseteq \{1,\dots,n\}} e^{\sigma_B N_B}$, where $\sigma_B\in\mathbb{R}$ and $N_B$ is the number of CLE loops surrounding the points $z_i$ for $i\in B$. This is made precise via an approximation procedure where as part of the proof we derive strong spatial independence results for CLE. The reweighting induces logarithmic singularities for the Liouville field at $z_1,\dots,z_n$ with a magnitude depending explicitly on $\sigma_1,\dots,\sigma_n$. We define the partition function of the surface, compute it for $n\in\{0,1\}$, and derive a recursive formula expressing the $n>1$ point partition function in terms of lower-order partition functions. The proof of the latter result is based on a continuum peeling process previously studied by Miller, Sheffield and Werner in the case $n=0$, and we derive an explicit formula for the generator of a boundary length process that can be associated with the exploration for general $n$. We use the recursive formula to partly characterize for which values of $(\sigma_B\colon B\subseteq \{1,\dots,n\})$ the partition function is finite. Finally, we give a new proof for the law of the conformal radius of CLE, which was originally established by Schramm, Sheffield, and Wilson.
\end{abstract}
\blfootnote{N.\,H., M.\,L. -- ETH Zürich, Rämistrasse 101, 8092 Zürich, Switzerland}

\maketitle
\setcounter{tocdepth}{1}
\tableofcontents
\newpage

\section{Introduction}

Liouville quantum gravity (LQG) is a theory of random fractal surfaces with conformal symmetries which arise as the scaling limit of discrete surfaces known as random planar maps. It has its origin in theoretical physics \cite{polyakov-qg1} and has been an active field of study in probability theory for the past fifteen years.

A particularly fruitful direction in the study of LQG has been the discovery of couplings between LQG surfaces and conformally invariant random fractal curves called  Schramm-Loewner evolutions (SLE) \cite{schramm0}. An extensive theory of conformal welding for LQG surfaces, where SLE curves arise as interfaces, has been developed \cite{shef-zipper,wedges} (see also \cite{hp-welding,ag-disk,ahs-disk,ahs-sle}). These results have had a number of applications for both LQG, SLE, and random planar maps, see \cite{ghs-mating-survey} for a survey.

In recent works, Miller, Sheffield and Werner \cite{msw-simple,msw-non-simple} study closely related couplings between LQG surfaces and so-called conformal loop ensembles (CLE) \cite{shef-cle,shef-werner-cle}, which are loop versions of SLE. They prove that a CLE drawn on top of an independent LQG disk (which is arguably the most natural LQG surface with a disk topology) breaks the disk into independent smaller LQG disks and explicitly describe an exploration of the CLE decorated LQG disk called the continuum percolation interface (CPI), which is the continuum counterpart of the peeling process for random planar maps. They consider both the regular LQG disk, which has disk topology, and the generalized LQG disk, which has pinch points, where by a pinch point we mean a point whose removal disconnects the generalized disk into two disjoint components. See Figure \ref{fig:cle-lqg-warmup}. Their approach allows them to obtain several new results about SLE and CLE.

In our paper we study a model which generalizes the one in \cite{msw-simple,msw-non-simple}. Our LQG disks will have a finite collection of marked points $\bm{z}=(z_i\colon i\in A)$ (for $A\subseteq\N$ finite) where the Liouville field (which is a distribution in the unit disk that describes the distortion of the Euclidean metric) has logarithmic singularities. These surfaces arise in a natural probabilistic construction where the law of an LQG disk with a CLE has been reweighted according to CLE loop nesting statistics around the points $\bm{z}$. In the remaining introduction we will first introduce the variants of CLE and LQG disks which will be studied in the paper (Sections \ref{sec:intro-def} to \ref{sss:gen-disk}) and then we state our main results (Section \ref{sec:intro-res}), namely exploration results for these reweighted loop decorated LQG disks, and recursive formulas as well as finiteness conditions for their partition functions.

\begin{figure}
	\centering
	\def\svgwidth{0.8\columnwidth}
	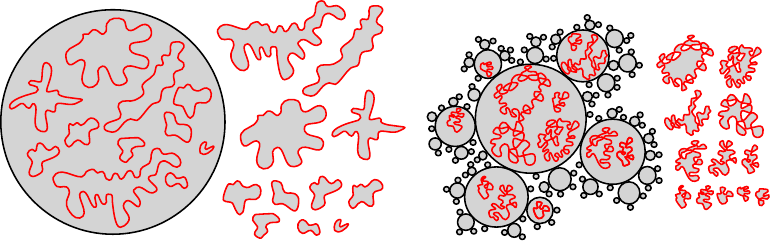
	\caption{\emph{Left.} The quantum surfaces obtained by cutting an LQG disk with parameter $\gamma$ with a (simple) non-nested CLE with parameter $\kappa=\gamma^2\in (8/3,4)$ are again independent quantum disks conditionally on all of their boundary lengths. \emph{Right.} The analogous statement holds when we decorate a generalized quantum disk of parameter $\gamma$ with a (non-simple) non-nested CLE of parameter $\kappa'=16/\gamma^2\in (4,8)$.}
	\label{fig:cle-lqg-warmup}
\end{figure}

\subsection{Motivation and parameters}
\label{sec:intro-def}

Let us first give some motivation for our model, and let $\gamma\in(0,2)$. The $\gamma$-LQG disk is a particularly natural $\gamma$-LQG surface, for example since it arises as scaling limits of random planar maps with disk topology \cite{bet-mier-disk,lqg-tbm3}. For concreteness, let us focus on the regular (not generalized) disk. As demonstrated by Miller, Sheffield, and Werner, it is natural to decorate this disk with an independent $\CLE_\kappa$ with $\kappa=\gamma^2$, which is a random countable collection of non-crossing nested loops whose law is conformally invariant. Let $h$ be the random distribution in $\D$ describing the unit boundary length $\gamma$-LQG disk and let $\Gamma$ denote the independent $\CLE_\kappa$ in $\D$. The field $h$ induces an area measure in $\D$ called the $\gamma$-LQG area measure which heuristically speaking has density $e^{\gamma h}$ relative to the Lebesgue area measure. It also induces a $\gamma$-LQG length measure along the loops of $\Gamma$.

For finite and non-empty $A\subseteq\N$, parameters $(\sigma_B\in\R\colon B\subseteq A)$, and points $\z=(z_i\in\D\,:\,i\in A)$ sampled independently according to the $\gamma$-LQG area measure, we are interested in reweighting the law of $(h,\z,\Gamma)$ by
\begin{align}
	\prod_{\emptyset\neq B\subseteq A} e^{\sigma_B N_B},
	\label{eq:reweight-nesting}
\end{align}
where $N_B$ is the number of CLE loops surrounding the points $z_i$ for $i\in B$ and not surrounding the points $z_i$ for $i\not\in B$. We often write $\sigma_i,N_i$ instead of $\sigma_{ \{i \} },N_{ \{i \} }$ for $i\in A$ to simplify notation. Since any fixed point is a.s.\ surrounded by a loop of a non-nested CLE, the iterative construction of a nested CLE from non-nested CLEs implies that $N_{i}=\infty$ a.s.\ for any $i\in A$ and therefore this reweighting does not make rigorous sense. When trying to define it rigorously, it is natural to apply an approximation procedure where one only considers the CLE loops of size above a certain threshold $\delta$. When defining the size of a CLE loop one wants to choose an \emph{intrinsic} definition, namely one which does not depend on the embedding $h$ of the $\gamma$-LQG disk into $\D$. For example, the Euclidean diameter of the loop would not be an appropriate measure, while the $\gamma$-LQG area surrounded by the loop or the $\gamma$-LQG length of the loop would be appropriate. Therefore, if $\sigma_i>0$ (resp.\ $\sigma_i<0$) then the reweighting we apply will both favor $\Gamma$ with a high (resp.\ low) density of loops around $z_i$ and $h$ which is particularly large (resp.\ small) near $z_i$. This explains why the reweighting \eqref{eq:reweight-nesting} affects the law of \emph{both} $\Gamma$ and $h$. Furthermore, one can argue at least heuristically that the effect of the reweighting should factorize in the sense that $\Gamma$ and $h$ are still independent (given $\z$) after the reweighting. In the rest of this subsection we will give precise definitions of the loop ensemble $\Gamma$ and field $h$ that we expect to get when applying the reweighting \eqref{eq:reweight-nesting}. Section \ref{sss:regular-disk} contains an explanation on how the heuristics above match the definitions made in this paper.

The $\gamma$-LQG disk decorated with an independent $\CLE_\kappa$ for $\kappa=\gamma^2\in(8/3,4)$ is believed to arise as the scaling limit of random planar maps decorated by a loop $O(n)$ model. Similarly, we conjecture that the variant of the $\gamma$-LQG disk obtained via the reweighting \eqref{eq:reweight-nesting} describes the scaling limit of random planar maps with a loop $O(n)$ model whose law has been reweighted according to nesting statistics relative to $\# A$ marked vertices. We refer to Section \ref{sec:discrete} for precise conjectures.

Throughout the text, several parameters will have to be tuned in a very particular way for the results to hold. Below, we first associate to each CLE parameter $\kappa$ the corresponding LQG parameter $\gamma$ and define how the CLE reweighting parameter $\sigma$ relates to the strength $\alpha^\kappa_\sigma$ of the corresponding logarithmic singularity of the Liouville field. The parameter $\rho^\kappa_\sigma$ appears in the renormalization term in the definition of reweighted CLE. The relationships between the various parameters will be justified later. The reader might wish to skip this definition and refer back to it whenever necessary. 
\begin{defn}
	\label{def:main-params}
	Fix $\kappa\in (8/3,8)$ and let $\gamma=\sqrt{\kappa}\wedge 4/\sqrt{\kappa}\,$. We will write $Q_\beta = \beta/2 + 2/\beta$ for $\beta>0$. Let $\sigma\mapsto \alpha^\kappa_\sigma$ be the strictly increasing bijection from $(-\infty,-\log(-\cos(4\pi/\kappa)))$ to $(Q_\gamma-\sqrt{\kappa}/4,Q_\gamma)$ given via
	\begin{align*}
		e^{-\sigma} = \frac{-\cos(4\pi/\kappa)}{-\cos(4\pi/\kappa - \pi\cdot 2\alpha^\kappa_\sigma/\sqrt{\kappa})}
	\end{align*}
	and let $\sigma\mapsto \rho^\kappa_\sigma$ the strictly increasing bijection from $\R$ to $(-1 + 2/\kappa + 3\kappa/32,\infty)$ given via
	\begin{align*}
		e^{-\sigma} = \frac{-\cos(4\pi/\kappa)}{\cos\left( \pi\sqrt{(1-4/\kappa)^2 - 8\rho^\kappa_\sigma/\kappa} \right)}\;.
	\end{align*}
	Note that $ \rho^\kappa_\sigma=\alpha^\kappa_\sigma(Q_\gamma - Q_{\alpha^\kappa_\sigma})$ when $\sigma < -\log(-\cos(4\pi/\kappa))$. Whenever $A\subseteq \N$ is finite, we let	
	\begin{align*}
		\mathfrak{S}_\kappa^A = \left\{\sigmab\in \R^{\{ B\colon B\subseteq A\}}\colon 
		\sigma_i\in(-\infty,-\log(-\cos(4\pi/\kappa))) \text{\ for\ }i\in A \;,\ \sigma_\emptyset=0\right\}\;.
	\end{align*}
	We also write $\frk S_\kappa = \cup_A \frk S^A_\kappa$ where the union is over all finite subsets $A$ of $\N$. To simplify notation, we will frequently write $\alpha_i$ instead of $\alpha^\kappa_{\sigma_i}$ for $i\in A$ when $\sigmab\in \frk S_\kappa^A$ is clear from the context. For $B\subseteq A$ we let $\sigmab|_B=(\sigma_C\colon C\subseteq B)$.
\end{defn}

\subsection{CLE weighted by nesting statistics}
\label{ss:cle-weighted}

We will now define the variant of CLE that we expect to get when performing the reweighting \eqref{eq:reweight-nesting}. First we state Theorem \ref{thm:loopexpconv}, which motivates our definition and is crucial for our weighted CLE to be well-defined. The result should be viewed as making sense of the joint moment generating function of the nesting statistics of a nested CLE, where a nesting statistic counts the number of CLE loops surrounding some points. The difficulty lies in the fact that single points are surrounded by an infinite number of loops, which implies that a regularization procedure is necessary to make sense of this moment generating function.

In the theorem statement and throughout the text, whenever $U$ is a disjoint union of proper simply connected domains and $z\in U$, we write $R(z,U)$ for the conformal radius of $z$ in the connected component of $U$ containing $z$. Furthermore given a closed curve $\eta$ and a point $z$, we say that $\eta$ surrounds $z$ if the index of $\eta$ around $z$ is $\pm 1$. We let $\eta^o$ denote the points  surrounded by the curve $\eta$.

\begin{thm}
	\label{thm:loopexpconv}
	Let $\kappa\in (8/3,8)$, $A\subseteq \N$ finite, and $D\subsetneq \C$ be simply connected. Moreover, consider distinct points $z_i\in D$ for $i\in A$ and constants $\sigma_B\in\R$ for all non-empty $B\subseteq A$. Let $\Gamma$ be a nested $\CLE_\kappa$ in $D$ and define $N^\epsilon_B(\z)$ to be the number of loops $\eta\in \Gamma$ such that $\eta$ surrounds $z_i$ if and only if $i\in B$ and such that $R(z_i,\eta^o)\ge \epsilon$ for all $i\in B$. For all $i\in A$, let $\eta_i\in \Gamma$ be the outermost loop surrounding $z_i$ but not $z_j$ for each $j\neq i$. Then
	\begin{align}
		\begin{split}
		&\prod_{i\in A} \frac{\epsilon^{\rho_{\sigma_i}^\kappa}}{c^\kappa_{\sigma_i}}\, \E\left(\exp\left( \sum_{B\subseteq A\colon \#B\ge 1} \sigma_B N^\epsilon_B(\z) \right) \right) \\
		&\quad \to \E\left(\prod_{i\in A} e^{\sigma_i}R(z_i,\eta_i^o)^{\rho_{\sigma_i}^\kappa} \exp\left( \sum_{B\subseteq A\colon \#B> 1} \sigma_B N^0_B(\z) \right) \right) < \infty\quad\text{as $\epsilon \downarrow 0$}
		\end{split}
		\label{eq:phi-limit}
	\end{align}
	for some (explicit) constants $c^\kappa_{\sigma_i}$ for $i\in A$.\footnote{\,Note that the sum in the expression for the limit does not range over singletons, i.e.\ $B$ with $\#B=1$, so that $N^0_B$ is a.s.\ finite for the considered sets $B$.} We define $\Phi^{\sigmab,\kappa}_D(\z)$ to be the limit on the right-hand side of the display above where $\z=(z_i\,:\,i\in A)$ and $\sigmab=(\sigma_B\,:\,B\subseteq A)$ (with $\sigma_\emptyset=0$). Moreover, for $\delta>0$ and $K\subseteq D$ compact, it holds that
	\begin{align*}
		\sup\{\Phi^{\sigmab,\kappa}_D(\z)\colon \z=(z_i\,:\,i\in A)\in K^A \text{ and } |z_i-z_j|\ge \delta\,\, \forall i\neq j \} &< \infty,\\
		\inf\{\Phi^{\sigmab,\kappa}_D(\z)\colon \z=(z_i\,:\,i\in A)\in K^A \text{ and } |z_i-z_j|\ge \delta\,\, \forall i\neq j \} &> 0.
	\end{align*}
\end{thm}
The proof of the theorem is found in Section \ref{sec:cle}. It will in particular rely on a strong spatial independence result for CLE (Theorem \ref{thm:cle-indep}). Two key properties of the function $\Phi^{\sigmab,\kappa}_D$ are its conformal covariance and a type of recursion relation which is closely related to the Markov property of CLE. These properties are inherited from the associated properties of $\CLE_\kappa$, and are stated precisely in Lemmas \ref{lem:phi-axioms} and \ref{lem:xi-law}. We refer to the end of Section \ref{sss:regular-disk} for a discussion of a possible relationship between $\Phi_D^{\sigmab,\kappa}$ and a conformal field theory (CFT) for CLE.

In light of Theorem \ref{thm:loopexpconv}, it is natural to expect that we get the following variant of CLE when performing the reweighting \eqref{eq:reweight-nesting}. Before giving the definition, we recall that a loop ensemble is a countable collection of (non-oriented) loops. For $\kappa\in(8/3,8)$ a $\CLE_\kappa$ is a random loop ensemble which comes in two variants, nested and non-nested, where a non-nested $\CLE_\kappa$ is obtained by sampling a nested $\CLE_\kappa$ and only keeping the loops not surrounded by any other loop, where we use the word surrounded as explained above Theorem \ref{thm:loopexpconv}. If $\Gamma$ is a loop ensemble in a domain $D\subseteq\C$ then we define the complementary components of $\Gamma$ to be the open connected components of the set $\cup_{\eta\in\Gamma\,}\eta^o$. The reader may want to come back to part (iii) of the definition below after reading Section \ref{sec:cle-explorations}, as CPIs are precisely defined there.

\begin{defn}
	Let $\kappa\in(8/3,8)$, $D\subsetneq\C$ be simply connected, $A\subseteq\N$ be finite, $\sigmab=(\sigma_B\colon B\subseteq A)$ with $\sigma_\emptyset=0$, and $\z=(z_i\,:\,i\in A)$ be as in Theorem \ref{thm:loopexpconv}.
	\begin{enumerate}[(i)]
		\item A non-nested $\CLE_\kappa^{\sigmab}(\z)$ in $D$ is a loop ensemble whose law has the following Radon-Nikodym derivative with respect to a non-nested $\CLE_\kappa$ $\Gamma$ in $D$
		\begin{align*}
			\frac{1}{\Phi_{D}^{\sigmab,\kappa}(\z)}
			\prod_{\eta\in\Gamma} e^{\sigma_{\cI(\eta^o)}} \prod_{U\in \mathfrak{U}(\eta^o)}
			\Phi_{U}^{\sigmab|_{\mathcal{I}(U)},\kappa}
			( \z|_{\mathcal{I}(U)} )\;,
		\end{align*}
		where $\mathfrak{U}(V)$ denotes the connected components of an open set $V$, $\cI(U):=\{i\in A\colon z_i\in U \}$ for $U\subseteq D$, and $\z|_{\cI(U)}=(z_i\colon i\in \cI(U))$.
		\item A nested $\CLE_\kappa^{\sigmab}(\z)$ is a loop ensemble with the law of $\cup_{k\ge 0} \Gamma_{k}$ where $\Gamma_k$ for $k\in\N_0$ are loop ensembles sampled iteratively as follows:
		\vskip5pt
		\begin{itemize}
			\item Let $\Gamma_0$ be a non-nested $\CLE_\kappa^{\sigmab}(\z)$.
			\item Given $\Gamma_0,\dots,\Gamma_{k-1}$, sample an independent non-nested $\CLE_\kappa^{\sigmab|_{\cI(U)}}(\z|_{\cI(U)})$ in $U$ for each complementary component $U$ of $\Gamma_{k-1}$. Let $\Gamma_k$ be the union of these loop ensembles.
		\end{itemize}
		\item Consider two prime ends $w_0$ and $w_\infty$ of $D$, and a parameter $\beta\in [-1,1]$. We say that $\lambda$ is a CPI from $w_0$ to $w_\infty$ with asymmetry parameter $\beta$ within nested $\CLE_\kappa^{\bm{\sigma}}(\bm{z})$ $\Gamma$ if the conditional law of $\lambda$ given the outermost loops of $\Gamma$ is the same as that of a CPI from $w_0$ to $w_\infty$ with asymmetry parameter $\beta$ within a $\CLE_\kappa$ given the outermost loops of this $\CLE_\kappa$ (see Section \ref{sec:cle-explorations}).
	\end{enumerate}
	\label{def:cle}
\end{defn}
If we do not specify where a $\CLE_\kappa^{\sigmab}(\z)$ is nested or non-nested we refer to the nested variant. We point out that the construction in (iii) is well-defined since the law of non-nested $\CLE_\kappa^{\bm{\sigma}}(\bm{z})$ is absolutely continuous with respect to the law of a non-nested $\CLE_\kappa$. Note however that the nested $\CLE_\kappa^{\sigmab}(\z)$ is singular with respect to a nested $\CLE_\kappa$ except if $\sigma_i=0$ for all $i\in A$, in which case it is absolutely continuous. Indeed, in both the case of a CLE and a reweighted CLE, one can see that $\lim_{n\rightarrow\infty}\log (R(z_i,(\eta_i^n)^o))/n$ exists a.s.\ and is a deterministic constant if $\eta^i_n$ denotes the $n$th CLE loop surrounding $z_i$, and these limits differ when one considers different $\sigma_i$ values.

\subsection{Regular LQG disks weighted by CLE nesting statistics}
\label{sss:regular-disk}

We will now introduce the LQG surface we expect to obtain upon applying the reweighting \eqref{eq:reweight-nesting}.

In fact, when we define the surface here in the introduction we focus on the even smaller parameter range where $\alpha_i\in(Q_\gamma-\sqrt{\kappa}/4,2)$ for all $i\in A$ since the formulas can be written in a simpler form in this case due to the existence of the $\alpha_i$-LQG area measure. The general case $\alpha_i\in(Q_\gamma-\sqrt{\kappa}/4,Q_\gamma)$ can be treated via a direct generalization (see Definition \ref{def:disk} and Proposition \ref{prop:limit-eps}). For $\Lambda\geq 0$, $\ell>0$, $\kappa\in(8/3,8)\setminus\{4\}$, $A\subseteq\N$ finite, and $\sigmab\in\frk S^A_\kappa$ such that $\alpha_i=\alpha^\kappa_{\sigma_i}\in(Q_\gamma-\sqrt{\kappa}/4,2)$ for all $i\in A$, we define a measure $\op{M}^{\sigmab,\kappa}_{\Lambda,\ell}$ on tuples $(h,\z)\in H^{-1}(\D)\times\D^A$ by setting
\begin{align}
	\op{M}^{\sigmab,\kappa}_{\Lambda,\ell}(dh,d\bm{z})=
	e^{-\Lambda \mu_h^\gamma(\D)}
	\Phi_\D^{\sigmab,\kappa}(\z)\prod_{i\in A} \mu^{\alpha_i}_h(dz_i)\,P(dh)\;,
	\label{eq5}
\end{align}
where $P$ is the probability measure describing the field of an embedding of a $\gamma$-LQG disk without marked points and with boundary length $\ell$, and $\mu^{\alpha_i}_h$ is the $\alpha_i$-LQG measure associated to the field $h$.

Let us explain how this definition relates to the motivation \eqref{eq:reweight-nesting} by performing the following back of the envelope computation. We present it in the singleton case of $A= \{i\}$, writing $z=z_i$ and $\alpha = \alpha_i$, but the generalization to more points is immediate. Recall that we let $h$ be an embedding of a $\gamma$-LQG disk and $\Gamma$ denotes an independent $\CLE_\kappa$ in the unit disk $\D$. Let us fix $\epsilon > 0$ and consider the $\gamma$-LQG metric ball with $\gamma$-LQG area $\epsilon$ centered at the point $z\in \D$. At a heuristic level, we approximate this metric ball by a Euclidean ball $B_r(z)$ with $r>0$ implicitly given by
\begin{align*}
	\epsilon = r^{2+\gamma^2/2} e^{\gamma h_r(z)} \;.
\end{align*}
Here $h_r(z)$ denotes the circle average of the field on $\partial B_r(z)$ and the right-hand side is an approximation of $\mu_h^\gamma(B_r(z))$. Let $N_\epsilon$ be the number of loops in $\Gamma$ surrounding $B_r(z)$, which we view as an intrinsic way of regularizing the loop count around the point $z$ as alluded to in Section \ref{sec:intro-def}. In accordance with \eqref{eq:reweight-nesting} we consider the reweighting
\begin{align*}
	\epsilon^{\alpha/\gamma - 1} e^{\sigma N_\epsilon} \cdot r^{\gamma^2/2} e^{\gamma h_r(z)} &= r^{\alpha(Q_\gamma - Q_\alpha)} e^{\sigma N_\epsilon}  r^{\alpha^2/2} e^{\alpha h_r(z)},
\end{align*}
where the prefactor involving $\epsilon$ is chosen precisely so that the right-hand side has a limit. We now see that the expression on the right-hand side matches with \eqref{eq5} in the $\epsilon\to 0$ limit by the definition of $\alpha$-LQG measures and by the definition made in Theorem \ref{thm:loopexpconv}. Needless to say, this is only a heuristic argument but it explains the origin of the definitions made here.

Note that we are excluding the case $\kappa = 4$ throughout this paper since this corresponds to $\gamma = 2$ and therefore a critical Liouville quantum gravity measure, which we do not consider in this work.

The reason for introducing the exponential term involving the parameter $\Lambda$ in the definition above is only that we will be able to show that choosing $\Lambda$ sufficiently large will make the measure above finite. By Girsanov's theorem for the $\alpha_i$-LQG area measure (see \cite[Proposition 3.4]{shef-kpz}), it is not hard to see that the measure $\op{M}^{\sigmab,\kappa}_{\Lambda,\ell}$ is supported on fields $(h,\z)$ that have an $\alpha_i$ logarithmic singularity at $z_i$ for each $i\in A$.

A (regular) $\gamma$-LQG surface is defined to be an equivalence class of tuples $(D,h,\z)$ (with $D\subseteq\C$ simply connected, $h$ a distribution on $D$, and $\z=(z_i\,:\,i\in A)$ a tuple of points in $D$ for $A\subseteq\N$ finite) where $(D,h,\z)\sim(D',h',\bm{z'})$ if there is a conformal map $\phi:D\to D'$ such that
\begin{align}
	h=h'\circ\phi+Q_\gamma\log|\phi'|\;,\qquad
	\phi(z_i)= z_i'\text{ for all }i\in A\;.
	\label{eq:cc}
\end{align}
We let $[(D,h,\z)]$ denote the $\gamma$-LQG surface given by the equivalence class of $(D,h,\z)$ and call $h$ an embedding of $[(D,h,\z)]$ into $D$.
See Section \ref{sec:gff-lqg} for further details and motivation.

For $\kappa\in(8/3,8)\setminus\{4 \}$ we consider the measure given by taking the pushforward of $\smash{\op{M}^{\sigmab,\kappa}_{\Lambda,\ell}}$ by the function which maps $(h,\z)$ to $[(\D,h,\z)]$. Using the conformal covariance of the function $\Phi_\D^{\sigmab,\kappa}$ and of the $\alpha_i$-LQG area measure for all $i\in A$ one can check that this measure on $\gamma$-LQG surfaces does not depend on the embedding $h$ that we chose for the $\gamma$-LQG disk above. This follows from the change of coordinates formula for $\mu^{\alpha_i}_h$ in combination with Lemma \ref{lem:phi-axioms}, and the precise values of $\alpha_i$ appearing in Definition \ref{def:main-params} are the only ones which make this property true. When $\kappa<4$ we write $\bar{\op{M}}^{\sigmab,\kappa}_{\Lambda,\ell}$ for this measure on $\gamma$-LQG surfaces.

We now specialize to the case $\kappa\in(8/3,4)$; the case $\kappa\in(4,8)$ will be discussed in relation with the generalized LQG disk below. For $\kappa\in(8/3,4)$ we define the \emph{weight} $\smash{W^{\sigmab,\kappa}_{\Lambda,\ell}}$ and the \emph{partition function} $Z^{\sigmab,\kappa}_{\Lambda,\ell}$ of the disk as follows
\begin{align}
	\begin{split}
		W^{\sigmab,\kappa}_{\Lambda,\ell}
		=|\op{M}^{\sigmab,\kappa}_{\Lambda,\ell}|
		\qquad\text{and}\qquad
		Z^{\sigmab,\kappa}_{\Lambda,\ell}
		= \ell^{-1-4/\kappa} \,W^{\sigmab,\kappa}_{\Lambda,\ell}\;,
	\end{split}
\label{eq:part-fcn}
\end{align}
where $|\cdot|$ denotes the total mass of a measure.\footnote{\,The scaling $\ell^{1+4/\kappa}$ we do when defining the partition function corresponds to having a disk with a marked point on the boundary whose conditional law given the field is that of a point sampled from the $\gamma$-LQG boundary measure. This is the natural convention in our setting since the disks we study will come naturally equipped with such a marked boundary point. We mention this since in the random planar map literature one frequently sees a scaling factor $\ell^{4/\kappa}$ which does not take into account the $1/\ell$ normalization factor arising from sampling a boundary point according to the $\gamma$-LQG measure.} Note that $W^{-,\kappa}_{0,\ell}=1$. The following scaling property is immediate
\begin{align}
    Z^{\sigmab,\kappa}_{\Lambda,\ell}=\ell^{-1-\frac{4}{\kappa}+\frac{2}{\sqrt{\kappa}}\sum_{i\in A}\alpha_i}\, Z^{\sigmab,\kappa}_{\Lambda\ell^2,1} \;.
    \label{eq22}
\end{align}
We will see in Theorem \ref{thm:markov-simple} and Remark \ref{rmk:jump-rates} that $\smash{Z^{\sigmab,\kappa}_{\Lambda,\ell}}$ is the natural definition of the partition function for the disk measure introduced above and we conjecture in Section \ref{sec:discrete} that $Z^{\sigmab,\kappa}_{\Lambda,\ell}$ describes the scaling limit of the associated random planar map partition function.

\subsection{Relationships with conformal field theory}
In this subsection we will briefly comment on relationships between the objects defined above and conformal field theory (CFT). First we remark that if $(h,\z)$ is sampled from $\smash{\op{M}^{\sigmab,\kappa}_{\Lambda,\ell}}$ (normalized to be a probability measure, provided this is possible) and we condition on the location of the marked points $\z$, then the conditional law of the field $h$ is the following: It is the field with singularities $((z_i,\alpha_i)\colon i\in A)$, total boundary length $\ell$, and cosmological constant $\Lambda\geq 0$ that appears in Liouville conformal field theory (LCFT) \cite{hrv-disk}, see Section \ref{sec:disk}. At the end of Section \ref{sec:disk} we explain the relationship between our partition function and correlation functions appearing in LCFT. In particular, we will see that for $\kappa\in(8/3,4)$
\begin{align}
	W^{{\sigmab},\kappa}_{\Lambda,\ell}
	& = \wh C \ell^{-2 +2Q_\gamma/\gamma} \int_{\D^A} \Phi^{{\sigmab},\kappa}_\D(\bm{z}) Z^{\op{LCFT}}_{\Lambda, \ell,  (\bm{\alpha},\z),( \bm{\beta},\w )}\,d\bm{z}\;,
	\label{eq-total-mass}
\end{align}
where $\wh C>0$ is an explicit constant and $Z^{\op{LCFT}}_{\Lambda, \ell, (\bm{\alpha},\z),( \bm{\beta},\w )}$ is a so-called fixed boundary length LCFT partition function.  The tuple $(\bm{\alpha},\z)$ indicates the logarithmic singularities within the domain while $(\bm{\beta},\w)$ are three logarithmic $\gamma$ singularities on the boundary of the domain. Choosing the location of the latter three singularities corresponds to fixing the embedding of the LQG surface. These LCFT partition functions are closely related to LCFT correlation functions. 

One can hope that \eqref{eq-total-mass} uniquely characterizes $\Phi^{{\sigmab},\kappa}_\D(\bm{z})$, which would provide a potential path for computing this interesting CLE observable if the other unknowns in the formula (namely, $W^{{\sigmab},\kappa}_{\Lambda,\ell}$ and $Z^{\op{LCFT}}_{\Lambda, \ell, (\bm{\alpha},\z),( \bm{\beta},\w )}$) were to be  identified. See also Remark \ref{rmk:bootstrap} and Section \ref{sec:disk} for more details on this.

The functions $\Phi^{{\sigmab},\kappa}_\D$ are of particular interest due to their potential relevance for the construction of a CFT for CLE. It is expected that one can define a CFT for CLE where the $n$-point function is related to connection probabilities for $n$ points in discrete loop models converging to CLE \cite{ijs16,prs16,js19,hjs20}. Such connection probabilities are further related to functions of the form of the right-hand side of \eqref{eq:phi-limit}, see e.g.\ \cite[Section 1.4]{as-cle}. The latter work considers the variant of $\Phi^{{\sigmab},\kappa}_\D$ on a sphere with three marked points and shows that this function is given by an imaginary counterpart of the three point structure constant in Liouville CFT (known as the DOZZ formula, see \cite{krv-dozz}). In particular, this variant of $\Phi^{{\sigmab},\kappa}_\D$ should define the three point structure constant of the CFT for CLE on a sphere and we expect that $\Phi^{{\sigmab},\kappa}_\D$ with more points is related to the higher order correlation functions of the theory. The field of such a theory is a variant of the CLE nesting field constructed in \cite{mww-extremes}.

Finally we remark that we expect  $\Phi^{{\sigmab},\kappa}_\D$ to satisfy the following asymptotic property as $z_i\to z_j$ for $i,j\in A$ distinct: 
\begin{align*}
    \Phi^{\sigmab,\kappa}_{\D}(\z) \sim \wh c^{\,\kappa}_{\sigmab} |z_i-z_j|^{\rho^\kappa_{\sigma_{i}}+\rho^\kappa_{\sigma_{j}}-\rho^\kappa_{\sigma_{i,j}} }
    \Phi^{\sigmab',\kappa}_{\D}(\z')\;,\qquad \z'=(z_k\,:\,k\neq i)\;,
\end{align*}
where $\wh c^{\,\kappa}_{\sigmab}$ is a (non-explicit) constant and $\sigmab'=(\sigma'_B \colon B\subseteq A\setminus\{i \} )$ is given by $\sigma'_B=\sigma_B$ for $ B\not\ni j$ and $\sigma'_B=\sigma_{\{i \}\cup B}$ for $B\ni j$.\footnote{\,We do not give a proof of this property but it follows immediately via a heuristic argument from Theorem \ref{thm:loopexpconv} and the Markov property of CLE if we approximate the number of CLE loops $\smash{N^0_{\{i,j\}}}$ around $z_i$ and $z_j$ by the number of CLE loops around $z_j$ with conformal radius larger than $|z_i-z_j|$.} Such a property for a correlation function  $\Phi^{{\sigmab},\kappa}_\D$ is related to fusion rules and operator product expansion.

\subsection{Generalized LQG disks weighted by CLE nesting statistics}
\label{sss:gen-disk}

The reader can skip definitions and results related to the generalized LQG disk on a first reading if  they wish, but we will still provide a very brief introduction here.

For $\gamma\in(\sqrt{2},2)$, the generalized $\gamma$-LQG disk can, roughly speaking, be described as a tree of regular $\gamma$-LQG disks. See Figure \ref{fig:cle-lqg-warmup}, right. The tree structure is encoded by the time reversal of a spectrally positive Lévy excursion $E:[0,\ell]\to[0,\infty)$ with exponent $4/\gamma^2$ and we call the duration $\ell$ of the excursion the generalized boundary length of the disk. Each jump of the excursion is associated with an independent regular $\gamma$-LQG disk of boundary length equal to the magnitude of the jump. We refer to Section \ref{sec:normal-lqg-disks} for more details.

In our setting, we will consider a variant of the generalized $\gamma$-LQG disk with additional marked points indexed by a finite set $A\subseteq\N$. Fix $\kappa'=16/\gamma^2\in(4,8)$, $\sigmab\in\frk S^A_{\kappa'}$, $\Lambda\geq 0$, and $\ell>0$.\footnote{\,We write $\kappa'$ instead of $\kappa$ in most of the paper when this parameter takes its value in $(4,8)$.} The measure on surfaces we consider is defined as follows 
\begin{align}
	\bar{\op{M}}^{{\sigmab},\kappa'}_{\Lambda,\ell} =
	\sum_{Q\in \Pi(e,A)} \prod_{t<\ell\colon \Delta e_t\neq 0} \op{M}^{{\sigmab}|_{\Delta Q_{t}},\kappa'}_{\Lambda,|\Delta e_t|} \,P(de)\;,
	\label{eq:gen-disk}
\end{align}
where
$P$ is the law of $E$ and $\Pi(e,A)$ is the set of non-increasing (in the sense of inclusion) càdlàg functions $Q\colon [0,\ell]\to \mathcal{P}(A)$ such that $Q_0=A$, $Q_{\ell-}=\emptyset$ and $\Delta Q_t := Q_{t-}\setminus Q_t = \emptyset$ whenever $\Delta e_t := e_t-e_{t-}= 0$. The interpretation of $Q_t$ is the collection of indices of points in the unexplored part of the generalized $\gamma$-LQG disk at time $t$, so in particular $\Delta Q_t$ is the collection of indices of points which are in the regular $\gamma$-LQG disk with boundary length $\Delta e_t$.
A surface sampled from $\bar{\op{M}}^{{\sigmab},\kappa'}_{\Lambda,\ell}$ should be viewed as the gluing of surfaces sampled from $\op{M}^{{\sigmab}|_{\Delta Q_{t}},\kappa'}_{\Lambda,|\Delta e_t|}$ into a tree structure according to the excursion $e$.

The weight $W^{\sigmab,\kappa}_{\Lambda,\ell}$ and the partition function $Z^{\sigmab,\kappa}_{\Lambda,\ell}$ of the disk are defined as follows
\begin{align}
\begin{split}
	W^{\sigmab,\kappa'}_{\Lambda,\ell}
	=|\bar{\op{M}}^{\sigmab,\kappa'}_{\Lambda,\ell}|
	\qquad\text{and}\qquad
	Z^{\sigmab,\kappa'}_{\Lambda,\ell}
	= \ell^{-1-4/\kappa'} \,W^{\sigmab,\kappa'}_{\Lambda,\ell},
\end{split}
\label{eq:part-fcn-gen}
\end{align}
and \eqref{eq22} still holds with $\Lambda \ell^{8/\kappa'}$ instead of $\Lambda \ell^2$ in the subscript on the right-hand side and with $\kappa'$ instead of $\kappa$. The exponent $8/\kappa'$ describes how the LQG area scales with the generalized boundary length of the disk.

\subsection{Main results on partition functions and the law of the exploration}
\label{sec:intro-res}

For $\kappa\in(8/3,4)$ consider a non-nested $\CLE_\kappa$ $\Gamma$ in $\D$ and fix distinct points $w_0,w_\infty\in \partial \D$. In \cite{msw-simple,cle-percolations}, Miller, Sheffield, and Werner consider a natural exploration of $\Gamma$. The exploration is a continuous curve $\lambda$ known as the continuum percolation interface (CPI) which starts at $w_0$ and ends at $w_\infty$. For a parameter $\beta\in [-1,1]$, first assign independently to each loop in $\Gamma$ a counterclockwise or clockwise orientation with probability $(1+\beta)/2$ and $(1-\beta)/2$, respectively. Then the CPI $\lambda$ is an $\SLE_{\kappa'}$-type curve ($\kappa'=16/\kappa$) that stays in the carpet and leaves the counterclockwise (resp.\ clockwise) loops to its right (resp.\ left). See Section \ref{sec:cle-explorations} for further details.

We consider an embedding $h$ of a (regular) $\gamma$-LQG disk of boundary length $\ell>0$ into the unit disk $\D$ with the property that a point sampled from the boundary measure of the disk is mapped to $w_0$ under the embedding and such that the clockwise (resp.\ counterclockwise) boundary arc from $w_0$ to $w_\infty$ has length $\ell_L$ (resp.\ $\ell_R$) and where $\ell_L,\ell_R>0$ are such that $\ell_L+\ell_R=\ell$. We suppose that the $\CLE_\kappa$ and CPI have been sampled independently of $h$. As explained later, by conformal invariance of CLEs and CPIs, we can view the CLE and CPI as being drawn on top of the $\gamma$-LQG disk.

There is a natural way to measure the length of SLE-type curves using the $\gamma$-LQG measure, which we refer to as $\gamma$-LQG or quantum natural lengths (see Section \ref{sec:gff-lqg}) Here, we suppose $\lambda$ has been parametrized by its quantum natural length and denote its total duration by $\zeta$.

For $t<\zeta$, let $\xi(t)$ denote the union of $\lambda([0,t])$ and the loops touching $\lambda([0,t])$, namely
\begin{align}
	\label{eq:xi0}
	\xi(t) := \lambda([0,t])\cup \bigcup_{\eta\in \Gamma\colon \eta\cap \lambda([0,t])\neq\emptyset} \eta\;. 
\end{align}
Let $D_t$ be the complementary component of $\D\setminus\xi(t)$ which has $w_\infty$ on its boundary. For $t\in [0, \zeta)$, the two points $\lambda(t)$ and $w_\infty$ divide the boundary $\partial D_t$ of $D_t$ into a left boundary arc and a right boundary arc. Denote the quantum boundary lengths of these by $L_t$ and $R_t$, respectively. We set $L_t=R_t=0$ for $t\ge \zeta$.

If the CPI $\lambda$ hits the left boundary of $D_{t-}$ at time $t$ then $\Delta L_t:=L_t-L_{t-}<0$ and the domain which is separated from $w_\infty$ by $\lambda$ at time $t$ has quantum boundary length equal to $|\Delta L_t|$. Furthermore, if the CPI $\lambda$ hits a clockwise oriented loop $\eta$ at time $t$ then $\Delta L_t:=L_t-L_{t-}>0$ and the quantum length of the loop is equal to $\Delta L_t$. The same statements hold with right, $R$, and counterclockwise instead of left, $L$, and clockwise, respectively. See Figure \ref{fig:CPI-LR}.

\begin{figure}
	\centering
	\def\svgwidth{0.9\columnwidth}
	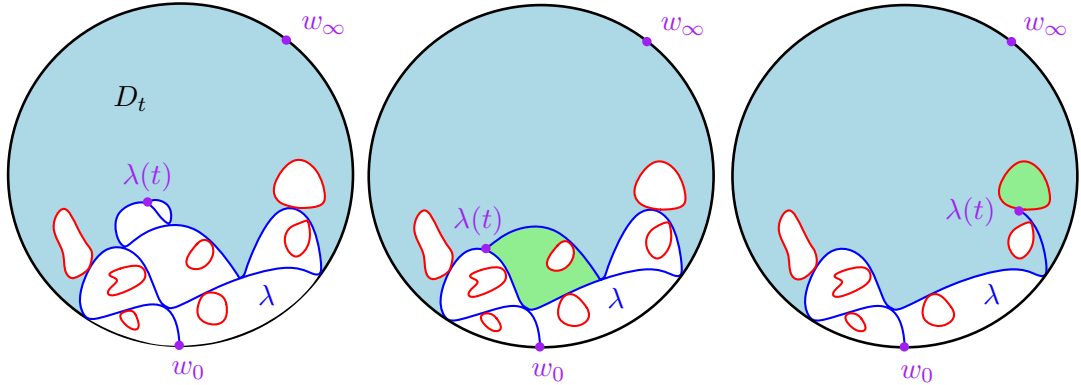
	\caption{\emph{Left}. Illustration of the CPI $\lambda$ (blue curve) until some time $t>0$ and the loops attached to $\lambda|_{[0,t]}$ (red). We have that $\xi(t)$ is the union of the blue curve and the red loops, while $D_t$ is the domain on blue background.
		\emph{Center}. We have $\Delta L_t<0$ since $\lambda$ hits the left boundary of the domain at time $t$. The domain which is cut off at time $t$ is shown in green.
		\emph{Right}. We have $\Delta R_t>0$ since $\lambda$ hits a counterclockwise oriented loop on its right side at time $t$. The domain which is cut off at time $t$ is shown in green.}
\label{fig:CPI-LR}
\end{figure}

The objects defined in the previous three paragraphs also make sense in the setting of LQG disks and CLE with additional marked points, as defined in Sections \ref{ss:cle-weighted} and \ref{sss:regular-disk}, and we will describe this case in more detail in this paragraph. Assume $\kappa\in(8/3,4 )$, $\sigmab\in\frk S_\kappa$, $\Lambda\geq 0$, and $\ell>0$ are such that $\smash{W^{\sigmab,\kappa}_{\Lambda,\ell}<\infty}$ (we address the admissibility question in Theorem \ref{thm:part-fcn-finite} below). First sample $(h,\z)$ according to $\op{M}^{\sigmab,\kappa}_{\Lambda,\ell}$ normalized to be a probability measure and then sample $\Gamma\sim{\CLE}_\kappa^{\sigmab}(\z)$ conditionally independently given $\z$. Let $\lambda$ be a CPI within $\Gamma$ as in Definition \ref{def:cle}, define $L$, $R$ and $\zeta$ as in the previous two paragraphs and for $t<\zeta$ let $P_t\subseteq A$ denote the indices of the points contained in $D_t$ (with $P_t=\emptyset$ for $t\ge \zeta$). We summarize all these processes in a tuple
\begin{align*}
	{\bf Y}=(L,R,P) \;.
\end{align*}
Above we considered the case $\kappa\in(8/3,4)$ but all of the objects can be defined in an analogous way for $\kappa'\in(4,8)$. Indeed, if $\kappa'\in (4,8)$, we consider a generalized LQG surface sampled from $\bar{\op{M}}^{{\sigmab},\kappa'}_{\Lambda,\ell}$, assuming the parameters are such that this measure has finite mass so that it can be renormalized to be a probability measure. Decorating the LQG surface with a $\CLE^{\sigmab}_{\kappa'}(\z)$ and a CPI, we can define $\zeta>0$ and a process ${\bf Y}=(L,R,P)$ as before. We refer to Section \ref{sec:background-msw} for a precise description of these objects in the case of no marked points; the construction is analogous in the case with marked points.

\begin{thm}\label{thm:markov-simple}
  Let $\kappa\in(8/3,4)$ (resp.\ $\kappa\in(4,8)$) and consider the surfaces and processes introduced above. The process ${\bf Y}$ is a strong Markov process and the rate at which the process jumps from $(\ell_L,\ell_R,B)$ to $(\ell_L+s,\ell_R,C)$ and $(\ell_L-s,\ell_R,C)$ for $C\subseteq B$, respectively, is given by the following for a constant $v^\kappa_\beta>0$
	\begin{align}
	-v^\kappa_\beta\cos(4\pi/\kappa)(1-\beta)\,\frac{
		Z^{\sigmab|_{C},\kappa}_{\Lambda,\ell_L+\ell_R+s}\,
		Z^{\sigmab|_{B\setminus C},\kappa}_{\Lambda,s}}
		{Z^{\sigmab|_{B},\kappa}_{\Lambda,\ell_L+\ell_R}}
	\qquad \op{and}\qquad
	v^\kappa_\beta\,\frac{
		Z^{\sigmab|_{C},\kappa}_{\Lambda,\ell_L+\ell_R-s}\,
		Z^{\sigmab|_{B\setminus C},\kappa}_{\Lambda,s}}
		{Z^{\sigmab|_{B},\kappa}_{\Lambda,\ell_L+\ell_R}} \;.
		\label{eq56}
	\end{align}
	The same holds for jumps from $(\ell_L,\ell_R,B)$ to $(\ell_L,\ell_R+s,C)$ and $(\ell_L,\ell_R-s,C)$, respectively, except that the multiplicative factor of $1-\beta$ is replaced by $1+\beta$ in the first formula.

	Moreover, conditionally on the process $\bf{Y}$ the quantum surfaces which are cut out by the CPI $\lambda$ (corresponding to the negative jumps of $L$ and $R$) and the quantum surfaces given by the inside of the discovered loops (corresponding to the positive jumps of $L$ and $R$) are independent regular (resp.\ generalized) quantum disks with marked points and boundary length as given by $\bf{Y}$.
\end{thm}

A more precise version of the theorem which also includes an explicit formula for the generator of the process $\bf Y$ is found in Theorem \ref{thm:markov-full}.

\begin{remark}
  The CPI is the continuum counterpart of the peeling process on planar maps, and the jump rates appearing in the theorem are reminiscent of those encountered in the peeling process, see \cite[Equation (14)]{budd-On-peeling} and Section \ref{sec:discrete}. Namely, the probability of a particular event is proportional to the product of the partition function of the surfaces we encounter on this event, see the numerators in \eqref{eq56}.
	\label{rmk:jump-rates}
\end{remark}

\begin{remark}
	Variants of the above theorem can be proved in a number of related settings. For example, instead of letting the CPI go towards a target point $w_\infty$ on the boundary of the disk, one can for instance explore towards one of the marked points $z_i$ or always go into the component with largest boundary length when the domain of exploration is cut in two. The latter exploration in the setting with no marked points is considered in \cite[Theorem 1.2]{msw-simple}. The transition rates in this case are obtained from the transition rates in the boundary-to-boundary exploration, and this proof carries through without changes to the setting of marked points.
\end{remark}

The special case of the theorem where $n=0$ was proved in \cite{msw-simple}. Our proof takes that result as a starting point and obtains the general result by applying a reweighting via Girsanov's theorem (see Proposition \ref{prop:chordal-zipper}).

From the proof leading to the above theorem one can prove a recursive formula for the partition function of the disk with $\#A\ge 2$ marked points in terms of the partition function of the disks with strictly fewer marked points. We remark that the process $(L,R)$ appearing in the theorem below can be described explicitly (up to a rescaling of time which does not affect the values of the expectations below) via a simple reweighting of a Lévy process, see Section \ref{sec:fragmentation-main}. This result arises from decomposing LQG surfaces using CPI curves as explained in the paragraph below the theorem statement.

\begin{thm}
	Let $\kappa\in(8/3,8)\setminus\{4 \}$, $A\subseteq\N$ finite, $\sigmab \in\frk S_\kappa^A$, $\Lambda\geq 0$, and $\ell>0$, and suppose $\#A\ge 2$. Define $X=L+R$, where $(L,R)$ is as in Theorem \ref{thm:markov-simple} for $A=\emptyset$ and $\Lambda=0$. Then
	\begin{align}
		W^{{\sigmab},\kappa}_{\Lambda,\ell} =\sum_{m\in\N_0} V_{\Lambda,\ell}^m\;,
	\end{align}
	where
	\begin{align*}
		V_{\Lambda,\ell}^0 &= \sum_{\emptyset\neq B\subsetneq A} \E_{(\ell/2,\ell/2)}\left( \sum_{t<\zeta} e^{\sigma_B 1(\Delta X_t>0)}\prod_{s<t}
		W^{-,\kappa}_{\Lambda,|\Delta X_s|}
		W^{{\sigmab}|_B,\kappa}_{\Lambda,|\Delta X_t|}
		W^{{\sigmab}|_{B\setminus A},\kappa}_{\Lambda,X_t} \right), \\
		V_{\Lambda,\ell}^{m+1}
		&= \E_{(\ell/2,\ell/2)}\left( \sum_{t<\zeta} \prod_{s< t} W^{-,\kappa}_{\Lambda,|\Delta X_s|}e^{\sigma_A 1(\Delta X_t>0)} V_{\Lambda,|\Delta X_t|}^m W^{-,\kappa}_{\Lambda,X_t} \right),\qquad m\in\N_0\;.
	\end{align*}
\label{thm:bootstrap}
\end{thm}

We now explain the proof idea and the intuition underlying the result. When we explore the CLE decorated LQG disk with the CPI as in Theorem \ref{thm:markov-simple} there is a positive chance that all the marked points are separated from the terminal point $w_\infty$ of the CPI at the same time, i.e., that they lie in the same complementary component of $\xi$. On the event that this happens, we can do a similar exploration with a new CPI in this complementary component, and we can iterate this procedure until not all the marked points lie in the same complementary component of $\xi$. Let $m\in\N_0$ denote the number of explorations we need to do before the first exploration where the points are separated. In our definition of $W^{\sigmab,\kappa}_{\Lambda,\ell}$ we can split the integral into a sum of integrals according to the value of $m$, and this gives the decomposition appearing in Theorem \ref{thm:bootstrap}. The proof of the theorem (found in Section \ref{sec:explorations}) builds on this intuition but also needs to rule out certain pathological behaviors of the CPI.

In the case $\# A\in\{0,1 \}$ the partition function has the following explicit closed formula. The function $\bar K_\nu$ (with $\nu>0$) appearing in the statement is given by the following formula
\begin{equation}
	\bar{K}_\nu(x):=\frac{2(x/2)^\nu}{\Gamma(\nu)}\,K_\nu(x)\;,\,\,x>0\;,\qquad\qquad
	\bar{K}_\nu(0)=1\;,
	\label{eq:Kbar}
\end{equation}
where $K_\nu$ is the modified Bessel function of the second kind (see the beginning of Section \ref{sec:levy-exc}).
\begin{thm}
	\label{thm:singlepoint-finite}
	Let $\kappa\in (8/3,8)\setminus \{4\}$, $A=\{i\}$, $\sigmab\in\frk S_\kappa^A$, $\Lambda\geq 0$, and $\ell>0$, and set $\alpha=\alpha^\kappa_{\sigma_i}$. Then $W^{-,\kappa}_{\Lambda,\ell},W^{{\sigmab},\kappa}_{\Lambda,\ell}<\infty$ and
	\begin{align*}
		W^{-,\kappa}_{\Lambda,\ell} &= \bar{K}_{4/\kappa}\left( 2\ell\left(\frac{\Lambda}{4\sin(\pi\gamma^2/4)}\right)^{(\kappa/8)\vee (1/2)}\, \right)\;, \\
		\frac{W^{{\sigmab},\kappa}_{\Lambda,\ell}}{W^{{\sigmab},\kappa}_{0,1}} &= \ell^{2\alpha/\sqrt{\kappa}} \bar{K}_{2/\sqrt{\kappa}\cdot(Q_\gamma-\alpha)}\left(2\ell\left(\frac{\Lambda}{4\sin(\pi\gamma^2/4)}\right)^{(\kappa/8)\vee (1/2)}\,\right) \;.
	\end{align*}
\end{thm}
We remark that the formulas in Theorem \ref{thm:singlepoint-finite} are an almost immediate consequence of \cite{ars-fzz} in the case $\kappa\in(8/3,4)$ (indeed, we will relate the weights appearing in the theorem to the partition functions in the cited paper), while they follow by combining the cited paper with an explicit computation for Lévy excursions in the case $\kappa\in(4,8)$. 

Note that the formulas in Theorem \ref{thm:singlepoint-finite} take the same form for $\kappa\in(8/3,4)$ and $\kappa\in(4,8)$ (the exponent $(\kappa/8)\vee(1/2)$ may appear at first to take a different form, but for both parameter ranges it represents how areas scale with boundary lengths). It is interesting that the formulas take the same form since the proofs are quite different for the two parameter ranges. The fact that the formulas in the two parameter ranges coincide supports the statement that the generalized disk should be viewed as the natural extension of the regular disk for $\kappa\in(4,8)$.

\begin{remark}
	When defining the weights $W^{{\sigmab},\kappa}_{\Lambda,\ell}$ and partition functions $Z^{{\sigmab},\kappa}_{\Lambda,\ell}$ we had to make some arbitrary choices for multiplicative constants involved in their definition. Conceptually, the partition function should give the \emph{relative} weight of different configurations, so, for example, multiplying all partition functions by a constant leaves the physical model unchanged. The degree of freedom we have lies precisely in the multiplicative constant for the zero-point and one-point partition functions, so if we fix these then the remaining partition functions are given. One can in particular observe that in Theorem \ref{thm:markov-simple}, replacing
	\begin{align*}
	    Z^{{\sigmab|_B},\kappa}_{\Lambda,\ell}\quad\text{by} \quad c\cdot\prod_{i\in B} c_i\cdot Z^{{\sigmab|_B},\kappa}_{\Lambda,\ell}
	\end{align*}
	for constants $c,c_i>0$, $i\in A$, and each $B\subseteq A$ leaves the jump rates unchanged.  Due to the arbitrariness in the multiplicative constant for the one-point partition function, we only consider the \emph{ratio} $W^{{\sigmab},\kappa}_{\Lambda,\ell}/W^{{\sigmab},\kappa}_{0,1}$ in Theorem \ref{thm:singlepoint-finite} but we remark that the denominator $W^{{\sigmab},\kappa}_{0,1}$ can be computed explicitly using results of \cite{remy-zhu-boundary}. For the zero-point partition function we fixed the multiplicative constant by requiring that $W^{-,\kappa}_{\Lambda,\ell}=1$.
	\label{rmk:arb-mult-const}
\end{remark}

\begin{remark}
	\label{rmk:bootstrap}
	Combining Theorems \ref{thm:bootstrap} and \ref{thm:singlepoint-finite}, one can in principle compute the weights $W^{\sigmab,\kappa}_{\Lambda,\ell}$ for an arbitrary number of marked points and any choice of $\sigmab$. Such recursive computation of partition functions is closely related to topological recursion, enumeration of planar maps and the loop equation in matrix models, see e.g.\ \cite{eynard-counting-surfaces,eynard-orantin-top-rec}.
	On the contrary, our approach is quite different from \cite{bpz-conformal,segal-cft,gkrv-bootstrap,gkrv-segal-axioms,wu-bootstrap-annulus}, where higher-order correlation functions in LCFT (which, like our weights $W^{\sigmab,\kappa}_{\Lambda,\ell}$, are closely related to partition functions) are expressed in terms of lower-order correlation functions via so-called conformal bootstrap; this recursive computation is of a different nature since the modulus is fixed, while we integrate over the modulus.
\end{remark}

A key question is when we have admissibility, namely for what parameters the partition function is finite. For $\kappa\in(8/3,8)\setminus\{ 4\}$ define the set $\cA_\kappa$ of admissible parameters by
\begin{align*}
    \cA_\kappa = \{({\sigmab},\Lambda)\in \frk S_\kappa\times [0,\infty)\,:\,Z^{\sigmab,\kappa}_{\Lambda,1}<\infty \}\;.
\end{align*}
We will not entirely identify the set $\cA_\kappa$ but the following theorem gives some of its properties.  
For example, assertion (ii) says that for more than one marked point it is necessary to have a positive cosmological constant $\Lambda >0$ in order to ensure admissibility. Another feature that is worth emphasizing is assertion (iv), which says that for an arbitrary number of marked points (and, more generally, for arbitrary values of $\sigma_i\in(-\infty,-\log(-\cos(4\pi/\kappa)))$) we can find $({\sigmab},\Lambda)\in\cA_\kappa$. Conversely, from assertion (vii) we see that if the number of points is at least two then we can find $({\sigmab},\Lambda)$ with $\Lambda>0$ such that $({\sigmab},\Lambda)\not\in\cA_\kappa$; combined with the previous sentence this implies that the set $\cA_\kappa$ is non-trivial in the setting of at least two points. Precisely identifying the set $\cA_\kappa$ is an interesting open problem.

\begin{thm}\label{thm:part-fcn-finite}
	Let $A\subseteq\N$ be finite and non-empty, ${\sigmab}\in\frk S^A_\kappa$, and $\Lambda\geq 0$.
	\begin{itemize}
		\item[(i)] If $\#A=1$ then $({\sigmab},\Lambda)\in\cA_\kappa$.
		\item[(ii)] If $\#A\geq 2$ and $\Lambda=0$ then $({\sigmab},\Lambda)\not\in\cA_\kappa$.
		\item[(iii)] If $({\sigmab},\Lambda)\in\cA_\kappa$ and  $B\subseteq A$ then $({\sigmab}|_B,\Lambda)\in\cA_\kappa$.
		\item[(iv)] If $\sigma_B<\sigma_i$ whenever $i\in B\subseteq A$ and $\#B\ge 2$ then $({\sigmab},\Lambda)\in\cA_\kappa$ for all $\Lambda>0$.
		\item[(v)] If $({\sigmab},\Lambda)\in\cA_\kappa$ for $\Lambda>0$ then $({\sigmab},\Lambda')\in\cA_\kappa$ for all $\Lambda'>0$.
		\item[(vi)] Suppose ${\sigmab}'$ is such that $\sigma'_{i}=\sigma_i$ for all $i\in A$ and $\sigma'_B\leq\sigma_B$ for all $B\subseteq A$ such that $\# B\geq 2$. Then if $({\sigmab},\Lambda)\in\cA_\kappa$ we have $({\sigmab}',\Lambda)\in\cA_\kappa$.
		\item[(vii)] Suppose ${\sigmab}'$ and $B'\subseteq A$ are such that $\#B'\geq 2$ and $\sigma'_B=\sigma_B$ for all $B\subseteq A,B'\neq B$. For $\sigma'_{B'}$ sufficiently large (depending on $\sigma_B$ for $B\neq B'$) we have $({\sigmab}',\Lambda)\not\in\cA_\kappa$.
	\end{itemize}
\end{thm}

We give an alternative proof of the following theorem, which was first proved in \cite{ssw-radii} via Itô calculus.\footnote{\,The expression for the density in Theorem \ref{thm:confradlaw} can be found just below \cite[Equation (2.22)]{brownian-first-passage}.} The proof is based on ideas from this paper in the $\kappa\neq 4$ case and the level line coupling of $\CLE_4$ with the GFF for $\kappa=4$. In the case $\kappa\neq 4$ we will use the explicit law of the boundary length process associated with the CPI for the case of $\#A=0$ marked points (as derived in \cite{msw-simple,msw-non-simple}) along with an exact computation for this process. See Section \ref{sec:ssw}. The theorem plays an important role in our proof of Theorem \ref{thm:loopexpconv} and we want to emphasize that our proof of Theorem \ref{thm:confradlaw} does not rely on Theorem \ref{thm:loopexpconv}.

\begin{thm}[\cite{ssw-radii}]
	\label{thm:confradlaw}
	Let $\Gamma$ be a non-nested $\CLE_\kappa$ in $\D$ for $\kappa\in (8/3,8)$. We write $\eta$ for the loop surrounding $0$ and recall that $\eta^o$ denotes the set of points surrounded by $\eta$. Then for all $\rho > -1 + 2/\kappa + 3\kappa/32$,
	\begin{align*}
		\E\left( R(0,\eta^o)^\rho \right) = \frac{-\cos(4\pi/\kappa)}{\cos\left( \pi\sqrt{(1-4/\kappa)^2 - 8\rho/\kappa} \right)}\;,
	\end{align*}
	while the expectation is infinite for $\rho \le -1 + 2/\kappa + 3\kappa/32$. Moreover, the density of $-\log R(0,\eta^o)$ on $(0,\infty)$ is
	\begin{align*}
		x\mapsto &\frac{-\cos(4\pi/\kappa)}{4\pi/\kappa}\sum_{n\ge 0} (-1)^n (n+1/2) \exp\left( \frac{(1-4/\kappa)^2}{8/\kappa}\,x- \frac{(n+1/2)^2}{8/\kappa}\,x\right)\\
		&= -\frac{4\sqrt{2\pi/\kappa}\,\cos(4\pi/\kappa)}{x^{3/2}} \sum_{n\ge 0} (-1)^n (n+1/2) \exp\left( \frac{(1-4/\kappa)^2}{8/\kappa}\,x - \frac{8\pi^2}{\kappa}(n+1/2)^2 x^{-1} \right)\;.
	\end{align*}
\end{thm}

\subsection{Outline}

In Section \ref{sec:discrete} we will describe some discrete counterparts of our results in the setting of random planar maps and formulate scaling limit conjectures. In Section \ref{sec:fragmentation-main} we study certain spatially inhomogeneous jump Markov processes (for which the process $\bf{Y}$ appearing in Theorem \ref{thm:markov-simple} is a special case), derive a formula for their generator, and do an explicit computation that is used later in the proof of Theorem \ref{thm:confradlaw}. In Section \ref{sec:levy-exc} we do computations for stable Lévy excursions which will be used when computing the partition function of the generalized LQG disk with zero or one marked points. We emphasize that the two latter sections can be read without any knowledge of LQG and CLE and only build on the theory of stochastic processes with values in $\R$ or in $\R^2$.

In Section \ref{sec:background} we present relevant background material on LQG and CLE, and in particular we give a precise description of the results on CLE explorations of \cite{msw-simple,msw-non-simple} that are reviewed briefly above. We then conclude the proof of Theorem \ref{thm:confradlaw} in Section \ref{sec:ssw} and give the proof of Theorem \ref{thm:loopexpconv} in Section \ref{sec:cle}. The technical bulk of the latter section is devoted to a spatial independence result for CLE. 

In Section \ref{sec:disk-cle-nesting} we give the precise definition of the LQG surface with marked points that appears in Theorem \ref{thm:markov-simple} and conclude the proof of Theorem \ref{thm:singlepoint-finite}. We also argue that the disk with marked points can be obtained via a limiting argument involving the disk with no marked points and we use the latter result to prove part of Theorem \ref{thm:markov-simple}. Finally, in Section \ref{sec:explorations} we use inputs from the previous sections to conclude the proofs of Theorems \ref{thm:markov-simple}, \ref{thm:bootstrap}, and \ref{thm:part-fcn-finite}.

The appendices contain various results and computations that are needed for the remainder of the paper. We mention in particular a result on renewal processes of independent interest in Appendix \ref{app:renewal}.\\

{\bf Notation.} We let $\N$ denote the positive integers $\N=\{1,2,\dots \}$ and let $\N_0$ denote the non-negative integers $\N_0 = \{0,1,2,\dots \}$. For $z\in\C$ and $r>0$ let $B_r(z)=\{w\in\C\,:\,|z-w|<r \}$ denote the open ball of radius $r$ centered at $z$. For an interval $I\subseteq\R$ we let $D(I)$ denote the set of real-valued càdlàg functions on $I$. We will often identify a curve with its trace.

For functions $f,g:S\to(0,\infty)$ and some set $S$ we write $f(x)\lesssim g(x)$ if there is a constant $C\in (0,\infty)$ such that for all $x\in S$ we have $f(x)\leq Cg(x)$. The dependence of $C$ on other parameters will be either clear from the context or stated explicitly.\\

{\bf Acknowledgments.} We thank the anonymous referees for their extensive comments and careful reading which helped us to greatly improve the manuscript. We also thank Wendelin Werner for several useful conversations, and we thank Jean Bertoin for providing the references \cite{kyprianou-levy-lamperti,silverstein-coharmonic-class}, Linxiao Chen for helpful discussions on random planar maps, and Xin Sun for helpful comments on the draft. Both authors were supported by grant 175505 of the Swiss National Science Foundation and both were part of SwissMAP. N.\,H.\ was also partially supported by Dr.\ Max Rössler, the Walter Haefner Foundation and the ETH Zürich Foundation and grant DMS-2246820 of the National Science Foundation.
\vskip10pt

\section{Discrete motivation}
\label{sec:discrete}

In this section, we will elaborate on the discrete counterparts of our results in the setting of random planar maps. For concreteness, we will describe all things in the setting of triangulations, but analogous constructions and conjectures can be stated in the setting of other planar maps (quadrangulations, general maps, etc.) as well. We call a planar map a triangulation with boundary if all faces have exactly three edges except for one distinguished face, called the exterior face, with arbitrary degree. The number of edges along the boundary of the exterior face is called the perimeter of the planar map. We allow the boundary of the exterior face to be non-simple.

Fix parameters
\begin{align*}
	A\subseteq \N\text{ finite}\;,\quad n\in (0,2)\;,\quad g,h,\Lambda\ge 0 \quad\text{and}\quad \sigma_B\in \R\text{ for all } B\subseteq A\text{ with $\sigma_\emptyset=0$}\;.
\end{align*}
Let $T=(V_T,E_T,F_T,e)$ be a finite planar triangulation with boundary of perimeter $\ell$, where $e$ is a distinguished (root) edge on the boundary. We consider a collection $\Gamma$ of vertex-disjoint (possibly nested) loops where the vertices of the loops are on the faces of $T$.  Figure \ref{fig:onmodel} illustrates the above definitions. For distinct points $x_i\in V_T$ with $i\in A$ and $\emptyset \neq B \subseteq A$, we let $N^\Gamma_B(\bm{x})$ denote the number of loops surrounding $x_i$ for each $i\in B$ but not surrounding the point $x_i$ whenever $i\notin B$. Moreover, we let $L(\Gamma)$ be the total length of all the loops in $\Gamma$. If $A=\emptyset $ we write ${\sigmab}=-$.

To each such triangulation $T$, configuration of loops $\Gamma$ on $T$, and $\bm{x}=(x_i\colon i\in A)$ where $x_i\in V_T$ with $i\in A$ are distinct, we associate the weight
\begin{align*}
	m_{\Lambda,\ell}^{{\sigmab},n,g,h}(T,\Gamma,\bm{x}) = g^{\#F_T - L(\Gamma)} h^{L(\Gamma)} n^{\#\Gamma}e^{-\Lambda \#V_T}\cdot \prod_{\emptyset \neq B\subseteq A} e^{\sigma_B N^\Gamma_B(\bm{x})}\;.
\end{align*}
We say that $(n,g,h,\Lambda,{\sigmab})$ is admissible if the total mass 
\begin{align*}
	|m_{\Lambda,\ell}^{{\sigmab},n,g,h}| := \sum_{(T,\Gamma,\bm{x})} m_{\Lambda,\ell}^{{\sigmab},n,g,h}(T,\Gamma,\bm{x})
\end{align*}
is finite for each $\ell\ge 1$, so that we can define a probability measure
\begin{align}
p_{\Lambda,\ell}^{{\sigmab},n,g,h}(T,\Gamma,\bm{x}) = \frac{m_{\Lambda,\ell}^{{\sigmab},n,g,h}(T,\Gamma,\bm{x})}{|m_{\Lambda,\ell}^{{\sigmab},n,g,h}|}\;. 
\label{eq:rpm-prob-measure}
\end{align}

When $A=\emptyset$ and $\Lambda=0$ this model is the classical loop $O(n)$ model on a planar map, which has been extensively studied using combinatorial and probabilistic techniques, see e.g.\ \cite{bbg-recursive-approach,budd-On-peeling,on-perimeter-cascade} and the references therein.
The model for non-zero ${\sigmab}$ is much less studied; however the asymptotics of the law of the nesting statistic $N^\Gamma_A(x)$ when $\#A=1$ have been investigated in \cite{on-nesting-statistics}. The paper \cite{ijs16} studies a loop model on a lattice where the weight of a configuration depends on the topological configuration of loops around a collection of marked points. We call the parameter $\Lambda$ the cosmological constant, as in, e.g.\ \cite{dkrv-lqg-sphere,hrv-disk}. 

The model defined above satisfies a very simple Markov property. Consider
\begin{align*}
	P=(T,\Gamma,\bm{x})\sim p_{\Lambda,\ell}^{{\sigmab},n,g,h}
\end{align*}
and define $\Gamma_*$ to be the outermost loops of $\Gamma$. For each loop $\eta\in \Gamma_*$ we let $P_\eta$ be the planar map enclosed by the loop $\eta$ (see Figure \ref{fig:onmodel}) and let ${\sigmab}|_{B_\eta}=(\sigma_C\colon C\subseteq B_\eta)$ where $B_\eta$ is the collection of $i\in A$ such that $x_i$ is surrounded by the loop $\eta$. We also write $\ell_\eta$ for the perimeter of $P_\eta$. Then conditionally on $(\ell_\eta\colon \eta\in \Gamma_*)$ and $(B_\eta\colon \eta\in\Gamma_*)$, the maps $(P_\eta\colon \eta\in \Gamma_*)$ are independent and have the laws
\begin{align*}
	P_\eta \sim p_{\Lambda,\ell_\eta}^{{\sigmab}|_{B_\eta},n,g,h} \quad\text{for all $\eta\in \Gamma_*$}\;.
\end{align*}
One can explore the planar map via a peeling process where at each step one peels through an edge, which means the following. Consider the part of the planar map that has been explored already, sample one of its boundary edges in a Markovian way, and consider the unexplored face on the other side of this edge. We then discover one of the following three scenarios: (i) We discover a face through which no loop passes, (ii) we discover (at once) an outermost loop (i.e.\ an element of $\Gamma_*$) together with all the faces intersecting it and surrounded by it, or 
(iii) we discover a face for which all three vertices lie on the boundary of the undiscovered part of the map. See \cite{budd-On-peeling} for a description in the setting of quadrangulations. One can write down the probabilities for each of these scenarios in terms of the partition function of the surfaces we get after doing the peeling step. The continuum results in this paper can be viewed as the continuum analog of this story: the case (ii) will correspond to the discovery of CLE loops by a CPI, the case (iii) to boundary intersections of the CPI and (i) results in the Lévy compensation in the $\kappa<4$ case (whereas in the $\kappa>4$ case, the scenario (i) disappears in the scaling limit).

\begin{figure}
	\centering
	\def\svgwidth{0.8\columnwidth}
	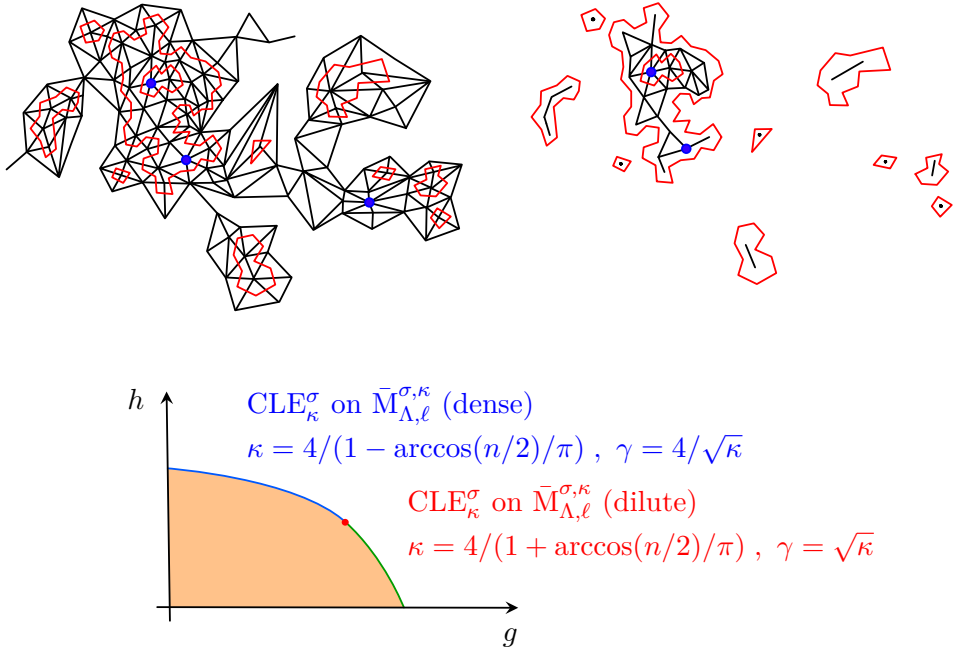
	\caption{\emph{Top.} An illustration of the spatial Markov property of the $O(n)$ planar map model weighted by nesting statistics: conditionally on the boundary lengths of the outermost loops and conditionally on the points surrounded by the outermost loops, the planar maps surrounded by these outermost loops are independent and have the law of planar maps with a loop $O(n)$ model and marked points. \emph{Bottom.} The orange area is the subcritical range where convergence to the continuum random tree is expected. The generic critical range where we conjecture convergence to the $\sqrt{8/3}\,$-LQG disk (also known as the Brownian disk) without loops is colored in green. Finally, the expected scaling limits of the blue range and the red point are indicated in the figure. All other parameters are inadmissible.}
	\label{fig:onmodel}
\end{figure}

\subsection{Planar maps without marked points}
\label{sec:planarnomarked}

In \cite{bbg-recursive-approach,budd-On-peeling} an asymptotic expression for the partition function of planar maps with the loop $O(n)$ model (i.e., the case $A=\emptyset$ and $\Lambda=0$) was established in the setting of quadrangulations when the loops were required to be rigid, i.e.\ a loop leaves and enters a quadrangle through opposite sides of a face. The scaling exponent they derived took a different form in four different phases: subcritical, generic critical, non-generic critical dense, and non-generic critical dilute. In our paper we focus only on what corresponds under scaling limit conjectures to the latter two phases, and we write $\mathcal{D}\subseteq (0,\infty)^2$ and $\{(g_*,h_*)\}\subseteq (0,\infty)^2$ for the respective parameter ranges. The analogous result of \cite{bbg-recursive-approach,budd-On-peeling} in the setting of triangulations would be the following for some constant $C(n,g,h)>0$ and with $\kappa=\kappa(n,g,h)$ and $\gamma=\gamma(n,g,h)$ as in Figure \ref{fig:onmodel},
\begin{align}
	\label{eq:pf-asymp}
	|m_{\ell,0}^{-,n,g,h}| \sim C(n,g,h)\ell^{-1-4/\kappa}\quad\text{as $\ell\to\infty$}\;.
\end{align}

It is expected that a random planar map sampled according to $(T_\ell,\Gamma_\ell)\sim p^{{\sigmab},n,g,h}_{0,\ell}$ converges in the scaling limit to a $\gamma$-LQG surface decorated by a $\CLE_\kappa$. To be slightly more precise, $T_\ell$ defines a metric measure space by equipping $T_\ell$ with its graph distance and giving each vertex unit mass, and $(T_\ell,\Gamma_\ell)$ defines a loop-decorated metric measure space. It is believed that if we rescale distances and mass appropriately as $\ell\to\infty$ then $(T_\ell,\Gamma_\ell)$ converges in the scaling limit to the $\gamma$-LQG disk decorated by an independent $\CLE_\kappa$. The convergence here is for some natural generalization of the Gromov-Hausdorff-Prokhorov topology \cite{adh-ghp} for loop-decorated metric measure spaces, e.g.\ the Gromov-Hausdorff-Prokhorov-Loop (GHPL) topology described in \cite{ghs-metric-peano}. We formulate the following precise conjecture.

\begin{conj}
	Let $n\in(0,2)$, $(g,h)\in\mathcal{D}\cup\{(g_*,h_*) \}$ and let $(T_\ell,\Gamma_\ell)\sim p^{-,n,g,h}_{0,\ell}$. Let $\kappa=\kappa(n,g,h)$ and $a(\kappa):=2\wedge (8/\kappa)$. Then there exist unique constants $d(\kappa)>0$ and $c,c'>0$ such that $(T_\ell,\Gamma_\ell)$ viewed as a loop-decorated metric measure space, where the metric is given by the graph distance multiplied by $c\ell^{-2/d(\kappa)}$ and the measure by the counting measure on the vertices multiplied by $c'\ell^{-a(\kappa)}$, converges in law to $\bar{\op{M}}_{0,1}^{-,\kappa}$ (normalized to be a probability measure) as $\ell\to\infty$ in the GHPL topology. This is a $\gamma$-LQG surface with $\gamma=\gamma(n,g,h)$.
	\label{conj0}
\end{conj}

In fact, the constant $d(\kappa)$ should be the almost sure Hausdorff dimension of the $\gamma$-LQG disk (see \cite{gm-uniqueness,dddf-lfpp} for further information on this). Conjectures of this kind can be found in much earlier literature, see e.g.\ the papers on the loop $O(n)$ model on planar maps referenced above. However, unlike in the cited papers, we construct a candidate limiting surface in the dense regime explicitly as a generalized (rather than regular) $\gamma$-LQG disk. The limiting surface should be a generalized (not regular) disk in the dense regime since, as explained above when describing the Markov property, the planar map has the same law as the map surrounded by an $O(n)$ loop conditioned on boundary length, and the latter map converges in the scaling limit to a generalized disk (i.e., a disk with pinch points corresponding to the pivotals of the $O(n)$ loop model in the dense regime). One can also verify that the scaling properties of the partition function \eqref{eq:pf-asymp} in the dilute and dense regimes are consistent with the respective continuum scaling properties of the regular and generalized disks (see Sections \ref{sss:regular-disk} and \ref{sss:gen-disk}, in particular \eqref{eq22}).

We conjectured above that in the dense regime, $(T_\ell,\Gamma_\ell)\sim p^{-,n,g,h}_{0,\ell}$ converges to a generalized (rather than regular) disk in the scaling limit. However, if we condition on the event that the boundary of $(T_\ell,\Gamma_\ell)$ is simple, we believe that we instead get a regular (not generalized) disk. In particular, the limiting behavior of the planar map depends on whether we require the boundary to be simple or not. In private communication, Linxiao Chen has provided evidence for this belief (in the case of quadrangulations with a rigid $O(n)$ decoration) by showing that the partition function for maps with simple boundary (in the dense regime) grows like $\ell^{-1-4/\gamma^2}$, namely similarly to the regular $\gamma$-LQG disk.

Note that the case $n=1$ and $g=h$ with $(g,h)\in \mathcal{D}$ describes percolation-decorated triangulations with monochromatic (vertex) boundary and percolation parameter $1/2$. Indeed, a percolation configuration with monochromatic boundary can be mapped to a collection of loops in the bulk of the planar map by sending each percolation configuration to the collection of interfaces between clusters of opposite colors. Uniformly sampled triangulations (which can also be seen as percolation-decorated triangulations with free boundary conditions and percolation parameter $1/2$ after forgetting the percolation decoration) do not have pinch points in the scaling limit (even when the boundary of the planar map is allowed to be non-simple). At first sight this may seem to contradict Conjecture \ref{conj0}, where we predict that we get a generalized disk (i.e., a disk with pinch points) in the case $n=1$, $g=h$, and $(g,h)\in\cD$. However, there is no contradiction since the boundary data are different in the two settings (monochromatic versus free). The crucial point to take away from this discussion is that conditioning  the vertex boundary to be monochromatic is expected to result in the appearance of pinch-points in the scaling limit while percolation-decorated random triangulations without this conditioning do not have pinch points in the scaling limit.

\subsection{Planar maps with marked points}

Motivated by Conjecture \ref{conj0} and the results stated in the introduction, we formulate the following conjecture. The topology of convergence is for a variant of the GHPL topology which accommodates the appearance of marked points. A key takeaway from the conjecture is the relationship between the $\sigma_i$'s and the $\alpha^\kappa_{\sigma_i}$'s, where the latter parameters describe the magnitude of the logarithmic singularities of the surface (see Definition \ref{def:main-params}).

\begin{conj}
	\label{conj1}
	Let $n\in(0,2)$ and let $(g,h)\in\mathcal{D}\cup\{(g_*,h_*)\}$. Moreover, let $\kappa=\kappa(n,g,h)$, $a(\kappa)$, $d(\kappa)$ and $c,c'>0$ be as in Conjecture \ref{conj0} and let $({\sigmab},\Lambda)\in\cA_\kappa$. Then $(n,g,h,\Lambda,{\sigmab})$ is admissible and we may sample
	\begin{align*}
		(T_\ell,\Gamma_\ell,\bm{x_\ell})\sim p^{{\sigmab},n,g,h}_{c'\Lambda \ell^{-a(\kappa)},\ell}\;.
	\end{align*}
	By rescaling the counting measure by $c'\ell^{-a(\kappa)}$ and the graph metric by $c\ell^{-2/d(\kappa)}$ we obtain a loop-decorated metric measure space which converges in law to $\bar{\op{M}}_{\Lambda,1}^{{\sigmab},\kappa}$ (normalized to be a probability measure) as $\ell\to\infty$ for the GHPL topology. Furthermore, if $({\sigmab},\Lambda)\in \cA_\kappa$ then for some function $C$, we have
	\begin{align}
		|m_{c'\Lambda \ell^{-a(\kappa)},\ell}^{{\sigmab},n,g,h}| \sim C(n,g,h,{\sigmab},\Lambda)\, \ell^{-1-4/\kappa+\sum_{i=1}^{n} 2\alpha_i/\sqrt{\kappa}}\qquad\text{as}\,\,\ell\to\infty
		\label{eq11}
	\end{align}
	and 
	\begin{align}
		\frac{C( n,g,h,{\sigmab},\Lambda )\cdot C( n,g,h,-,\Lambda )^{\#A-1}	}
		{\prod_{i\in A} C( n,g,h,\sigma_i,\Lambda)}
		=
		\frac{Z^{{\sigmab},\kappa}_{\Lambda,1}\cdot 
		(Z^{-,\kappa}_{\Lambda,1})^{\#A-1}}
		{\prod_{i\in A} Z^{\sigma_i,\kappa}_{\Lambda,1}}\;. 
		\label{eq9}
	\end{align}
\end{conj}

The measures $p^{{\sigmab},n,g,h}_{\Lambda,\ell}$ and $\bar{\op{M}}^{{\sigmab},\kappa}_{\Lambda,\ell}$ are (at least heuristically) obtained from $p^{-,n,g,h}_{\Lambda,\ell}$ and $\bar{\op{M}}^{-,\kappa}_{\Lambda,\ell}$ respectively, by performing a reweighting according to the number of loops surrounding the marked points on the surface. Given Conjecture \ref{conj0} this makes the scaling limit result of Conjecture \ref{conj1} plausible, and the asymptotics of the partition functions in \eqref{eq11} are reasonable to expect in analogy with the continuum.

In \eqref{eq9} notice that we only claim equality of certain \emph{ratios} of the constants $C( n,g,h,{\sigmab},\Lambda)$ rather than statements for the constants $C( n,g,h,{\sigmab},\Lambda)$ themselves. When defining the partition functions in both the discrete and the continuum there is some arbitrariness in the normalizing constants. Indeed, the transition probabilities of the peeling process in the discrete and the transition rates of the continuum process $\bf Y$ are all functions of
\begin{align*}
	\frac{|m^{{\sigmab},n,g,h}_{\Lambda,\ell}|\cdot |m^{-,n,g,h}_{\Lambda,\ell}|^{\#A-1}}{\prod_{i\in A} |m^{\sigma_i,n,g,h}_{\Lambda,\ell}|} \quad\text{and}\quad \frac{Z^{{\sigmab},\kappa}_{\Lambda,1}\cdot 
	(Z^{-,\kappa}_{\Lambda,1})^{\#A-1}}
	{\prod_{i\in A} Z^{\sigma_i,\kappa}_{\Lambda,1}}
\end{align*}
respectively, modulo the degree of freedom in the continuum corresponding the speed of the process. See Remark \ref{rmk:arb-mult-const}.

Conjecture \ref{conj1} concerns the case when $({\sigmab}, \Lambda) \in \cA_\kappa$. One can ask about properties of the discrete model for $({\sigmab}, \Lambda) \not\in \cA_\kappa$. We will not give a complete picture of this case, but remark that when $\sigma_i \gg 1$ then the model is not admissible since we can make $m^{{\sigmab},n,g,h}_{0,\ell}(T,\Gamma,\bm{x})$ arbitrarily large by considering a map $T$ with a long and thin ``tube'' with many loops wrapping around the tube, such that the tube separates $x_i$ and the boundary of $T$; by increasing the length of the tube and adding more loops (which increases $N^\Gamma_{\{i \}}(\bm{x})$), the weight $m^{{\sigmab},n,g,h}_{0,\ell}(T,\Gamma,\bm{x})$ increases. We have non-admissibility when $\sigma_{ij} \gg 1$ for a similar reason, namely it can be proved by letting $x_i$ and $x_j$ be close and considering a long and thin tube separating $x_i$ and $x_j$ from the boundary of the map $T$. Even if $({\sigmab},\Lambda) \not \in \cA_\kappa$ it might be possible to understand the asymptotic geometry of the planar map, e.g.\ by considering infinite measures instead of probability measures and/or allowing infinite planar maps.

\subsection{Other discrete models} 

We remark that $\gamma$-LQG disks with $\alpha$ singularities are also expected to arise as the scaling limit of planar maps with other statistical physics models than the loop $O(n)$ model. In the more general setting one considers tuples $(M,P,\bm{x})$ with $M=(V_M,E_M,F_M)$ a planar map, $P$ a statistical physics model on $M$, and $\bm{x}=(x_i\,:\,i\in A)$, $x_i\in V_M$. One associates a weight $w(M,P,\bm{x})$ to each such tuple. For certain choices of $w$ we expect that one still gets a surface defined as in \eqref{eq5} or \eqref{eq:gen-disk} in the scaling limit, but typically with the function $\Phi$ defined differently when $n>1$. One still needs $\Phi$ to satisfy the conformal covariance property of Lemma \ref{lem:phi-axioms} in order for the reweighting one does in the continuum to be intrinsic to the surface (i.e., independent of the chosen embedding), which in particular gives a unique choice for $\Phi$ (modulo multiplication by a constant) when $n=1$. One example of a way to define $w$ is to let it be the indicator function for some rare event for $P$ at the vertices $\bm{x}$ (e.g.\ an arm event). We believe that rigorous scaling limit results could be established in the special case of critical percolation on triangulations with one uniformly chosen vertex conditioned to have a certain configuration of arms to the boundary of the triangulation, building on \cite{bhs-site-perc,ghs-metric-peano}.

\section{Reweighted Lévy processes: generator and explicit computation}
\label{sec:fragmentation-main}

In this section we study certain spatially inhomogeneous jump Markov processes. The processes are constructed by applying reweightings to stable Lévy processes. We derive a formula for their generator (Theorem \ref{thm:weighted-levy}) and do an explicit computation (Proposition \ref{prop:power-theta}).

Determining the generator is  technically challenging and builds in particular on new techniques that exploit the scale invariance of the process to control its fluctuations close to its starting point, as well as more standard tools like Palm's formula for Poisson point processes. We expect that variants of our techniques also work for a large class of other Markov processes that are given by reweighted Lévy processes. The explicit calculation is done via a martingale argument involving the generator of a simple variant of the Markov process.

The Markov processes we consider would typically arise in settings similar to our paper, namely in the setting of growth-fragmentation processes where
the jumps represent splitting/fragmentation events,
the rates of the various jumps can be expressed in terms of the respective partition functions,
and we keep track of a fixed number of marked points.
In this section we consider jumps rates (or partition functions) of a more general form than the rest of the paper. We emphasize that the proofs in this section do not build on properties of LQG and CLE but are done via direct analysis of the Markov processes. 

Let us start by recalling the notion of a generator of a Markov process. Let $E$ be a metric space and let $X=(X_t\colon t\ge 0)$ be a time-homogeneous càdlàg Markov process on $E$ so $X$ is a random variable in the space $D([0,\infty),E)$ of càdlàg functions. Then for $f\colon E\to \R$ let
\begin{align}
	\label{eq:generator-def}
	\mathcal{G}f(x) = \lim_{t\downarrow 0} \frac{1}{t}\,\E_x\left( f(X_t)-f(x) \right)
\end{align}
provided the limit exists for all $x\in E$. We call the operator $\mathcal{G}$ the generator of $X$. Sometimes, uniform convergence in $x$ is required in \eqref{eq:generator-def} but here it is more convenient to work with the definition given above. Throughout this section, we will not determine the whole set of functions $f$ for which \eqref{eq:generator-def} exists but instead show the existence of and give the expression for $\mathcal{G}f$ for certain nice (namely, smooth and compactly supported) functions $f$.

The section is divided into two subsections where we consider the cases with and without, respectively, Lévy compensation. These correspond to the cases where the LQG disks and the CLE are simple ($\kappa\in(8/3,4)$) and non-simple ($\kappa'\in(4,8)$), respectively. At a heuristic level, the Lévy compensation of the boundary length process of a CPI occurs while the CPI is moving in between CLE loops (see also the discrete intuition right before the start of Section \ref{sec:planarnomarked}). It is not so surprising that a compensation is needed in the simple case and not in the non-simple case since in the former case the CLE loops do not touch and the trunk of the CPI ``spends more time'' in between the loops.

\subsection{With Lévy compensation}
\label{subsec:simple-fragmentation}

Let $\kappa\in (8/3,4)$ and $\beta\in [-1,1]$. We consider independent compensated Lévy processes $L$ and $R$ such that the generator of the process $(L,R)$ is given by
\begin{align}
	\begin{split}
	\mathcal{L}_{\kappa,\beta} f(l,r) &= \int_\R \left( f(l+h,r) - f(l,r) - h\,\frac{\partial f(l,r)}{\partial l}\right)\,\mu_L(dh) \\
	&\qquad + \int_\R \left( f(l,r+h) - f(l,r) -h\,\frac{\partial f(l,r)}{\partial r}\right)\,\mu_R(dh)
	\end{split}
	\label{eq:Lgen}
\end{align}
for $f\in C^\infty_c(\R^2)$ where $\mu_L$ and $\mu_R$ (the jump measures of $L$ and $R$) are given by
\begin{align*}
	\mu_L(dh)/dh &= \frac{-\cos(4\pi/\kappa)(1-\beta)/2}{h^{1+4/\kappa}}\,1_{(0,\infty)}(h) + \frac{1/2}{|h|^{1+4/\kappa}}\,1_{(-\infty,0)}(h)\;,\\
	\mu_R(dh)/dh &= \frac{-\cos(4\pi/\kappa)(1+\beta)/2}{h^{1+4/\kappa}}\,1_{(0,\infty)}(h) + \frac{1/2}{|h|^{1+4/\kappa}}\,1_{(-\infty,0)}(h)\;.
\end{align*}
See \cite[Theorem 19.10]{kallenberg}. More informally, $L$ is a jump process which transitions from $l$ to $l+h$ at rate  $-\cos(4\pi/\kappa)(1-\beta)/2\cdot h^{-1-4/\kappa}$ and from $l$ to $l-h$ at rate $1/2\cdot h^{-1-4/\kappa}$. Note that this is exactly a stable Lévy process of index $4/\kappa \in (1,2)$ and therefore we need compensation, which is reflected in the generator expression via the compensation term $-h\partial f(l,r)/\partial l$. The definition is of course chosen completely analogously in the case of $R$.

Note that the compensation term in \eqref{eq:Lgen} is so that $(L,R)$ satisfies stable scaling with exponent $4/\kappa$, i.e.\ $((L_{\lambda^{4/\kappa}t}/\lambda,R_{\lambda^{4/\kappa}t}/\lambda)\colon t\ge 0)$ has the same law as $(L,R)$ for any $\lambda>0$. From the definition, we see that $\mathcal{L}_{\kappa,\beta}f$ is a bounded function for $f\in C^\infty_c(\R^2)$. Also, as stated in \cite[Lemma 19.21]{kallenberg}, for $f\in C^\infty_c(\R^2)$ the process
\begin{align}
	t\mapsto f(L_t,R_t)-f(L_0,R_0) - \int_0^{t} \mathcal{L}_{\kappa,\beta}f(L_s,R_s)\,ds
	\label{eq:L-martingale}
\end{align}
is a martingale. We also let $X=L+R$.

In \cite{msw-simple} a certain reweighting is performed which we now explain. We remark that besides the existence of the reweighted process $(L',R')$ considered below, the results of Proposition \ref{prop:simple-msw-process} are new. For $l,r>0$ and $f\in C^\infty_c((0,\infty)^2)$ we let
\begin{align}
	w_\kappa(l,r) = (l+r)^{-1-4/\kappa}1(l,r>0)\quad \text{and }\quad \mathcal{G}_{\kappa,\beta} f(l,r)=\frac{1}{w_\kappa(l,r)}\,\mathcal{L}_{\kappa,\beta}(w_\kappa f)(l,r) \;.
	\label{eq:w}
\end{align}
By convention, we also write $\mathcal{G}_{\kappa,\beta}f(0,0)=0$. From the definition, one can check that $\mathcal{G}_{\kappa,\beta}f$ is a bounded function.

\begin{prop}
	\label{prop:simple-msw-process}
	For each $l,r>0$ there is a strong Markov càdlàg process $((L'_t,R'_t)\colon t\geq 0)$ starting at $(l,r)$, taking values in $(0,\infty)^2\cup\{(0,0)\}$ and absorbed at $(0,0)$ (i.e., after the process $(L',R')$ hits $(0,0)$ it stays at this point) such that
	\begin{align}
		\label{eq:compensated-weighting}
		\E_{(l,r)} \left(g\left(L'|_{[0,t]},R'|_{[0,t]}\right);t<\zeta'\right) = \E_{(l,r)} \left(\frac{w_\kappa(L_t,R_t)}{w_\kappa(l,r)} g\left(L|_{[0,t]},R|_{[0,t]}\right);t<\tau_0\right)
	\end{align}
	whenever $g\colon D([0,t])^2\to [0,\infty]$ is measurable on the space of real-valued càdlàg functions, $t\ge 0$, $\zeta'=\inf\{t\ge 0\colon L'_t=R'_t=0\}$ and $\tau_0 = \inf\{ t\ge 0\colon L_t\le 0 \text{ or } R_t \le 0\}$. Moreover, if we define $X'=L'+R'$, then $X'_t\to 0$ as $t\to \infty$ almost surely. Let $f\in C^\infty_c((0,\infty)^2)$. Then
    \begin{align}
		\label{eq:simple-msw-process}
		t\mapsto f(L'_t,R'_t) - f(L'_0,R'_0)-\int_0^{t} \mathcal{G}_{\kappa,\beta}f(L'_s,R'_s)\,ds
    \end{align}
	is a (càdlàg) martingale. Finally, $\mathcal{G}_{\kappa,\beta}$ is the generator of $(L',R')$ on $C^\infty_c((0,\infty)^2)$.
\end{prop}

\begin{proof}
	We begin by defining a Markov process $(L',R')$ with values in $(0,\infty)^2\cup \{(0,0)\}$ by specifying its transition kernel; for $t\ge 0$ and $l,r>0$ let
	\begin{align*}
		p^t_{(l,r)}(dl',dr') &= \left( 1- \E_{(l,r)}\left(\frac{w_\kappa(L_t,R_t)}{w_\kappa(l,r)};t<\tau_0\right)\right) \delta_{(0,0)}(dl',dr') \\
		&\quad + \frac{w_\kappa(l',r')}{w_\kappa(l,r)}\,\P_{(l,r)}(L_t\in dl', R_t\in dr', t<\tau_0)\;,\\
		p^t_{(0,0)} &= \delta_{(0,0)}\;.
	\end{align*}
	The semigroup property of these transition kernels is straightforward to verify and to see that $(p^t\colon t\ge 0)$ is a family of Markov kernels, it therefore suffices to show that the coefficient in front of $\delta_{(0,0)}$ above is non-negative. To this end, we consider smooth $h_n\in C^\infty_c(\R)$ such that $h_n\ge 0$, $h_n|_{(1/n,n)}=1$ and $h_n|_{(1/(2n),2n)^c}=0$.
	Let $\tau_m=\inf\{t\ge 0\colon L_t\notin (1/m,m)\text{ or } R_t\notin (1/m,m) \}$. Then we observe that
	\begin{align*}
		&\E_{(l,r)}\left(h_n(L_{t\wedge \tau_m})h_n(R_{t\wedge \tau_m})w_\kappa(L_{t\wedge \tau_m},R_{t\wedge \tau_m})\right) \\
		&\quad = w_\kappa(l,r) + \int_0^t \E_{(l,r)}\left(\mathcal{L}_{\kappa,\beta}((u,v)\mapsto h_n(u)h_n(v)w_\kappa(u,v))(L_s,R_s);s<\tau_m\right)\,ds
	\end{align*}
	for all $n$ with $l,r\in (1/n,n)$. The key is that by Lemma \ref{lem:simple-integrals-comp}, we have $\mathcal{L}_{\kappa,\beta} w_\kappa = 0$. Since $L_s,R_s\in [1/m,m]$ when $s<\tau_m$ it is thus easy to see that the integral term above goes to $0$ as $n\to \infty$ for $m$ fixed. The claim then follows by applying Fatou's lemma by first letting $n\to \infty$ and then taking $m\to \infty$.

	As mentioned above, $(p_t\colon t\ge 0)$ is a family of Markov kernels and hence a process $(L',R')$ with these transition kernels can be constructed. To see that $(L',R')$ has a càdlàg version, we will construct several supermartingales and argue that these processes having càdlàg versions implies the existence of a càdlàg version for $(L',R')$.

	Let $\rho_L=\P(L_1>0)$ and $\rho_R = \P(R_1>0)$ be the positivity parameters of the Lévy processes $L$ and $R$. We also let $\widehat{\rho}_L=1-\rho_L$ and $\widehat{\rho}_R=1-\rho_R$. By the definition of $L$ and $R$ we cannot have $\widehat{\rho}_L,\widehat{\rho}_R>0$. We show that
	\begin{align*}
		(X')^{1+4/\kappa}\quad\text{and}\quad (L')^{\wh{\rho}_L\cdot 4/\kappa}(X')^{1+4/\kappa}
	\end{align*}
	are supermartingales and so (by Doob's regularization theorem for martingales) there is a version of $(L',R')$ such that these two processes are càdlàg and we will switch to such a version. Working with this version, we see that $X'$ is càdlàg and hence $L'$ is also càdlàg and, finally, the process $R'=X'-L'$ is càdlàg.

	We next establish the supermartingale property. By \eqref{eq:compensated-weighting} applied with $g(u,v)=(u_t+v_t)^{1+4/\kappa}$ we have
	\begin{align*}
		\E_{(l,r)}\left( (X'_t)^{1+4/\kappa} \right) &= (l+r)^{1+4/\kappa}\,\P_{(l,r)}( t<\tau_0)\le (l+r)^{1+4/\kappa}\;.
	\end{align*}
	Then applying \eqref{eq:compensated-weighting} applied with $g(u,v)=(u_t)^{\wh\rho_L\cdot 4/\kappa}(u_t+v_t)^{1+4/\kappa}$,
	\begin{align*}
		\E_{(l,r)}\left( (L'_t)^{\wh{\rho}_L\cdot 4/\kappa}(X'_t)^{1+4/\kappa} \right) &\le (l+r)^{1+4/\kappa} \,\E_{(l,r)}\left((L_t)^{\wh{\rho}_L\cdot 4/\kappa}; \inf_{[0,t]}L>0\right) \\
		&= (l+r)^{1+4/\kappa} l^{\wh{\rho}_L\cdot 4/\kappa}\;.
	\end{align*}
	The final equality is established in \cite[Section 5.4]{kyprianou-levy-lamperti} (this martingale was first identified in \cite[Theorem 2]{silverstein-coharmonic-class}, see the remarks below \cite[(1.15)]{silverstein-coharmonic-class}). The Markov property of $(L',R')$ implies the supermartingale property. By supermartingale convergence, we have that $\lim_{t\to \infty} (X'_t)^{1+4/\kappa}$ exists almost surely and since
	\begin{align*}
		\E_{(l,r)}\left( (X'_t)^{1+4/\kappa} \right) &= (l+r)^{1+4/\kappa}\,\P_{(l,r)}( t<\tau_0 ) \to 0\quad\text{as $t\to \infty$}
	\end{align*}
	we in fact have $(X'_t)^{1+4/\kappa}\to 0$ as $t\to \infty$ a.s.\ and thus $X'_t\to 0$ as $t\to \infty$ a.s.

	The claim \eqref{eq:compensated-weighting} can now be checked using the definition of the transition kernels by first considering functions $g$ that only depend on finitely many marginals and then using a monotone class argument. To get the strong Markov property, note that for $f\in C_c^\infty((0,\infty)^2)$ the map
	\begin{align*}
		(l,r)&\mapsto \E_{(l,r)}(f(L'_t,R'_t))
		= \frac{1}{w_\kappa(l,r)}\,\E_{(l,r)}\left((w_\kappa f)(L_t,R_t);\inf_{[0,t]} L>0,\inf_{[0,t]}R>0\right)\\
		&= \E_{(0,0)}\left(\frac{(w_\kappa f)(l+t^{\kappa/4}L_1,r+t^{\kappa/4} R_1)}{w_\kappa(l,r)};\inf_{[0,1]} L>-lt^{-\kappa/4},\inf_{[0,1]}R>-rt^{-\kappa/4}\right)
	\end{align*}
	from $(0,\infty)^2\cup\{(0,0)\}$ to $\R$ is continuous and bounded. The strong Markov property can be deduced from this using the same reasoning that shows that Feller processes are strong Markov processes, see for instance \cite[Theorem 19.17]{kallenberg}.

	Finally, we need to show that the process defined in \eqref{eq:simple-msw-process} is a martingale. By the Markov property of $(L',R')$ it suffices to prove that
	\begin{align}
		\label{eq:result-simple-gen}
		\E_{(l,r)}(f(L'_t,R'_t))=f(l,r) + \int_0^t \E_{(l,r)}(\mathcal{G}_{\kappa,\beta}f(L_s,R_s))\,ds
	\end{align}
	for all $t\ge 0$. To see this, note that by \eqref{eq:L-martingale}
\begin{align*}
		t\mapsto(w_\kappa f)(L_t,R_t)-(w_\kappa f)(l,r) - \int_0^{t} \mathcal{L}_{\kappa,\beta}(w_\kappa f)(L_s,R_s)\,ds
\end{align*}
	is a martingale starting from $0$ and hence we can use optional stopping at the stopping time $t\wedge \tau_0$ to obtain
	\begin{align*}
		\E_{(l,r)}((w_\kappa f)(L_{t\wedge\tau_0},R_{t\wedge \tau_0}))= (w_\kappa f)(l,r) + \E_{(l,r)}\left( \int_0^{t} \mathcal{L}_{\kappa,\beta}(w_\kappa f)(L_s,R_s)1(s<\tau_0)\,ds \right)\;.
	\end{align*}
	The claim \eqref{eq:result-simple-gen} now follows by first interchanging the integral and the expectation and then using the definition of the process $(L',R')$. To obtain the generator, note that by right-continuity of $(L,R)$ and the fact that $\mathcal{L}_{\kappa,\beta}(w_\kappa f)$ is a bounded continuous function, we get
	\begin{align*}
		\frac{1}{t}\int_0^{t} \mathcal{L}_{\kappa,\beta}(w_\kappa f)(L_s,R_s)1(s<\tau_0)\,ds \to \mathcal{L}_{\kappa,\beta}(w_\kappa f)(l,r)\quad\text{as $t\downarrow 0$}
	\end{align*}
	provided $(L,R)$ starts from $(l,r)$. The generator result now follows from dominated convergence since $\mathcal{L}_{\kappa,\beta}(w_\kappa f)$ is a bounded function and hence the terms in the limit above are bounded by a deterministic constant.
\end{proof}

\begin{remark}
	In fact, the process $(L')^{\wh{\rho}_L\cdot 4/\kappa}(R')^{\wh{\rho}_R\cdot 4/\kappa}(X')^{1+4/\kappa}$ is a martingale where we define $\widehat{\rho}_L = \P(L_1<0)$ and $\widehat{\rho}_R = \P(R_1<0)$. Indeed,
	\begin{align*}
		&\E_{(l,r)}\left( (L'_t)^{\wh{\rho}_L\cdot 4/\kappa} (R'_t)^{\wh{\rho}_R\cdot 4/\kappa} (X'_t)^{1+4/\kappa} \right) \\
		&= (l+r)^{1+4/\kappa} \,\E_{(l,r)}\left((L_t)^{\wh{\rho}_L\cdot 4/\kappa}(R_t)^{\wh{\rho}_R\cdot 4/\kappa}; \inf_{[0,t]}L>0, \inf_{[0,t]}R>0\right) \\
		&= (l+r)^{1+4/\kappa} \,\E_{(l,r)}\left((L_t)^{\wh{\rho}_L\cdot 4/\kappa}; \inf_{[0,t]}L>0\right)  \E_{(l,r)}\left((R_t)^{\wh{\rho}_R\cdot 4/\kappa}; \inf_{[0,t]}R>0\right) \\
		&= (l+r)^{1+4/\kappa} r^{\wh{\rho}_R\cdot 4/\kappa} l^{\wh{\rho}_L\cdot 4/\kappa}
	\end{align*}
	and the Markov property of $(L',R')$ yields the martingale property. While we are not making use of this martingale anywhere in this paper, we believe it to be potentially useful in understanding the dependence of some functionals of $(L',R')$ on $\beta$.
\end{remark}

The main results of this section will be about reweighting the law of $(L',R')$. We begin by performing a reweighting of the Lévy process $(L,R)$, where we need the following result on Poisson point processes.
\begin{lemma}
	\label{lem:ppp-reweighting}
	Let $\xi$ be a Poisson random measure with intensity $\mu$ on a measure space $(E,\mathcal{E})$. Also, let $f\colon E\to \R$ be measurable such that
	\begin{align*}
		\int \mu(dx) |e^{f(x)}-1|<\infty \quad\text{and let}\quad  Z := \exp\left(\, \int \xi(dx) f(x) - \int \mu(dx) (e^{f(x)}-1) \right)\;.
	\end{align*}
	Then $\E(Z)=1$. Now reweight the law of $\xi$ by $Z$. Then under the reweighted measure, $\xi$ is a Poisson random measure with intensity measure $e^f\mu$.
\end{lemma}
\begin{proof}
	Let $\xi'$ be a Poisson random measure of intensity $e^f\mu$. For $g\colon E\to [0,\infty)$ measurable we have (the second and final equalities follow from Campbell's formula)
	\begin{align*}
		\E\left(Z\cdot e^{\,\int \xi(dx) g(x)} \right) &= \E\left( e^{\,\int \xi(dx) (f(x)+g(x))} \right) e^{\,-\int \mu(dx) (e^{f(x)}-1)} \\
		&= e^{\,\int \mu(dx) (e^{f(x)+g(x)}-1)}e^{\,-\int \mu(dx) (e^{f(x)}-1)}\\
		&= e^{\,\int e^{f(x)}\mu(dx)\, (e^{g(x)}-1)} = \E\left(e^{\,\int \xi'(dx) g(x)}\right)
	\end{align*}
	which implies the lemma by a monotone class argument.
\end{proof}

The following lemma is now an immediate consequence of Lemma \ref{lem:ppp-reweighting}. We would like to stress that the Lévy compensation term appearing in the generator of $(L,R)$ is unaffected by the reweighting below. We again refer to \cite[Theorem 19.10]{kallenberg} for information on generators of Lévy processes. The assumption $1-W^\emptyset(h)\lesssim h^2$ ensures the absolute continuity between $(L^\emptyset,R^\emptyset)$ and $(L,R)$ below.

\begin{lemma}
	\label{lem:twisted-jump-simple}
	Suppose  $W^\emptyset\colon [0,\infty)\to (0,1]$ is measurable with $1-W^\emptyset(h)\lesssim h^2$. Let $ (L^\emptyset,R^\emptyset)$ be two independent Lévy processes with generator defined for $l,r\in \R$ and $f\in C^\infty_c(\R^2)$ by
	\begin{align*}
		\mathcal{L}^\emptyset_{\kappa,\beta}f(l,r) &= \int_\R\left( W^\emptyset(|h|) (f(l+h,r) - f(l,r)) - h\,\frac{\partial f(l,r)}{\partial l}\right)\, \mu_L(dh)\\
		&\quad + \int_\R \left( W^\emptyset(|h|) (f(l,r+h) - f(l,r)) -h\,\frac{\partial f(l,r)}{\partial r}\right)\,\mu_R(dh) \;.
	\end{align*}
	Then whenever $f\in C^\infty_c(\R^2)$, the process
	\begin{align*}
		t\mapsto f(L^\emptyset_t,R^\emptyset_t)-f(L^\emptyset_0,R^\emptyset_0) - \int_0^{t} \mathcal{L}^\emptyset_{\kappa,\beta}f(L^\emptyset_s,R^\emptyset_s)\,ds
	\end{align*}
	is a martingale. We let $X^\emptyset = L^\emptyset + R^\emptyset$. Then for all $g\colon D([0,t])^2\to [0,\infty]$ we have
	\begin{align*}
		\E_{(l,r)}\left(g(L^\emptyset|_{[0,t]},R^\emptyset|_{[0,t]})\right) = \frac{ \E_{(l,r)}\left(\,\prod_{s<t} W^\emptyset(|\Delta X_s|)\cdot g(L|_{[0,t]},R|_{[0,t]})\right) }{\exp\left(t\int (\mu_L+\mu_R)(dh)(W^\emptyset(|h|)-1)\right) } \;.
	\end{align*}
	We define $\tau^\emptyset_0=\inf\{t\ge 0\colon L^\emptyset_t \le 0\text{ or } R^\emptyset_t\le 0\}$ and $\alpha = \int (\mu_L+\mu_R)(dh)(W^\emptyset(|h|)-1)$.
\end{lemma}
\begin{proof}
	The martingale property of the process defined in the statement follows from standard results on Lévy processes, see for instance \cite[Theorem 19.10 and Lemma 19.21]{kallenberg}, so it only remains to prove the second part. Note that by definition of the Lévy processes $(L,R)$ and $(L^\emptyset,R^\emptyset)$ we have
	\begin{align*}
		\left(t\mapsto \sum_{s\le t\colon |\Delta X_s|>1/n} \Delta L_s - t\mu_L(\R\setminus [-1/n,1/n])\right) &\to L \\
		\left(t\mapsto \sum_{s\le t\colon |\Delta X_s|>1/n} \Delta R_s - t\mu_R(\R\setminus [-1/n,1/n])\right) &\to R \\
 		\left(t\mapsto \sum_{s\le t\colon |\Delta X_s|>1/n} \Delta L^\emptyset_s - t\mu_L(\R\setminus [-1/n,1/n])\right) &\to L^\emptyset \\
		\left(t\mapsto \sum_{s\le t\colon |\Delta X_s|>1/n} \Delta R^\emptyset_s - t\mu_R(\R\setminus [-1/n,1/n])\right) &\to R^\emptyset
	\end{align*}
	uniformly on compact sets in probability as $n\to \infty$, where we note in particular that the processes $L^\emptyset$ and $R^\emptyset$ were defined via the generator in such a way that the compensation term is the same as for $L$ and $R$, respectively. We now apply Lemma \ref{lem:ppp-reweighting} (individually) to the Poisson point processes $\sum_{t\ge 0\colon \Delta L_t\neq 0} \delta_{(t,\Delta L_t)}$ and $\sum_{t\ge 0\colon \Delta R_t\neq 0} \delta_{(t,\Delta R_t)}$.
\end{proof}

Throughout this paper, we will make frequent use of Palm's formula for Poisson point processes. The theory of Palm distributions for general point processes appears for instance as a special case in \cite[Chapter 6]{kallenberg-measures} but let us state the classical case for Poisson point processes (without proof) here for future reference.

\begin{lemma}
	\label{lem:palm-abstract}
	Let $\xi$ be a Poisson random measure with intensity $\mu$ on a measurable space $(E,\mathcal{E})$. Moreover, consider a measurable function $f\colon E\times \mathcal{M}(E)\to [0,\infty]$ where $\mathcal{M}(E)$ is the collection of measures on $E$ endowed with the $\sigma$-algebra generated by the evaluation maps $\nu\mapsto \nu(A)$ for $A\in \mathcal{E}$. Then
	\begin{align*}
		\E\left(\int f(x,\xi)\,\xi(dx)\right) = \int \E(f(x,\xi+\delta_x))\,\mu(dx).
	\end{align*}
	The result extends to $f$ taking values in $\R$ provided either side of the equality above is finite when $f$ is replaced by $|f|$.
\end{lemma}

Let us return to the study of the processes $(L,R)$, $(L',R')$ and $(L^\emptyset,R^\emptyset)$. A key input is the following Palm's formula type result that allows us to replace the expected sum of the jumps of a process by an integral over time and over space with respect to some intensity jump measure.

This lemma will be used in the following context: If we consider a random surface and, using a peeling-type process, cut it into smaller surfaces with boundary lengths given by the jumps of a Markov process, then we can express the weight of the whole surface in terms of an expectation of a sum over the jumps of the Markov process. The lemma allows us to rewrite this expectation as an expectation of an integral with respect to some intensity measure of the jumps.

\begin{lemma}
	\label{lem:msw-palm-result}
	Consider the setting of Lemma \ref{lem:twisted-jump-simple}.
	Suppose that $f\colon D([0,\infty))^2\times [0,\infty)\times \R\to [0,\infty]$ is measurable with $f(\cdot,\cdot,\cdot,0)=0$. Then
	\begin{align*}
		&\E_{(l,r)}\left(\sum_{s<t} f(L^\emptyset,R^\emptyset,s,\Delta X^\emptyset_s);t<\tau^\emptyset_0\right) \\
		&=\E_{(l,r)}\left( \int_0^{t\wedge \tau^\emptyset_0}\,ds\int W^\emptyset(|h|)\mu_L(dh)\, a_s^t(L^\emptyset|_{[0,s)},R^\emptyset|_{[0,s)},L_s^\emptyset+h,R^\emptyset_s) \right)\\
		&\quad + \E_{(l,r)}\left( \int_0^{t\wedge \tau^\emptyset_0}\,ds\int W^\emptyset(|h|)\mu_R(dh)\, a_s^t(L^\emptyset|_{[0,s)},R^\emptyset|_{[0,s)},L_s^\emptyset,R^\emptyset_s+h) \right)
	\end{align*}
	where $a_s^t(u_L,u_R,l',r')= \E_{(l',r')}( f(u_L\oplus L^\emptyset,u_R\oplus R^\emptyset,s,h);t-s<\tau^\emptyset_0 )\,1(l',r'>0)$ and where $\oplus$ denotes the concatenation of two functions. Moreover
	\begin{align*}
		&\E_{(l,r)}\left(\sum_{
			s<\zeta'
		}\prod_{s'<s} W^\emptyset(|\Delta X'_{s'}|) \cdot g(\Delta X'_s,L'_{\cdot+s},R'_{\cdot +s})\right) \\
		&= \E_{(l,r)}\left(\int_0^{\zeta'} \prod_{s'<s} W^\emptyset(|\Delta X'_{s'}|)\int \frac{w_\kappa(L_s'+h,R_s')\,\mu_L(dh)}{w_\kappa(L_s',R_s')}\, \E_{(L'_s+h,R'_s)}\left( g(h,L',R') \right) \,ds\right) \\
		&\quad + \E_{(l,r)}\left(\int_0^{\zeta'} \prod_{s'<s} W^\emptyset(|\Delta X'_{s'}|)\int \frac{w_\kappa(L_s',R_s'+h)\,\mu_R(dh)}{w_\kappa(L_s',R_s')}\, \E_{(L'_s,R'_s+h)}\left( g(h,L',R') \right) \,ds\right)
	\end{align*}
	whenever $g\colon \R\times D([0,\infty))^2\to [0,\infty]$ is measurable with $g(0,\cdot,\cdot)=0$.
\end{lemma}

\begin{proof}
	Since the jumps of $(L^\emptyset,R^\emptyset)$ form a Poisson point process with an explicit intensity measure, by applying Palm's formula (see Lemma \ref{lem:palm-abstract}), we get
	\begin{align*}
		&\E_{(l,r)}\left(\sum_{s<t} f(L^\emptyset,R^\emptyset,s,\Delta X^\emptyset_s);t<\tau^\emptyset_0\right) \\
		&\quad = \int_0^t ds\int W^\emptyset(|h|)\mu_L(dh)\, \E_{(l,r)}\left(f(L^\emptyset+h\,1_{[s,\infty)},R^\emptyset,s,h);B^L_{h,t,s}\right) \\
		&\qquad + \int_0^t ds\int W^\emptyset(|h|)\mu_R(dh)\, \E_{(l,r)}\left(f(L^\emptyset,R^\emptyset+h\,1_{[s,\infty)},s,h);B^R_{h,t,s}\right)
	\end{align*}
	where
	\begin{align*}
		B^L_{h,t,s}=\{\inf\nolimits_{\,[0,t]} L^\emptyset, \inf\nolimits_{\,[0,t]} R^\emptyset, \inf\nolimits_{\,[t,s]} L^\emptyset+h, \inf\nolimits_{\,[t,s]} R^\emptyset>0\} \;,\\
		B^R_{h,t,s}=\{\inf\nolimits_{\,[0,t]} L^\emptyset, \inf\nolimits_{\,[0,t]} R^\emptyset, \inf\nolimits_{\,[t,s]} L^\emptyset, \inf\nolimits_{\,[t,s]} R^\emptyset+h>0\}\;.
	\end{align*}
	By conditioning $(L^\emptyset,R^\emptyset)$ on its values up to time $s$ and using the Markov property of $(L^\emptyset,R^\emptyset)$ we get
	\begin{align*}
		&\E_{(l,r)}\left(f(L^\emptyset+h\,1_{[s,\infty)},R^\emptyset,s,h);B^L_{h,t,s}\right) = \E_{(l,r)}\left( a_s^t(L^\emptyset_{[0,s)},R^\emptyset_{[0,s)},L^\emptyset_s+h,R^\emptyset_s) ;s<\tau^\emptyset_s \right)\;
	\end{align*}
	We analogously reexpress the second term to deduce the result.
	For the second part, we begin by observing that monotone convergence implies
	\begin{align*}
		&\E_{(l,r)}\left(\sum_{s<\zeta'}\prod_{s'<s} W^\emptyset(|\Delta X'_{s'}|) \cdot g(\Delta X'_s,L'_{\cdot+s},R'_{\cdot +s})\right) \\
		&= \lim_{n\to \infty}\sum_{k\ge 0} \E_{(l,r)}\left( \sum_{s<k2^{-n}}\prod_{s'<s} W^\emptyset(|\Delta X'_{s'}|) \cdot g(\Delta X'_s,L'_{\cdot+s},R'_{\cdot +s})1(\zeta'\in (k2^{-n},(k+1)2^{-n}] ) \right)\;.
	\end{align*}
	By the Markov property of $(L',R')$ at time $k2^{-n}$ and the construction of $(L',R')$ we get
	\begin{align*}
		&\E_{(l,r)}\left( \sum_{s<k2^{-n}}\prod_{s'<s} W^\emptyset(|\Delta X'_{s'}|) \cdot g(\Delta X'_s,L'_{\cdot+s},R'_{\cdot +s})1(\zeta'\in (k2^{-n},(k+1)2^{-n}] ) \right) \\
		&= \E_{(l,r)}\Biggl(\, \sum_{s<k2^{-n}} \left.\E_{(L'_{k2^{-n}},R'_{k2^{-n}})}\left(g(h,u_L\oplus L',u_R\oplus R'); \zeta'\le 2^{-n}\right)\right|_{\substack{h=\Delta X'_s\\u_L=L'_{s+\cdot}|_{[0,k2^{-n}]}\\u_R=R'_{s+\cdot}|_{[0,k2^{-n}]}}}\\
		&\qquad\qquad\qquad\qquad \cdot \prod_{s'<s} W^\emptyset(|\Delta X'_{s'}|) \cdot 1(k2^{-n}<\zeta')  \,\Biggr)\;.
	\end{align*}
	By the construction of $(L',R')$ and the definition of $(L^\emptyset, R^\emptyset)$ the expression above equals
	\begin{align*}
		&e^{\alpha k2^{-n}} \, \E_{(l,r)}\Biggl(\, \sum_{s<k2^{-n}} \, \left.\E_{(L^\emptyset_{k2^{-n}},R^\emptyset_{k2^{-n}})}\left(g(h,u_L\oplus L',u_R\oplus R'); \zeta'\le 2^{-n}\right)\right|_{\substack{h=\Delta X^\emptyset_s\\u_L=L^\emptyset_{s+\cdot}|_{[0,k2^{-n}]}\\u_R=R^\emptyset_{s+\cdot}|_{[0,k2^{-n}]}}}\\
		&\qquad\qquad\qquad\qquad\cdot  \prod_{s'\in [s,k2^{-n})} W^\emptyset(|\Delta X^\emptyset_{s'}|)^{-1} \,\frac{w_\kappa(L^\emptyset_{k2^{-n}},R^\emptyset_{k2^{-n}})}{w_\kappa(l,r)}\, 1(k2^{-n}<\tau^\emptyset_0) \Biggr)
	\end{align*}
	We now apply the first part of the lemma, rewrite the expression and take the limit as $n\to \infty$ (again using monotone convergence) to deduce the result.
\end{proof}

In the remainder of this section, we consider a finite subset $A\subseteq \N$ and smooth functions $W^B\colon (0,\infty)\to (0,\infty)$ for each $B\subseteq A$ such that
\begin{gather}
	\label{eq:conditions-simple}
	\begin{split}
		W^\emptyset\text{ is non-increasing}\;,\quad W^\emptyset(h)\le 1\;,\quad 1-W^\emptyset(h)\lesssim h^2 \quad\text{for $h>0$}\;,\\
		\text{and}\quad W^B(x)\lesssim W^B(y)\quad\text{for all $B\subseteq A$ and whenever $0<y\le x\le 2y$}\;.
	\end{split}
\end{gather}
By convention we will write $W^\emptyset(0)=1$ and $W^B(0)=0$ if $B\neq \emptyset$. We also consider constants $\sigma_B\in \R$ for $B\subseteq A$ such that $\sigma_\emptyset = 0$.

For $u\colon [0,\infty)\to \R$ càdlàg we let $\Pi(u,B)$ be the set of non-increasing (in the sense of inclusion) functions $Q\colon [0,\infty)\to \mathcal{P}(A)$ such that $Q_0=B$, $Q_t=\emptyset$ for $t$ sufficiently large and $\Delta Q_t := Q_{t-}\setminus Q_t = \emptyset$ whenever $\Delta u_t := u_t-u_{t-}= 0$.

\begin{thm}
	\label{thm:weighted-levy}
	For each $B\subseteq A$ and $l,r>0$, we now define a process $\bm{Y}^W=(L^W,R^W,P^W)$ taking values in $(0,\infty)^2\times \mathcal{P}(A)$, starting at $(l,r,B)$ and absorbed at $(0,0,\emptyset)$ by
	\begin{equation}
		\begin{split}
			&\E_{(l,r,B)}\left(g(L^W,R^W,P^W)\right)\\
			&=\frac{1}{W^B(l+r)}\,\E_{(l,r)}\left( \,\sum_{Q\in \Pi(X',B)} g(L',R',Q)\prod_{s<\zeta'\colon \Delta X'_s\neq 0} e^{\sigma_{Q_{s-}\setminus Q_s}1(\Delta X'_s>0)}\, W^{Q_{s-}\setminus Q_s}(\Delta X'_s) \right)\;.
		\end{split}
		\label{eq:weighted-levy}
	\end{equation}
	For this to be well-defined, we also assume that the right-hand side of the equation is $1$ for $g=1$ and for all $B\subseteq A$ and $l,r>0$. We also let $\zeta^W = \inf\{t\ge 0\colon \bm{Y}_t^W = (0,0,\emptyset)\}$. Then $\bm{Y}^W$ is a strong Markov process and the following is well-defined
	\begin{align*}
		&\mathcal{G}_{\kappa,\beta}^W f(l,r,B) \\
		&\quad:= f(l,r,B)\int (\mu_L+\mu_R)(dh) (W^\emptyset(|h|)-1) \\
		&\qquad + \frac{1}{w_\kappa(l,r)W^B(l+r)}\, \mathcal{L}^\emptyset_{\kappa,\beta}((u,v)\mapsto w_\kappa(u,v) W^B(u+v) f(u,v,B))(l,r) \\
		&\qquad + \sum_{C\subsetneq B}\int \mu_L(dh)\, \frac{W^{B\setminus C}(|h|)w_\kappa(l+h,r)W^C(l+r+h)}{w_\kappa(l,r)W^B(l+r)}\,e^{\sigma_{B\setminus C}1(h>0)}\,f(l+h,r,C) \\
		&\qquad + \sum_{C\subsetneq B}\int \mu_R(dh)\, \frac{W^{B\setminus C}(|h|)w_\kappa(l,r+h)W^C(l+r+h)}{w_\kappa(l,r)W^B(l+r)}\,\,e^{\sigma_{B\setminus C}1(h>0)}\,f(l,r+h,C)
	\end{align*}
	for $f\colon (0,\infty)^2\times \mathcal{P}(A)\to \R$ with $f(\cdot,B)\in C^\infty_c((0,\infty)^2)$ for all $B\subseteq A$. Furthermore,
	\begin{align*}
		\lim_{t\downarrow 0}\frac{1}{t}\, \E_{(l,r,B)}\left(f(\bm{Y}^W_t) - f(l,r,B) \right)  = \mathcal{G}_{\kappa,\beta}^W f(l,r,B)
	\end{align*}
	i.e.\ $\mathcal{G}_{\kappa,\beta}^W$ is the generator of $\bm{Y}^W$ on the set of functions $f$ considered above.
\end{thm}

The proof is a quite technical and we will need some preparatory lemmas. The first one given below is the simple ingredient, namely a Palm's formula result.

\begin{lemma}
	\label{lem:weighted-levy-palmstep}
	Let $\bm{Y}^W$ be as in Theorem \ref{thm:weighted-levy}. Consider $f\colon (0,\infty)^2\times \mathcal{P}(A)\to [0,\infty]$ measurable and fix $M'\in [0,\infty]$. For $B\subseteq A$, $l,r>0$ and $t\ge 0$ define
	\begin{align*}
		E^{M'}_t(u,v)&:=\{u([0,t]),v([0,t])\subseteq (1/M',M') \}\;,\\
		a^t_{(l,r)}(B,C,C_0,\dots,C_{k-1})&:=\E_{(l,r)}\Biggl(\, \prod_{s<t} W^\emptyset(|\Delta X'_s|)\sum_{t_0<\cdots<t_{k-1}<\,t} \prod_{i=0}^{k-1} e^{\sigma_{C_i}1(\Delta X'_{t_i}>0)}\frac{W^{C_i}(|\Delta X'_{t_i}|)}{W^\emptyset(|\Delta X'_{t_i}|)}\\
		&\qquad\qquad\qquad\cdot W^C(X'_t)f(L'_t,R'_t,C);E^{M'}_t(L',R')\Biggr)\;,\\
		\mu^U_{(l,r)}(dx,dy) &:= e^{\sigma_U 1(x+y>0)}\,\frac{W^U(|x+y|)}{W^\emptyset(|x+y|)}\,\left(\mu_L(dx)\delta_0(dy)  +\delta_0(dx)\mu_R(dy)\right)\;,\\
		\alpha &= \int (\mu_L+\mu_R)(dh)(W^\emptyset(|h|)-1)
	\end{align*}
	when $u,v\in D([0,\infty))$, $C\subseteq B$ and $B\setminus C = C_0\cup \cdots \cup C_{k-1}$ is a union of disjoint sets. With these definitions we can write
	\begin{align*}
		\E_{(l,r,B)}(f(\bm{Y}^W_t);E^{M'}_t(L^W,R^W)) &=\frac{1}{W^B(l+r)}\,\sum a^t_{(l,r)}(B,C,C_0,\dots,C_{k-1})
	\end{align*}
	where the sum ranges over all $C\subseteq B$ and all (ordered) partitions $B\setminus C=C_0\cup\cdots \cup C_{k-1}$ into disjoint non-empty sets. Moreover $a^t_{(l,r)}(B,C,C_0,\dots,C_{k-1})$ equals
	\begin{align}
		\label{eq:iterated-palm-simple}
		\begin{split}
		&\frac{e^{\alpha t}}{w_\kappa(l,r)}\int_{0<t_0<\cdots<t_{k-1}<t}dt_0 \cdots dt_{k-1}\,\\
		&\quad  \E_{(l,r)}\Biggl( \int_{(\R^2)^k}\prod_{i=1}^{k-1} \mu^{C_i}_{(L^\emptyset_{t_i}+x_0+\cdots+x_{i-1},R^\emptyset_{t_i}+y_0+\cdots+y_{i-1})}(dx_i,dy_i)\\
		&\qquad\qquad\quad \cdot w_\kappa(L^\emptyset_{t}+x_0+\cdots+x_{k-1},R^\emptyset_{t}+y_0+\cdots+y_{k-1})\\
		&\qquad\qquad\quad \cdot f(L^\emptyset_{t}+x_0+\cdots+x_{k-1},R^\emptyset_{t}+y_0+\cdots+y_{k-1},C)\\
		&\qquad\qquad\quad \cdot W^C(L^\emptyset_{t}+x_0+\cdots+x_{k-1}+R^\emptyset_{t}+y_0+\cdots+y_{k-1})\\
		&\qquad\qquad\quad \cdot 1\left(E^{M'}_t\left(L^\emptyset + \sum_{i=0}^{k-1} x_i 1_{[t_i,\infty)}, R^\emptyset + \sum_{i=0}^{k-1} y_i 1_{[t_i,\infty)}\right)\right) \Biggr)\;.
		\end{split}
	\end{align}
\end{lemma}

\begin{proof}
	The first claim is easy to see from the definitions and the second claim follows by applying Lemma \ref{lem:twisted-jump-simple} and then the first part of Lemma \ref{lem:msw-palm-result} several times.
\end{proof}

\begin{lemma}
	\label{lem:smp-fragmentation}
	Let $\bm{Y}^W$ be as in Theorem \ref{thm:weighted-levy}. Then $\bm{Y}^W$ satisfies the strong Markov property.
\end{lemma}

\begin{proof}
	The weak Markov property of $\bm{Y}^W$ is a simple consequence of the definition and the weak Markov property of $(L',R')$. Consider now $f\colon (0,\infty)^2\times \mathcal{P}(A)\to \R$ with $f(\cdot,B)\in C^\infty_c((0,\infty)^2)$ for all $B\subseteq A$. Using the first part of Lemma \ref{lem:weighted-levy-palmstep} with $M'=\infty$, the definition of $(L',R')$ and stable scaling of $(L,R)$ it is not difficult to see  that
	\begin{align*}
		(l,r)\mapsto \E_{(l,r,B)}(f(\bm{Y}^W_t))
	\end{align*}
	is continuous from $(0,\infty)^2$ to $\R$ for each $B\subseteq A$ (here one needs to make use of the continuity of $W^C$ and the fact that $f(\cdot,\cdot,C)$ has compact support for each $C\subseteq B$). Using standard ideas from the theory of Feller processes (see e.g.\ \cite[Theorem 19.17]{kallenberg}) it is then classical to deduce that the process $\bm{Y}^W$ has the strong Markov property.
\end{proof}

A key insight which allows analytical tools to be applied in the proof of Theorem \ref{thm:weighted-levy} without having detailed information about the functions $W^B$ near $0$ and $\infty$ for $B\subseteq A$ is established with the following lemma. The main idea of the argument applies to general Markov chains and requires only some continuity properties of the process in the starting point; in our case we get these via stable scaling and the properties \eqref{eq:conditions-simple} of the functions $W^B$.

\begin{lemma}
	\label{lem:generator-truncation}
	Let $\bm{Y}^W$ be as in Theorem \ref{thm:weighted-levy}. Fix $l,r>0$, $B\subseteq A$, $\epsilon>0$ and $M\in (0,\infty)$. Then there exists $\infty>M'>M$ such that
	\begin{align}
		\P_{(l,r,B)}( L^W([0,t])\not\subseteq (1/M',M') , (L^W,R^W)_t\in (1/M,M)^2) \le \epsilon\cdot t\quad\text{for all $t\ge 0$}\;.
		\label{eq:analytical}
	\end{align}
	By symmetry the analogous expression with $L^W$ and $R^W$ interchanged also holds.
\end{lemma}

\begin{proof}
	Fix $\epsilon_0\in(0,1)$ sufficiently small and $M_0$ sufficiently large such that $l/(1+\epsilon_0),l,r,r(1+\epsilon_0) \in (1/M_0,M_0)$ and
	\begin{align*}
		[1/M,M]\subseteq \big((1+\epsilon_0)^{-1}/M_0,r(1-(1+\epsilon_0)^{-2}+M_0)\big)\;.
	\end{align*}
	Let $l_-=l/(1+\epsilon_0)$ and $r_+=r(1+\epsilon_0)$ and fix some $l_+\in (l_-,l)$ and $r_-\in (r,r_+)$. The significance of these conditions will become apparent towards the end of the proof.

	Note that for any $n\ge 1$, $c\in (0,1)$ and any $M'>M_0$ we have the following lower bound (where throughout, we set $M''=2M'$)
	\begin{equation}
	\begin{split}
		&\P_{(l_+,r_-,B)}\left( L^W([0,1])\not\subseteq (1/M'',M'') , 1<\zeta^W \right)\ge \sum_{i=0}^{n-1} \P_{(l_+,r_-,B)}(A_i)\quad\text{where} \\
		&\quad A_i= \{ L^W([0,ci/n])\subseteq (1/M'',M''),L^W([ci/n,c(i+1)/n])\not\subseteq (1/M'',M''),\\
		&\quad\qquad\qquad (L^W,R^W)_{c(i+1)/n}\in (1/M_0,M_0)^2,1<\zeta^W \}\;.
	\end{split}
	\label{eq:split-event}
    \end{equation}
	The idea of the proof is to show that each of the probabilities $\P_{(l_+,r_-,B)}(A_i)$ on the right-hand side of the displayed inequality are (modulo multiplication by a constant) at least as large as the left-hand side of \eqref{eq:analytical} with $t=c/n$. We will prove this by using the Markov property of $(L',R')$ and that the properties of $(L',R')$ do not depend too strongly on its starting point. From this we get that the left-hand side of \eqref{eq:analytical} is (modulo multiplication by a constant) bounded above by the left-hand side of \eqref{eq:split-event} times $t$. Observing that the left-hand side of \eqref{eq:split-event} converges to 0 as $M''\to\infty$ we conclude the proof. In order to carry out this strategy, we first observe that we have the inclusion $A_i'\subseteq A_i$ where
	\begin{align*}
		&A_i'= \{  L^W([0,ci/n])\subseteq (l_-,l), R^W([0,ci/n])\subseteq (r,r_+), P^W_{ci/n}=B,\\ &\qquad\quad L^W([ci/n,c(i+1)/n])\not\subseteq (1/M'',M''), (L^W,R^W)_{c(i+1)/n}\in (1/M_0,M_0)^2,1<\zeta^W \}\;.
	\end{align*}
	The Markov property of $\bm{Y}^W$ as stated in Lemma \ref{lem:smp-fragmentation} at times $i/n$ and $(i+1)/n$ implies that
	\begin{equation}
		\begin{split}
		\P_{(l_+,r_-,B)}(A_i')&\ge \P_{(l_+,r_-,B)} ( L^W([0,ci/n])\subseteq (l_-,l), R^W([0,ci/n])\subseteq (r,r_+), P^W_{ci/n}=B ) \\
		&\quad \cdot \inf_{\substack{l'\in (l_-,l)\\r'\in (r,r_+)}} \P_{(l',r',B)}( L^W([0,c/n])\not\subseteq (1/M'',M''),(L^W,R^W)_{c/n}\in (1/M_0,M_0)^2 )\\
		&\quad \cdot \inf_{\substack{l',r'\in (1/M_0,M_0)\\ C\subseteq B}} \P_{(l',r',C)}(1-c(i+1)/n<\zeta^W)\;.
	\end{split}
	\label{eq2}
	\end{equation}
	From the definitions we can see that the first and third terms in the product above are lower bounded by a positive constant which does not depend on $M''$, $c$, $i$ or $n$. Combining this with \eqref{eq:split-event}, \eqref{eq2}, and the fact that $A'_i\subseteq A_i$ yields that
	\begin{align}
		\label{eq:infimum-comparison-step}
		\begin{split}
		&\frac{1}{n}\,\P_{(l_+,r_-,B)}\left( L^W([0,1])\not\subseteq (1/M'',M'') , 1<\zeta^W \right) \\
		&\qquad\gtrsim \inf_{\substack{l'\in (l_-,l)\\r'\in (r,r_+)}} \P_{(l',r',B)}( L^W([0,c/n])\not\subseteq (1/M'',M''),(L^W,R^W)_{c/n}\in (1/M_0,M_0)^2 )
		\end{split}
	\end{align}
	for all $n\ge 1$ and all $M'>M_0$. It remains to understand the dependence of the terms in the infimum on the right-hand side on $(l',r')$; more precisely we must show that (modulo multiplication by a constant) it is lower bounded by the considered probability with $(l',r')=(l,r)$ and the considered intervals slightly changed.

	For $l'\in (l_-,l]$, $r'\in [r,r_+)$ and $x'=l'+r'$ by Lemma \ref{lem:weighted-levy-palmstep} and by unpacking the definitions we get
	\begin{align}
		\label{eq:huge-upper-bound}
		\begin{split}
		&\P_{(l',r',B)}( L^W([0,c/n])\not\subseteq (1/M'',M''),(L^W,R^W)_{c/n}\in (1/M_0,M_0)^2 ) \\
		&= \frac{1}{W^B(x')(x')^{-1-4/\kappa}}\\
		&\quad \cdot\sum\E_{(0,0)}\Biggl( W^C(x'+X_{c/n})(x'+X_{c/n})^{-1-4/\kappa}\,1(l'+L_{c/n},r'+R_{c/n}\in (1/M_0,M_0))\\
		&\qquad \cdot \sum_{0<t_0<\cdots<t_{k-1}<\,c/n} \prod_{i=0}^{k-1} e^{\sigma_{C_i}1(\Delta X_{t_i}>0)} W^{C_i}(|\Delta X_{t_i}|) \prod_{s\in (0,c/n)\setminus\{t_0,\dots,t_{k-1}\}} W^\emptyset(|\Delta X_s|) \\
		&\qquad \cdot 1\left(l'+\inf_{[0,c/n]}L>0,r'+\inf_{[0,c/n]}R>0,l'+L([0,c/n])\not\subseteq (1/M',M') \right) \Biggr)
		\end{split}
	\end{align}
	where the first sum ranges over all $C\subseteq B$ and all (ordered) partitions $B\setminus C=C_0\cup\cdots \cup C_{k-1}$ into disjoint non-empty sets. We now analyze the dependence on $(l',r')$ for each of the product terms in the expectation The integrand in the expectation is not monotonic in $l'$ but is (up to multiplicative factors) monotonic in $r'$. To deal with the $l'$ dependence, we observe that, by stable scaling when $L_0=R_0=0$,
	\begin{align}
		\label{eq:stable-rescale-law}
		(L,R) \stackrel{d}{=} (l'/l)\cdot ((L,R)_{(l/l')^{4/\kappa}t}\colon t\ge 0) =: (\widetilde{L},\widetilde{R})\;.
	\end{align}
	Also let $\widetilde{X}=\widetilde{L}+\widetilde{R}$ and write $x=l+r$. Since $(L,R)\stackrel{d}{=} (\wt L,\wt R)$, the right-hand side of \eqref{eq:huge-upper-bound} is unchanged if we replace $(L,R)$ by $(\wt L,\wt R)$ and we will do this when bounding it from below. For the first term in the expectation in the display \eqref{eq:huge-upper-bound}, we note that
	\begin{align*}
		&\{ l'+\widetilde{L}_{c/n}\in (1/M_0,M_0),r'+\widetilde{R}_{c/n}\in (1/M_0,M_0) \} \\
		&\supseteq\{ l + L_{c(l/l')^{4/\kappa}/n}\in (l/l')(1/M_0,M_0), r +  R_{c(l/l')^{4/\kappa}/n}\in r-r'l/l'+(l/l') (1/M_0,M_0)\} \\
		&\supseteq\{ l + L_{c(l/l')^{4/\kappa}/n}\in ((1+\epsilon_0)^{-1}/M_0,M_0) ,\\
		&\qquad  r +  R_{c(l/l')^{4/\kappa}/n} \in ((1+\epsilon_0)^{-1}/M_0,r(1-(1+\epsilon_0)^2)+M_0) \}
	\end{align*}
	and remark that $x\mapsto W^C(x)x^{-1-4/\kappa}$ is upper and lower bounded by finite and positive constants on each compact set contained in $(0,\infty)$. For the second term in the expectation in \eqref{eq:huge-upper-bound}, we use \eqref{eq:conditions-simple}, to lower bound
	\begin{align*}
		&\sum_{0<t_0<\cdots<t_{k-1}<\,c/n,} \prod_{i=0}^{k-1} e^{\sigma_{C_i}1(\Delta \widetilde{X}_{t_i}>0)} W^{C_i}(|\Delta \widetilde{X}_{t_i}|) \prod_{
			s\in (0,c/n)\setminus\{t_0,\dots,t_{k-1}\}
		} W^\emptyset(|\Delta \widetilde{X}_s|)\\
		&\gtrsim \sum_{0<t_0<\cdots<t_{k-1}<\,c(l/l')^{4/\kappa}/n} \prod_{i=0}^{k-1} e^{\sigma_{C_i}1(\Delta X_{t_i}>0)} W^{C_i}(|\Delta X_{t_i}|)
		\prod_{
			s\in (0,c(l/l')^{4/\kappa}/n)\setminus\{t_0,\dots,t_{k-1}\}
		} W^\emptyset(|\Delta X_s|)\;.
	\end{align*}
	Note that in the display above, only times where the respective process in question jumps contribute to the sums and products.
	Finally, for the third term in the expectation in \eqref{eq:huge-upper-bound},
	\begin{align*}
		&\{ l'+\inf\nolimits_{[0,c/n]}\widetilde{L}>0,r'+\inf\nolimits_{[0,c/n]}\widetilde{R}>0,l'+\widetilde{L}([0,c/n])\not\subseteq (1/M'',M'') \} \\
		&\supseteq \{ l+\inf\nolimits_{[0,c(l/l')^{4/\kappa}/n]}L>0,r+\inf\nolimits_{[0,c(l/l')^{4/\kappa}/n]}R>0,\\
		&\qquad\quad l+ L([0,c/n])\not\subseteq (1/M'',(1+\epsilon_0)M'') \}.
	\end{align*}
	We now use the three last displays to bound the right-hand side of \eqref{eq:huge-upper-bound} (with $(\wt L,\wt R)$ instead of $(L,R)$) and get
	\begin{align*}
		&\P_{(l',r',B)}( L^W([0,c/n])\not\subseteq (1/M'',M''),(L^W,R^W)_{c/n}\in (1/M_0,M_0)^2 ) \\
		&\qquad\gtrsim  \P_{(l,r,B)}( L^W([0,c/n])\not\subseteq (1/M'',(1+\epsilon_0)M''),(L^W,R^W)_{c(l/l')^{4/\kappa}/n}\in (a,b)^2 )
	\end{align*}
	for all $l'\in (l_-,l)$, $r'\in (r,r_+)$ where $(a,b)=((1+\epsilon_0)^{-1}/M_0,r(1-(1+\epsilon_0)^{-2})+M_0)$. We can further lower bound this using the Markov property to obtain
	\begin{align*}
		&\P_{(l',r',B)}( L^W([0,c/n])\not\subseteq (1/M'',M''),(L^W,R^W)_{c/n}\in (1/M_0,M_0)^2 ) \\
		&\gtrsim \P_{(l,r,B)} (L^W([0,c/n])\not\subseteq (1/M'',(1+\epsilon_0)M''), (L^W,R^W)_{c/n}\in (1/M,M)^2)\\
		&\quad\cdot \inf_{\substack{l'',r''\in (1/M,M)\\C\subseteq B}} \P_{(l'',r'',C)}((L^W,R^W)_{c(l/l')^{4/\kappa}/n-c/n}\in (a,b)^2  )
	\end{align*}
	By construction, $[1/M,M]\subseteq (a,b)$ and one can see from the definitions that the infimum in the display above is bounded below by a positive constant uniformly in $l'\in (l_-,l)$. Combining this with \eqref{eq:infimum-comparison-step} yields
	\begin{align*}
		&\P_{(l,r,B)}(L^W([0,c/n])\not\subseteq (1/M'',(1+\epsilon_0)M''), (L^W,R^W)_{c/n}\in (1/M,M)^2) \\
		&\qquad\lesssim \frac{1}{n} \,\P_{(l_+,r_-,B)}\left( L^W([0,1])\not\subseteq (1/M'',M'') , 1<\zeta^W \right)
	\end{align*}
	for all $n\ge 1$ and $c\in (0,1)$. Using that $\epsilon_0\in(0,1)$ this further gives that for some constant $C_0>0$ and all $t\in(0,1)$ and $M'>M_0$ (recalling that $M''=2M'$),
	\begin{align*}
		&\P_{(l,r,B)}(L^W([0,t])\not\subseteq (1/M',M'), (L^W,R^W)_{t}\in (1/M,M)^2) \\
		&\qquad\leq C_0t \,\P_{(l_+,r_-,B)}\left( L^W([0,1])\not\subseteq (1/M'',M'') , 1<\zeta^W \right).
	\end{align*}
	The lemma now follows from this inequality and
	\begin{align*}
		\label{eq:first-conv-to-zero}
		\P_{(l_+,r_-,B)}\left( L^W([0,1])\not\subseteq (1/M'',M'') , 1<\zeta^W \right) \to 0
	\end{align*}
	as $M''\to \infty$.
\end{proof}

\begin{lemma}
	\label{lem:generator-finite}
	Consider the setting as in Theorem \ref{thm:weighted-levy}. Then $\int_0^1 W^B(x)\,dx<\infty$ for all $B\subseteq A$.
\end{lemma}

\begin{proof}
	The case $B=\emptyset$ is clear by \eqref{eq:conditions-simple}, so assume that $B\neq \emptyset$ and consider $l=r=1/2$. By the assumptions of Theorem \ref{thm:weighted-levy} we get that
	\begin{align*}
		\infty &> W^B(1) \ge \E_{(l,r)}\left(\sum_{t<\zeta'} e^{\sigma_B 1(\Delta X'_t>0)}\prod_{s\neq t} W^\emptyset(|\Delta X'_s|)\cdot W^B(|\Delta X'_t|) \right) \\
		&= \E_{(l,r)}\Biggl(\int_0^{\zeta'} \prod_{s<t}W^\emptyset(|\Delta X'_s|)\int \frac{w_\kappa(L'_s+h,R'_s)W^B(|h|)e^{\sigma_B 1(h>0)}\mu_L(dh)}{w_\kappa(L'_t,R'_t)}\,W^\emptyset(X'_t+h) \,dt \\
		&\quad + \int_0^{\zeta'} \prod_{s<t}W^\emptyset(|\Delta X'_s|)\int \frac{w_\kappa(L'_s,R'_s+h)W^B(|h|)e^{\sigma_B 1(h>0)}\mu_R(dh)}{w_\kappa(L'_t,R'_t)}\,W^\emptyset(X'_t+h) \,dt\Biggr)\;.
	\end{align*}
	where to go to the second line we use the second statement in Lemma \ref{lem:msw-palm-result} and that, by assumption, the right-hand side of \eqref{eq:weighted-levy} is equal to 1 for $g=1$. If we had that $\int_0^1 W^B(x)\,dx=\infty$ then the integrand in the final expectation above would be almost surely infinite resulting in a contradiction.
\end{proof}

\begin{proof}[Proof of Theorem \ref{thm:weighted-levy}]
	The strong Markov property appears as Lemma \ref{lem:smp-fragmentation}. We will now prove the main claim of the theorem. Let us remark that all expressions in the definition of $\mathcal{G}_{\kappa,\beta}^W f$ are well-defined by Lemma \ref{lem:generator-finite}. By Lemma \ref{lem:weighted-levy-palmstep} with $M'=\infty$ and applied to the function $(u,v,C)\mapsto f(u,v,B)1(B=C)$, and by the expression of the generator of $(L^\emptyset,R^\emptyset)$ as given in Lemma \ref{lem:twisted-jump-simple} we get by the definitions and by Lemma \ref{lem:twisted-jump-simple} that
	\begin{align*}
		&\E_{(l,r,B)}\left(f(\bm{Y}^W_t);P^W_t=B\right) \\
		&= \frac{e^{\alpha t}}{w_\kappa(l,r)W^B(l+r)}\E_{(l,r)}\left(f(L^\emptyset_{t\wedge \tau_0^\emptyset},R^\emptyset_{t\wedge \tau_0^\emptyset},B)W^B(X^\emptyset_{t\wedge \tau_0^\emptyset})w_\kappa(L^\emptyset_{t\wedge \tau_0^\emptyset},R^\emptyset_{t\wedge \tau_0^\emptyset}) \right)
	\end{align*}
	Now by the martingale statement in Lemma \ref{lem:twisted-jump-simple} and optional stopping at time $t\wedge \tau_0^\emptyset$ we obtain that
	\begin{align*}
		&\E_{(l,r)}\left(f(L^\emptyset_{t\wedge \tau_0^\emptyset},R^\emptyset_{t\wedge \tau_0^\emptyset},B)W^B(X^\emptyset_{t\wedge \tau_0^\emptyset})w_\kappa(L^\emptyset_{t\wedge \tau_0^\emptyset},R^\emptyset_{t\wedge \tau_0^\emptyset}) \right) \\
		&= f(l,r,B)W^B(l+r)w_\kappa(l,r) + \E_{(l,r)}\left(\int_0^{t} \mathcal{L}^\emptyset_{\kappa,\beta}g (L^\emptyset_s,R^\emptyset_s)1(s<\tau_0^\emptyset)\,ds\right)
	\end{align*}
	where $g(l',r'):=f(l',r',B)W^B(l'+r')w_\kappa(l',r')$. Since $\mathcal{L}^\emptyset_{\kappa,\beta} g$ is a bounded function and $(L^\emptyset,R^\emptyset)$ is right-continuous we obtain (as in Proposition \ref{prop:simple-msw-process}) by dominated convergence that
	\begin{align*}
		\frac{1}{t} \,\E_{(l,r)}\left(\int_0^{t} \mathcal{L}^\emptyset_{\kappa,\beta}g (L^\emptyset_s,R^\emptyset_s)1(s<\tau_0^\emptyset)\,ds\right) \to \mathcal{L}^\emptyset_{\kappa,\beta}g(l,r)
	\end{align*}
	as $t\downarrow 0$. Combining the above observations yields that
	\begin{align*}
		&\lim_{t\downarrow 0} \frac{1}{t}\left(\E_{(l,r,B)}\left(f(\bm{Y}^W_t);P^W_t=B\right) - f(l,r,B)\right) \\
		&= f(l,r,B)\int (\mu_L+\mu_R)(dh) (W^\emptyset(|h|)-1) \\
		&\qquad + \frac{1}{w_\kappa(l,r)W^B(l+r)}\, \mathcal{L}^\emptyset_{\kappa,\beta}((u,v)\mapsto w_\kappa(u,v) W^B(u+v) f(u,v,B))(l,r)\;.
	\end{align*}
	This in particular implies the result when $B=\emptyset$.

	Let $M$ be such that $f(\cdot,\cdot,C)$ has support contained in $(1/M,M)^2$ for all $C\subseteq A$. We will now show that
	\begin{align*}
		 &\lim_{t\downarrow 0}\frac{1}{t}\,\E_{(l,r,B)}( f(\bm{Y}^W_t);E^{M'}_t(L^W,R^W),P^W_t\neq B) \\
		 &= \sum_{\emptyset\neq C_0\subseteq B}\int \mu_L(dh)\, \frac{W^{C_0}(|h|)w_\kappa(l+h,r)W^{B\setminus C_0}(l+r+h)}{w_\kappa(l,r)W^B(l+r)}\,e^{\sigma_{B\setminus C}1(h>0)}\,f(l+h,r,B\setminus C_0) \\
		 &\quad + \sum_{\emptyset\neq C_0\subseteq B}\int \mu_R(dh)\, \frac{W^{ C_0}(|h|)w_\kappa(l,r+h)W^{B\setminus C_0}(l+r+h)}{w_\kappa(l,r)W^B(l+r)}\,\,e^{\sigma_{B\setminus C}1(h>0)}\,f(l,r+h,B\setminus C_0)
	\end{align*}
	for all $M'>M$ where $E^{M'}_t(u,v)=\{u([0,t]),v([0,t])\subseteq (1/M',M')\}$ whenever $u,v\in D([0,\infty))$. Lemma \ref{lem:generator-truncation} (with the constant $M$ appearing in the statement of this lemma having been defined above) then implies the theorem. The display above is an immediate consequence of \eqref{eq:iterated-palm-simple} within Lemma \ref{lem:weighted-levy-palmstep}; the key is that the indicator in the final line of \eqref{eq:iterated-palm-simple} yields deterministic bounds on the other quantities in the integral and we then see that if $k>1$ then the integral is $O(t^2)$ and if $k=1$ we obtain exactly the limit above.
\end{proof}

For the computation of the law of the conformal radius of a simple conformal loop ensemble, we will need to determine a particular functional associated to the process $(L',R')$. Thus we conclude this section with the following explicit computation. The proof starts by expressing $A_\pm$ (as defined in the proposition below) in terms of an integral over time, using the Palm's formula type result of Lemma \ref{lem:msw-palm-result}. Then the key trick is to observe that the integrand can be expressed in terms of $\mathcal{G}_{\kappa,\beta} f$ for an appropriate function $f$, which allows us to use a martingale argument with \eqref{eq:simple-msw-process} to compute the integral. The fact that we can find such an $f$ explicitly is very particular to the setting considered here and does not generalize to arbitrary functionals.
\begin{prop}
	\label{prop:power-theta}
	Suppose that $l,r>0$ such that $l+r=1$ and let
	\begin{align*}
		A_+ = \E_{(l,r)}\left(\sum_{t<\zeta'\colon \Delta X'_t>0} |\Delta X'_t|^\theta \right)\quad\text{and}\quad A_- = \E_{(l,r)}\left(\sum_{t<\zeta'\colon \Delta X'_t<0} |\Delta X'_t|^\theta \right)\;.
	\end{align*}
	Then $A_+<\infty$, $A_-<1$ and
	\begin{align*}
		\frac{A_+}{1-A_-} = \frac{\cos(4\pi/\kappa)}{\cos(4\pi/\kappa-\pi\theta)}
	\end{align*}
	whenever we have $\theta\in (4/\kappa+1/2,4/\kappa+3/2)$. Also, if $\theta\notin (4/\kappa, 1+8/\kappa)$ then $A_+=\infty$. Recall that the subscript $(l,r)$ in the display above indicates that $(L',R')$ starts from $(l,r)$.
\end{prop}

\begin{proof}
	By the second part of Lemma \ref{lem:msw-palm-result} with $W^\emptyset\equiv 1$ and $g(x,l,r)=x^\theta1(x>0)$ we get
	\begin{align*}
		A_+ = \int_0^\infty\E_{(l,r)}\left( \int_0^\infty \mu_L(dh)\, \frac{h^\theta w_\kappa(L'_t+h,R'_t)}{w_\kappa(L'_t,R'_t)} + \int_0^\infty \mu_R(dh)\, \frac{h^\theta w_\kappa(L'_t,R'_t+h)}{w_\kappa(L'_t,R'_t)};t<\zeta'\right)\,dt\;.
	\end{align*}
	Note that if $\theta\notin (4/\kappa, 1+8/\kappa)$ then
	\begin{align*}
		&\int_0^\infty \mu_L(dh)\, \frac{h^\theta w_\kappa(L'_t+h,R'_t)}{w_\kappa(L'_t,R'_t)} + \int_0^\infty \mu_R(dh)\, \frac{h^\theta w_\kappa(L'_t,R'_t+h)}{w_\kappa(L'_t,R'_t)} \\
		&\quad = -\cos(4\pi/\kappa)\int_0^\infty \frac{h^\theta\,dh}{h^{1+4/\kappa}}\, \frac{(L'_t + R'_t)^{1+4/\kappa}}{(L'_t+R'_t+h)^{1+4/\kappa}}=\infty\quad\text{for all $t<\zeta'$ a.s.}
	\end{align*}
	This in particular implies the final assertion of the proposition.

	Suppose henceforth that $\theta\in (4/\kappa+1/2,4/\kappa+3/2)$ and define $\sigma$ by $e^{-\sigma}=\cos(4\pi/\kappa)/\cos(4\pi/\kappa-\pi\theta)$. Applying Lemma \ref{lem:msw-palm-result} to $A_-$, similarly as it was applied to $A_+$ above, we obtain
	\begin{gather*}
		e^\sigma A_++A_-= \int_{0}^\infty\E_{(l,r)}\left( g(L'_t,R'_t);t<\zeta'\right)\,dt\quad\text{where}\\
		g(l',r')=\int_{-\infty}^\infty \mu_L(dh)\, \frac{e^{\sigma 1(h>0)}|h|^\theta w_\kappa(l'+h,r')}{w_\kappa(l',r')} + \int_{-\infty}^\infty \mu_R(dh)\, \frac{e^{\sigma 1(h>0)}|h|^\theta w_\kappa(l',r'+h)}{w_\kappa(l',r')}\;.
	\end{gather*}
	Note that $g$ is well-defined and finite since $\theta\in (4/\kappa,1+8/\kappa)$.

	The key is now that if we let $f(l',r')=(l'+r')^\theta 1(l',r'>0)$ then $\mathcal{G}_{\kappa,\beta}f(l',r') = -g(l',r')$ for all $l',r'>0$,  where we recall the definition of the generator $\mathcal{G}_{\kappa,\beta}$ appearing in \eqref{eq:w}. To see this, by scaling, it suffices to consider $l'+r'=1$. Let us write $\theta'=\theta-1-4/\kappa$. Then
	\begin{align*}
		\mathcal{G}_{\kappa,\beta}f(l',r') + g(l',r') &= \frac{1}{w_\kappa(l',r')}\,\mathcal{L}_{\kappa,\beta}(fw_\kappa)(l',r') + g(l',r')\\
		&= -\cos(4\pi/\kappa)\int_0^\infty \frac{dh}{h^{1+4/\kappa}}\left((1+h)^{\theta'}-1-\theta'h \right) \\
		&\quad + \int_0^\infty \frac{dh}{2h^{1+4/\kappa}}\left( (1-h)^{\theta'}1(h<l') -1+\theta'h \right) + \int_0^{l'} \frac{h^\theta\,dh}{2(1-h)^{1+4/\kappa}} \\
		&\quad + \int_0^\infty \frac{dh}{2h^{1+4/\kappa}}\left( (1-h)^{\theta'}1(h<r') -1+\theta'h \right) + \int_0^{r'} \frac{h^\theta\,dh}{2(1-h)^{1+4/\kappa}} \\
		&\quad -\cos(4\pi/\kappa-\pi\theta) \int_0^\infty \frac{h^{\theta'}\,dh}{(1+h)^{1+4/\kappa}}\\
		&= 0
	\end{align*}
	by Lemmas \ref{lem:simple-integrals-comp} and \ref{lem:beta-results}. Putting the results obtained so far together, we deduce that
	\begin{align}
		\label{eq:prefinal-step-ssw-computation}
		e^{\sigma}A_+ + A_- = -\int_0^\infty \E_{(l,r)}\left( \mathcal{G}_{\kappa,\beta} f(L'_t,R'_t);t<\zeta'\right)\,dt\;.
	\end{align}
	In order to conclude the proof it is sufficient to show that the right-hand side of \eqref{eq:prefinal-step-ssw-computation} is equal to $1$.

	We consider functions $h_n\in C_c^\infty(\R)$ with $h_n\ge 0$, $h_n|_{(2^{-n},2^n)}=1$ and $h_n|_{(2^{-n-1},2^{n+1})^c}=0$. Also let $\zeta'_m =\inf\{t\ge 0\colon L'_t\notin (2^{-m},2^m)\text{ or } R'_t\notin(2^{-m},2^m)\}$. By Proposition \ref{prop:simple-msw-process} (in particular, \eqref{eq:simple-msw-process}) and the optional stopping theorem at time $t\wedge\zeta'_m$, for $t\ge 0$
	\begin{align}
		\label{eq:intermediate-ssw-generator}
		\E_{(l,r)}(u_n(L'_{t\wedge\zeta'_m},R'_{t\wedge \zeta'_m}))= 1 + \int_0^t \E_{(l,r)} \left( \mathcal{G}_{\kappa,\beta}u_n(L'_s,R'_s) ; s<\zeta'_m \right)\,ds
	\end{align}
	where $u_n(l',r')=h_n(l')h_n(r')f(l',r')$ and provided that $l,r\in (2^{-n},2^n)$. We now consider the limits $n\to\infty$, $m\to\infty$ and then $t\to \infty$ in \eqref{eq:intermediate-ssw-generator} on both sides.

	We begin with the right-hand side. For each $m\ge 1$ there exists a constant $C_m$ such that $|\mathcal{G}_{\kappa,\beta}u_n| \le C_m$ on $[2^{-m},2^m]^2$ and for all $n\ge 1$. Thus for the $n\to\infty$ limit, we can use dominated convergence to get
	\begin{align}
		\label{eq:ssw-comp-dom-conv}
		\int_0^t \E_{(l,r)} \left( \mathcal{G}_{\kappa,\beta}u_n(L'_s,R'_s) ; s<\zeta'_m \right)\,ds \to \int_0^t \E_{(l,r)} \left( \mathcal{G}_{\kappa,\beta}f(L'_s,R'_s) ; s<\zeta'_m \right)\,ds\le 0
	\end{align}
	as $n\to \infty$ for all $m\ge 1$ and $t>0$.
	Observe now that since $\mathcal{G}_{\kappa,\beta}f=-g\le 0$ we can use monotone convergence for the $m\to \infty$ and $t\to \infty$ limit to get
	\begin{align*}
		\lim_{t\to \infty}\lim_{m\to \infty}\int_0^t \E_{(l,r)} \left( \mathcal{G}_{\kappa,\beta}f(L'_s,R'_s) ; s<\zeta'_m \right)\,ds = \int_0^\infty  \E_{(l,r)}\left( \mathcal{G}_{\kappa,\beta} f(L'_s,R'_s);s<\zeta'\right)\,ds\;.
	\end{align*}
	Therefore, in order to conclude the proof it is sufficient to show that the left-hand side of \eqref{eq:intermediate-ssw-generator} goes to zero as we send $n\to\infty$, then $m\to\infty$, and finally $t\to\infty$.

	Fatou's lemma and \eqref{eq:intermediate-ssw-generator} together with \eqref{eq:ssw-comp-dom-conv} yield
	\begin{align*}
		\E_{(l,r)}((X'_{t\wedge\zeta'_m})^\theta) \le \liminf_{n\to \infty} \E_{(l,r)}(u_n(L'_{t\wedge\zeta'_m},R'_{t\wedge \zeta'_m}))\le 1\;.
	\end{align*}
	Let $\theta'>\theta$ such that $\theta'\in (4/\kappa+1/2,4/\kappa+3/2)$. Then by the same argument the above display also holds with $\theta$ replaced by $\theta'$ and therefore the family $((X'_{t\wedge\zeta'_m})^\theta\colon m\ge 1, t>0)$ is uniformly integrable. Consequently since $X'_{t\wedge\zeta'_m} \to X'_t$ a.s. as $m\to \infty$ and $X'_t\to 0$ a.s. as $t\to \infty$ we deduce that (using monotone convergence for the first equality)
	\begin{align}
		\label{eq:error-term-to-zero}
		\lim_{t\to\infty}\lim_{m\to\infty} \lim_{n\to\infty}\E_{(l,r)}(u_n(L'_{t\wedge\zeta'_m},R'_{t\wedge \zeta'_m}))=\lim_{t\to\infty}\lim_{m\to\infty} \E_{(l,r)}((X'_{t\wedge\zeta'_m})^\theta) = 0,
	\end{align}
	which concludes the proof.
\end{proof}

\subsection{Without compensation}
\label{subsec:nonsimple-fragmentations}

Let $\kappa'\in (4,8)$ and $\beta\in [-1,1]$. Consider two independent Lévy processes $L$ and $R$ (this time without compensation) where this time the generator of $(L,R)$ is
\begin{align*}
	\mathcal{L}_{\kappa',\beta}f(l,r) &= \int_\R (f(l+h,r) - f(l,r))\,\mu_L(dh)  + \int_\R (f(l,r+h) - f(l,r))\,\mu_R(dh)
\end{align*}
for $f\in C^\infty_c(\R^2)$ and where the jump measures $\mu_L$ and $\mu_R$ of $L$ and $R$ are given by
\begin{align*}
	\mu_L(dh)/dh &= \frac{-\cos(4\pi/\kappa')(1-\beta)/2}{h^{1+4/\kappa'}}\,1_{(0,\infty)}(h) + \frac{1/2}{|h|^{1+4/\kappa'}}\,1_{(-\infty,0)}(h)\;,\\
	\mu_R(dh)/dh &= \frac{-\cos(4\pi/\kappa')(1+\beta)/2}{h^{1+4/\kappa'}}\,1_{(0,\infty)}(h) + \frac{1/2}{|h|^{1+4/\kappa'}}\,1_{(-\infty,0)}(h)\;.
\end{align*}
By definition, $\mathcal{L}_{\kappa',\beta}f$ is a bounded function whenever $f\in C^\infty_c(\R^2)$ and in this case (see again for instance \cite[Theorem 19.10 and Lemma 19.21]{kallenberg}), the process
\begin{align*}
	t\mapsto f(L_t,R_t) - f(L_0,R_0) - \int_0^t \mathcal{L}_{\kappa',\beta}f(L_s,R_s)\,ds
\end{align*}
is a martingale.

Let $X=L+R$. Just like in the previous section, we now want to consider a reweighting construction. For $l,r>0$ and $f\in C^\infty_c(\R^2)$ we let
\begin{align*}
	w_{\kappa'}(l,r) = (l+r)^{-1-4/\kappa'}1(l,r>0)\quad \text{and }\quad \mathcal{G}_{\kappa',\beta} f(l,r)=\frac{1}{w_{\kappa'}(l,r)}\,\mathcal{L}_{\kappa',\beta}(w_{\kappa'} f)(l,r) \;.
\end{align*}
By convention, we write $\mathcal{G}_{\kappa',\beta} f(0,0)=0$ and one can check that $\mathcal{G}_{\kappa',\beta}f$ is a bounded function. Let us remark that from Lemma \ref{lem:nonsimple-integrals-comp} we get that for $f\in C^\infty_c((0,\infty)^2)$ and $l,r>0$
\begin{align*}
	\mathcal{G}_{\kappa',\beta} f(l,r) &= \int_\R (f(l+h,r) - f(l,r))\,\frac{w_{\kappa'}(l+h,r)\mu_L(dh)}{w_{\kappa'}(l,r)} \\
	&\quad + \int_\R (f(l,r+h) - f(l,r))\,\frac{w_{\kappa'}(l,r+h)\mu_R(dh)}{w_{\kappa'}(l,r)}\;.
\end{align*}
Throughout this section we will not give the proofs since they are all essentially identical to the ones in the previous section but only highlight the small differences (the intermediate lemmas needed for some of the proofs are also exactly analogous).

The only differing ingredient in the proof of the following proposition is that Lemma \ref{lem:nonsimple-integrals-comp} is used instead of \ref{lem:simple-integrals-comp} as was the case in the proof of Proposition \ref{prop:simple-msw-process}.

\begin{prop}
	\label{prop:nonsimple-msw-process}
	For each $l,r>0$ there is a strong Markov càdlàg process $(L',R')$ starting at $(l,r)$, taking values in $(0,\infty)^2\cup\{(0,0)\}$ and absorbed at $(0,0)$ such that
	\begin{align*}
		\E_{(l,r)} \left(g\left(L'|_{[0,t]},R'|_{[0,t]}\right);t<\zeta'\right) = \E_{(l,r)} \left(\frac{w_{\kappa'}(L_t,R_t)}{w_{\kappa'}(l,r)} g\left(L|_{[0,t]},R|_{[0,t]}\right);t<\tau_0\right)
	\end{align*}
	whenever the function $g\colon D([0,t])^2\to [0,\infty]$ is measurable on the space of càdlàg functions, $t\ge 0$, $\zeta'=\inf\{t\ge 0\colon L'_t=R'_t=0\}$ and where $\tau_0 = \inf\{ t\ge 0\colon L_t\le 0 \text{ or } R_t \le 0\}$. Moreover, $X'_t\to 0$ as $t\to \infty$ almost surely where $X'=L'+R'$. Furthermore for $f\in C^\infty_c((0,\infty)^2)$ the process
	\begin{align*}
		t\mapsto f(L'_t,R'_t) - \int_0^{t\wedge \zeta'} \mathcal{G}_{\kappa',\beta}f(L'_s,R'_s)\,ds
	\end{align*}
	is a (càdlàg) martingale. Finally, $\mathcal{G}_{\kappa',\beta}$ is the generator of $(L',R')$ on $C^\infty_c((0,\infty)^2)$.
\end{prop}

Using Lemma \ref{lem:ppp-reweighting} we immediately obtain the following result.

\begin{lemma}
	\label{lem:twisted-jump-nonsimple}
	Suppose  $W^\emptyset\colon [0,\infty)\to (0,1]$ is measurable with $1-W^\emptyset(h)\lesssim h^{8/\kappa'}$. Let $ (L^\emptyset,R^\emptyset)$ be two independent Lévy processes with generator defined for $l,r>0$ by
	\begin{align*}
		\mathcal{L}^\emptyset_{\kappa',\beta}f(l,r) &= \int_\R  (f(l+h,r) - f(l,r))\, W^\emptyset(h)\mu_L(dh) \\
		&\quad + \int_\R  (f(l,r+h) - f(l,r))\, W^\emptyset(h)\mu_R(dh) \;.
	\end{align*}
	In fact, whenever $f\in C^\infty_c(\R^2)$, the process
	\begin{align*}
		t\mapsto f(L^\emptyset_t, R^\emptyset_t)-f(L^\emptyset_0,R^\emptyset_0) - \int_0^{t} \mathcal{L}^\emptyset_{\kappa',\beta}f(L^\emptyset_s,R^\emptyset_s)\,ds
	\end{align*}
	is a martingale. We let $X^\emptyset = L^\emptyset + R^\emptyset$. Then for all $g\colon D([0,t])^2\to [0,\infty]$ we have
	\begin{align*}
		\E_{(l,r)}\left(g(L^\emptyset|_{[0,t]},R^\emptyset|_{[0,t]})\right) = \frac{ \E_{(l,r)}\left(\,\prod_{s<t} W^\emptyset(|\Delta X_s|)\cdot g(L|_{[0,t]},R|_{[0,t]})\right) }{\exp\left(t\int (\mu_L+\mu_R)(dh)(W^\emptyset(|h|)-1)\right) } \;.
	\end{align*}
	We define $\tau^\emptyset_0=\inf\{t\ge 0\colon L^\emptyset_t \le 0\text{ or } R^\emptyset_t\le 0\}$ and $\alpha = \int (\mu_L+\mu_R)(dh)(W^\emptyset(|h|)-1)$.
\end{lemma}

As in the previous section, we consider a finite subset $A\subseteq \N$ and smooth functions $W^B\colon (0,\infty)\to (0,\infty)$ for each $B\subseteq A$ such that
\begin{gather}
\begin{split}
	W^\emptyset\text{ is non-increasing}\;,\quad W^\emptyset(h)\le 1\;,\quad 1-W^\emptyset(h)\lesssim h^{8/\kappa'} \quad\text{for $h>0$} \\
	\text{and}\quad W^B(x)\lesssim W^B(y)\quad\text{for all $B\subseteq A$ and whenever $0<y\le x\le 2y$}.
	\end{split}
	\label{eq:conditions-nonsimple}
\end{gather}
By convention we will write $W^\emptyset(0)=1$ and $W^B(0)=0$ if $B\neq \emptyset$. We also consider constants $\sigma_B\in \R$ for $B\subseteq A$ with $\sigma_\emptyset = 0$.

Recall also the following definition: for $u\colon [0,\infty)\to \R$ càdlàg we let $\Pi(u,B)$ be the set of càdlàg non-increasing (in the sense of inclusion) functions $Q\colon [0,\infty)\to \mathcal{P}(A)$ such that $Q_0=B$, $Q_t=\emptyset$ for $t$ sufficiently large and $\Delta Q_t = Q_{t-}\setminus Q_t = \emptyset$ whenever $\Delta u_t = 0$.

\begin{thm}
	\label{thm:nonsimple-weighted-levy}
	For each $B\subseteq A$ and $l,r>0$, we now define a new process $\bm{Y}^W=(L^W,R^W,P^W)$ taking values in $(0,\infty)^2\times \mathcal{P}(A)$, starting at $(l,r,B)$ and absorbed at $(0,0,\emptyset)$ by
	\begin{align*}
		&\E_{(l,r,B)}\left(g(L^W,R^W,P^W)\right)\\
		&=\frac{1}{W^B(l+r)}\,\E_{(l,r)}\left( \,\sum_{Q\in \Pi(X',B)} g(L',R',Q)\prod_{s<\zeta'\colon \Delta X'_s\neq 0} e^{\sigma_{Q_{s-}\setminus Q_s}1(\Delta X'_s>0)}\, W^{Q_{s-}\setminus Q_s}(\Delta X'_s) \right)\;.
	\end{align*}
	For this to make sense, we also assume that the right-hand side of the equation is $1$ for $g=1$ and for all $B\subseteq A$ and $l,r>0$. Moreover, we let $\zeta^W = \inf\{t\ge 0\colon \bm{Y}^W = (0,0,\emptyset)\}$. Then $\bm{Y}^W$ is a strong Markov process and the following expression is well-defined:
	\begin{align*}
		&\mathcal{G}_{\kappa',\beta}^W f(l,r,B) \\
		&\quad:= f(l,r,B)\int (\mu_L+\mu_R)(dh) (W^\emptyset(h)-1) \\
		&\qquad + \frac{1}{w_{\kappa'}(l,r)W^B(l+r)}\, \mathcal{L}^\emptyset_{\kappa',\beta}((u,v)\mapsto w_{\kappa'}(u,v) W^B(u+v) f(u,v,B))(l,r) \\
		&\qquad + \sum_{C\subsetneq B}\int \mu_L(dh)\, \frac{W^{B\setminus C}(|h|)w_{\kappa'}(l+h,r)W^C(l+r+h)}{w_{\kappa'}(l,r)W^B(l+r)}\,e^{\sigma_{B\setminus C}1(h>0)}\,f(l+h,r,C) \\
		&\qquad + \sum_{C\subsetneq B}\int \mu_R(dh)\, \frac{W^{B\setminus C}(|h|)w_{\kappa'}(l,r+h)W^C(l+r+h)}{w_{\kappa'}(l,r)W^B(l+r)}\,\,e^{\sigma_{B\setminus C}1(h>0)}\,f(l,r+h,C)
	\end{align*}
	for $f\colon (0,\infty)^2\times \mathcal{P}(A)\to \R$ with $f(\cdot,B)\in C^\infty_c((0,\infty)^2)$ for all $B\subseteq A$. Then
	\begin{align*}
		\lim_{t\downarrow 0} \frac{1}{t}\, \E_{(l,r,B)}\left(f(\bm{Y}^W_t) - f(l,r,B)\right) = \mathcal{G}_{\kappa',\beta}^W f(l,r,B)
	\end{align*}
	i.e.\ $\mathcal{G}_{\kappa',\beta}^W$ is the generator of $\bm{Y}^W$ on the set of functions $f$ considered above.
\end{thm}

The proof of the following proposition is exactly analogous to the one of Proposition \ref{prop:power-theta} except that Lemma \ref{lem:nonsimple-integrals-comp} is used instead of Lemma \ref{lem:simple-integrals-comp} in the computations.

\begin{prop}
	\label{prop:nonsimple-power-theta}
	Suppose that $l,r>0$ such that $l+r=1$ and let
	\begin{align*}
		A_+ = \E_{(l,r)}\left(\sum_{t<\zeta'\colon \Delta X'_t>0} |\Delta X'_t|^\theta \right)\quad\text{and}\quad A_- = \E_{(l,r)}\left(\sum_{t<\zeta'\colon \Delta X'_t<0} |\Delta X'_t|^\theta \right)\;.
	\end{align*}
	Then $A_+<\infty$, $A_-<1$ and $A_+/(1-A_-) = \cos(4\pi/\kappa')/\cos(4\pi/\kappa'-\pi\theta)$ whenever we have $\theta\in (4/\kappa'+1/2,4/\kappa'+3/2)$. Also, if $\theta\notin (4/\kappa', 1+8/\kappa')$ then $A_+=\infty$.
\end{prop}

\section{Some functionals of Lévy excursions}
\label{sec:levy-exc}

In this section we will carry out computations for stable Lévy excursions. These will be used when finding the partition function for the generalized LQG disk with zero marked points or one marked point in Section \ref{sec:disk}.

Suppose that $\nu \in (1,2)$. Throughout this section, let $(B_t)_{t\geq 0}$ be a stable Lévy process with exponent $\nu$ and no negative jumps started at $0$. We fix the normalization as in \cite{curien-kortchemski-looptree-def}, that is by requiring that
\begin{align*}
	\E(e^{-\lambda B_t}) = e^{-t\lambda^\nu}\quad\text{for $\lambda>0$ and $t\ge 0$}\;.
\end{align*}
The jump measure of the Lévy process $B$ is $\pi(dh)=\Gamma(-\nu)^{-1} h^{-1-\nu}1_{(0,\infty)}(h)\,dh$. Let us define $\tau_y = \inf\{t\ge 0\colon B_t<-y\}$ and $\tau_{y-}=\inf\{t\ge 0\colon B_t\le -y\}$ for $y\ge 0$. Note that we have $\tau_y\neq\tau_{y-}$ precisely if $B|_{[\tau_{y-},\tau_{y}]}$ is an excursion from $-y$ to $-y$ staying strictly above level $-y$ on the interval $(\tau_{y-},\tau_{y})$.

Then by excursion theory (see for instance \cite[Chapter 22]{kallenberg}) there exists a random non-negative process $(b_t)$ starting at $0$ and absorbed again at $0$ at time $1$ with the following property. Let $\wt{n}_\ell$ be the law of $(\ell^{\,1/\nu}b_{t/\ell}\colon t\ge 0)$, i.e., $\widetilde{n}_\ell$ is a probability measure on excursions of duration $\ell$ which is obtained by rescaling $b$ appropriately. Define
\begin{align*}
	\wt{n} = \frac{1}{\nu\cdot \Gamma(1-1/\nu)} \int_0^\infty d\ell\,\ell^{-1-1/\nu}\,\wt{n}_\ell\;,
\end{align*}
so that $\wt n$ is an infinite measure on excursions of non-fixed duration obtained by combining the measures $\wt n_\ell$ (called the excursion measure of the Lévy process). Then
\begin{align*}
	\xi:=\sum_{y\ge 0\colon \tau_y > \tau_{y-}} \delta_{(y,B_{(\tau_{y-}+\cdot)\wedge \tau_y}-B_{\tau_{y-}})}\quad\text{is a PPP with intensity}\quad \lambda|_{(0,\infty)}\otimes \wt{n}
\end{align*}
on $(0,\infty)\times D([0,\infty))$ where $\lambda$ is the Lebesgue measure on $\R$ and $D([0,\infty))$ is the space of càdlàg functions. In other words, the excursions made by $B$ above its running infimum indexed by the value of the infimum have law of a Poisson point process (PPP) with intensity $\lambda|_{(0,\infty)}\otimes \wt{n}$. The process $(\ell^{\,1/\nu}b_{t/\ell}\colon t\ge 0)$ is called the spectrally positive stable Lévy excursion with exponent $\nu$ and length $\ell$. We refer to \cite[Section 3.1.1]{curien-kortchemski-looptree-def} for more information on these constructions.

We will now compute some explicit functionals associated to such Lévy excursions. Recall the definition of the modified Bessel functions of the second kind:
\begin{align*}
	K_\nu(x)= \frac{1}{2(x/2)^\nu} \int_0^\infty e^{-(x/2)^2 y-1/y}\,\frac{dy}{y^{1+\nu}}\;.
\end{align*}
We will now prove the following theorem; as the proof will reveal, it is quite remarkable that the expectations of the functionals below have a completely explicit form.

\begin{thm}
	\label{thm:levy-excursions}
	Consider $\nu\in (1,2)$, $c>0$, $\theta \in (1,1+1/\nu)$ and let
	\begin{align*}
		g(x) &= \frac{2(cx/2)^\nu}{\Gamma(\nu)}\,K_\nu(cx)\quad\text{for $x>0$}\quad\text{and}\quad g(0)=1\;, \\
		f(x) &= x^{1+\nu} K_{1+\nu - \theta\nu}(cx)\quad\text{for $x>0$}\quad\text{and}\quad f(0)=0 \;.
	\end{align*}
	Let $c'=c^\nu 2^{1-\nu}$. Then for all $\ell>0$,
	\begin{align*}
		\E\left( \prod_{t\le 1} g(\ell^{1/\nu}\Delta b_t)\right) &= \frac{2(c'\ell/2)^{1/\nu}}{\Gamma(1/\nu)}\,K_{1/\nu}(c'\ell)\;,\\
		\E\left( \sum_{t\le 1} f(\ell^{1/\nu}\Delta b_t)\,\prod_{s\neq t} g(\ell^{1/\nu}\Delta b_s)\right) &= \frac{\sqrt{1/\nu}\,\Gamma(-1/\nu)\sin(\pi(1+1/\nu-\theta))}{\sqrt{\nu}\,\Gamma(-\nu)\sin(\pi(1+\nu-\theta\nu))} \,\ell^{1+1/\nu} K_{1+1/\nu-\theta}(c'\ell)\;.
	\end{align*}
\end{thm}

To prove the theorem we need two lemmas, one on Lévy processes and the other on the computation of some integrals involving the modified Bessel functions of the second kind. The latter is Lemma \ref{lem:residue-bessel-computation}.

\begin{lemma}
	\label{lem:palm-levy-excursion}
	Consider $\nu\in (1,2)$. Let $G\colon [0,\infty)\to [0,\infty)$ be twice continuously differentiable with $G(0)=G'(0)=0$. Suppose that  $\lambda>0$ and $\rho_\lambda>0$ are such that
	\begin{align}
		\label{eq:palm-key-statement}
		\lambda = \frac{1}{\Gamma(-\nu)} \int_0^\infty \frac{dh}{h^{1+\nu}}\,(e^{-\rho_\lambda h - G(h)}-1+\rho_\lambda h) \;.
	\end{align}
	Note that the integral in the display above is finite for any $\rho_\lambda>0$ because of the assumption $G(0)=G'(0)=0$. Then with $B$ as above,
	\begin{align}
		\label{eq44}
		\begin{split}
		-\rho_\lambda &= \log \E\left( e^{-\lambda \tau_1-\sum_{t< \tau_1} G(\Delta B_t)} \right) \\
		&= \frac{1}{\nu \Gamma(1-1/\nu)} \int_0^\infty \frac{d\ell}{\ell^{1+1/\nu}}\left( e^{-\lambda \ell}\,\E\left( e^{-\sum_{t\le 1} G(\ell^{1/\nu}\Delta b_t) } \right) - 1\right)\;.
		\end{split}
	\end{align}
	Let $F\colon [0,\infty)\to [0,\infty)$ be twice continuously differentiable with $F(0)=F'(0)=0$. Then  
	\begin{align}
		\label{eq45}
		\begin{split}
		&-\frac{d}{d\lambda}(e^{-\rho_\lambda})\,\frac{1}{\Gamma(-\nu)} \int_0^\infty \frac{dh}{h^{1+\nu}}\,e^{-\rho_\lambda h -G(h)}F(h) = \E\left( \sum_{t<\tau_1} F(\Delta B_t)\, e^{-\lambda \tau_1-\sum_{s< \tau_1} G(\Delta B_s)} \right)\\
		&\qquad = \frac{e^{-\rho_\lambda}}{\nu \Gamma(1-1/\nu)} \int_0^\infty \frac{d\ell}{\ell^{1+1/\nu}}\,\E\left( \sum_{t\le 1} F(\ell^{1/\nu}\Delta b_t)\, e^{-\lambda\ell - \sum_{s\le 1} G(\ell^{1/\nu}\Delta b_s)} \right),
		\end{split}
	\end{align}
	where now we need to assume that a $\rho_\lambda>0$ satisfying \eqref{eq:palm-key-statement} is defined on some open set of $\lambda>0$ values in order for the derivative to be well-defined.
\end{lemma}

\begin{proof}
	Let $M$ be the following càdlàg process
	\begin{align*}
		M_T = e^{-\rho_\lambda B_T -\lambda T - \sum_{t\le T} G(\Delta B_t)}\;\quad T\ge 0\;.
	\end{align*}
	We will first argue that $\E(M_T)=1$ for any $T\ge 0$. Since $B$ has stationary and independent increments, it will follow that $M$ is a martingale. By the construction of the compensated Lévy process $B$ as a compensated sum over its jumps (see e.g.\ \cite[Chapter I]{bertoin-book}), we have
	\begin{align*}
		M^n_T:=e^{-\rho_\lambda \left(\sum_{t\le T\colon \Delta B_t>1/n} \Delta B_t - T\,\Gamma(-\nu)^{-1}\int_{1/n}^\infty h\cdot h^{-1-\nu}\,dh\right) - \sum_{t\le T\colon \Delta B_t>1/n} G(\Delta B_t)-\lambda T} \to M_T
	\end{align*}
	in probability as $n\to \infty$. Campbell's formula and the definition of $\lambda$ imply that
	\begin{align}
		\log \E(M^n_T) = -\lambda T+\frac{T}{\Gamma(-\nu)} \int_{1/n}^\infty (e^{-\rho_\lambda h - G(h)}-1+\rho_\lambda h) \frac{dh}{h^{1+\nu}} \to 0\quad\text{as $n\to \infty$}\;.
		\label{eq43}
	\end{align}
	A similar computation shows that $(M^n_T\colon n\ge 1)$ is bounded in $L^2$ and in particular uniformly integrable. Therefore $\E(M^n_T)\to \E(M_T)$ as $n\to \infty$. Using this, \eqref{eq:palm-key-statement}, and \eqref{eq43}, we get that $\E(M_T)=1$ as required.
	The fact that $M$ is a martingale and optional stopping now implies that for any $T\ge 0$,
	\begin{align*}
		1&=\E(M_{T\wedge \tau_1}) = \E\left(e^{-\rho_\lambda B_{T\wedge \tau_1} -\lambda (T\wedge \tau_1) - \sum_{t\le T\wedge \tau_1} G(\Delta B_t)} \right)\;.
	\end{align*}
	Sending $T\to\infty$ and using dominated convergence with dominating function $e^{\rho_\lambda}$ we get the first equality of \eqref{eq44}. 
	For the second equality of \eqref{eq44}, we observe that
	\begin{align}
		\label{eq:exponent-ppp-excursion}
		\lambda \tau_1 +\sum_{t<\tau_1} G(\Delta B_t) = \int_{(0,1)\times D([0,\infty))}\xi(dy,de) \left(\lambda \zeta_e + \sum_{t<\zeta_e} G(\Delta e_t)\right)\quad\text{a.s.}
	\end{align}
	where $\zeta_e=\inf\{t>0\colon e_t=0\}$. Hence since $n$ is the intensity measure of the Poisson point process $\xi$, we obtain
	\begin{align*}
		\log \E\left( e^{-\lambda \tau_1-\sum_{t< \tau_1} G(\Delta B_t)} \right) = \int_{(0,1)\times D([0,\infty))} n(dy,de) \left( e^{\lambda \zeta_e + \sum_{t<\zeta_e} G(\Delta e_t)} - 1 \right)\;.
	\end{align*}
	The result follows from the definition of $n$.
	
	The proof of \eqref{eq45}
	will make use of Palm's formula in several places (see Lemma \ref{lem:palm-abstract}). Let us start with the excursion theory perspective as this is somewhat simpler. Indeed, we write
	\begin{align*}
		&X:=\sum_{t<\tau_1} F(\Delta B_t)\, e^{-\lambda \tau_1-\sum_{s< \tau_1} G(\Delta B_s)} \\
		&\quad = \int_{(0,1)\times D([0,\infty))}\xi(dy,de) \left(\sum_{t<\zeta_e} F(\Delta e_t) \right) e^{- \int_{(0,1)\times D([0,\infty))}\xi(dy',de') \left(\lambda \zeta_{e'} + \sum_{s<\zeta_{e'}} G(\Delta e'_s)\right)}\;.
	\end{align*}
	By Palm's formula we obtain
	\begin{align*}
		\E(X) &=  \int_{(0,1)\times D([0,\infty))}n(dy,de) \left(\sum_{t<\zeta_e} F(\Delta e_t) \right) \\
		&\qquad \cdot \E\left(e^{- \int_{(0,1)\times D([0,\infty))}(\xi+\delta_{(y,e)})(dy',de') \left(\lambda \zeta_{e'} + \sum_{s<\zeta_{e'}} G(\Delta e'_s)\right)}\right)\\
		&= \int_{(0,1)\times D([0,\infty))}n(dy,de) \left(\sum_{t<\zeta_e} F(\Delta e_t) e^{-\lambda \zeta_e - \sum_{s<\zeta_e} G(\Delta e_s)}\right) \\
		&\qquad \cdot \E\left(e^{- \int_{(0,1)\times D([0,\infty))}\xi(dy',de') \left(\lambda \zeta_{e'} + \sum_{s<\zeta_{e'}} G(\Delta e'_s)\right)}\right) \;.
	\end{align*}
	Note that by \eqref{eq:exponent-ppp-excursion}, the final expectation in the display above equals $e^{-\rho_\lambda}$ and the second equality of \eqref{eq45}
	 follows again from the definition of $n$. 
	 
	Lastly, we compute $\E(X)$ using the Poisson point process structure of the jumps of $B$. We will need the notation
	\begin{align*}
		B^{(t,h)} = B + h\cdot 1_{[t,\infty)}\quad\text{and}\quad \tau_1^{(t,h)} = \inf\{s\ge 0\colon B^{(t,h)}_s<-1\}\;.
	\end{align*}
	Again by Palm's formula we get
	\begin{align*}
		\E(X) &= \E\left( \sum_{t\ge 0} F(\Delta B_t)1(t<\tau_1)\, e^{-\lambda \tau_1-\sum_{s< \tau_1} G(\Delta B_s)} \right) \\
		&= \frac{1}{\Gamma(-\nu)}\int_0^\infty dt\int_0^\infty \frac{dh}{h^{1+\nu}}\,F(h)\,\E\left( e^{-\lambda \tau^{(t,h)}_1-\sum_{s< \tau_1} G(\Delta B^{(t,h)}_s)}; t<\tau^{(t,h)}_1 \right)\;.
	\end{align*}
	Note that $\{t<\tau_1^{(t,h)}\}=\{t<\tau_1\}$ a.s. and by using the strong Markov property of $B$ at time $\tau_1$ we get
	\begin{align*}
		&\E\left( e^{-\lambda \tau^{(t,h)}_1-\sum_{s< \tau_1} G(\Delta B^{(t,h)}_s)}; t<\tau^{(t,h)}_1 \right) = \E\left( e^{-\lambda \tau^{(t,h)}_1-\sum_{s< \tau_1} G(\Delta B^{(t,h)}_s)}; t< \tau_1 \right) \\
		&\qquad = e^{-G(h)}\,\E\left( e^{-\lambda \tau_1-\sum_{s< \tau_1} G(\Delta B_s)}; t<\tau_1 \right)\E\left( e^{-\lambda \tau_h-\sum_{s< \tau_h} G(\Delta B_s)}\right) \\
		&\qquad = e^{-G(h)}\,\E\left( e^{-\lambda \tau_1-\sum_{s< \tau_1} G(\Delta B_s)}; t<\tau_1 \right) e^{-\rho_\lambda h}
	\end{align*}
	Plugging this into the expression for $\E(X)$ derived above yields
	\begin{align*}
		\E(X)=  \frac{1}{\Gamma(-\nu)}\int_0^\infty \frac{dh}{h^{1+\nu}}\,F(h)e^{-G(h)-\rho_\lambda h}\,\E\left( e^{-\lambda \tau_1-\sum_{s< \tau_1} G(\Delta B_s)}\tau_1 \right)\;.
	\end{align*}
	The first equality of \eqref{eq45} now follows by using that the expectation on the right-hand side of the last display is equal to $d(e^{-\rho_\lambda})/d\lambda$. 
\end{proof}

\begin{proof}[Proof of Theorem \ref{thm:levy-excursions}]
	Let $F(\ell) = f(\ell)/g(\ell)$ and $G(\ell)=-\log g(\ell)$. Let us first compute $\rho_\lambda$ for $\lambda>0$ in the setting of Lemma \ref{lem:palm-levy-excursion}. We have
	\begin{align*}
		\lambda &= \frac{1}{\Gamma(-\nu)}\int_0^\infty \frac{dh}{h^{1+\nu}} \left(e^{-\rho_\lambda h}\frac{2(ch/2)^\nu}{\Gamma(\nu)}\,K_\nu(ch) -1 +\rho_\lambda h\right) \\
		&= c^\nu 2^{1-\nu} \begin{cases}
			\cos(\nu\theta_\lambda )&\colon \text{if }\rho_\lambda\in (0,c], \text{ where } \theta_\lambda \in [0,\pi/2)\text{ is defined by } \rho_\lambda = c \cos(\theta_\lambda)\;, \\
			\cosh(\nu\theta_\lambda) &\colon \text{if } \rho_\lambda>c, \text{ where } \theta_\lambda > 0 \text{ is defined by } \rho_\lambda=c\cosh(\theta_\lambda)\;,
		\end{cases}
	\end{align*}
	by Lemma \ref{lem:residue-bessel-computation} and Euler's reflection formula. This can be inverted:
	\begin{align*}
		\rho_\lambda = c \begin{cases}
			\cos(\psi_\lambda/\nu) &\colon \text{if } \lambda\in(0,c^\nu 2^{1-\nu}],\text{ where } \psi_\lambda \in [0,\pi/2) \text{ and } \lambda= c^\nu 2^{1-\nu}\cos(\psi_\lambda) \;,\\
			\cosh(\psi_\lambda/\nu) &\colon \text{if } \lambda >c^\nu 2^{1-\nu},\text{ where } \psi_\lambda  >0 \text{ and }\lambda= c^\nu 2^{1-\nu}\cosh(\psi_\lambda)\;.
		\end{cases}
	\end{align*}
	Moreover, again by Lemma \ref{lem:residue-bessel-computation} and Euler's reflection formula
	\begin{align*}
		&\frac{1}{\nu\Gamma(1-1/\nu)} \int_0^\infty \frac{d\ell}{\ell^{1+1/\nu}}\left( e^{-\lambda\ell}\,\frac{2(c'\ell/2)^{1/\nu}}{\Gamma(1/\nu)}\,K_{1/\nu}(c'\ell)-1 \right) \\
		&\qquad = - (c')^{1/\nu}\, 2^{1-1/\nu}\cdot \begin{cases}
			\cos(\psi_\lambda/\nu) &\colon \text{if } \lambda\in(0,c'],\text{ where } \psi_\lambda \in [0,\pi/2) \text{ and } \lambda= c'\cos(\psi_\lambda) \;,\\
			\cosh(\psi_\lambda/\nu) &\colon \text{if } \lambda >c',\text{ where } \psi_\lambda  >0 \text{ and }\lambda= c'\cosh(\psi_\lambda)\;.
		\end{cases}\\
		&\qquad = -\rho_\lambda
	\end{align*}
	by the definition of $c'$. The first part of the theorem then follows from the injectivity of the Laplace transform and the first part of Lemma \ref{lem:palm-levy-excursion}. For the second part of the theorem, again by a Laplace inversion argument and the now the second part of Lemma \ref{lem:palm-levy-excursion}, it suffices to show that
	\begin{gather*}
		-\frac{d(e^{-\rho_\lambda})/d\lambda}{\Gamma(-\nu)}\int_0^\infty \frac{e^{-\rho_\lambda h}\, dh}{h^{1+\nu}} \, f(h) = \frac{e^{-\rho_\lambda}}{\nu\Gamma(1-1/\nu)}\int_0^\infty \frac{e^{-\lambda\ell}\, d\ell}{\ell^{1+1/\nu}}\, C \ell^{1+1/\nu} K_{1+1/\nu-\theta}(c'\ell)\\
		\text{where}\quad C = \frac{\sqrt{1/\nu}\,\Gamma(-1/\nu)\sin(\pi(1+1/\nu-\theta))}{\sqrt{\nu}\,\Gamma(-\nu)\sin(\pi(1+\nu-\theta\nu))}\;.
	\end{gather*}
	This can be verified directly via Lemma \ref{lem:residue-bessel-computation}. In the computation it is useful to note that we have $\psi_\lambda = \nu\theta_\lambda$.
\end{proof}

\section{Background on CLE and LQG}
\label{sec:background}

In this section we present various background materials which are needed in the remainder of the paper and in order to give precise statements of our main results. In Section \ref{sec:cle-explorations} we review results on CLE explorations from \cite{cle-percolations,msw-simple,msw-non-simple}. In Section \ref{sec:gff-lqg} we introduce the Gaussian free field and the notions of regular and generalized LQG surfaces. In Section \ref{sec:normal-lqg-disks} we define regular and generalized LQG disks. In the case of the regular disk we allow the disk to have a given number of marked points at given locations. Our definition coincides with the one given in \cite{hrv-disk} and we will consider this disk in both the regular and the generalized cases in Section \ref{sec:disk-cle-nesting}. Finally, in Section \ref{sec:background-msw} we review results from \cite{msw-simple,msw-non-simple} on explorations of CLE decorated LQG disks and explain some consequences of these results.

\subsection{CLE explorations}
\label{sec:cle-explorations}
In this text, we assume basic familiarity with SLE \cite{werner-notes,lawler-book} and CLE \cite{shef-cle,shef-werner-cle}. We will however present one of the outcomes from \cite{cle-percolations} that will be used in this paper. The ideas below are illustrated in Figure \ref{fig:looptrunk}.

\begin{figure}
	\centering
	\def\svgwidth{0.8\columnwidth}
	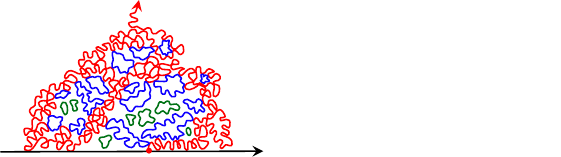
	\caption{\emph{Left.} A CPI (red) in a $\CLE_\kappa$ with $\kappa\in (8/3,4)$ and with parameter $\beta$ has $\SLE_{\kappa'}(\rho',\kappa'-6-\rho')$ as its marginal law and the exploration path is a $\SLE_\kappa^\beta(\kappa-6)$ curve. \emph{Right.} When $\kappa\in (4,8)$, then a the marginal law of a CPI (red) in a $\CLE_{\kappa'}$ and with parameter $\beta$ is a $\SLE_\kappa(\rho,\kappa-6-\rho)$ curve and the obtained exploration path is then a full $\SLE_{\kappa'}^\beta(\kappa'-6)$ curve. Here $\kappa$ and $\kappa'$ are related by $\kappa'=16/\kappa$. The blue loops are those discovered by the CPI while the green ones are not discovered by it.}
	\label{fig:looptrunk}
\end{figure}

The concept that is important is that of a continuum percolation interface (CPI). We begin with the simple case $\kappa < 4$ (the $\kappa=4$ case is slightly different \cite{cle-percolations,lehmkuehler-cle4} and will not be needed here). Fix a parameter $\beta\in [-1,1]$ and let $\Gamma$ be a non-nested $\CLE_\kappa$ in $\H$. Conditionally on $\Gamma$, we independently assign to each loop a counterclockwise orientation with probability $(1+\beta)/2$ and a clockwise orientation with probability $(1-\beta)/2$. This way, we obtain a collection $\Gamma_\beta$ of oriented loops. A CPI $\lambda$ from $0$ to $\infty$ in $\Gamma_\beta$ is a (random) continuous curve in $\overline{\H}$ that is generated by a Loewner chain, does not intersect the interior of any loops in $\Gamma_\beta$ and obeys the following properties:
\begin{enumerate}[(i)]
	\item The law of $(\Gamma_\beta,\lambda)$ satisfies scale invariance.
	\item All counterclockwise (resp.\ clockwise) loops in $\Gamma_\beta$ intersecting $\lambda$ are on the right (resp.\ on the left) of the curve $\lambda$.
	\item Fix $t\ge 0$ and let $K_t$ be the union of $\lambda([0,t])$ and the filling of all loops in $\Gamma$ intersecting $\lambda([0,t])$. We condition on $\lambda|_{[0,t]}$ and all the loops in $\Gamma_\beta$ intersecting $\lambda([0,t])$.  Given this information, the conditional law of $(\Gamma_\beta,\lambda)$ is given by sampling independent non-nested $\CLE_\kappa$ (with i.i.d. orientations as before) in the bounded complementary components of $K_t$ and an independent copy of $(\Gamma_\beta,\lambda)$ mapped conformally into the unbounded complementary component of $K_t$ where we concatenate $\lambda|_{[0,t]}$ with the curve in the mapped in domain.
\end{enumerate}
Details of this definition can be found in \cite[Section 2.3]{cle-percolations}. By \cite[Section 4]{cle-percolations} this uniquely characterizes the law of $(\Gamma_\beta,\lambda)$, however almost surely the conditional law of $\lambda$ given $\Gamma_\beta$ is non-atomic as established in \cite{msw-sle-range}. Informally this means that almost surely, we cannot deterministically obtain $\lambda$ from $\Gamma_\beta$. By \cite[Theorem 1.6]{msw-simple}, the law of $\lambda$ is a $\SLE_{\kappa'}(\rho',\kappa'-6-\rho')$ curve where $\kappa'=16/\kappa$ and
\begin{align*}
	\rho' = \frac{2}{\pi}\,\arctan\left( \frac{\sin(\pi \kappa'/2)}{1+\cos(\pi\kappa'/2)-2/(1-\beta)}\right) \in [\kappa'-6,0]
\end{align*}
where the branch of $\arctan(\cdot)$ is chosen for each value $\beta$ such that the value on the right lies in the specified interval. Moreover, the curve obtained by following $\lambda$ and tracing each loop that $\lambda$ intersects according to its orientation and in chronological order, is a so-called $\SLE_\kappa^\beta(\kappa-6)$ curve (see \cite{shef-cle} and also \cite{werner-wu-explorations}).

When $\kappa'\in (4,8)$ there is an analogous concept of a CLE percolation interface which we will also abbreviate by CPI. Again, we fix $\beta\in [-1,1]$, let $\Gamma$ be a non-nested $\CLE_{\kappa'}$ in $\H$ and obtain a CLE with orientations $\Gamma_\beta$ in precisely the same way as in the simple case. It turns out that a.s. there exists a unique simple curve $\lambda$ from $0$ to $\infty$ in $\overline{\H}$ which does not intersect the interior of any loop in $\Gamma_\beta$ and which has the property that each counterclockwise (resp.\ clockwise) loop in $\Gamma_\beta$ that intersects $\lambda$ is right (resp.\ left) of $\lambda$. When $\beta=1$ (resp.\ $\beta=-1$), $\lambda$ simply traces the negative (resp.\ positive) real line. If $\beta\in (-1,1)$ then $\lambda$ has the law of a $\SLE_\kappa(\rho,\kappa-6-\rho)$ curve where $\kappa=16/\kappa'$ and
\begin{align*}
	\rho = \frac{2}{\pi}\,\arctan\left( \frac{\sin(\pi\kappa/2)}{1+\cos(\pi\kappa/2)-2/(1-\beta)} \right) - 2\in (-2,\kappa-4)
\end{align*}
as established in \cite[Theorem 1.3]{msw-non-simple}. As in the simple case, the conditional law of $\Gamma_\beta$ given $\lambda$ and all loops in $\Gamma_\beta$ that $\lambda$ intersects is obtained by sampling independent non-nested $\CLE_{\kappa'}$ (with i.i.d. orientations as before) within each complementary component of the union of $\lambda$ and the filling of all the loops it touches (see \cite[Section 6]{cle-percolations}). Lastly, one can obtain a curve by moving along $\lambda$ and tracing a loop whenever it is first hit by $\lambda$ according to its orientation. This can be defined rigorously and the curve so obtained is a full $\SLE_{\kappa'}^\beta(\kappa'-6)$; see \cite[Section 6]{cle-percolations} for definitions and details on these results.

In the above discussion, we explained the background on CPIs in the upper half plane starting at $0$ and ending at $\infty$. If we consider an arbitrary simply connected domain with two distinguished prime ends, then by mapping the CPI and CLE to this domain along a conformal transformation which sends $0$ (resp.\ $\infty$) to the first (resp.\ second) distinguished prime end, then this allows us to generalize the above definition to the setting of arbitrary simply connected domains. We call the tuple $(\Gamma,\lambda)$ a CLE coupled with a CPI with asymmetry parameter $\beta$.

\subsection{The Gaussian free field and Liouville quantum gravity}
\label{sec:gff-lqg}

For any domain $D\subseteq \C$, let $C^\infty_c(D)'$ denote the space of distributions (also known as generalized functions) on the domain $D$ with its usual topology; the $\sigma$-algebra is generated by the evaluation maps $h\to h(f)$ for $f\in C^\infty_c(D)$. Note that this space is not metrizable and hence not Polish. For technical reasons, it will therefore be advantageous to also work with a subspace of $C^\infty_c(D)'$ on some occasions (the disadvantage of this subspace is that it is not invariant under precompositions with conformal transformations). If $D$ has harmonically non-trivial boundary, we let $H_0^1(D)$ be the Sobolev space which is the Hilbert space closure of $C^\infty_c(D)$ with respect to the inner product
\begin{align*}
	(f,g)_\nabla = \frac{1}{2\pi} \int_D \nabla f(z)\cdot \nabla g(z) \,dz
\end{align*}
for $f,g\in C_c^\infty(D)$ (the $2\pi$ factor makes the explicit Green's function expressions below simpler). Its Hilbert space dual $H^{-1}(D):= H^1_0(D)'$ embeds continuously into $C^\infty_c(D)'$. The space $H^{-1}(D)$ is a metric space and it is also separable and hence Polish.

Let $G_\text{D}\colon \D\times \D\to (0,\infty]$ be the Dirichlet Green's function on $\D$
\begin{align*}
	G_\text{D}(z,w) = -\log(|z-w|)+\log(|1-z\bar{w}|)
\end{align*}
The Dirichlet Gaussian free field on $\D$ (see \cite{shef-gff, powell-werner-gff}) is the random element $h_\text{D}$ of $H^{-1}(\D)$ such that $h_\text{D}(f)$ is a centered Gaussian random variable for each $f\in C^\infty_c(\D)$ and such that
\begin{align*}
	\E( h_\text{D}(f)h_\text{D}(g)) = \int_{\D\times \D} G_\text{D}(z,w)f(z)g(w)\,dz dw \;.
\end{align*}
One can deterministically associate to $h_\text{D}$ its circle average approximation which is a collection of random variables $((h_\text{D})_\epsilon(z)\colon \epsilon>0, z\in (1-\epsilon)\D)$ which is continuous in its parameters, see \cite{shef-kpz}. Then
\begin{align}
	\Var( (h_\text{D})_\epsilon(z) ) = \log(1/\epsilon) + R(z,\D) + o(1)\quad\text{as $\epsilon\to 0$ for all $z\in \D$}\;.
\end{align}
If $\gamma\in (0,2)$ then there is a measure $\mu^\gamma_{h_\text{D}}$ (called the $\gamma$-LQG area measure) supported in $\D$ such that whenever $f\in C_b(\C)$ (i.e.\ $f$ is continuous and bounded), we have
\begin{align*}
	\int_{(1-\epsilon)\D} f(z)\epsilon^{\gamma^2/2}e^{\gamma (h_\text{D})_\epsilon(z)}\,dz \to \int_\D f(z)\,\mu^\gamma_{h_\text{D}}(dz)
\end{align*}
as $\epsilon \to 0$ in $L^1$ and
\begin{align}
	\label{eq:gff-annealed}
	\E\left(\,\int_\D f(z)\,\mu^\gamma_{h_\text{D}}(dz)\right) = \int_\D f(z)R(z,\D)^{\gamma^2/2}\,dz\;.
\end{align}
This type of result is classical and due to Kahane \cite{kahane}, see also \cite{shef-kpz} and \cite{berestycki-gmt-elementary} (the conformal radius factor appears in order to make this definition agree with the one involving circle averages). Note that if $\mathfrak{h}$ is a random function inducing a distribution in $H^{-1}(\D)$, then (in view of the circle average definition) it is natural to define
\begin{align}
	\label{eq:extend-area-gff}
	\mu^\gamma_{h_\text{D}+\mathfrak{h}}(dz) := e^{\gamma \mathfrak{h}(z)}\,\mu^\gamma_{h_\text{D}}(dz)\;.
\end{align}
There is an analogous concept of an LQG boundary measure which we introduce in the particular case of a Neumann GFF on the unit disk. Throughout this paper, $G\colon\D\times\D\to (0,\infty]$ denotes the Green's function for the Neumann GFF in $\D$ with mean zero on $\partial\D$, namely
\begin{align*}
	G(z,w) = -\log( |z-w||1-z\bar{w}| )\;.
\end{align*}
An important property of $G$ and $G_\text{D}$ which underlies conformal covariance properties of several measures appearing in this paper is that for any Möbius transformation $f\colon \D\to \D$ we have
\begin{align*}
	G(f(z),f(w)) = G(z,w) - \log |f'(z)| - \log |f'(w)|\quad\text{and}\quad G_\text{D}(f(z),f(w)) = G_\text{D}(z,w)\;.
\end{align*}
The Neumann GFF $h^0$ on $\D$ with mean zero on $\partial \D$ is then the random element of $H^{-1}(\D)$ such that $h^0(f)$ is a centered Gaussian random variable for all $f\in C^\infty_c(\D)$ and
\begin{align*}
	\E( h^0(f)h^0(g)) = \int_{\D\times \D} G(z,w)f(z)g(w)\,dz dw \;.
\end{align*}

One key property of $h^0$ is the following spatial Markov property that it satisfies: Consider disjoint balls $B_{r_i}(z_i)\subseteq \D$ for $i=1,\dots,n$. Then there are random harmonic functions $\mathfrak{h}_i$ on $B_{r_i}(z_i)$ and i.i.d.\ Dirichlet Gaussian free fields $h_i$ in $\D$ for $i\le n$ such that $(\mathfrak{h}_i\colon i\le n)$ and $(h_i\colon i\le n)$ are independent, and
\begin{align*}
	h^0|_{\cup_i B_{r_i}(z_i)} = \sum_{i=1}^n h_i\left(\frac{\cdot - z_i}{r_i}\right) 1_{B_{r_i}(z_i)} + \sum_{i=1}^n \mathfrak{h}_i1_{B_{r_i}(z_i)}\;.
\end{align*}
We call the first summand the projection of $h^0$ onto $H^{-1}(\cup_{i} B_{r_i}(z_i))$ and the second summand the harmonic extension of $h^0$ restricted to the complement of $\cup_i B_{r_i}(z_i)$. Note that the same is true for any continuous function $\mathfrak{g}$ on $\D$: We can uniquely decompose $\mathfrak{g}$ into a function supported on $\cup_{i} B_{r_i}(z_i)$ and a continuous function which agrees with $\mathfrak{g}$ on $\D \setminus \cup_{i} B_{r_i}(z_i)$ and is harmonic on $\cup_{i} B_{r_i}(z_i)$. Thus this projection decomposition extends to the case when $h^0$ is replaced by a field which is the sum of a continuous function and an independent field which is absolutely continuous with respect to $h^0$.

To understand the LQG area and boundary measure associated to $h^0$, let us observe that
\begin{align}
	\label{eq:nm-dirichlet-gff}
	G(z,w)=G_\text{D}(z,w) -2\log(|1-z\bar{w}|)\;.
\end{align}
By the Markov property explained above, we can construct a random continuous function $\mathfrak{h}$ on $\D$ (extended to be $0$ on $\C\setminus\D$) independent of a Dirichlet GFF $h_\text{D}$ on $\D$ such that $\mathfrak{h}$ is a centered Gaussian process with
\begin{align*}
	\E(\mathfrak{h}(z)\mathfrak{h}(w))= -2\log(|1-z\bar{w}|)\quad\text{for all $z,w\in \D$}\;.
\end{align*}
By \eqref{eq:nm-dirichlet-gff} $h^0 = h_\text{D}+ \mathfrak{h}$ is then a Neumann GFF on $\D$ with mean zero on the boundary. Thus $\mu_{h^0}^\gamma$ is defined via \eqref{eq:extend-area-gff} and it follows that for each $f\in C_b(\C)$ and all $\gamma\in (0,2)$,
\begin{align*}
	\int_{(1-\epsilon)\D} f(z)\epsilon^{\gamma^2/2} e^{\gamma h^0_\epsilon(z)}\,dz \to \int_\D f(z)\,\mu^\gamma_{h^0}(dz)\quad\text{in $L^1$ as $\epsilon\to 0$}\;.
\end{align*}
We can now define the semicircle approximations $(h^0_\epsilon(z)\colon \epsilon\in (0,2), z\in \partial \D)$ to $h^0$ and $(\mathfrak{h}_\epsilon(z)\colon \epsilon\in (0,2), z\in \partial \D)$ to $\mathfrak{h}$, see again \cite{shef-kpz}. The $\gamma$-LQG boundary measure associated to $h^0$ when $\gamma\in (0,2)$ is the measure $\nu^\gamma_{h^0}$ supported on $\partial \D$ such that for all $f\in C_b(\partial \D)$,
\begin{align}
	\int_{\partial\D} f(w)\epsilon^{\gamma^2/4}e^{\gamma h^0_\epsilon(w)/2}\,dw\to \int_{\partial\D} f(w)\,\nu^\gamma_{h^0}(dw)\quad\text{in $L^1$ as $\epsilon\to 0$}\;.
	\label{eq:lqg-bdy}
\end{align}
Note that $\Var(h^0_\epsilon(z))=2\log(1/\epsilon)+ o(1)$ as $\epsilon \to 0$ and hence the expected total mass of $\nu^\gamma_{h^0}$ is $2\pi$.
It is not hard to deduce from this, that the above convergence also holds when $h^0_\epsilon$ is replaced by $\mathfrak{h}_\epsilon$ and it follows that $\nu^\gamma_{h^0}$ has a version which is measurable with respect to $\mathfrak{h}$. Again, when $\mathfrak{g}$ is a random function inducing a distribution in $H^{-1}(D)$, we set
\begin{align*}
	\nu^\gamma_{h^0+\mathfrak{g}}(dw) = e^{\gamma \mathfrak{g}(w)/2}\,\nu^\gamma_{h^0}(dw)\;.
\end{align*}

This construction can be generalized by replacing the Lebesgue measure (as the background measure) by more general measures, in particular those induced by the Minkowski content of (random) curves. This way one can associate to segments of SLE-type curves a $\gamma$-LQG length. We refer the reader to \cite{benoist-lqg-chaos} for details on this.

Let us recall the definition of a decorated $\gamma$-LQG surface (the labeling of the marked points by arbitrary index sets will simplify the notation at a later stage in the paper). Below $\partial D$ is the set of prime ends of a domain $D$ and $\bar{D}=D\cup \partial D$. In the definition below, we will also use the notation $\phi^{-1}(\bm{z}) = (\phi^{-1}(z_i)\colon i\in A)$, $\phi^{-1}(\bm{w}) = (\phi^{-1}(w_j)\colon j\in B)$ and $\phi^{-1}\circ \Gamma = \{\phi^{-1}\circ \eta'\colon \eta'\in \Gamma\}$.

\begin{defn}
	For $\gamma\in(0,2)$ we define an equivalence relation on tuples $(D,h,\z,\w,\eta,\Gamma)$, where $D\subseteq\C$ is simply connected, $h\in C^\infty_c(D)'$, $\bm{z}=(z_i\colon i\in A)$ with $z_i \in D$ for $i\in A\subseteq \N$, $\bm{w}=(w_j\colon j\in B)$ with $w_j \in \partial D$ for $j\in B\subseteq \N$, $\eta$ is either a placeholder value $-$ or a continuous curve in $\bar{D}$ viewed up to reparametrization, and $\Gamma$ is a collection of loops (each loop viewed up to reparametrization). We let
	\begin{align*}
		(D,h,\bm{z},\bm{w},\eta,\Gamma)\sim_\gamma(D',h\circ\phi+Q_\gamma\log|\phi'|, \phi^{-1}(\bm{z}), \phi^{-1}(\bm{w}), \phi^{-1}\circ \eta, \phi^{-1}\circ \Gamma)
	\end{align*}
	if $\phi:D'\to D$ is a conformal map. A regular decorated $\gamma$-LQG surface is an equivalence class of tuples $(D,h,\bm{z},\bm{w},\eta,\Gamma)$. We let $[(D,h,\bm{z},\bm{w},\eta,\Gamma)]$ denote the equivalence class associated to $(D,h,\bm{z},\bm{w},\eta,\Gamma)$. We write $-$ for the empty tuple and the empty collection and use the shorthand notation $(D,h,\bm{z},\bm{w},-,-)=(D,h,\bm{z},\bm{w})$. We call $[(D,h,\bm{z},\bm{w})]$ a regular $\gamma$-LQG surface.
	\label{def:lqg-surface}
\end{defn}

\begin{figure}
	\centering
	\def\svgwidth{0.7\columnwidth}
	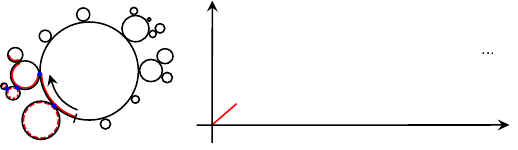
	\caption{We illustrate the encoding of a looptree with a distinguished point (marked with a line segment) when the total length of all the loops in the looptree is finite. We trace the loops in clockwise order starting at the distinguished point. In this simple example, the process $e$ increases linearly with slope $1$ except when a loop is completed in which case it jumps down by the height of the loop. The dotted lines represent the completed loops at this point of the exploration and the blue points correspond the start of tracing a new loop.}
	\label{fig:looptree-def}
\end{figure}
The motivation behind the above definition (in particular, for the transformation rule for the field $h$) is that the $\gamma$-LQG area measure and $\gamma$-LQG boundary measure are invariant under the map $\phi$ in the following sense 
\begin{align}
	\mu^\gamma_h = \phi_*\mu^\gamma_{h\circ\phi+Q_\gamma\log|\phi'|}\;,
	\qquad 
	\nu^\gamma_h = \phi_*\nu^\gamma_{h\circ\phi+Q_\gamma\log|\phi'|}\;. 
	\label{eq:meas-inv} 
\end{align}
In fact, above we defined $\nu^\gamma_h$ only for the unit disk, and by using that \eqref{eq:meas-inv} holds for the unit disk we get a natural definition of $\nu^\gamma_h$ for a general simply connected domain $D$ by defining
\begin{align*}
	\nu^\gamma_h := \phi_*\nu^\gamma_{h\circ\phi+Q_\gamma\log|\phi'|}\;,
\end{align*}
where $\phi:\D\to D$ is conformal. See  \cite{shef-wang-lqg-coord} for a proof that \eqref{eq:meas-inv} holds simultaneously for all conformal transformations $\phi$.

What the definition of $\gamma$-LQG surface above excludes is the possibility of pinch points that the surface might have. The definition of such generalized $\gamma$-LQG surfaces relies on the notion of looptrees as introduced in \cite{curien-kortchemski-looptree-def} (see also \cite[Section 10]{wedges} where they appear in the context of forested lines). We will recall the motivation for the following definition below.

\begin{defn}
	\label{def:general-lqg-surface}
	A generalized (decorated) $\gamma$-LQG surface is a pair
	\begin{align*}
		S=(e,\{(t,S_t)\colon t\ge 0, \Delta e_t \neq 0\})
	\end{align*}
	where $e\colon [0,\infty)\to [0,\infty)$ is a non-constant càdlàg function with only negative jumps such that $e_0=0$ and $e_t=0$ for all $t\ge \zeta_e=\inf\{s>0\colon e_s=0\}$, and $S_t$ are regular (decorated) $\gamma$-LQG surfaces whenever $\Delta e_t=e_t-e_{t-} \neq 0$. We call $\zeta_e$ the boundary length of $S$.
\end{defn}

We will now explain how $e$ encodes a looptree. The construction is illustrated in Figure \ref{fig:looptree-def}. We define an equivalence relation $\sim_e$ on the interval $[0,\zeta_e]$ by specifying
\begin{align*}
	t\sim_e \sup\{s<t\colon e_s\le e_t\}\quad\text{when}\quad \Delta e_t\neq 0\;.
\end{align*}
The quotient topological space $[0,\zeta_e]/\!\!\sim_e$ is still a metric space (i.e., the quotient pseudometric is a metric) and is called the looptree associated to the function $e$. Moreover, this topological space has a distinguished point, namely the equivalence class $[0]$. Also, each time $t\in [0,\zeta_e]$ with $\Delta e_t \neq 0$ corresponds to one particular loop in the looptree and each of these loops has a distinguished point $[t]$ on it. The intuitive picture of $S$ is now that we \enquote{glue} each $S_t$ when $\Delta e_t\neq 0$ into the loop corresponding to time $t$ according to the boundary length of the quantum surface; to fix this gluing uniquely, it is performed so that the first marked boundary point of each quantum surface is glued to the distinguished point on the respective loop (needless to say, this assumes that each quantum surface has at least one marked boundary point). Moreover, if some of the decorated quantum surfaces have a continuous curve as part of its decoration, the picture is that we concatenate these continuous curve segments so that the curve starts at $[0]$ (provided this is possible).

\subsection{Regular and generalized Liouville quantum gravity disks}
\label{sec:normal-lqg-disks}

The Liouville quantum gravity (LQG) disk is an LQG surface of particular importance since it arises as the scaling limit of random planar maps with disk topology, see Section \ref{sec:discrete}. As we will see, when this disk is coupled to an independent conformal loop ensemble (CLE) and reweighted according to the CLE nesting statistics around $n$ marked points, the field of the disk will get logarithmic singularities at the marked points. It turns out that the latter field is precisely equal to a field which arises in Liouville conformal field theory (LCFT) on the disk as defined by Huang, Rhodes, and Vargas \cite{hrv-disk}. In Definition \ref{def:lcft-field} right below we give their definition of the boundary length $\ell$ Liouville field with marked points.
\begin{defn}
	Suppose that $\alpha_i\in\R$ for $i\in A$ and $\beta_j\in\R$ for $j\in B$ for some finite possibly empty index sets $A,B\subseteq \N$ satisfy
	\begin{align}
		\alpha_i<Q_\gamma\,\forall i,\qquad
		\beta_j<Q_\gamma\,\forall j,\qquad Q_\gamma-\sum_{i\in A} \alpha_i-\frac{1}{2} \sum_{j\in B}\beta_j < \frac{2}{\gamma}\wedge\min\{Q_\gamma -\beta_j\colon j\in B\}\;.
		\label{eq:seiberg}
	\end{align}
	Let $h^0$ denote a Neumann GFF in $\D$ with mean zero on $\partial\D$. For distinct points $z_i\in\D$ where $i\in A$ and $w_j\in\partial\D$ for $j\in B$ define
	\begin{align*}
	h = h^0
	+ \sum_{i\in A} \alpha_i G(\cdot,z_i)
	+ \sum_{j\in B}\frac{\beta_j}{2} G(\cdot,w_j),
	\qquad
	 h_*= h-\frac{2}{\gamma}\log\nu_{h}^\gamma(\partial\D) + \frac{2}{\gamma}\log \ell\;.
	\end{align*}
	For $\Lambda \ge0$ and $\ell>0$ we define a probability measure by
	\begin{align}
		P^{(\bm{\alpha},\bm{z}),(\bm{\beta},\bm{w})}_{\Lambda,\ell}(X) := \frac{C^{(\bm{\alpha},\bm{z})}_{(\bm{\beta},\bm{w})}\E\left( 1_X\left(h_*\right)e^{-\Lambda \mu_{h_*}^\gamma(\D)}\nu_h^\gamma(\partial \D)^{ 2Q_\gamma/\gamma-\sum_{i\in A}2\alpha_i/\gamma-\sum_{j\in B}\beta_j/\gamma} \right) }{Z^{(\bm{\alpha},\bm{z}),(\bm{\beta},\bm{w})}_{\Lambda,\ell}}
		\label{eq:reweighted-measure}
	\end{align}
	for $X\subseteq C^\infty_c(D)'$ measurable where
	\begin{align*}
		C^{(\bm{\alpha},\bm{z})}_{(\bm{\beta},\bm{w})} = \prod_{i\in A} R(z_i,\D)^{-\alpha_i^2/2}\prod_{\substack{i,i'\in A\\ i< i'}} e^{\alpha_i\alpha_{i'} G(z_i,z_{i'})}\prod_{\substack{j,j'\in B\\ j< j'}} e^{\beta_j/2\cdot \beta_{j'}/2\cdot G(w_j,w_{j'})} \prod_{\substack{i\in A\\ j\in B}} e^{\alpha_i \beta_j/2\cdot G(z_i,w_j)}
	\end{align*}
	and where the denominator in \eqref{eq:reweighted-measure} is the normalizing constant which ensures that \eqref{eq:reweighted-measure} defines a probability measure. The measure is called the Liouville (disk) field with marked points $((z_i,\alpha_i)\colon i\in A)$, $((w_j,\beta_j)\colon j\in B)$, cosmological constant $\Lambda$ and boundary length $\ell$.
	\label{def:lcft-field}
\end{defn}

\begin{remark}
	It follows from the proof of \cite[Corollary 3.10]{hrv-disk} that if \eqref{eq:seiberg} holds, then the boundary length $\nu_h(\partial \D)$ has finite moments of sufficiently high order so that \eqref{eq:reweighted-measure} is well-defined as a probability measure.\footnote{\,There is a typo in the statement of \cite[Proposition 3.2]{hrv-disk} where there is an incorrect minus sign in the exponent of the Poincaré metric at an interior point; in our setting the correct definition to make things consistent with the remainder of \cite{hrv-disk} is given above where the exponent on the conformal radius is negative.} The cited corollary is for the case $\Lambda=0$ but having $\Lambda>0$ is only making the expectation smaller; however we require $\alpha_i<Q_\gamma$ such that $\mu_{h_*}^\gamma(\D)<\infty$. We also remark that in the setting of no bulk points ($A=\emptyset$) and $\beta_j=\gamma$ for all $j\in B$ we need to have $\#B\geq 3$ in order for \eqref{eq:seiberg} to be satisfied.
	\label{rmk:seiberg}
\end{remark}

One natural question is whether there are relations between disks with different singularities. The proposition below resolves this question in the special case where we show how to construct a Liouville field with one $\alpha$ bulk singularity and one $\gamma$ boundary singularity by starting with a Liouville field with no bulk singularities and three $\gamma$ boundary singularities, forgetting two of the boundary marked points and sampling a bulk singularity  essentially from the $\alpha$-LQG measure. In Section \ref{sec:ssw} and in Section \ref{sec:disk}, this proposition will be key as it allows us to explicitly compute some partition functions and in particular establish their finiteness.

\begin{prop}
	\label{prop:resample-lcft-point}
	Consider distinct and counterclockwise oriented points $w_a,w_b,w_c\in \partial \D$. Whenever $w_a=\widetilde{w}_a,\widetilde{w}_b,\widetilde{w}_c$ are distinct, let $\psi_{\bm{\wt{w}}}\colon \D\to \D$ be the unique conformal or anticonformal transformation mapping $\bm{\wt{w}} = (\widetilde{w}_a,\widetilde{w}_b,\widetilde{w}_c)$ to $\bm{w} = (w_a,w_b,w_c)$. Then
	\begin{align*}
		&Z^{(\alpha,0),(\gamma,1)}_{\Lambda,\ell} \int P^{(\alpha,0),(\gamma,1)}_{\Lambda,\ell}(dk) \int_{\partial \D} \nu_{k}^\gamma(d\wt{w}_b)\int_{\partial\D} \nu_{k}^\gamma(d\widetilde{w}_c)\, f(k\circ \psi_{\bm{\wt{w}}}^{-1}+Q_\gamma \log |(\psi_{\bm{\wt{w}}}^{-1})'|,\psi_{\bm{\wt{w}}}(0)) \\
		&\qquad = 2c_{\bm{w}}\ell^2\int_{\D}\,dz\, R(z,\D)^{\alpha(Q_\gamma - Q_\alpha)} Z^{(\alpha,z),((\gamma,\gamma,\gamma),\bm{w})}_{\Lambda,\ell} \int P^{(\alpha,z),((\gamma,\gamma,\gamma),\bm{w})}_{\Lambda,\ell}(dk) f(k,z)
	\end{align*}
	where $c_{\bm{w}}=\exp(-(G(w_a,w_b)+G(w_b,w_c)+G(w_c,w_a))/2)$. Moreover, if $\alpha<2$ we also have
	\begin{align*}
		&\int_{\D}\,dz\, R(z,\D)^{\alpha(Q_\gamma - Q_\alpha)} Z^{(\alpha,z),((\gamma,\gamma,\gamma),\bm{w})}_{\Lambda,\ell} \int P^{(\alpha,z),((\gamma,\gamma,\gamma),\bm{w})}_{\Lambda,\ell}(dk) f(k,z) \\
		&\quad = Z^{(-,-),((\gamma,\gamma,\gamma),\bm{w})}_{\Lambda,\ell}\int P^{(-,-),((\gamma,\gamma,\gamma),\bm{w})}_{\Lambda,\ell}(dk) \int_\D \mu_k^\alpha(dz)\,R(z,\D)^{\alpha(Q_\gamma - Q_\alpha)} f(k,z)\;.
	\end{align*}
\end{prop}

The key in the proof of the proposition above is the following elementary change of variables result. It will allow us in the proof of the proposition to switch from an integral over an interior point to an integral over two boundary points.

\begin{lemma}
	\label{lem:simple-integral-change}
	Let $g\colon \D\to [0,\infty]$ be measurable. Let $I$ be the set of distinct $(u,v)\in(\partial\D\setminus\{1\})^2$ Fix $(w_-,w_+)\in I$ such that $(1,w_-,w_+)$ is ordered counterclockwise. Then
	\begin{align*}
		&\int_I g(\psi_{\wt{w}_\pm}(0))|\psi_{\wt{w}_\pm}'(1)\psi_{\wt{w}_\pm}'(\widetilde{w}_b)\psi_{\wt{w}_\pm}'(\widetilde{w}_c)| \cdot |\psi_{\wt{w}_\pm}'(0)|^2 \,d\widetilde{w}_-\,d\widetilde{w}_+ \\
		&\qquad = 2e^{-(G(1,w_-)+G(1,w_+)+G(w_-,w_+))/2}\int_\D g(z)\,dz
	\end{align*}
	where $\psi_{\wt{w}_\pm}\colon \D\to\D$ is the unique conformal or anticonformal transformation mapping $(1,\widetilde{w}_-,\widetilde{w}_+)$ to $(1,w_-,w_+)$ whenever $(\widetilde{w}_-,\widetilde{w}_+)\in I$.
\end{lemma}

\begin{proof}
	Let $f\colon \H\to \D$ given by $f(u)=(u-i)/(u+i)$ be the Möbius transformation mapping $(\infty,i)$ to $(1,0)$ with inverse $f^{-1}(z) = i(1+z)/(1-z)$. By performing a change of coordinates, for any measurable function $\widetilde{g}\colon \H\to [0,\infty]$ and $x_-<x_+$ we have
	\begin{align}
		\label{eq:first-parm-coord-change}
		\begin{split}
		&\int_\H \widetilde{g}(u)\,du = \int_{\R^2} \widetilde{g}\left(\psi^*_{\wt{x}_\pm}(i)\right) \frac{(x_+-x_-)^2}{(\widetilde{x}_+-\widetilde{x}_-)^3}\,1(\widetilde{x}_-<\widetilde{x}_+)\,d\widetilde{x}_- d\widetilde{x}_+\\
		&\quad\text{where}\quad \psi^*_{\wt{x}_\pm} = \frac{(x_+-x_-)(\cdot)+x_-\widetilde{x}_+ - x_+\widetilde{x}_-}{\widetilde{x}_+ - \widetilde{x}_-}\;.
		\end{split}
	\end{align}
	Let $x_\pm = f^{-1}(w_\pm)$. If $\widetilde{x}_\pm = f^{-1}(\widetilde{w}_\pm)$ then $f\circ \psi^*_{\wt{x}_\pm}=\psi_{\wt{w}_\pm}\circ f$ and hence
	\begin{align}
		\label{eq:derivative-relations}
		\begin{split}
		&\psi'_{\wt{w}_\pm}(\widetilde{w}_+) = \frac{f'(x_+)}{f'(\widetilde{x}_+)}\,\frac{x_+-x_-}{\widetilde{x}_+-\widetilde{x}_-}\;,\quad \psi'_{\wt{w}_\pm}(\widetilde{w}_-) = \frac{f'(x_-)}{f'(\widetilde{x}_-)}\,\frac{x_+-x_-}{\widetilde{x}_+-\widetilde{x}_-}\;,\\
		&\psi'_{\wt{w}_\pm}(0) = \frac{f'(\psi^*_{\wt{x}_\pm}(i))}{f'(i)}\,\frac{x_+-x_-}{\widetilde{x}_+-\widetilde{x}_-}\;,\quad \psi'_{\wt{w}_\pm}(1)= \frac{\widetilde{x}_+-\widetilde{x}_-}{x_+-x_-}\;.
		\end{split}
	\end{align}
	We now consider $\widetilde{g}=(g\circ f)\cdot|f'|^2$ and let $I'$ be the set of $(u,v)\in I$ such that $(1,u,v)$ is counterclockwise ordered. Then by two changes of coordinates and \eqref{eq:first-parm-coord-change} we get
	\begin{align*}
		&\int_\D g(z)\,dz = \int_\H g(f(u))|f'(u)|^2\,du \\
		&= \int_{\R^2} g((f\circ \psi^*_{\wt{x}_\pm})(i)) |f'(\psi^*_{\wt{x}_\pm}(i))|^2 \,\frac{(x_+-x_-)^2}{(\widetilde{x}_+-\widetilde{x}_-)^3}\,1(\widetilde{x}_-<\widetilde{x}_+)\,d\widetilde{x}_- d\widetilde{x}_+ \\
		&= \int_{I'} g((f\circ \psi^*_{f^{-1}(\wt{w}_\pm)})(i))\,|f'(\psi^*_{f^{-1}(\wt{w}_\pm)}(i))|^2\,\frac{(x_+-x_-)^2 |(f^{-1})'(\widetilde{w}_-)|\cdot |(f^{-1})'(\widetilde{w}_+)|}{(f^{-1}(\widetilde{w}_+)-f^{-1}(\widetilde{w}_-))^3}\,d\widetilde{w}_-\,d\widetilde{w}_+ \;.
	\end{align*}
	Note that $(f\circ \psi^*_{f^{-1}(\wt{w}_\pm)})(i)=\psi_{\wt{w}_\pm}(0)$ and if $\widetilde{x}_\pm = f^{-1}(\widetilde{w}_\pm)$ then by \eqref{eq:derivative-relations},
	\begin{align*}
		&|f'(\psi^*_{\wt{x}_\pm}(i))|^2\,\frac{(x_+-x_-)^2 |(f^{-1})'(\widetilde{w}_-)|\cdot |(f^{-1})'(\widetilde{w}_+)|}{(\widetilde{x}_+-\widetilde{x}_-)^3} \\
		&= |\psi'_{\wt{w}_\pm}(\widetilde{w}_+)|\cdot |\psi'_{\wt{w}_\pm}(\widetilde{w}_-)|\cdot |\psi'_{\wt{w}_\pm}(1)|\cdot |\psi'_{\wt{w}_\pm}(0)|^2\,\frac{ |f'(i)|}{|f'(x_-)|\cdot |f'(x_+)|(x_+-x_-)}\;.
	\end{align*}
	Finally
	\begin{align*}
		\frac{ |f'(i)|}{|f'(x_-)|\, |f'(x_+)|(x_+-x_-)} &= \frac{|(f^{-1})'(w_-)|\cdot |(f^{-1})'(w_+)|}{2(f^{-1}(w_+)-f^{-1}(w_-))} = \frac{1}{|1-w_-|\,|1-w_+|\, |w_--w_+|} \\
		&= e^{(G(1,w_-)+G(1,w_+)+G(w_-,w_+))/2}\;.
	\end{align*}
	Putting everything together yields the claim. Note that the extra factor of $2$ in the statement arises since on the right-hand side of \eqref{eq:first-parm-coord-change} we only consider tuples $(\widetilde{x}_-,\widetilde{x}_+)$ that give rise to a conformal (rather than anticonformal) transformation $\psi_{\wt{x}_\pm}^*$.
\end{proof}

\begin{proof}[Proof of Proposition \ref{prop:resample-lcft-point}]
	Without loss of generality we may assume that $\Lambda =0$ and $\ell=1$; the general case follows by scaling and the fact that $\mu_k^\gamma(\D)$ is a measurable function of the field $k$. Also, without loss of generality $w_a=1$. In the notation of Definition \ref{def:lcft-field} and by Girsanov's theorem we get
	\begin{align*}
		&\E\left(\int_{\partial \D}\int_{\partial \D}\nu_{h^0}^\gamma(d\widetilde{w}_b)\nu_{h^0}^\gamma(d\widetilde{w}_c) g(\widetilde{w}_b,\widetilde{w}_c,h^0)\right) \\
		&\quad = \int_{\partial \D}\int_{\partial \D}e^{(\gamma/2)^2G(\wt{w}_b,\wt{w}_c)}\E\left( g(\widetilde{w}_b,\widetilde{w}_c,h+\gamma/2\cdot G(\cdot,\widetilde{w}_b)+\gamma/2\cdot G(\cdot,\widetilde{w}_c))\right) \,d\widetilde{w}_b\,d\widetilde{w}_c
	\end{align*}
	for any measurable function $g\colon (\partial \D)^2\times C^\infty_c(D)'\to [0,\infty]$; this in particular yields
	\begin{align*}
		&Z^{(\alpha,0),(\gamma,1)}_{0,1} \int P^{(\alpha,0),(\gamma,1)}_{0,1}(dk) \int_{\partial \D} \nu_{k}^\gamma(d\wt{w}_b)\int_{\partial\D} \nu_{k}^\gamma(d\widetilde{w}_c)\, f(k\circ \psi_{\bm{\wt{w}}}^{-1}+Q_\gamma \log |(\psi_{\bm{\wt{w}}}^{-1})'|,\psi_{\bm{\wt{w}}}(0)) \\
		& = \int_{(\partial \D\setminus\{1\})^2} \,d\widetilde{w}_b\,d\widetilde{w}_c\, Z^{(\alpha,0),((\gamma,\gamma,\gamma),\bm{\wt{w}})}_{0,1} \int P^{(\alpha,0),((\gamma,\gamma,\gamma),\bm{\wt{w}})}_{0,1}(dk) f(k\circ \psi_{\bm{\wt{w}}}^{-1}+Q_\gamma \log |(\psi_{\bm{\wt{w}}}^{-1})'|,\psi_{\bm{\wt{w}}}(0)) \\
		& = \int_{(\partial \D\setminus\{1\})^2} \,d\widetilde{w}_b\,d\widetilde{w}_c\, \prod_i |\psi_{\bm{\wt{w}}}'(\widetilde{w}_i)|\cdot |\psi_{\bm{\wt{w}}}'(0)|^{\alpha(Q_\gamma-\alpha/2)}\\
		&\qquad\qquad\cdot Z^{(\alpha,\psi_{\bm{\wt{w}}}(0)),((\gamma,\gamma,\gamma),\bm{w})}_{0,1} \int P^{(\alpha,\psi_{\bm{\wt{w}}}(0)),((\gamma,\gamma,\gamma),\bm{w})}_{0,1}(dk) f(k,\psi_{\bm{\wt{w}}}(0))
	\end{align*}
	where the second equality follows from the change of coordinates formula \cite[Theorem 3.5]{hrv-disk}. Note that
	\begin{align*}
	|\psi_{\bm{\wt{w}}}'(0)|^{\alpha(Q_\gamma-\alpha/2)} = R(\psi_{\bm{\wt{w}}}(0),\D)^{\alpha(Q_\gamma-Q_\alpha)}|\psi_{\bm{\wt{w}}}'(0)|^2
	\end{align*}
	The first claim now follows from Lemma \ref{lem:simple-integral-change}. The second claim follows from another application of Girsanov's theorem.
\end{proof}

\begin{defn}
	\label{def:normal-unit-disk}
	Let $k$ be the Liouville field with $\#A=0$, $\#B=3$, $\beta_j=\gamma$ for all $j\in B$, where the singularities are located at distinct points $w_j\in \partial \D$ for $j\in B$ and we have boundary length $\ell$ and cosmological constant $0$. Then the $\gamma$-LQG surface $[(\D,k,\w)]$ is called the (regular) $\gamma$-LQG disk with boundary length $\ell$ and three marked boundary points. We get the (regular) $\gamma$-LQG disk with boundary length $\ell$ and zero (resp.\ one, two) marked points by keeping zero (resp.\ one, two) of the marked points.
\end{defn}

It follows from \cite[Theorem 3.5]{hrv-disk} that the law of $[(\D,k,\bm{w})]$ does not depend on the locations of the marked points $(w_j\colon j\in B)$. The paper \cite{cercle-quantum-disk} shows that this definition agrees with the one given in \cite[Definition 4.21]{wedges}.

We will now define the generalized $\gamma$-LQG disk with (generalized) boundary length $\ell>0$. This will be a generalized $\gamma$-LQG surface in the sense of Definition \ref{def:general-lqg-surface}. Recall the definition of spectrally positive stable Lévy excursions from Section \ref{sec:levy-exc}.

\begin{defn}
	\label{def:gen-disk-no-points}
	Suppose that $\sqrt{2}<\gamma<2$ and $\ell>0$. Let $(E_t\colon t\in[0,\ell])$ be the time reversal of a spectrally positive stable Lévy excursion with exponent $4/\gamma^2$ and duration $\ell$. Conditionally on $E$, let $S_t$ be independent $\gamma$-LQG disks with one marked point sampled from the $\gamma$-LQG boundary measure and with boundary lengths $|\Delta E_t|$ whenever $t\in(0,\ell)$ is such that $\Delta E_t\neq 0$. Then
	\begin{align*}
		(E,\{ (t, S_t) \colon t<\ell, \Delta E_t\neq 0 \})
	\end{align*}
	is called the generalized $\gamma$-LQG disk with boundary length $\ell$.
\end{defn}

We emphasize the time reversal of the excursion in the definition. Indeed, a positive excursion with only positive jumps is easier to construct, but the process with only negative jumps is more natural from the perspective of the looptree.

\subsection{Exploring CLE decorated LQG disks}
\label{sec:background-msw}

\begin{figure}
	\centering
	\def\svgwidth{0.8\columnwidth}
	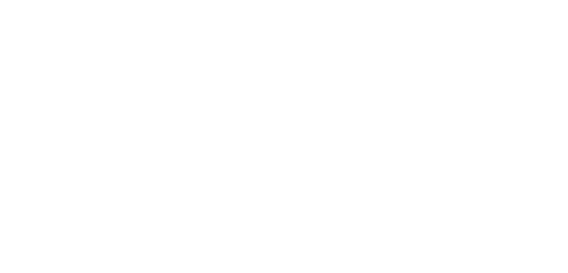
	\caption{\emph{Left.} We consider a regular $\gamma$-LQG disk decorated by a $\CLE_\kappa$ (red loops) and a CPI with asymmetry parameter $\beta$; the CPI up to some time $t$ is colored in blue and the remaining part of the CPI is colored in green. The future part of the disk at time $t$ (which is the surface in which the exploration continues) is shown in light blue. Note that the disks colored in white are precisely the disks that touch the CPI up to time $t$. \emph{Right.} The figure illustrates a generalized $\gamma$-LQG disk decorated by a $\CLE_{\kappa'}$ and a CPI with asymmetry parameter $\beta$ with the same coloring convention as in the simple case.}
	\label{fig:msw-review}
\end{figure}

In this section, we briefly review the results by Miller, Sheffield and Werner on exploring LQG surfaces decorated by independent CLE. The case where the quantum disk is decorated by simple CLE appears in their paper \cite{msw-simple} and the case where the generalized quantum disk is decorated by a non-simple CLE appears in \cite{msw-non-simple}.

We begin by explaining the simple case. See the left-hand side of Figure \ref{fig:msw-review} for an illustration of the ideas below. We consider parameters
\begin{align*}
	\sqrt{8/3}<\gamma<2\;,\quad \kappa=\gamma^2\;,\quad \beta\in [-1,1]\quad\text{and}\quad \ell_L,\ell_R>0 \;.
\end{align*}
Let $(\D,h)$ be any embedding of a regular $\gamma$-LQG disk $S=[(\D,h)]$ with boundary length $\ell=\ell_L+\ell_R$ into $\D$. Let $w_0\in \partial \D$ be chosen uniformly according to the $\gamma$-LQG boundary measure and let $w_\infty\in \partial D$ be such that the clockwise boundary arc from $w_0$ to $w_\infty$ has length $\ell_L$. Then the counterclockwise arc from $w_0$ to $w_\infty$ has length $\ell_R$.

Conditionally on $h$ and on $(w_0,w_\infty)$, let $(\Gamma,\lambda)$ be a non-nested $\CLE_\kappa$ in $\D$ coupled with a CPI from $w_0$ to $w_\infty$ with asymmetry parameter $\beta$ (see Section \ref{sec:cle-explorations}). By the conformal invariance of the law of $(\Gamma,\lambda)$, the law of the decorated regular $\gamma$-LQG surface $[(\D,h,w_0,w_\infty,\lambda,\Gamma)]$ does not depend on the embedding $(\D,h)$ of the quantum disk. We may assume that $\lambda$ is parametrized according to $\gamma$-LQG length and call its total duration $\zeta$. Moreover, for each $t< \zeta$ we write $\xi(t)$ for the union of $\lambda([0,t])$ and all the loops in $\Gamma$ touching it (see \eqref{eq:xi0}). Denote by $D_t$ the complementary component of $\D\setminus \xi(t)$ containing the boundary point $w_\infty$.

Let $L_t$ (resp.\ $R_t$) denote the $\gamma$-LQG length of the clockwise (resp.\ counterclockwise) boundary arc of $D_t$ from $\lambda_t$ to $b$. By convention, we set $L_t=R_t=0$ for $t\ge \zeta$. It turns out that this process has a càdlàg version $((L_t,R_t)\colon t\ge 0)$ which takes values in $(0,\infty)^2\cup\{(0,0)\}$. Let $X=L+R$. To help the reader compare the construction here to the one in the non-simple case presented below, let us also define the regular $\gamma$-LQG surface $S_t := [(D_t,h|_{D_t})]$.

When $\Delta X_t:=X_t-X_{t-}\neq 0$ we let $\Delta D_t = (\cap_{t'<t} D_{t'})\setminus D_t$ and we call the quantum surface $\Delta S_t := [(\Delta D_t,h|_{\Delta D_t})]$ the cut out quantum surface at time $t$.  The surface $\Delta S_t$ has boundary length $|\Delta X_t|$.
Note that $L$ (resp.\ $R$) has a positive jump at time $t$ if $\lambda$ hits a clockwise (resp.\ counterclockwise) oriented loop $\eta$ for the first time at time $t$, and we have $\Delta D_t=\eta^o$ and hence $\Delta S_t = [(\eta^o,h|_{\eta^o})]$ in this case. On the other hand, $L$ (resp.\ $R$) has a negative jump at time $t$ if $\lambda$ hits the boundary of its past hull at time $t$ on the left (resp.\ right) side. In the latter case, we call $\Delta D_t$ a trunk component.

\begin{thm}[\cite{msw-simple}]
	\label{thm:msw-simple}
	There exists a constant $v_\kappa^\beta \in (0,\infty)$ such that
	\begin{align*}
		(L,R) \stackrel{d}{=} ((L_{v_\kappa^\beta t}',R'_{v_\kappa^\beta t})\colon t\ge 0)
	\end{align*}
	where $(L',R')$ is defined in Proposition \ref{prop:simple-msw-process}. Moreover, conditionally on $(L,R)$ the cut out quantum surfaces which correspond to jump times $t\ge 0$ with $\Delta X_t\neq 0$ are independent regular $\gamma$-LQG disks with boundary lengths $|\Delta X_t|$.
\end{thm}

\begin{remark}
	In Section \ref{sec:fragmentation-main} the variables $L$ and $R$ were used to denote two independent Lévy processes. In all other parts of the paper, including this section, $L$ and $R$ denote boundary length processes within LQG disks.
\end{remark}

In order to use LQG tools to understand CLE, it is useful to split a regular $\gamma$-LQG disk decorated by a CLE into the components obtained by restricting the surface to all the individual CLE loops. The following corollary explains how one can evaluate functionals depending on the restrictions to the CLE loops.

\begin{cor}
	\label{cor:msw-simple-full}
	Let $\Psi$ be a measurable function on the space of regular $\gamma$-LQG surfaces with values in $[0,\infty]$. Consider a regular $\gamma$-LQG disk $S$ with boundary length $\ell>0$ and decorated with an independent non-nested $\CLE_\kappa$. Let $(S_n\colon n\ge 1)$ be the $\gamma$-LQG surfaces cut out by the CLE loops with some arbitrary ordering. Then
	\begin{gather*}
		\E_\ell\left(\sum_{n\ge 1} \Psi(S_n)\right)=\sum_{k\ge 0} f_k(\ell)\quad\text{where}\quad  f_0(\ell) :=  \E_{(\ell/2,\ell/2)}\left(\sum_{t\in(0,\zeta)\colon \Delta X_t>0} \Psi(\Delta S_t) \right)
		 \\
		 \text{and}\quad f_k(\ell) := \E_{(\ell/2,\ell/2)}\left( \sum_{t\in (0,\zeta)\colon \Delta X_t<0} f_{k-1}(|\Delta X_t|) \right)
	\end{gather*}
	for $k\ge 1$. The subscript $\ell$ on the left-hand side of the statement indicates the total boundary length of the $\gamma$-LQG disk $S$ and a subscript of the form $(\ell_L,\ell_R)$ indicates that the CPI exploring the LQG disk is started with boundary lengths $(\ell_L,\ell_R)$.
\end{cor}

\begin{proof}
	The proof is based on an exploration procedure for iteratively discovering all the $\gamma$-LQG surfaces cut out by the CLE loops on $S$.
	
	We first pick a boundary point from $S$ uniformly according to the $\gamma$-LQG boundary measure (independently of the CLE) and consider the unique point at distance $\ell/2$ along the boundary away from it. We then sample a CPI with asymmetry parameter $\beta$ within the CLE and let $(L,R)$, $X$, and $\Delta S_t$ be defined as in the discussion above. Then the $\gamma$-LQG surfaces $\Delta S_t$ in the case $\Delta X_t>0$ correspond to the case of CLE loops intersecting the CPI being cut out.
	
	Within each of the trunk components $\Delta S_t$ (i.e., $\Delta X_t<0$) we now independently sample a boundary point uniformly from its $\gamma$-LQG boundary measure and subsequently also consider the unique point at distance $|\Delta X_t|/2$ (along the boundary of the trunk component) away from this point. We can now again consider a CPI with asymmetry parameter $\beta$ from the first to the second of these boundary points in each of the trunk components. Again, the CLE loops intersecting the newly constructed CPI cut out a collection of $\gamma$-LQG surfaces and we also obtain a collection of trunk components for each of the CPIs.
	
	By iterating this procedure within each of the trunk components, we obtain a collection of $\gamma$-LQG surfaces $(S_n'\colon n\ge 1)$. Moreover, by construction and Theorem \ref{thm:msw-simple} we have
	\begin{align}
		\label{eq:explored-explicit}
		\E_\ell\left(\sum_{n\ge 1} \Psi(S'_n)\right)&=\sum_{k\ge 1} f_k(\ell)\;.
	\end{align}
	Clearly, $(S_n'\colon n\ge 1)$ forms a subcollection of $(S_n\colon n\ge 1)$, so to establish the corollary, it suffices to show that these collections are identical. 
	To prove this, it is sufficient to establish the following identity
	\begin{align}
		\label{eq:explored-vs-unexplored}
		\E_\ell\left(\sum_{n\ge 1} A(S'_n)\right) = \E_\ell\left(\sum_{n\ge 1} A(S_n)\right)
	\end{align}
	where $A([(D,h)]) = \mu^\gamma_h(D)$ denotes the total area of $[(D,h)]$. We will now verify this identity. 
	The right-hand side of \eqref{eq:explored-vs-unexplored} equals $\E_\ell(A(S))=\E_1(A(S))\,\ell^2$ since (modulo boundary curves with zero LQG area) $S$ can be viewed as a disjoint union of the surfaces $S_n$. To prove that the left-hand side of \eqref{eq:explored-vs-unexplored} is also equal to $\E_1(A(S))\,\ell^2$ we will use the explicit formula \eqref{eq:explored-explicit} and Proposition \ref{prop:power-theta} with $\theta=2$. Indeed, in the notation of the proposition just mentioned we have $f_0(\ell) = A_+ \E_1(A(S))\,\ell^2$ and by induction we see that $f_k(\ell) = A_+A_-^k \E_1(A(S))\, \ell^2$. Therefore by Proposition \ref{prop:power-theta} we obtain
	\begin{align*}
		\E_\ell\left(\sum_{n\ge 1} \Psi(S'_n)\right)&=\sum_{k\ge 1} f_k(\ell) = \frac{A_+}{1-A_-} \,\E_1(A(S))\,\ell^2 = \E_1(A(S))\, \ell^2\;.
	\end{align*}
	The claim follows.
\end{proof}

We will also need a more general version of the previous corollary in order to prove Theorem \ref{thm:bootstrap}. In the following, recall the definition of $\Pi(X,B)$ just before the statement of Theorem \ref{thm:weighted-levy}. The next result will be used in the context where $\Psi'(S)$ is the weight of a $\gamma$-LQG surface with no marked points and $\Psi_i(S)$ is the weight of a $\gamma$-LQG disk with one marked point, indexed by $i \in A$ (see \eqref{eq:part-fcn}). Applied in this context, the corollary says that the weight $f^B(\ell)$ of an LQG disk with the marked points contained in distinct outermost CLE loops can be computed using an iterative CPI exploration. The requirement that the marked points are in distinct outermost CLE loops corresponds to the fact that in the definition of $f^B(\ell)$ we only sum over $m\in\N_0^B$ with distinct components.

\begin{cor}
	\label{cor:msw-simple-full-more}
	Let $(\Psi_i\colon i\in A)$ and $\Psi'$ be measurable functions on the space of regular $\gamma$-LQG surfaces with values in $[0,\infty]$. Consider the same setup as in Corollary \ref{cor:msw-simple-full}. We assume that
	\begin{align}
		\label{eq:psi0-condition}
		\E_\ell(\Psi'(S))=\E_\ell\left( \prod_{n\ge 1} \Psi'(S_n)\right)\;.
	\end{align}
	For $B\subseteq A$ let
	\begin{align*}
		f^B(\ell) &= \E_\ell\left(\sum_{m\in \N_0^B} \prod_{i\in B} \Psi_i(S_{m_i}) \prod_{n'\notin \{m_i\colon i\in B\}} \Psi'(S_{n'})\right)\;.
	\end{align*}
	where the sum is only over tuples with distinct components. Then $f^B(\ell)=\sum_{k\ge 0} f^B_k(\ell)$ where
	\begin{align*}
		f^B_0(\ell) &= \E_{(\ell/2,\ell/2)}\Biggl(\,\sum_{Q\in \Pi^+(X,B)}\ \prod_{t\in (0,\zeta)\colon \Delta X_t< 0,\, \Delta Q_t\neq \emptyset} f^{\Delta Q_t}(|\Delta X_t|)  \\
		&\qquad\qquad\qquad \cdot \prod_{t\in (0,\zeta)\colon \Delta Q_t=\emptyset} \Psi'(\Delta S_t) \prod_{t\in (0,\zeta)\colon \Delta X_t>0,\, \Delta Q_t=\{i\}} \Psi_i(\Delta S_t)\Biggr)\;, \\
		f^B_{k+1}(\ell) &= \E_{(\ell/2,\ell/2)}\left( \sum_{t<\zeta\colon \Delta X_t < 0} f^B_k(|\Delta X_t|) \prod_{s\in (0,\zeta)\setminus\{t\}} \Psi'(\Delta S_s) \right)\quad\text{for $k\ge 0$}
	\end{align*}
	where we write $\Pi^+(X,B)$ for the set of $Q\in \Pi(X,B)$ such that $\#\Delta Q_t\le 1$ whenever $\Delta X_t>0$ and we require that $\Delta Q_t\neq B$ whenever $\Delta X_t<0$.
\end{cor}

\begin{proof}
	Let us condition on both the CLE and the LQG surface, and fix $m \in \N_0^B$ with $(m_i \colon i \in B)$ distinct. The idea is to define an iterative CPI exploration until we have discovered all loops $\eta_i$ corresponding to the surfaces $(S_{m_i})$, i.e., each such loop $\eta_i$ intersects at least one of the CPIs. Suppose we are exploring the CLE decorated LQG disk with a CPI. If all the $\eta_i$ are in the same cut out component (and none intersect the CPI), we do another CPI exploration in the one cut out component containing all the $\eta_i$. We repeat this until not all of the $\eta_i$ are in the same cut out component of the CPI. This procedure terminates by the argument in the previous corollary showing that the iterative CPI exploration discovers all CLE loops.
	The corollary follows since $f^B_k$ is equal to the expectation defining $f^B$ restricted to the event that we need exactly $k+1$ such chordal explorations in order for not all $\eta_i$ to be in the same complementary component of the CPI.
	Indeed, the terms in the sum in the definition of $f_0^B(\ell)$ correspond precisely to the event that not all $\eta_i$ are in the same complementary component in CPI exploration.
	The iterative definition for $f^B_{k+1}(\ell)$ precisely captures repeating the construction in one of the cut out components.
\end{proof}

The situation is analogous in the setting of a generalized $\gamma$-LQG disk decorated by a non-simple CLE. See now the right side of Figure \ref{fig:msw-review} for a graphical depiction of the following definitions and results. We fix parameters
\begin{align*}
	\sqrt{2} < \gamma < 2\;,\quad \kappa = 16/\gamma^2\;,\quad \beta\in [-1,1] \quad\text{and}\quad \ell_L,\ell_R>0\;.
\end{align*}
The detailed definitions are a bit more involved than in the simple case considered above and we refer to \cite{msw-non-simple} for more details. Consider a generalized $\gamma$-LQG disk with boundary length $\ell=\ell_L+\ell_R$, that is,
\begin{align*}
	S=(E,\{(s,[(\D,h_s,w_0^s)])\colon s<l+r,\Delta E_s\neq 0\})
\end{align*}
where $E$ is the time reversal of a spectrally positive stable Lévy excursion with exponent $4/\gamma^2$ and duration $\ell$, and conditionally on $E$, we have that $(\D,h_s)$ is an embedding of an independent regular $\gamma$-LQG disk of boundary lengths $|\Delta E_s|$ and $w_0^s$ is independently picked according to the $\gamma$-LQG boundary measures for each $s<\ell$ with $\Delta E_s\neq 0$. As explained in Section \ref{sec:gff-lqg} we view the LQG surfaces as ``glued into'' the loops of the looptree defined by $E$, and we do this so that the point $w_0^s$ is identified with the point on the loop that is closest to the root of the looptree in the metric of the looptree.

The times $0$ and $\ell_L$ (in the parametrization of the excursion $E$) correspond to two points on the looptree associated to $E$ (see Section \ref{sec:gff-lqg}) and there exists a unique curve of minimal length from the first to the second of the points above with the property that it only traces loops of the looptree in clockwise order -- we call this curve a clockwise geodesic. Recall that each $s<\ell$ for which $\Delta E_s\neq 0$ is associated with a loop in the looptree and let $I(E,\ell_L)$ denote the collection of such times corresponding to loops that intersect the clockwise geodesic mentioned above. Furthermore, for each $s\in I(E,\ell_L)$, we let $L(E,\ell_L)_s$ be the length of the intersection of the clockwise geodesic with the loop corresponding to the time $s$.

Now, for each regular $\gamma$-LQG disk of the generalized $\gamma$-LQG disk we sample an independent $\CLE_{\kappa'}$ (on $\D$). Moreover, after sampling i.i.d.\ orientations for all the $\CLE_{\kappa'}$ loops (clockwise and counterclockwise with probability $(1-\beta)/2$ and $(1+\beta)/2$, respectively), we deterministically associate to each $s\in I(E,\ell_L)$ a CPI with asymmetry parameter $\beta$ from $w_0^s$ to $w_\infty^s$ where $w_\infty^s$ is chosen so that the clockwise boundary segment from $w_0^s$ to $w_\infty^s$ has $\gamma$-LQG length $L(E,\ell_L)_s$ with respect to the field $h_s$; formally, this yields a decorated generalized $\gamma$-LQG surface where the surfaces corresponding to the loops connecting the (generalized) boundary points corresponding to $0$ and $\ell_L$ are decorated by a collection of loops together with a curve and all other surfaces are only decorated by a collection of loops. We parametrize the CPI by the $\gamma$-LQG length of the CPI in the generalized quantum disk, i.e., it is the sum of all the $\gamma$-LQG lengths of the CPIs in the regular $\gamma$-LQG disks. Moreover, we let $\zeta$ be the total $\gamma$-LQG length of the CPI.

For $t\in [0,\zeta)$, we (informally) define a generalized $\gamma$-LQG surface $S_t$ as follows (see again \cite{msw-non-simple} for details). Let us consider the CPI stopped at time $t$. For $s<\ell$ with $\Delta E_s\neq 0$, let $A^s_t$ be the set of points in $\D$ such that the corresponding point in the $[(\D,h_s)]$ (which is a part of $S$) can be connected within $S$ to the generalized boundary point associated to $l$ along a path which does not intersect the CPI stopped at time $t$ and which does not intersect the interior of any of the CLE loops intersecting the CPI stopped at time $t$. This is illustrated in the right-hand side of Figure \ref{fig:msw-review}, where the purple (resp.\ white) points are contained in $A_t^s$ (resp.\ $\D\setminus A_t^s$).

It can be argued that the collection of sets $(A^s_t\colon s<\ell,\;\Delta E_s\neq 0)$ defines a generalized $\gamma$-LQG disk $S_t$ with two marked points; the point visited by the CPI at time $t$ and the terminal point of the CPI (the latter of which is equal to the  boundary point of the original generalized disk corresponding to time $\ell_L$ for $E$).
For $t<\zeta$, the clockwise (resp.\ counterclockwise) generalized $\gamma$-LQG length from the first to the second one of these generalized boundary points is called $L_t$ (resp.\ $R_t$). If $t\ge \zeta$, we simply set $L_t=R_t=0$. As in the simple case, it turns out that $(L,R)$ has càdlàg version $((L,R)\colon t\ge 0)$ which takes values in $(0,\infty)^2\cup\{(0,0)\}$ and we also make the definition $X=L+R$.

Whenever $t\in(0,\zeta)$ is such that $\Delta X_t\neq 0$, the collection of sets $(\Delta A^s_t\colon s \in(0,l+r),\;\Delta E_s\neq 0)$ with $\Delta A^s_t = (\cap_{t'<t} A^s_{t'})\setminus A^s_t$ also defines a generalized $\gamma$-LQG surface $\Delta S_t$ (called the cut out generalized quantum surfaces at time $t$). As in the simple case, if $\Delta X_t>0$ then the CPI discovers a CLE loop at time $t$
and $\Delta S_t$ is the generalized $\gamma$-LQG surface given by the interior of the CLE loop. If $\Delta X_t<0$, this corresponds to the CPI hitting its past hull and cutting out what we call a trunk component.

\begin{thm}[\cite{msw-non-simple}]
	\label{thm:msw-non-simple}
	There exists a constant $v_\kappa^\beta \in (0,\infty)$ such that
	\begin{align*}
		(L,R) \stackrel{d}{=} ((L_{v_\kappa^\beta t}',R'_{v_\kappa^\beta t})\colon t\ge 0)
	\end{align*}
	where $(L',R')$ is defined in Proposition \ref{prop:nonsimple-msw-process}. Furthermore, conditionally on $(L,R)$ the cut out (generalized) quantum surfaces which correspond to jump times $t\ge 0$ with $\Delta X_t\neq 0$ are independent generalized $\gamma$-LQG disks with boundary length $|\Delta X_t|$.
\end{thm}

\begin{cor}
	\label{cor:msw-nonsimple-full}
	Let $\Psi$ be a measurable function on the space of generalized $\gamma$-LQG surfaces with values in $[0,\infty]$. Consider a generalized $\gamma$-LQG disk $S$ with (generalized) boundary length $\ell>0$ with each of its (regular) components decorated by an independent non-nested $\CLE_{\kappa'}$. Let $(S_n\colon n\ge 1)$ be the (generalized) $\gamma$-LQG surfaces cut out by the CLE loops. Then
	\begin{gather*}
		\E_\ell\left(\sum_{n\ge 1} \Psi(S_n)\right)=\sum_{k\ge 0} f_k(\ell)\quad\text{where}\quad  f_0(\ell) :=  \E_{(\ell/2,\ell/2)}\left(\sum_{t\in(0,\zeta)\colon \Delta X_t>0} \Psi(\Delta S_t) \right)
		 \\
		 \text{and}\quad f_k(\ell) := \E_{(\ell/2,\ell/2)}\left( \sum_{t\in (0,\zeta)\colon \Delta X_t<0} f_{k-1}(|\Delta X_t|) \right)
	\end{gather*}
	for $k\ge 1$. The subscript $\ell$ on the left-hand side of the statement indicates the total boundary length of the $\gamma$-LQG disk $S$ and a subscript of the form $(\ell_L,\ell_R)$ indicates that the CPI exploring the LQG disk is started with boundary lengths $(\ell_L,\ell_R)$.
\end{cor}

\begin{proof}
	The proof is entirely analogous to that of Corollary \ref{cor:msw-simple-full} except that we now use Proposition \ref{prop:nonsimple-power-theta} with $\theta = \gamma^2/2$ instead of Proposition \ref{prop:power-theta} with $\theta=2$.
\end{proof}

\begin{remark}
	Corollary \ref{cor:msw-simple-full-more} also has an analogous version for the theory of generalized $\gamma$-LQG disks discussed above (with only the word ``simple'' replaced by ``generalized'' and Corollary \ref{cor:msw-nonsimple-full} in place of Corollary \ref{cor:msw-simple-full}).
\end{remark}

\section{Law of the CLE conformal radius}
\label{sec:ssw}

In this section, we will give a new proof of Theorem \ref{thm:confradlaw} when $\kappa\neq 4$ using techniques from Liouville quantum gravity. We also present a new proof using similar ideas in the $\kappa=4$ case. The cases of simple and non-simple CLEs will be considered separately although the proofs are essentially identical and will rely on the papers \cite{msw-simple} and \cite{msw-non-simple}, respectively (see Section \ref{sec:background-msw} for a review).

\begin{thm}
	\label{thm:ssw-simple-lqg}
	Let $\kappa\in(8/3,4)$ and $\alpha\in (2/\gamma+\gamma/4,2)$ where $\gamma=\sqrt{\kappa}$. Let $\Gamma$ be a non-nested $\CLE_\kappa$ in $\D$ and let $\eta_0$ be the loop in $\Gamma$ that surrounds $0$. Then
	\begin{align*}
		c:=\E( R(0,\eta^o_0)^{\alpha(Q_\gamma - Q_\alpha)} ) = \frac{\cos(4\pi/\kappa)}{\cos(4\pi/\kappa-\pi\cdot 2\alpha/\gamma)}\;.
	\end{align*}
\end{thm}

\begin{proof}
	The key idea of the proof is the following. To a regular $\gamma$-LQG surface $[(D,k)]$ we associate a value
	\begin{align*}
		\Psi([(D,k)]) = \int_D R(z,D)^{\alpha(Q_\gamma-Q_\alpha)}\mu^\alpha_k(dz) \in [0,\infty]
	\end{align*}
	provided the measure $\mu^\alpha_k$ is defined (and $\Psi([(D,k)])=0$ otherwise, say). It is not hard to check that $\Psi$ is well-defined (i.e., does not depend on the representative $(D,k)$ of its equivalence class) by using the definition of a regular $\gamma$-LQG surface (Definition \ref{def:lqg-surface}) and how $\mu^\alpha_k$ transforms under the application of a conformal map (see \eqref{eq:meas-inv}).
	
	Let $(\D,h)$ be an embedding of a $\gamma$-LQG disk with boundary length $1$ into the unit disk $\D$ which is independent of $\Gamma$,  i.e., $h$ and $\Gamma$ are independent. Then
	\begin{align*}
		\E\left( \Psi([(\D,h)]) \right) =\E\left( \int_\D R(z,\D)^{\alpha(Q_\gamma-Q_\alpha)}\mu^\alpha_h(dz) \right) <\infty\;,
	\end{align*}
	where the finiteness is a consequence of the definition of $h$ as a $\gamma$-LQG disk and the second part of Proposition \ref{prop:resample-lcft-point} with $f=1$ and $\Lambda=0$ (note that $\alpha>\gamma/2$ so that the Liouville partition functions appearing on both sides of the second equation in the proposition are finite, see Remark \ref{rmk:seiberg}). Moreover, $[(\D,h+2/\gamma\cdot \log \ell)]$ is a regular $\gamma$-LQG disk with boundary length $\ell$ by Definition \ref{def:normal-unit-disk} and
	\begin{align}
		\E\left( \Psi([(\D,h+2/\gamma\cdot \log \ell)]) \right) = \ell^{2\alpha/\gamma} \,\E\left( \Psi([(\D,h)]) \right)\;.
		\label{eq:psi-scaling}
	\end{align}
	For any $z\in \D$, let $\eta_z$ be the loop around $z$ (if no such loop exists, which is a probability zero event for any fixed $z$, we take $\eta_z$ to be the curve tracing $\partial \D$, say). Recalling that $c=\E(R(0,\eta_0^o)^{\alpha(Q_\alpha - Q_\gamma)})\in (0,\infty]$ and using conformal invariance of CLE, we get
	\begin{align*}
		\E(R(z,\eta_z^o)^{\alpha(Q_\alpha - Q_\gamma)}) = c \cdot R(z,\D)^{\alpha(Q_\alpha - Q_\gamma)}\quad\text{for all $z\in \D$}\;.
	\end{align*}
	We now first use the previous equation, exchange integration with expectation and rewrite the expression within the expectation to obtain
	\begin{align*}
		c\cdot\E\left( \Psi([(\D,h)]) \right) &= \E\left( \int_\D \E(R(z,\eta_z^o)^{\alpha(Q_\gamma-Q_\alpha)})\mu^\alpha_h(dz) \right) = \E\left( \int_\D R(z,\eta_z^o)^{\alpha(Q_\gamma-Q_\alpha)}\mu^\alpha_h(dz) \right) \\
		&= \E\left(\sum_{\eta\in \Gamma} \int_{\eta^o} R(z,\eta^o)^{\alpha(Q_\gamma-Q_\alpha)}\mu^\alpha_h(dz) \right) = \E\left(\sum_{n\ge 1}  \Psi(S_n) \right)
	\end{align*}
	where $(S_n)$ denotes the quantum surfaces cut out of $[(\D,h)]$ by $\Gamma$, noting that we are taking the expectation over both the LQG field and the CLE.
	
	Let $(L',R')$ and $X'=L'+R'$ be as in Proposition \ref{prop:simple-msw-process} and conditionally on $(L',R')$ let $S'_t$ be independent regular $\gamma$-LQG surfaces with boundary length $|\Delta X'_t|$ whenever $\Delta X'_t\neq 0$. By Corollary \ref{cor:msw-simple-full} and Theorem \ref{thm:msw-simple} we have
	\begin{align*}
		\E\left(\sum_{n\ge 1}  \Psi(S_n) \right) = \sum_{k\ge 0} f_k(1)
	\end{align*}
	where
	\begin{align*}
	f_0(\ell) &= \E_{(\ell/2,\ell/2)}\left( \sum_{t\in(0,\zeta')\colon \Delta X'_t>0} \Psi(S'_t) \right) \\
		&=  \E\left( \Psi([(\D,h)]) \right)\ell^{2\alpha/\gamma}\,\E_{(1/2,1/2)}\left( \sum_{t\in(0,\zeta')\colon \Delta X'_t>0}  (\Delta X'_t)^{2\alpha/\gamma} \right),\\
	f_k(\ell) &= \E_{(\ell/2,\ell/2)}\left( \sum_{t\in (0,\zeta)\colon \Delta X'_t<0} f_{k-1}(|\Delta X'_t|) \right),
	\end{align*}
	where we used \eqref{eq:psi-scaling} to get the second equality. Note that in the first of the previous two displays, we are taking an expectation over the random variable $(S_n)$ while in the second one we are taking one over $(L',R',X',(S_t'))$. By Proposition \ref{prop:power-theta} with $\theta=\alpha\gamma/2$ we see by induction (and using the notation of the aforementioned proposition) that for all $k\ge 1$, $f_k(\ell) = \E\left( \Psi([(\D,h)]) \right)\ell^{2\alpha/\gamma} A_+ A_-^k$. Using this, we obtain that
	\begin{align*}
		\sum_{k\ge 1} f_k(1) =  \frac{\cos(4\pi/\kappa)}{\cos(4\pi/\kappa-\pi\cdot 2\alpha/\gamma)}\,\E\left( \Psi([(\D,h)]) \right)
	\end{align*}
	and the claim follows.
\end{proof}

\begin{thm}
	\label{thm:ssw-nonsimple-lqg}
	Let $\kappa'\in(4,8)$ and $\alpha\in (1/\gamma+\gamma/2,2)$ with $\gamma=4/\sqrt{\kappa'}$. Let $\Gamma$ be a non-nested $\CLE_{\kappa'}$ in $\D$ and let $\eta_0$ be the loop in $\Gamma$ that surrounds $0$. Then
	\begin{align*}
		c := \E( R(0,\eta^o_0)^{\alpha(Q_\gamma - Q_\alpha)} ) = \frac{\cos(4\pi/\kappa')}{\cos(4\pi/\kappa'-\pi\cdot \alpha\gamma/2)}\;.
	\end{align*}
\end{thm}

\begin{proof}
	The proof is very similar to the proof of Theorem \ref{thm:ssw-simple-lqg}. We define the following functional on the space of generalized $\gamma$-LQG surfaces
	\begin{align*}
		\Psi(e,\{ (t, [(D_t,k_t)]) \colon t\ge 0, \Delta e_t\neq 0 \}) = \sum_{t\ge 0\colon \Delta e_t\neq 0} \int_{D_t} R(z,D_t)^{\alpha(Q_\gamma-Q_\alpha)}\mu^\alpha_{k_t}(dz) \in [0,\infty]
	\end{align*}
	provided the measure $\mu^\alpha_{k_t}$ is defined for all $t$ and define the functional to be $0$ otherwise. Again, the integrals do not depend on the particular equivalence class of the $\gamma$-LQG surfaces $[(D_t,k_t)]$. Let $(E_t\colon t\in[0,1])$ be the time reversal of a spectrally positive stable Lévy excursion with exponent $4/\gamma^2$ and duration $1$ and conditionally on $E$, consider independent regular $\gamma$-LQG quantum disks $[(\D,h_t)]$ for any $t\in [0,1]$ with $\Delta E_t\neq 0$. Then by stable scaling
	\begin{align*}
		S_\ell = (\ell^{\gamma^2/4}E_{\cdot/\ell},\{(t,[(\D,h_t+2/\gamma \cdot \log (\ell^{\gamma^2/4}|\Delta E_{t/\ell}|))]\colon t<\ell, \Delta E_{t/\ell}\neq 0)\})
	\end{align*}
	is a generalized $\gamma$-LQG disk with boundary length $\ell$. By definition, we therefore obtain
	\begin{align*}
		\E\left(\Psi(S_\ell) \right) &= \E\left(\sum_{t<\ell\colon \Delta E_{t/\ell}\neq 0} (\ell^{\gamma^2/4}|\Delta E_{t/\ell}|)^{2\alpha/\gamma} \right)\,\E\left( \int_\D R(z,\D)^{\alpha(Q_\gamma-Q_\alpha)}\mu^\alpha_{h}(dz) \right) \\
		&= \ell^{\alpha\gamma/2} \,\E\left(\sum_{t<1\colon \Delta E_{t}\neq 0} |\Delta E_{t}|^{2\alpha/\gamma} \right)\,\E\left( \int_\D R(z,\D)^{\alpha(Q_\gamma-Q_\alpha)}\mu^\alpha_{h}(dz) \right)
	\end{align*}
	noting that the two expectations on the right-hand side are finite -- for the first expectation, this is a consequence of Theorem \ref{thm:levy-excursions} and for the second expectation this follows from Proposition \ref{prop:resample-lcft-point} with $f=1$. If $S$ is a generalized $\gamma$-LQG disk with boundary length $1$ then we can decorate this surface with an independent $\CLE_{\kappa'}$ on each of its regular $\gamma$-LQG disks components. By a similar argument as in the proof of Theorem \ref{thm:ssw-simple-lqg} we get
	\begin{align*}
		c\cdot \E\left(\Psi(S) \right) = \E\left(\sum_{n\ge 1}  \Psi(S_n) \right)
	\end{align*}
	where $(S_n)$ denotes the collection of generalized $\gamma$-LQG surfaces cut out by the $\CLE_{\kappa'}$ decoration. The conclusion now follows similarly as in the proof of Theorem \ref{thm:ssw-simple-lqg}, now using Corollary \ref{cor:msw-nonsimple-full}, Theorem \ref{thm:msw-non-simple} and Proposition \ref{prop:nonsimple-power-theta}.
\end{proof}

The $\kappa=4$ case is actually the one which is understood best due to the level-line coupling with the Gaussian free field by Miller and Sheffield (see e.g.\ \cite{asw-btls} for a self-contained presentation). For the sake of completeness, we also give a proof of the computation of the law of the conformal radius in a similar spirit as in the $\kappa \neq 4$ case. Note however that the CLE decoration is not independent of the field but instead determined by it. While the proof is interesting, let us remark that the law of the conformal radius can also be read off more directly in the proof of the level-line coupling (see again \cite{asw-btls}) or one could attempt to take the limit $\kappa\to 4$ in the theorems above.

\begin{prop}
	\label{prop:ssw-critical-levelline}
	Consider $\Gamma\sim\CLE_4$ (non-nested) in $\D$ and let $\eta_0$ be the loop in $\Gamma$ that surrounds $0$. Then for any $\gamma\in (0,2)$,
	\begin{align*}
		\E\left( R(0,\eta^o_0)^{\gamma^2/2}\right)=1/\cosh(\pi\gamma)\;.
	\end{align*}
\end{prop}

\begin{proof}
	Let us recall the level-line coupling of a $\CLE_4$ with a Dirichlet GFF due to Miller and Sheffield: Let $h$ be a Dirichlet GFF in $\D$ as introduced in Section \ref{sec:gff-lqg}, let us write $\Gamma=\{\widetilde{\eta}_i\colon i\ge 1\}$ and conditionally on $\Gamma$, let $(h_i\colon i\ge 1)$ be i.i.d.\ Dirichlet GFFs in $\widetilde{\eta}_i^o$ and $(\sigma_i\colon i\ge 1)$ be independent i.i.d.\ and uniform on $\{\pm 1\}$; then
	\begin{align*}
		h \stackrel{d}{=} \sum_{i\ge 1} (\pi\sigma_i+h_i)1_{\widetilde{\eta}_i^o}\;.
	\end{align*}
	It follows that for $\gamma\in (0,2)$,
	\begin{align}
		\label{eq:ssw-gff-key}
		\E\left(\,\int_\D \mu^\gamma_h(dz)\right)=\E\left( \sum_{i\ge 1} \int_{\widetilde{\eta}_i^o} \mu^\gamma_{\pi\sigma_i+h_i}(dz)\right) = \E\left( \sum_{i\ge 1} \E\left(\int_{\widetilde{\eta}_i^o} \mu^\gamma_{\pi\sigma_i+h_i}(dz)\mid \Gamma\right)\right)\;.
	\end{align}
	By \eqref{eq:gff-annealed} we get
	\begin{align*}
		\E\left(\,\int_\D \mu^\gamma_h(dz)\right) &= \int_{\D} R(z,\D)^{\gamma^2/2}\,dz\;,\\
		\E\left(\int_{\widetilde{\eta}_i^o} \mu^\gamma_{\pi\sigma_i+h_i}(dz)\mid \Gamma\right) &= \cosh(\gamma\pi) \int_{\widetilde{\eta}_i^o} R(z,\widetilde{\eta}_i^o)^{\gamma^2/2}\,dz\quad\text{a.s.,}
	\end{align*}
	where we use for the second identity that $\E(e^{\gamma\pi\sigma_i})=\cosh(\gamma\pi)$.
	For $z\in \D$, let $\eta_z$ be the a.s.\ unique loop in $\Gamma$ surrounding the point $z$ taking $\eta_z$ to be the curve tracing $\partial \D$ if no such loop exists. Then $\E(R(z,\eta_z^o)^{\gamma^2/2})=R(z,\D)^{\gamma^2/2}\,\E(R(0,\eta_0)^{\gamma^2/2})$ for $z\in \D$ by conformal invariance of $\Gamma$  and thus by \eqref{eq:ssw-gff-key},
	\begin{align*}
		\cosh(\gamma\pi)^{-1}\int_{\D} R(z,\D)^{\gamma^2/2}\,dz &= \E\left(\sum_{i\ge 1}\int_{\widetilde{\eta}_i^o} R(z,\widetilde{\eta}_i^o)^{\gamma^2/2}\,dz\right) = \E\left( \int_\D R(z,\eta_z^o)^{\gamma^2/2}\,dz \right) \\
		&= \int_\D \E(R(z,\eta_z^o)^{\gamma^2/2})\,dz =\E(R(0,\eta_0^o)^{\gamma^2/2})\int_\D R(z,\D)^{\gamma^2/2}\,dz\;.
	\end{align*}
	The claim follows.
\end{proof}

Combining the results above we conclude the proof of Theorem \ref{thm:confradlaw}.

\begin{proof}[Proof of Theorem \ref{thm:confradlaw}]
	Let us first consider $\kappa\in(8/3,4)$ and $\gamma=\sqrt{\kappa}$. If $\alpha\in (2/\gamma+\gamma/4,2)$, let $\rho = \alpha(Q_\gamma-Q_\alpha)\in (-1+2/\kappa+3\kappa/32,\sqrt{\kappa}+4/\sqrt{\kappa}-4)$. Since
	\begin{align*}
		\sqrt{(1-4/\kappa)^2-8\rho/\kappa} = 4/\kappa -2\alpha/\kappa + 1
	\end{align*}
	by Theorem \ref{thm:ssw-simple-lqg} we get
	\begin{align*}
		\E\left(R(0,\eta_0^o)^\rho \right) &= \E\left(R(0,\eta_0^o)^{\alpha(Q_\gamma - Q_\alpha)} \right) = \frac{\cos(4\pi/\kappa)}{\cos(4\pi/\kappa-\pi\cdot 2\alpha/\gamma)} 
		= \frac{-\cos(4\pi/\kappa)}{\cos\left( \pi\sqrt{(1-4/\kappa)^2 - 8\rho/\kappa} \right)}
	\end{align*}
	when $\rho\in (-1+2/\kappa+3\kappa/32,\sqrt{\kappa}+4/\sqrt{\kappa}-4)$. In particular, the two functions
	\begin{align*}
		\rho\mapsto \E(R(o,\eta_0^o)^\rho)\quad\text{and}\quad \rho\mapsto \frac{-\cos(4\pi/\kappa)}{\cos\left( \pi\sqrt{(1-4/\kappa)^2 - 8\rho/\kappa} \right)}
	\end{align*}
	are both analytic functions on $\{\rho\in \C\colon \Re \rho>-1+2/\kappa+3\kappa/32\}$ (for the second function this statement is obvious) and agree on a set containing an accumulation point. The two functions therefore are identical on $\{\rho\in \C\colon \Re \rho>-1+2/\kappa+3\kappa/32\}$ as required.

	In the $\kappa \in(4,8)$ case we let $\gamma=4/\sqrt{\kappa}$, consider $\alpha\in (1/\gamma+\gamma/2,2)$ and define $\rho = \alpha(Q_\gamma-Q_\alpha)\in(-1+2/\kappa+3\kappa/32,\sqrt{\kappa}+4/\sqrt{\kappa}-4)$. Then we have
	\begin{align*}
		\sqrt{(1-4/\kappa)^2-8\rho/\kappa}= 4/\kappa-\alpha\gamma/2+1
	\end{align*}
	and the result follows for $\rho\in (-1+2/\kappa+3\kappa/32,\sqrt{\kappa}+4/\sqrt{\kappa}-4)$ by Theorem \ref{thm:ssw-nonsimple-lqg} and the same reasoning as in the $\kappa<4$ case. We again get the full range of $\rho$ values by analytic continuation.

	Finally, we consider the $\kappa=4$ case. Proposition \ref{prop:ssw-critical-levelline} yields the claim for $\rho\in (0,2)$. The moment generating function of the random variable $-\log R(0,\eta_0^o)$ characterizes its law and we see that the moment generating function of $-\log R(0,\eta_0^o)$ agrees with that of the random variable $\tau_{\pm 1}=\inf\{t\ge 0\colon |B_t|=\pi\}$ where $B$ is a standard Brownian motion. The moment generating function of $\tau_{\pm 1}$ is explicit on the entire range where it is finite.
\end{proof}

\section{CLE weighted by nesting statistics}
\label{sec:cle}

In this section we prove Theorem \ref{thm:loopexpconv}, which justifies that the CLE weighting considered in the introduction is natural, since, given a tuple of points, it corresponds in some sense to reweighting the law of a CLE coupled to an LQG surface by the number of CLE loops surrounding any subset of these points. See Section \ref{sss:regular-disk} for a description of this weighting in the case of a regular LQG disk. A key input is Theorem \ref{thm:cle-indep}, which is a spatial independence property for CLE that is necessary in order to establish finiteness of the (claimed) limit in Theorem \ref{thm:loopexpconv}. Theorem \ref{thm:cle-indep} will be proved in Section \ref{sec:cle-k-small} for $\kappa\in(8/3,4]$ and in Section \ref{sec:cle-k-big} for $\kappa\in(4,8)$.

\subsection{Weighting CLE by nesting statistics: Proof of Theorem \ref{thm:loopexpconv}}

As discussed in the introduction, we want to mimic in the continuum the definition of the $O(n)$ model weighted by the number of loops surrounding points. Due to the fact that the number of loops in a nested CLE surrounding a single point is almost surely infinite, a renormalization is needed to weight a CLE by the exponential of the number of loops surrounding points. Making sense of this rigorously will be the main achievement of this section.
\begin{thm}
	Let $\delta>0$ and let $z_1,\dots,z_n\in \D$ be such that
	\begin{align*}
		|z_i-z_j|>\delta\quad\text{and}\quad|z_i|<1-\delta\qquad\forall i,j\le n\,,\,i\neq j\;.
	\end{align*}
	Let $\Gamma$ be a nested $\CLE_\kappa$ in $\D$ for $\kappa\in(8/3,8)$ and define $\eta_i$ to be the outermost loop surrounding $z_i$ which does not surround $z_j$ for all $j\neq i$. Suppose $\rho_i>-1+2/\kappa+3\kappa/32$ for all $i=1,\dots,n$. Then there is a constant $C>0$ depending only on $\delta,\rho_1,\dots,\rho_n$ such that
	\begin{align*}
	\E\left(\, \prod_{i=1}^n R(z_i,\eta_i^o)^{\rho_i} \right) \le C\;.
	\end{align*}
	\label{thm:cle-indep}
\end{thm}

The proof of the theorem is deferred to Sections \ref{sec:cle-k-small} and \ref{sec:cle-k-big} in the simple and non-simple case, respectively. The proof of this result will make strong use of a quasi-independence property of the CLE: the conformal radii of the outermost loops only surrounding one particular point do not correlate strongly with each other; indeed the proofs in Sections \ref{sec:cle-k-small} and \ref{sec:cle-k-big} will exploit two different constructions of CLE from which this quasi-independence can be deduced. Before we can establish Theorem \ref{thm:loopexpconv} we will also need two lemmas on CLE. Related results also appear in \cite{mww-extremes}, however, both the finiteness and convergence of exponential moments of nesting statistics presented here is new.

\begin{lemma}
	\label{lem:clelocalasymptotic}
	Let $\Gamma$ be a nested $\CLE_\kappa$ in $\D$ where $\kappa\in (8/3,8)$. Moreover, whenever $\epsilon>0$ we write $N_\epsilon$ for the number of loops $\eta\in \Gamma$ surrounding $0$ and with $R(0,\eta^o)>\epsilon$. Then
	\begin{align*}
		\E(e^{\sigma N_\epsilon})<\infty \quad\text{for all $\epsilon>0$, $\sigma\in \R$}\quad
		\text{and}\quad \epsilon^{\rho^\kappa_\sigma} \,\E(e^{\sigma N_\epsilon})\to c^\kappa_\sigma\quad\text{as $\epsilon\downarrow 0$}
	\end{align*}
	for some explicit constant $c^\kappa_\sigma \in (0,\infty)$.
\end{lemma}

\begin{proof}
	Let $(\eta_i\colon i\ge 0)$ be the loops in $\Gamma$ surrounding $0$ ordered according to their nesting structure (i.e., $\eta_i$ surrounds $\eta_{i+1}$ for $i\ge 0$) with the convention that $\eta_0$ parametrizes $\partial \D$. Then by the spatial Markov property of $\CLE_\kappa$
	\begin{align*}
		X_i = -\log R(0,\eta_i^o) + \log R(0,\eta_{i-1}^o)\;,\quad i\ge 1
	\end{align*}
	are i.i.d. and their density $f$ and Laplace transform $\mathcal{L}f$ are given in Theorem \ref{thm:confradlaw}. Also note that $N_\epsilon = \sup\{n\ge 0\colon X_1+\cdots + X_n \le \log(1/\epsilon)\}$ for $\epsilon>0$. Also  by Definition \ref{def:main-params} we have $\sigma = -\log \mathcal{L}f(\rho^\kappa_\sigma)$. The claim now follows from Proposition \ref{prop:renewalresult} where the conditions can be verified using Theorem \ref{thm:confradlaw}.
\end{proof}

\begin{lemma}
	\label{lem:superexp-cle-tail}
	Let $z\in \D\setminus \{0\}$ and let $\Gamma$ be a nested $\CLE_\kappa$ in $\D$ with $\kappa\in (8/3,8)$. We also let $N_{0z}$ be the number of loops in $\Gamma$ surrounding both $0$ and $z$. Then for all $\delta\in (0,1)$ and $\theta\ge 0$ there is a constant $C>0$ such that $\E(\exp (\theta N_{0z}))\le C$ for $|z|\ge \delta$.
\end{lemma}

\begin{proof}
	Let $\eta_i$ be the $2i$-th loop surrounding $0$ (counted from the boundary of $\D$) with the convention that $\eta_0$ parametrizes $\partial \D$ and let us write $A_w$ for the event that $\eta_1$ surrounds $w$. Then $\P(A_w)\to 0$ as $|w|\to 1$; indeed, if this was not the case then there would exist $(w_n)$ with $|w_n|\to 1$ and $\limsup_{n\to\infty} \P(A_{w_n})>0$ and hence $\P(A_{w_n}\ \text{i.o.})>0$ which is impossible since $\eta_1\subseteq \D$. Let $N'_{0z}$ be the largest $i$ such that $\eta_i$ surrounds $z$ so that in particular
	\begin{align}
		\label{eq:only-second-bound}
		N'_{0z}\le 2N_{0z}+1\;.
	\end{align}
	Let $\phi_i\colon \D\to C_i$ be conformal transformations with $\phi_i(0)=0$ and $\phi_i'(0)>0$ where $C_i$ is the connected component of $\eta_i^o$ containing $0$. Consider $\psi>\theta$, let $\epsilon>0$ be such that $\P(A_w)\le e^{-\psi}$ for all $|w|\ge 1-\epsilon$, and define $\tau_z = \inf\{i\ge 0\colon |\phi^{-1}_i(z)|\ge 1-\epsilon\}$ where by convention $\phi^{-1}_i(z)=1$ if $z\notin C_i$. By the Markov property of nested $\CLE_\kappa$,
	\begin{align*}
		\E\left( e^{\theta N'_{0z}}\right) = \E\left( e^{\theta\tau_z}\left.\E\left(e^{\theta N'_{0w}}\right)\right|_{w=\phi^{-1}_{\tau_z}(z)}\right) \le \E\left( e^{\theta\tau_z}\right) \sup_{|w|\ge 1-\epsilon}\E\left( e^{\theta N'_{0w}}\right).
	\end{align*}
	Let us treat the two factors separately. First of all, by the growth theorem for univalent functions \cite[Theorem 3.21]{lawler-book} applied to $\varphi_i=\phi_i/\phi_i'(0)$ at the value $\phi^{-1}(z)$ whenever $z\in C_i$,
	\begin{align*}
		|\varphi_i(\phi^{-1}_i(z))| \le \frac{|\phi^{-1}_i(z)|}{(1-|\phi^{-1}_i(z)|)^2} \le \frac{1}{(1-|\phi^{-1}_i(z)|)^2}.
	\end{align*}
	Since $|\varphi_i(\phi^{-1}_i(z))| = |z|/R(0,\eta_i^o)$, we obtain that
	\begin{align*}
		|\phi^{-1}_i(z)| \ge 1 - \sqrt{\frac{R(0,\eta_i^o)}{|z|}}\quad\text{whenever $z\in C_i$}\;.
	\end{align*}
	Let $M_z$ be the number of $i\ge 1$ such that $R(0,\eta_i^o)\ge |z|\epsilon^2$. Then by the above estimate, $\tau_z \le M_z + 1$ and therefore by Lemma \ref{lem:clelocalasymptotic}, $\E(e^{\theta \tau_z})\le e^\theta \,\E(e^{\theta M_z})\le C$ for some constant $C<\infty$ only depending on $\delta$. Let us now consider the second factor. For $n\ge 0$ and $|w|\ge 1-\epsilon$, by the Markov property for nested $\CLE_\kappa$,
	\begin{align*}
		\P(N'_{0w}\ge n+1)=\E\left( \left. \P(A_v)\right|_{v=\phi^{-1}_n(w)}; N'_{0w}\ge n, z\in C_n \right)\;.
	\end{align*}
	By the Schwarz lemma, $|\phi^{-1}_n(w)|\ge |w|$ and thus $\P(N'_{0w}\ge n+1) \le e^{-\psi}\, \P(N_{0w}\ge n)$ for $|w|\ge 1-\epsilon$. Therefore $\P(N'_{0w}\ge n) \le \exp(-\psi n)$ for all $n\ge 0$ and $|w|\ge 1-\epsilon$. The claim follows by combining everything and using \eqref{eq:only-second-bound}.
\end{proof}

We can now combine the above to deduce the main result.

\begin{proof}[Proof of Theorem \ref{thm:loopexpconv}]
	Without loss of generality $A=[n]=\{1,\dots,n\}$.
	Let $C_i$ be the connected component of $\eta_i^o$ containing $z_i$ for $i\le n$ (recall that $\eta_i$ denotes the outermost loop in $\Gamma$ surrounding $z_i$ but not $z_j$ for $j \neq i$). We explore the nested CLE $\Gamma$ from the outside until we discover the loops $(\eta_1,\dots,\eta_n)$. By the strong Markov property of the CLE and since $R(z_i,\eta_i^o)=R(z_i,C_i)$, we get
	\begin{align*}
		&\prod_{i=1}^n \epsilon^{\rho_{\sigma_i}^\kappa}\, \E\left(e^{\sum_{B\subseteq [n]\colon \#B\ge 1} \sigma_B N^\epsilon_B(\z)} \right) \\
		&\quad = \E\left(\,\prod_{i=1}^n e^{\sigma_i1(R(z_i,C_i)\ge \epsilon)}\,\epsilon^{\rho_{\sigma_i}^\kappa} \left.\! \E\left( e^{\sigma_iN^{\epsilon_i}}\right)\right|_{\epsilon_i = \epsilon/R(z_i,C_i)}\cdot e^{ \sum_{B\subseteq [n]\colon \#B> 1} \sigma_B N^\epsilon_B(\z)} \right)
	\end{align*}
	where $N^\epsilon$ is the number of loops $\eta$ in a $\CLE_\kappa$ $\wt{\Gamma}$ in $\D$ surrounding $0$ with $R(0,\eta^o)\ge \epsilon$. From Lemma \ref{lem:clelocalasymptotic} it follows that
	\begin{align}
		\label{eq:approx-cle-proof}
		\begin{split}
		&\prod_{i=1}^n e^{\sigma_i1(R(z_i,\eta_i^o)\ge \epsilon)}\,\epsilon^{\rho_{\sigma_i}^\kappa} \left.\! \E\left( e^{\sigma_iN^{\epsilon_i}}\right)\right|_{\epsilon_i = \epsilon/R(z_i,\eta_i^o)}\cdot e^{ \sum_{B\subseteq [n]\colon \#B> 1} \sigma_B N^\epsilon_B(\z)} \\
		&\quad\to \prod_{i=1}^n e^{\sigma_i}c_{\sigma_i}^\kappa R(z_i,\eta_i^o)^{\rho^\kappa_{\sigma_i}}\cdot e^{ \sum_{B\subseteq [n]\colon \#B> 1} \sigma_B N^0_B(\z)}\quad\text{a.s. as $\epsilon\downarrow 0$}\;.
		\end{split}
	\end{align}
	Also by Lemma \ref{lem:clelocalasymptotic}, there exists a constant $c>1$ such that for all $\epsilon\le 1$
	\begin{align*}
		\epsilon^{\rho_{\sigma_i}^\kappa} \left.\! \E\left( e^{\sigma_iN^{\epsilon_i}}\right)\right|_{\epsilon_i = \epsilon/R(z_i,\eta_i^o)} &\le c R(z_i,\eta_i^o)^{\rho_{\sigma_i}^\kappa} 1(\epsilon \le R(z_i,\eta_i^o)) + \epsilon^{\rho_{\sigma_i}^\kappa} 1(\epsilon > R(z_i,\eta_i^o)) \\
		&\le 2c R(z_i,\eta_i^o)^{\rho_{\sigma_i}^\kappa\wedge 0}\quad\text{for all $i\le n$}\;.
	\end{align*}
	In the first line of the above display, we split into the cases $\epsilon \le R(z_i,\eta_i^o)$ (where Lemma \ref{lem:clelocalasymptotic} is applicable) and $\epsilon > R(z_i,\eta_i^o)$ where the expectation equals on the left-hand side equals $1$.
	Thus we can deduce the first part of the theorem using dominated convergence for the sequence in \eqref{eq:approx-cle-proof} and using the following dominating function
	\begin{align*}
		\prod_{i=1}^n 2c\, e^{\sigma_i\vee 0} R(z_i,\eta_i^o)^{\rho_i\wedge 0}\cdot \prod_{B\subseteq [n]\colon \#B> 1} e^{\sigma_B N^\epsilon_B(\z)} \;.
	\end{align*}
	The fact that this expression has finite expectation follows from Theorem \ref{thm:cle-indep} and Lemma \ref{lem:superexp-cle-tail} after using Hölder's inequality with exponents $p$ and $(q_B\colon B\subseteq [n], \#B>1)$, taking the exponent $p$ sufficiently close to $1$ so that $(\rho_i\wedge 0)p > -1+2/\kappa+3\kappa/32$ for all $i\le n$ (noting that $-1+2/\kappa+3\kappa/32<0$).
	
	For the universal upper bound at the end of the statement of the theorem we use Hölder's inequality with the above exponents to get
	\begin{align*}
		&\E\left(\prod_{i\in [n]} R(z_i,\eta_i^o)^{\rho_{\sigma_i}^\kappa} \exp\left( \sum_{B\subseteq [n]\colon \#B> 1} \sigma_B N^0_B(\z) \right) \right) \\
		&\quad \le \E\left( \prod_{i\in [n]} e^{\sigma_i}R(z_i,\eta_i^o)^{p\rho_{\sigma_i}^\kappa}\right)^{1/p} \prod_{B\subseteq [n]\colon \#B> 1} \E\left(e^{q_B\sigma_B N^0_B(\z)}\right)^{1/q_B}\;.
	\end{align*}
	The result now follows from Theorem \ref{thm:cle-indep} and Lemma \ref{lem:superexp-cle-tail}.

	It remains to establish the universal lower bound. Let $\widetilde{\eta}_\epsilon$ be the outermost loop in $\widetilde{\Gamma}$ surrounding $0$ and contained in $\epsilon\D$. For each $i$, let $\phi_i\colon \D\to D$ be a conformal map with $\phi_i(0)=z_i$. Then by compactness of $K$ and since $\delta>0$ there exists $\epsilon_0=\epsilon_0(D,K,\delta)\in (0,1)$ such that $\phi_i(\epsilon_0\D)\cap\{z_1,\dots,z_n\}=\{z_i\}$ for all $i$. Thus by conformal invariance of CLE, for any $c>0$,
	\begin{align*}
		\P(R(z_i,\eta_i^o)>c)\ge \P( |\phi'_i(0)|R(0,\widetilde{\eta}_{\epsilon_0}) > c).
	\end{align*}
	Moreover, $|\phi'(0)|\ge \delta_0$ for some $\delta_0=\delta_0(D,K,\delta)$ (again, this follows from the compactness of $K$ and the fact that $\delta>0$). Hence, by a union bound,
	\begin{align*}
		\P(R(z_i,\eta_i^o)>c\ \forall i) \ge 1- n\cdot \P( R(0,\widetilde{\eta}_{\epsilon_0}) \le c/\delta_0) \ge 1/2
	\end{align*}
	for $c=c(D,K,\delta,n)>0$ sufficiently small and by decreasing it further, we can also ensure that $R(z_i,D)<1/c$ for all $i$. By Hölder's inequality,
	\begin{align*}
		\frac{1}{2}\prod_{i\in [n]} c^{|\rho_i|/2} &\le \E\left( \prod_{i\in [n]} R(z_i,\eta_i^o)^{\rho^\kappa_{\sigma_i}/2} \right) \\
		&\le \E\left(\prod_{i\in [n]} R(z_i,\eta_i^o)^{\rho_{\sigma_i}^\kappa} \prod_{B\subseteq [n]\colon \#B> 1}e^{ \sigma_B N^0_B(\z)}\right)^{1/2} \E\left(\prod_{B\subseteq [n]\colon \#B> 1}e^{ - \sigma_B N^0_B(\z)}\right)^{1/2}.
	\end{align*}
	The claim now follows since the second term on the right-hand side can be bounded from above by a constant that only depends on $D$, $K$, $\delta$, $n$ and $\bm{\sigma}$.
\end{proof}

The next lemma will not be used until Section \ref{sec:cle-k-big}. However, since the flavor of the proof is similar to other proofs given in this section, we state and prove it here.

\begin{lemma}
	Let $\Gamma$ be a $\CLE_\kappa$ in $\D$ with $\kappa\in (8/3,8)$ and let $r\in(0,1)$. If $\wt{\eta}_r$ is the outermost loop surrounding $0$ that is contained in $r\D$, then $\E(R(0,\wt{\eta}_r^o)^\rho )<\infty$ if and only if $\rho>-1+2/\kappa+3\kappa/32$. Furthermore, if $\bar\eta_r$ is the outermost loop surrounding 0 that is surrounded by $\wt\eta_r$ then we have $\E(R(0,\bar{\eta}_r^o)^\rho )<\infty$.
	\label{prop:moment-loops-smaller-domain}
\end{lemma}
\begin{proof}
	The second assertion of the lemma is an immediate consequence of the first assertion by the Markov property of $\CLE_\kappa$ and Theorem \ref{thm:confradlaw}.
	When $\rho\le -1+2/\kappa+3\kappa/32$ or $\rho\ge 0$, the claim also follows directly from Theorem \ref{thm:confradlaw}, so suppose that $\rho\in (-1+2/\kappa+3\kappa/32,0)$. We first prove the result for a single $r\in (0,1)$ (sufficiently close to $1$) and will at the end deduce the general case. Let $\eta_i$ be the $2i$-th loop surrounding $0$ (counted from the boundary of $\D$) with the convention that $\eta_0$ parametrizes $\partial \D$. Since $\eta_1\subseteq \D$ a.s.\ (i.e., the loop $\eta_1$ a.s.\ does not touch the boundary of the domain),
	\begin{align}
		\label{eq:secondloopboundary}
		p_r := \P(\eta_1 \subseteq r\D)\to 1\quad\text{as $r\uparrow 1$}\quad\text{for $n\ge 0$}\;.
	\end{align}
	Let $C_i$ be the connected component of $\eta_i^o$ containing $0$ and let $\phi_i\colon \D\to C_i$ be the unique conformal transformation with $\phi_i(0)=0$ and $\phi'_i(0)>0$. For each $n\in\N$ define an event $A_n$ as follows
	\begin{align*}
		A_n =\left\{ \phi_i(\eta_{i+1})\not\subseteq r\D \text{ for $i=0,\dots,n-1$ and } \phi_n(\eta_{n+1})\subseteq r\D \right\}\;.
	\end{align*}
	By the spatial Markov property of $\CLE_\kappa$, $\P(A_n)=(1-p_r)^n p_r$. In particular, almost surely $A_n$ occurs for some $n\ge 0$ when $p_r>0$. Moreover, by the Schwarz lemma, on the event $A_n$ we have the inclusion $\wt{\eta}_r^o \supseteq \eta^o_{n+1}$. Hence when $1/p+1/q=1$,
	\begin{align*}
		\E(R(0,\wt{\eta}_r^o)^\rho)&\le \sum_{n\ge 0} \E(R(0,\eta_{n+1}^o)^\rho;A_n) \le \sum_{n\ge 0} \E(R(0,\eta_{n+1})^{\rho p})^{1/p} \P(A_n)^{1/q} \\
		&= \sum_{n\ge 0} \E(R(0,\eta_1)^{p\rho})^{(n+1)/p} (1-p_r)^{n/q} p_r^{1/q}\;.
	\end{align*}
	By Theorem \ref{thm:confradlaw}, we can take $p>1$ sufficiently close to $1$ such that $\E(R(0,\eta_1)^{p\rho})<\infty$ and then by \eqref{eq:secondloopboundary} we take $r\in (0,1)$ sufficiently close to $1$ such that
	\begin{align*}
		\E(R(0,\eta_1)^{p\rho})^{1/p} (1-p_r)^{1/q} < 1\;.
	\end{align*}

	Thus $\E(R(0,\wt{\eta}_r^o)^\rho)<\infty$ for $r$ sufficiently close to $1$. We now prove that $\E(R(0,\wt{\eta}_{r^k}^o)^\rho)\le \E(R(0,\wt{\eta}_r^o)^\rho)^k$ for $k\ge 1$ by induction from which the general statement directly follows. The $k=1$ case is trivial. For the induction step, let $C_r$ be the connected component of $\wt{\eta}_r$ containing $0$ and let $\psi_r\colon \D\to C_r$ be the unique conformal transformation with $\psi_r(0)=0$ and $\psi_r'(0)>0$. By the Schwarz lemma $|\psi(z)|\le r|z|$ for $z\in \D$ so that $\psi_r(r^{k-1}\D) \subseteq r^k\D$ and hence by the spatial Markov property of $\CLE_\kappa$,
	\begin{align*}
		\E(R(0,\wt{\eta}_{r^k}^o)^\rho) \le \E(R(0,\wt{\eta}_{r^{k-1}}^o)^\rho)\E(R(0,\wt{\eta}_{r}^o)^\rho)
	\end{align*}
	since $\rho\le 0$. The claim follows.
\end{proof}

\subsection{Properties of the nesting statistic}
\label{sec:cle-nesting-stat}

In this short section, we list some key properties that the function $\Phi^{{\sigmab},\kappa}_{D}$ introduced in Theorem \ref{thm:loopexpconv} satisfies. In the later sections on LQG, these properties (rather than the precise definition) of the function $\Phi^{{\sigmab},\kappa}_{D}$ will be used. In this section, we will make use of the terminology from Section \ref{sec:cle-explorations}, in particular the notion of a CPI. There are completely analogous statements for the two lemmas in the $\kappa=4$ case which we omit since they will not be used in this paper and the description of the CPIs is slightly different in that case.

\begin{lemma}
	\label{lem:phi-axioms}
	Fix $\kappa\in (8/3,8)$, $\bm{\sigma}=(\sigma_B\colon \emptyset \neq B\subseteq A)$ and consider a conformal transformation $\psi\colon D\to D'$ for $D$ and $D'$ simply connected strict subsets of $\C$. Let $\Gamma$ be a nested $\CLE_\kappa$ in $D$ and write $\Gamma_0$ for the collection of its outermost loops. For $\bm{z}\in D^A$ (with $z_i\neq z_j$ for $i\neq j$) and $U\subseteq D$, we write $\mathcal{I}(U)=\{i\in A\colon z_i\in U\}$. Then
	\begin{align}
		\label{eq:phi-axioms-no-cpi}
		\begin{split}
			\Phi^{{\sigmab},\kappa}_{D'}(\psi(z_i)\,:\,i\in A) &= \prod_{i\in A} |\psi'(z_i)|^{\rho^\kappa_{\sigma_i}}\cdot \Phi^{{\sigmab},\kappa}_D(\bm{z})\;,\\
			\Phi^{{\sigmab},\kappa}_D(\bm{z}) &= \E\left(\, \prod_{\eta\in \Gamma_0} e^{\sigma_{\mathcal{I}(\eta^o)}} \prod_{U\in \mathfrak{U}(\eta^o)} \Phi^{{\sigmab}|_{\mathcal{I}(U)},\kappa}_{U}(\bm{z}|_{\mathcal{I}(U)}) \right)
		\end{split}
	\end{align}
	where $\mathfrak{U}(V)$ denotes the set of connected components of an open set $V$.
	
	Now suppose that $\kappa\neq 4$ and consider a CPI $\lambda$ (with some asymmetry parameter) in $\Gamma$ from one prime end $w_0$ to another prime end $w_\infty$ of $D$. Let $L_\lambda$ be the collection of all loops in $\Gamma_0$ that $\lambda$ intersects and let $C_\lambda$ be the collection of all open complementary components of the union of $\lambda$ with all loops in $L_\lambda$. Then
	\begin{align}
		\label{eq:phi-axiom-cpi}
		\Phi^{{\sigmab},\kappa}_D(\bm{z}) &= \E\left(\, \prod_{\eta\in L_\lambda} e^{{\sigma}_{\mathcal{I}(\eta^o) } } \cdot \prod_{U\in C_\lambda} \Phi^{{\sigmab}|_{\mathcal{I}(U)},\kappa}_{U}(\bm{z}|_{\mathcal{I}(U)}) \right)\;.
	\end{align}
\end{lemma}

\begin{proof}
	The first equality of \eqref{eq:phi-axioms-no-cpi} is satisfied by conformal invariance of $\CLE_\kappa$ and since we have $R(z_i,U)|\psi'(z_i)|=R(\psi(z_i),\psi(U))$ for any set $U\subseteq D$ which is a disjoint union of simply connected open sets (this follows directly from the definition of conformal radii). The second equality will follow by conditioning on the outermost loops $\Gamma_0$ of the nested $\CLE_\kappa$. In the following, we use the notation appearing in Theorem \ref{thm:loopexpconv} and  we will split the outermost CLE loops into those surrounding exactly one point and those surrounding at least two points. We obtain
	\begin{align}
		\label{eq:phi-outer-gamma}
		\begin{split}
		&\E\left( \prod_{i\in A} e^{\sigma_i}R(z_i,\eta_i^o)^{\rho_{\sigma_i}^\kappa} \prod_{B\subseteq A\colon \#B> 1}  e^{ \sigma_B N^0_B(\z)} \mid \Gamma_0 \right) \\
		&\qquad=\prod_{i\in A\colon \eta_i\in \Gamma_0} e^{\sigma_i} R(z_i,\eta_i^o)^{\rho^\kappa_{\sigma_i}} \cdot \left(\, \prod_{\eta\in \Gamma_0\setminus \{\eta_i\colon i\in A\}} e^{\sigma_{\mathcal{I}(\eta^o)}} \prod_{U\in \mathfrak{U}(\eta^o)} \Phi^{{\sigmab}|_{\mathcal{I}(U)},\kappa}_{U}(\bm{z}|_{\mathcal{I}(U)}) \right) \\
		&\qquad= \prod_{\eta\in \Gamma_0} e^{\sigma_{\mathcal{I}(\eta^o)}} \prod_{U\in \mathfrak{U}(\eta^o)} \Phi^{{\sigmab}|_{\mathcal{I}(U)},\kappa}_{U}(\bm{z}|_{\mathcal{I}(U)})
		\end{split}
	\end{align}
	where the second equality follows from the definitions. The second equality in \eqref{eq:phi-axioms-no-cpi} thus follows by taking expectations on both sides and by using the definition of $\Phi^{\bm{\sigma},\kappa}_D(\bm{z})$.

	The statement \eqref{eq:phi-axiom-cpi} is obtained as follows. First of all, note that $\Gamma$ and $\lambda$ are conditionally independent given $\Gamma_0$. Therefore
	\begin{align*}
		\E\left( \prod_{i\in A} e^{\sigma_i}R(z_i,\eta_i^o)^{\rho_{\sigma_i}^\kappa} \prod_{B\subseteq A\colon \#B> 1}  e^{ \sigma_B N^0_B(\z)} \mid (\lambda,\Gamma_0) \right) = \prod_{\eta\in \Gamma_0} e^{\sigma_{\mathcal{I}(\eta^o)}} \prod_{U\in \mathfrak{U}(\eta^o)} \Phi^{{\sigmab}|_{\mathcal{I}(U)},\kappa}_{U}(\bm{z}|_{\mathcal{I}(U)})
	\end{align*}
	by \eqref{eq:phi-outer-gamma}. By the tower property we get
	\begin{align*}
		&\E\left( \prod_{i\in A} e^{\sigma_i}R(z_i,\eta_i^o)^{\rho_{\sigma_i}^\kappa} \prod_{B\subseteq A\colon \#B> 1}  e^{ \sigma_B N^0_B(\z)} \mid (\lambda,L_\lambda) \right) \\
		&\qquad= \E\left( \prod_{\eta\in \Gamma_0} e^{\sigma_{\mathcal{I}(\eta^o)}} \prod_{U\in \mathfrak{U}(\eta^o)} \Phi^{{\sigmab}|_{\mathcal{I}(U)},\kappa}_{U}(\bm{z}|_{\mathcal{I}(U)}) \mid (\lambda,L_\lambda) \right)\;.
	\end{align*}
	By Section \ref{sec:cle-explorations}, the conditional law of $\Gamma_0$ given $(\lambda,L_\lambda)$ is obtained by taking the union of $L_\lambda$ together with an independent non-nested $\CLE_\kappa$ $\Gamma_U$ within each $U\in C_\lambda':= C_\lambda\setminus \cup_{\eta \in L_\lambda} \mathfrak{U}(\eta^o)$. We deduce that
	\begin{align*}
		&\E\left(\prod_{i\in A} e^{\sigma_i}R(z_i,\eta_i^o)^{\rho_{\sigma_i}^\kappa} \prod_{B\subseteq A\colon \#B> 1}  e^{ \sigma_B N^0_B(\z)} \mid (\lambda,L_\lambda) \right) \\
		&\ = \prod_{\eta\in L_\lambda} e^{\sigma_{\mathcal{I}(\eta^o)}}  \prod_{U\in \mathfrak{U}(\eta^o)} \Phi^{{\sigmab}|_{\mathcal{I}(U)},\kappa}_{U}(\bm{z}|_{\mathcal{I}(U)}) \cdot \prod_{U\in C_\lambda'} \mathbb{E}\left( \prod_{\eta \in \Gamma_U} e^{\sigma_{\mathcal{I}(\eta^o)}} \prod_{V\in \mathfrak{U}(\eta^o)} \Phi^{{\sigmab}|_{\mathcal{I}(V)},\kappa}_{V}(\bm{z}|_{\mathcal{I}(V)}) \mid \Gamma_U\right) \\
		&\ = \prod_{\eta\in L_\lambda} e^{{\sigma}_{\mathcal{I}(\eta^o) } } \cdot \prod_{U\in C_\lambda} \Phi^{{\sigmab}|_{\mathcal{I}(U)},\kappa}_{U}(\bm{z}|_{\mathcal{I}(U)})
	\end{align*}
	where the final equality follows from the second part of \eqref{eq:phi-axioms-no-cpi} applied to $\Phi^{{\sigmab}|_{\mathcal{I}(U)},\kappa}_{U}(\bm{z}|_{\mathcal{I}(U)})$. The claim \eqref{eq:phi-axiom-cpi} now follows by taking the expectation on both sides.
\end{proof}

\begin{lemma}
	Consider $\kappa\in (8/3,8)\setminus\{4\}$, $\beta\in [-1,1]$ and a simply connected domain $D\subsetneq \C$. We also fix two distinct prime ends $w_0$ and $w_\infty$. Let $(\Gamma,\lambda)$ be a coupling of a $\CLE_\kappa$ in $D$ and a CPI with asymmetry parameter $\beta$ from $w_0$ to $w_\infty$ in $D$.
	
	Similarly, we let $(\Gamma',\lambda')$ be the coupling of a $\CLE_\kappa^{\bm{\sigma}}(\z)$ with a corresponding CPI from $w_0$ to $w_\infty$ in $D$ (see Definition \ref{def:cle}).
	Recall the notation of $(L_\lambda,C_\lambda)$ from Lemma \ref{lem:phi-axioms} and define $(L_{\lambda'},C_{\lambda'})$ analogously. Also recall the other notation from this lemma. Then the Radon-Nikodym derivative of $(\lambda',L_{\lambda'})$ with respect to $(\lambda,L_\lambda)$ is given by
	\begin{align*}
	\frac{1}{\Phi^{{\sigmab},\kappa}_{D}(\z)}\, \prod_{\eta\in L_\lambda} e^{{\sigma}_{\mathcal{I}(\eta^o) } }  \cdot \prod_{U\in C_\lambda} \Phi^{{\sigmab}|_{\mathcal{I}(U)},\kappa}_{U}(\bm{z}|_{\mathcal{I}(U)}) \;.
	\end{align*}
	The conditional law of $\Gamma'$ given $(\lambda',L_{\lambda'})$ is given by taking an independent $\CLE^{\bm{\sigma}|_{\mathcal{I}(U)}}_\kappa(z|_{\mathcal{I}(U)})$ for each $U\in C_{\lambda'}$ together with the loops $L_{\lambda'}$.
	\label{lem:xi-law}
\end{lemma}

\begin{proof}
	By definition, the Radon-Nikodym derivative of $\Gamma'_0$ with respect to $\Gamma_0$ is given by
	\begin{align*}
		\frac{1}{\Phi^{\bm{\sigma},\kappa}_D(\bm{z})}\, \prod_{\eta\in \Gamma_0} e^{\sigma_{\mathcal{I}(\eta^o)}} \prod_{U\in \mathfrak{U}(\eta^o)} \Phi^{{\sigmab}|_{\mathcal{I}(U)},\kappa}_{U}(\bm{z}|_{\mathcal{I}(U)}).
	\end{align*}
	Therefore the Radon-Nikodym derivative of $(\lambda',L_{\lambda'})$ with respect to $(\lambda,L_\lambda)$ equals
	\begin{align}
		\label{eq34}
		\begin{split}
		&\frac{1}{\Phi^{\bm{\sigma},\kappa}_D(\bm{z})}\,\E\left( \prod_{\eta\in \Gamma_0} e^{\sigma_{\mathcal{I}(\eta^o)}} \prod_{U\in \mathfrak{U}(\eta^o)} \Phi^{{\sigmab}|_{\mathcal{I}(U)},\kappa}_{U}(\bm{z}|_{\mathcal{I}(U)}) \mid (\lambda,L_\lambda) \right) \\
		&\qquad= \frac{1}{\Phi^{{\sigmab},\kappa}_{D}(\z)}\, \prod_{\eta\in L_\lambda} e^{{\sigma}_{\mathcal{I}(\eta^o) } } \cdot \prod_{U\in C_\lambda} \Phi^{{\sigmab}|_{\mathcal{I}(U)},\kappa}_{U}(\bm{z}|_{\mathcal{I}(U)})
		\end{split}
	\end{align}
	where the equality follows from Section \ref{sec:cle-explorations} as in Lemma \ref{lem:phi-axioms}. To see the final part of the lemma, by the definition of a nested $\CLE_\kappa^{\bm{\sigma}}(\bm{z})$, it suffices to describe the conditional law of $\Gamma'_0$ given the process $(\lambda',L_{\lambda'})$. The statement in this case is a consequence again of reweighting the conditional law of $\Gamma_0$ given $(\lambda,L_\lambda)$ described in Section \ref{sec:cle-explorations}.
\end{proof}

\subsection{Spatial independence for simple CLE}
\label{sec:cle-k-small}

The aim of this section will be to establish Theorem \ref{thm:cle-indep} in the case $\kappa\in (8/3,4]$. For this, we rely on the construction of nested $\CLE_\kappa$ using the Brownian loop soup. We thus begin by quickly reviewing this construction. Consider $D\subseteq \C$ and define the following infinite and non-atomic measure  on $C([0,\infty),\C)$ by
\begin{align*}
	\mu_D = \int_D dz \int_0^\infty dt\,\frac{p_t(z,z)}{t}\, P^t_z(\,\cdot\,; \omega([0,t])\subseteq D)\;.
\end{align*}
Here $p_t(z,w)= \exp(-|z-w|^2/(2t))/(2\pi t)$ is the heat kernel, $P^t_z$ is the law of a Brownian bridge in $\C$ from to $z$ to $z$ of duration $t$, i.e., it is the law of $z+B_{\,\cdot\,\wedge\, t} - (1\wedge (\cdot/t) )\cdot B_t$ where $B$ is a two-dimensional standard Brownian motion and $\omega$ is the coordinate process on $C([0,\infty),\C)$. The measure $\mu_D$ is called the \emph{Brownian loop measure} and was introduced by Lawler and Werner in \cite{lawler-werner-soup}.

 Two key properties of the Brownian loop measure are its conformal invariance and its restriction property. To state them properly, let $\sim$ be an equivalence relation on $C(\partial \D,\C)$ such that $\eta\sim\wt\eta$ if there is an orientation-preserving homeomorphism $\psi:\partial\D\to\partial\D$ such that $\eta(s)=\wt\eta(\psi(s))$ for all $s\in\partial\D$. Then $C(\partial \D,\C)/\!\sim$ is the set of loops on $\C$ viewed modulo reparametrization of time. The measure $\mu_D$ induces a measure $\wt{\mu}_D$ on $C(\partial \D,\C)/\!\sim$. Conformal invariance means that $([\ell]\mapsto [\phi\circ \ell])_*\wt{\mu}_D = \wt{\mu}_{D'}$ for every conformal transformation $\phi\colon D\to D'$, and the restriction property means that $\wt{\mu}_D$ restricted to the loops that are contained in $D'$ is the measure $\wt{\mu}_{D'}$ whenever $D'\subseteq D\subseteq \C$. The conformal invariance can be found in \cite{lawler-werner-soup} while the restriction property is immediate from the definition.

Let $\mathcal{L}_c$ be a Poisson point process with intensity $c\mu_D$ (called the Brownian loop soup of intensity $c$ in $D$). $\mathcal{L}_c$ is a.s.\ an infinite collection of loops and we say that two loops $\ell,\ell'\in \mathcal{L}_c$ are in the same cluster if there exist loops $\ell_0,\dots,\ell_n\in \mathcal{L}_c$ such that $\ell_0=\ell$, $\ell_n=\ell'$, and $\ell_i$ intersects $\ell_{i+1}$ for $i=0,\dots,n-1$. For each cluster, we can consider the closure of the union of all the loops within it; whenever the outer boundary of a cluster is a simple curve, we add it to a set $\mathcal{C}_c$ (all curves are viewed up to reparametrization). In fact, it turns out that almost surely, all clusters have a simple outer boundary. The remarkable result due to Sheffield and Werner \cite[Theorem 1.6]{shef-werner-cle} is that $\mathcal{C}_c$ is a CLE.

\begin{thm}[\cite{shef-werner-cle}]
	\label{thm:cle-werner-sheffield}
	When $c\in (0,1]$, $\mathcal{C}_c$ is almost surely an infinite collection of disjoint simple loops and has the law of
	a non-nested $\CLE_\kappa$ in $D$ with $\kappa\in (8/3,4]$ where $\kappa$ and $c$ are related via $c= (3\kappa-8)(6-\kappa)/(2\kappa)$.
\end{thm}

It is also true that $\mathcal{L}_c$ encodes an entire nested $\CLE_\kappa$. Indeed, let $\mathcal{C}^1_c = \mathcal{C}_c$ and inductively define $\mathcal{C}^{n+1}_c$ from $\mathcal{C}^n_c$ and $\mathcal{L}_c$ as follows: Let $\mathcal{L}^{n+1}_c$ be all the loops in $\mathcal{L}_c$ that are surrounded by a loop in $\mathcal{C}^n_c$ but do not intersect any loop in $\mathcal{C}^n_c$. The collection $\mathcal{C}^{n+1}_c$ is then obtained from $\mathcal{L}^{n+1}_c$ as its set of outer cluster boundaries (analogous to the construction of $\mathcal{C}_c$ from $\mathcal{L}_c$). Finally, let $\mathcal{C}^\infty_c = \cup_{n\ge 1} \mathcal{C}^n_c$.

\begin{prop}
	\label{prop:nestedcle-loopsoup}
	When $c\in (0,1]$, $\mathcal{C}^\infty_c$ is a nested $\CLE_\kappa$ in $D$ with $\kappa\in (8/3,4]$ where $\kappa$ and $c$ are related by $c= (3\kappa-8)(6-\kappa)/(2\kappa)$.
\end{prop}

\begin{proof}
	This is an immediate consequence of \cite[Theorem 1]{werner-qian-loopclusters} and Theorem \ref{thm:cle-werner-sheffield}.
\end{proof}

\begin{figure}
	\centering
	\def\svgwidth{0.8\columnwidth}
	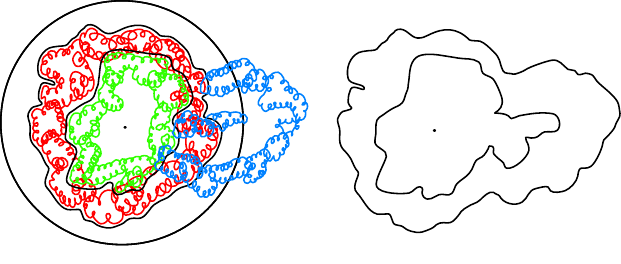
	\caption{Illustration of the proof of Theorem \ref{thm:cle-indep} for $\kappa\in(8/3,4]$. \emph{Left.} The clusters of the loop soup restricted to the ball $B_{\delta/2}(z_i)$ that touch the outermost and second outermost cluster are drawn in red and green, respectively. The blue loops are not contained in $B_{\delta/2}(z_i)$ and we therefore
	obtain larger clusters. \emph{Right.} These are the (nested) clusters of the whole CLE. Note that some of the red loops on the left no longer touch the boundary of the new outermost cluster and hence now contribute to the boundary of the second outermost cluster.}
	\label{fig:loopsoup}
\end{figure}

We can now establish Theorem \ref{thm:cle-indep} in the case where the CLE is simple.

\begin{proof}[Proof of Theorem \ref{thm:cle-indep} for {$\kappa\in(8/3,4]$}]
	Without loss of generality $D=\D$ and $\rho_1,\dots,\rho_n\le 0$. The proof will exploit the coupling in Proposition \ref{prop:nestedcle-loopsoup} of a nested $\CLE_\kappa$ $\mathcal{C}^\infty_c$ with a Brownian loop soup $\mathcal{L}_c$. Note that the balls $B_{\delta/2}(z_i)\subseteq \D$ for all $i=1,\dots,n$ are pairwise disjoint. For each $i\le n$, if $\mathcal{L}_c(i)$ denotes the set of loops in $\mathcal{L}_c$ that are contained in $B_{\delta/2}(z_i)$ then $\mathcal{L}_c(i)$ has the law of a Brownian loop soup in $B_{\delta/2}(z_i)$ and hence defines a nested $\CLE_\kappa$ $\mathcal{C}_c^\infty(i)$ in $B_{\delta/2}(z_i)$ via the outer boundaries of its clusters. For $i\le n$, let
	\begin{align*}
		(\eta^-_{i,k}\colon k\ge 1)\quad\text{and}\quad (\eta^+_{i,k}\colon k\ge 1)
	\end{align*}
	denote the loops surrounding $z_i$ in $\mathcal{C}_c^\infty(i)$ and $\mathcal{C}^\infty_c$, respectively, such that both collections of loops are ordered according to their nesting structure.

	Fix $i\le n$. The key property we need is that $(\eta^-_{i,k})^o\subseteq (\eta^+_{i,k})^o$ for all $k\ge 1$. We prove this key property by
	induction on $k$. See Figure \ref{fig:loopsoup}. The induction hypothesis $k=1$ is clear. For the induction step, we assume that the claim holds for $k$ and we wish to deduce it for $k+1$. This follows from the simple observation that a Brownian loop that is strictly contained in $(\eta^-_{i,k})^o$ is also strictly contained in $(\eta^+_{i,k})^o$.

	In the notation of the theorem, let $N$ be the minimal integer such that
	\begin{align*}
		(\eta^+_{i,N})^o \subseteq \eta_i^o\quad\text{for all $i\le n$}\;.
	\end{align*}
	Combining this with the inclusion $(\eta^-_{i,k})^o\subseteq (\eta^+_{i,k})^o$ whenever $i,k\ge 1$ and applying Hölder's inequality with exponents $p,q>1$ such that $1/p+1/q=1$ yields
	\begin{align*}
		&\E\left(\, \prod_{i=1}^n R(z_i,\eta_i^o)^{\rho_i} \right) \le \sum_{k\ge 1} \E\left(\, \prod_{i=1}^n R(z_i,(\eta_{i,k}^-)^o)^{\rho_i} ;N=k\right) \\
		&\le \sum_{k\ge 1} \E\left(\, \prod_{i=1}^n R(z_i,(\eta_{i,k}^-)^o)^{\rho_i p}\right)^{1/p} \P(N=k)^{1/q} = \sum_{k\ge 1} \prod_{i=1}^n \E\left(R(z_i,(\eta_{i,k}^-)^o)^{\rho_i p}\right)^{1/p} \P(N=k)^{1/q}\\
		&= \sum_{k\ge 1} \prod_{i=1}^n (\delta/2)^{\rho_i} \E\left(R(0,\wt{\eta}^o)^{\rho_i p}\right)^{k/p} \P(N=k)^{1/q}
	\end{align*}
	where $\wt{\eta}$ is the outermost loop in $\mathcal{C}^\infty_c$ surrounding $0$, and we used independence of the CLEs $(\mathcal{C}^\infty_c(i)\colon i\le n)$ and the conformal invariance of CLE. We take $p>1$ sufficiently close to $1$ such that $\rho_ip>-1+2/\kappa+3\kappa/32$ for all $i\le n$. Observe that
	\begin{align*}
		N \le \max\{N_{ij}\colon 1\le i<j\le n\}\le \sum_{1\le i<j\le n} N_{ij}
	\end{align*}
	where $N_{ij}$ is the number of loops in $\mathcal{C}^\infty_c$ surrounding both $z_i$ and $z_j$. Now take $\theta>0$ such that
	\begin{align*}
		\prod_{i=1}^n \E\left(R(0,\wt{\eta}^o)^{\rho_i p}\right)^{1/p} < e^{\theta/q}\;.
	\end{align*}
	By Markov's inequality, Hölder's inequality and Lemma \ref{lem:superexp-cle-tail},
	\begin{align*}
		\P(N=k)&\le e^{-\theta k}\,\E\left(e^{\theta N}\right) \le e^{-\theta k}\,\E\left(\prod_{1\le i<j\le n} e^{\theta N_{ij}}\right) \\
		&\le e^{-\theta k}\prod_{1\le i<j\le n}\E\left(e^{n(n-1)/2\cdot \theta N_{ij}}\right)^{2/(n(n-1))} \le C\cdot e^{-\theta k}
	\end{align*}
	for some constant $C>0$ where $C$ only depends on $\delta$, $\theta$ and the number of points $n$. The result is now immediate.
\end{proof}

\subsection{Spatial independence for non-simple CLE}
\label{sec:cle-k-big}

In this section we will prove Theorem \ref{thm:cle-indep} for $\kappa'\in(4,8)$. The idea of the proof is to use a coupling between $\CLE_{\kappa'}$ and a GFF $h$ based on imaginary geometry and then use spatial independence properties of the GFF. Throughout the section we set $\kappa=16/\kappa'$. 

We say that $h$ is a GFF with \emph{clockwise (resp.\ counterclockwise) boundary data} on $(\D,-i)$ if it is a Dirichlet GFF in $\D$ with boundary data given by
\begin{align*}
	-\pi/\sqrt{{\kappa'}}+\chi\pi/2+\chi\op{arg}(z) \quad\text{and}\quad \pi/\sqrt{{\kappa'}}-3\chi\pi/2+\chi\op{arg}(z)\;,
\end{align*}
respectively, where $\chi=2/\sqrt{{\kappa}}-\sqrt{{\kappa}}/2$ and $\op{arg}(z)\in[-\pi/2,3\pi/2)$. For $D$ an arbitrary simply connected domain with a distinguished prime end $z_0$ we say that $h$ is a GFF on $(D,z_0)$ with clockwise (resp.\ counterclockwise) boundary data if $h\circ\phi-\chi\op{arg}(\phi')$ is a GFF with clockwise (resp.\ counterclockwise) boundary data in $\D$ for a conformal map $\phi:\D\to D$ satisfying $\phi(-i)=z_0$. We use the name clockwise (resp.\ counterclockwise) since in the coupling between the GFF and the branching $\SLE_{\kappa'}$ described in (i) and (ii) right below, if the branching $\SLE_{\kappa'}$ encloses a domain $D$ by tracing its boundary in clockwise (resp.\ counterclockwise) direction, then $h|_D$ will have clockwise (resp.\ counterclockwise) boundary data.

\begin{lemma}\label{prop:harm-ext}
	Let $D\subseteq\C$ be a domain, let $\delta,p>0$, let $z\in D$ be such that $\op{dist}(z,\C\setminus D)>\delta$, and let $h$ be a GFF in $D$ with clockwise or counterclockwise boundary data. Let $\frk h$ denote the harmonic extension  of $h|_{D\setminus B}$ to $B$ where $B=B_{\delta/2}(z)$. Then there is a constant $c=c(D,\delta,p)\in(0,\delta/4)$ such that for all $r<c$ we have
	\begin{align*}
		\E \left(\exp \left(p\|\frk h|_{B_r(z)}\|^2_\nabla \right)\right)<c^{-1} \;.
	\end{align*}
\end{lemma}
\begin{proof}
	Using e.g.\ \cite[Section 2.2, Theorem 7]{evans-book} (and the harmonicity of the function $\mathfrak{h}$) we can see that there is a universal constant $c_0>0$ such that for all $w\in B_{\delta/2}(z)$,
	\begin{align*}
	\begin{split}
		|\nabla\frk h(w)|
		\leq c_0\op{ave}(|\frk h|; \partial B_\delta(z) ) /\delta\;,
	\end{split}
	\end{align*}
	where $\op{ave}(|\frk h|; \partial B_\delta(z) )$ denotes the average of $|\frk h|$ over $\partial B_\delta(z)$.
	Therefore for $r\in (0,\delta/2)$ and by Jensen's inequality,
	\begin{align}
	\|\frk h|_{B_r(z)}\|_\nabla^2
	\leq \pi c_0^2 (r/\delta)^2\,\op{ave}(\frk h^2; \partial B_{\delta}(z) ) \;.
	\label{eq7}
	\end{align} 
	Again using Jensen's inequality implies
	\begin{align}
	\exp\left(\pi pc^2_0 (r/\delta)^2\op{ave}(\frk h^2; \partial B_\delta(z) )\right)
	\le \op{ave} \left( e^{\pi pc^2_0 (r/\delta)^2 \frk h^2}  ; \partial B_\delta(z) \right)\;.
	\label{eq8}
	\end{align}
	By combining \eqref{eq7} and \eqref{eq8} we get
	\begin{align*}
	\begin{split}
		\E \left(\exp(p\|\frk h|_{B_r(z)}\|^2_\nabla)\right)
		&\le
		\E \left( \exp\left(\pi pc^2_0 (r/\delta)^2 \op{ave}(\frk h^2; \partial B_\delta(z) )\right) \right) \\
		&\le \op{ave} \left( \E\left( e^{\pi pc^2_0 (r/\delta)^2 \frk h^2} \right) ; \partial B_\delta(z) \right).
	\end{split}
	\end{align*}
	The lemma now follows since the random variables $\frk h(w)$ for $w\in \partial B_\delta(z)$ are
	Gaussian with variance and expectation bounded above by a constant depending only on $D$ and $\delta$. 
\end{proof}
 
We will now recall that there exists a coupling between the following three random objects, such that each object determines the other three objects:
\begin{itemize}
	\item[(i)] a GFF $h$ on $(\D,-i)$ with clockwise boundary data,
	\item[(ii)] a branching $\SLE_{\kappa'}$ $\cL=(\lambda_z\colon z\in\cQ)$ in $\D$ started from $-i$ with force point at $(-i)^-$ (where $\cQ$ is the set of rational points in $\D$), 
	\item[(iii)] a $\CLE_{\kappa'}$ $\Gamma$ in $\D$.
\end{itemize}
There is also a coupling with counterclockwise instead of clockwise boundary data in (i) and with $(-i)^+$ instead of $(-i)^-$ in (ii). We will focus on the case of clockwise boundary data and force point at $(-i)^-$ and only point out the cases where the coupling is different for  
counterclockwise boundary data and force point at $(-i)^+$. We refer to Figure \ref{fig:cle-nonsimple} for an illustration of the coupling.

A branching $\SLE_{\kappa'}$ in $\D$ started from $-i$ with force point at $(-i)^-$ is a set of curves $\cL=(\lambda_z\colon z\in\cQ)$ such that for each $z\in\cQ$, the curve $\lambda_z$ has the law of a radial $\SLE_{\kappa'}({\kappa'}-6)$ from $-i$ to $z$ with the following force point behavior: The force point is initially located infinitesimally to the left (i.e., clockwise) of the starting point $-i$; the force point flips to the other side whenever the curve completes a clockwise (resp.\ counterclockwise) loop around $z$ and before the flip, the force point lies right (resp.\ left) of the tip of the curve.
Furthermore, for any $z,w\in\cQ$, $\lambda_z$ and $\lambda_w$ are identical until $z$ and $w$ lie in different complementary components of the curve, and afterwards the two curves evolve independently.

The coupling between the objects in (i) and (ii) is based on the theory of imaginary geometry \cite{ig1,ig4}. 
Recall from \cite{ig4} that for each $z\in\cQ$, one can define the following curves:
\begin{itemize}
	\item An east-going flow line $\xi^{\op{E}}_z: [0,\infty) \to \ol\D$ for $h$ started at $z$ which by convention merges into the curve $\partial\D$ upon hitting it and traces a segment of $\partial\D$ in counterclockwise direction until hitting $-i$.
	\item A west-going flow line $\xi^{\op{W}}_z: [0,\infty) \to \ol\D$ for $h$ started at $z$ which terminates at $-i$ but does not merge into $\partial\D$ before that.
\end{itemize}

The flow lines are determined locally by $h$ in the sense that if $z\in U\subseteq\D$ for some fixed domain $U$ then the curves $\xi^{\op{E}}_z, \xi^{\op{W}}_z$ stopped when they first leave $U$ are a.s.\ determined by $h|_U$. Furthermore, if an east-going (resp.\ west-going) flow line hits another east-going (resp.\ west-going) flow line, then these two flow lines merge, while if an east-going (resp.\ west-going) flow line hits a west-going (resp.\ east-going) flow line, then it will bounce off without crossing the west-going (resp.\ east-going) flow line.

We next review some properties of the coupling between (i) and (ii) here and refer the reader to \cite[Section 4]{ig4} for a complete description. For any $z\in\cQ$ the curve $\lambda_z$ can be constructed by first sampling its left and right boundaries (viewed as curves in the universal cover of $\D\setminus\{z \}$), which are given by $\xi^{\op{W}}_z$ and $\xi^{\op{E}}_z$, respectively. Suppose $V\subseteq\D$ is a complementary component of $\D\setminus (\xi^{\op{W}}_z\cup \xi^{\op{E}}_z)$ which lies between $\xi^{\op{W}}_z$ and $\xi^{\op{E}}_z$, and let $x\in\partial V$ (resp.\ $y\in\partial V$) denote the point of intersection between $\xi^{\op{W}}_z$ and $\xi^{\op{E}}_z$ on $\partial V$ which is ordered last (resp.\ first).

We obtain $\lambda_z$ restricted to $V$ as follows: To the field $h|_V$ we associate the counterflow line in $V$ from $x$ to $y$ (see \cite[Theorem 1.1]{ig1}), which has the marginal law of a chordal $\SLE_{\kappa'}(\kappa'/2-4;\kappa'/2-4)$ in $V$ (from $x$ to $y$); this is a boundary filling SLE with force points. If the GFF has clockwise boundary data and the force point of $\cL$ is at $(-i)^+$ (i.e., counterclockwise or to the right of $-i$) then the coupling described above is identical.
	
	For $z\in\cQ$ let $V$ be a complementary component of $\lambda_z$, i.e., $V$ is a connected component of the set $\D\setminus\lambda_z$. For a connected set $I\subseteq \partial V\cap \lambda_z$ we say that $I$ is oriented clockwise (resp.\ counterclockwise) if the points of $I$ are visited in clockwise (resp.\ counterclockwise) order by $\lambda_z$. In particular, this definition can be applied to the full boundary $\partial V$ (i.e., the case $I=\partial V$) when $\partial V\subseteq\lambda_z$. If $\partial V\cap\partial\D\neq\emptyset$ then we say that $\partial V$ is oriented clockwise (resp.\ counterclockwise) if $\partial V\cap\lambda_z$ is oriented clockwise (resp.\ counterclockwise) and the boundary data of $h$ are clockwise (resp.\ counterclockwise). The crucial property of the coupling with $h$ is the following: If $\partial V$ is ordered clockwise (resp.\ counterclockwise) then the conditional law of $h|_V$ given $\lambda_z$ is that of a GFF with clockwise (resp.\ counterclockwise) boundary data and the fields in the different complementary components of $\lambda_z$ are conditionally independent given $\lambda_z$.
	
	We define a function $\Upsilon_{\op{GFF}\to \SLE}$ which takes as input a distribution on some simply connected domain $D$ and outputs a collection of curves on that domain, such that in the coupling above we have $\cL=\Upsilon_{\op{GFF}\to \SLE}(h)$ almost surely.

	\begin{figure}
		\centering
		\def\svgwidth{.95\columnwidth}
		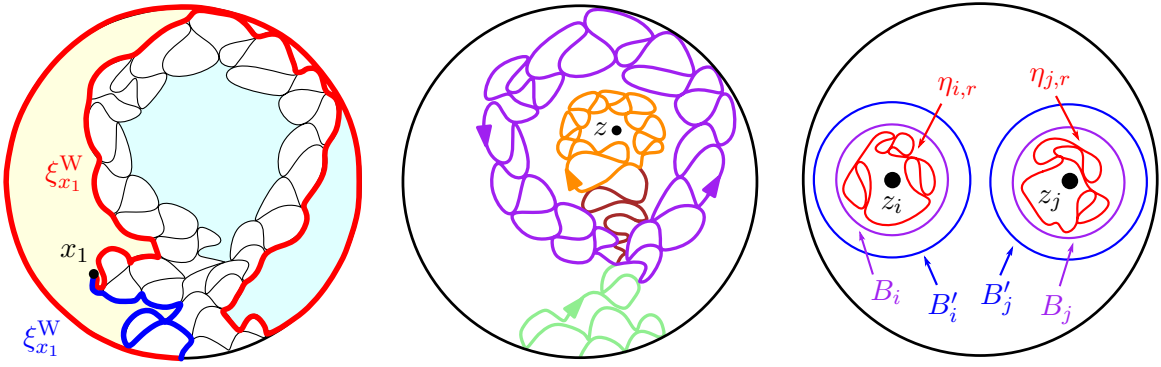
		\caption{
		\emph{Left.} Illustration of the coupling between (i) and (ii) and the definition of a pocket (see Definition \ref{def:pocket}). The figure shows $\lambda_{x_1}$ for some $x_1\in\cQ$ (black), $\xi^{\op{W}}_{x_1}$ (blue), and $\xi^{\op{E}}_{x_1}$ (red). The two domains in light blue are pockets while the domain in light yellow is not a pocket since its red boundary arc is ordered clockwise and its blue boundary arc is ordered counterclockwise. 
		{\emph{Center.}} Illustration of the coupling between (ii) and (iii). The path obtained by concatenating the green, purple, brown, and orange path segments is $\lambda_z|_{[0,\tau_2]}$, while the path segments $\lambda_z|_{[\sigma_1,\tau_1]}$ (purple) and $\lambda_z|_{[\sigma_2,\tau_2]}$ (orange) are contained in the outermost and in the second outermost CLE loop surrounding $z$, respectively. 
		\emph{Right.} Illustration of the proof of Theorem \ref{thm:cle-indep} for ${\kappa'}\in(4,8)$. The relative size of the various circles is not to scale.}
	\label{fig:cle-nonsimple}
	\end{figure}

	\begin{figure}
		\centering
		\def\svgwidth{.95\columnwidth}
		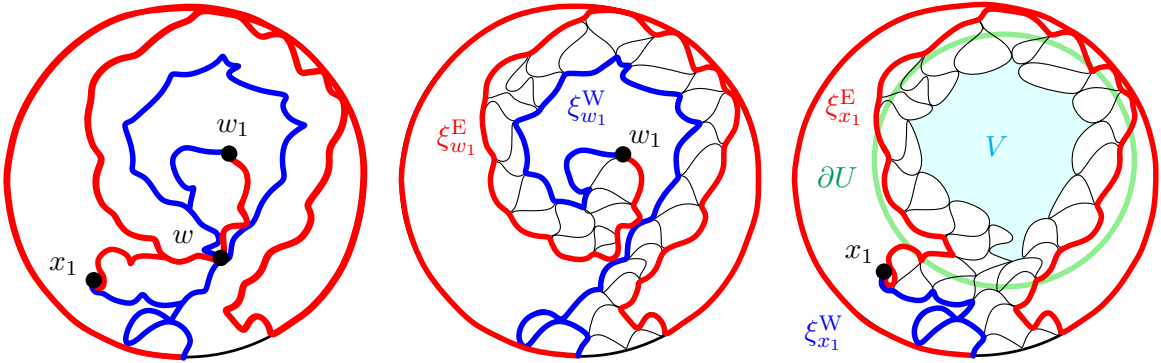
		\caption{Illustration of the proof of Lemma \ref{prop:i-ii}.
		{\emph{Left.}} $\wh\xi\,^{\op{W}}_{\!w}$, $\smash{\xi^{\op{W}}_{w_1}}$, $\smash{\xi^{\op{W}}_{x_1}}$ (blue), $\wh\xi\,^{\op{E}}_{\!w}$, $\xi^{\op{E}}_{w_1}$, and $\xi^{\op{E}}_{x_1}$ (red).
		{\emph{Center.}} $\lambda_{w_1}$ (black), $\xi^{\op{W}}_{w_1}$ (blue), and $\xi^{\op{E}}_{w_1}$ (red).
		{\emph{Right.}} $\lambda_{x_1}$ (black), $\xi^{\op{W}}_{x_1}$ (blue), $\xi^{\op{E}}_{x_1}$ (red), and a pocket $V$ (light blue). The pocket $V$ is a pocket in $U$, where $U$ is the disk with boundary $\partial U$ (light green).
		}
	\label{fig:cle-nonsimple-proof}
\end{figure}

	\begin{defn}
		Consider a coupling of the objects in (i) and (ii) above. A domain $V\subseteq \D$ is called a pocket if it is a complementary component of $\lambda_z$ for some $z\in\cQ$ and if $\partial V$ is  ordered either clockwise or counterclockwise. For $U\subseteq \D$ we say that $V$ is a pocket in $U$ if $V$ is a pocket and the closure of $V$ is contained in $U$.
		\label{def:pocket}
	\end{defn}
	If $z,w\in\cQ$, $V$ is a complementary component of $\lambda_z$, and $w\in V$, then $\lambda_w$ can be written as the concatenation of $\lambda_z$ until this curve encloses $V$ and a curve $\lambda'$ contained in $\ol V$. Let $\cL|_V$ denote the collection of paths consisting of these curves $\lambda'$ for all $w\in\cQ\cap V$.

	\begin{lemma}\label{prop:i-ii}
		The set of pockets $V$ in $U$ is a measurable function of $h|_U$, and for any pocket $V\subseteq U$ it holds a.s.\ that $\Upsilon_{\op{GFF}\to \SLE}(h|_V)$ is well-defined and equal to $\cL|_V$.
	\end{lemma}
	
	\begin{proof}
		We encourage the reader to look at Figure \ref{fig:cle-nonsimple-proof} while reading the proof.
		We will argue that with probability 1, a domain $V\subseteq\D$ with simple boundary is a pocket if and only if we can find a sequence of points $(w_j)\subseteq V\cap\cQ$ converging to some $w\in\D$ and curves $\wh\xi\,^{\op{W}}_{\!w},\wh\xi\,^{\op{E}}_{\!w}$ starting at $w$ and ending at $-i$ such that the following holds (or the following with the role of left/right and E/W swapped):
		\begin{itemize}
			\item The boundary of $V$ is equal to an initial segment of $\wh\xi\,^{\op{W}}_{\!w}$.
			\item The curves $\xi^{\op{W}}_{w_j},\xi^{\op{E}}_{w_j}$ merge into $\wh\xi\,^{\op{W}}_{\!w},\wh\xi\,^{\op{E}}_{\!w}$, respectively, such that $\xi^{\op{W}}_{w_j}\Delta \wh\xi\,^{\op{W}}_{\!w}$ and $\xi^{\op{E}}_{w_j}\Delta \wh\xi\,^{\op{E}}_{\!w}$ have diameters converging to $0$ as $j\to \infty$.
			\item The curve $\xi^{\op{E}}_{w_j}$ contains $\wh\xi\,^{\op{E}}_{\!w}$, $\wh\xi\,^{\op{W}}_{\!w}\cap V=\emptyset$ and $\wh\xi\,^{\op{E}}_{\!w}\cap V=\emptyset$.
		\end{itemize}
		The first assertion of the lemma then follows since flow lines are locally determined by the field to which they are coupled. 
		
		If $V$ is a pocket, pick $x_1\in\cQ$ such that $V$ is a complementary component of $\lambda_{x_1}$ and let $\tau$ denote the time at which $\lambda_{x_1}$ encloses $V$ (i.e., the time at which $\lambda_{x_1}$ finises tracing $\partial V$). Existence of appropriate points $\smash{(w_j)}$ is now immediate since we can let $\smash{(w_j)\subseteq V\cap\cQ}$
		be points converging to $\lambda_{x_1}(\tau)$ and we can let $\wh\xi\,^{\op{W}}_{\!w}$ and $\wh\xi\,^{\op{E}}_{\!w}$ be the left and right, respectively, boundary of $\lambda_{x_1}$ infinitesimally before time $\tau$.
		
		Conversely, suppose we are given points $(w_j)$ and curves $\wh\xi\,^{\op{W}}_{\!w},\wh\xi\,^{\op{E}}_{\!w}$ satisfying the given properties and assume without loss of generality that $\partial V$ is equal to an initial segment of $\wh\xi\,^{\op{W}}_{\!w}$ and not of $\wh\xi\,^{\op{E}}_{\!w}$.
		We claim that $\wh\xi\,^{\op{W}}_{\!w}$ and $\wh\xi\,^{\op{E}}_{\!w}$ are the left and right, respectively, boundaries of $\lambda_{z_1}$ infinitesimally before this curve encloses $V$. If the claim was not true then $\lambda_{z_1}$ would enter $V$ at some point $\wh z\in\partial V\setminus\{w \}$ before finishing to trace $\partial V$, so (by definition of $\cL$) all the curves $\lambda_{z_j}$ would enter $V$ at $\wh z$. This contradicts the assumption that the left and right boundaries of $\smash{\lambda_{z_j}}$ converge to the curves $\wh\xi\,^{\op{W}}_{\!w}$ and $\wh\xi\,^{\op{E}}_{\!w}$ so the claim is correct. Furthermore, by letting $x_1\in V^c\cap\cQ$ be a point which is \emph{not} lying between $\xi^{\op{E}}_{z_1}$ and $\xi^{\op{W}}_{z_1}$ and which is in the same complementary component of $\lambda_{z_1}$ infinitesimally before this curve finishes tracing $\partial V$, we get
		 that $\lambda_{x_1}\cap\lambda_{z_1}$ is equal to $\lambda_{x_1}$ stopped upon enclosing $V$ and that $V$ is a complementary component of $x_1$. 	
		Also observe that $\partial V$ is oriented counterclockwise, while it would have been oriented clockwise if we had assumed that $\partial V$ is equal to an initial segment of $\wh\xi\,^{\op{E}}_{\!w}$ instead of $\wh\xi\,^{\op{W}}_{\!w}$.
		
		We have that $\Upsilon_{\op{GFF}\to \SLE}(h|_V)$ is a.s.\ well-defined for all pockets $V$ since $h|_V$ has the law of a GFF with clockwise boundary data (if $\partial V$ is oriented clockwise) or counterclockwise boundary data (if $\partial V$ is oriented counterclockwise). We get further that $\Upsilon_{\op{GFF}\to \SLE}(h|_V)=\cL|_V$ a.s.\ since by the construction of $\cL$ in terms of $h$ we have that $\cL|_V$ is determined a.s.\ by the flow lines of the field $h|_V$.
\end{proof} 

The coupling between the objects in (ii) and (iii) is from \cite{shef-cle} and is also reviewed in e.g.\ \cite{msw-gasket} and \cite[Section 3.6.3 and Remark 3.28]{ghs-mating-survey}. We will not give a complete description of the construction but recall some key properties. 

We say that $\lambda_z$ \emph{completes a counterclockwise turn around $z$} at time $\tau$ if $\lambda_z$ hits its left boundary at time $\tau$. We say that $\lambda_z$ makes a counterclockwise turn around $z$ during $[\sigma,\tau]$ if it completes a clockwise turn around $z$ at time $\tau$ and $\sigma=\sup\{t\in(0,\tau)\,:\,\lambda_z(t)=\lambda_z(\tau) \}$. Define clockwise turns in the same manner with right instead of left. Let $[\sigma_1,\tau_1]$ be the first interval during which $\lambda_z$ makes a counterclockwise turn around $z$. Iteratively, for each even (resp.\ odd) $k\in\N$ let $[\sigma_k,\tau_k]$ denote the first interval after $[\sigma_{k-1},\tau_{k-1}]$ during which $\lambda_z$ makes a clockwise (resp.\ counterclockwise) turn around $z$. The coupling between the objects in (ii) and (iii) is such that $\lambda_z([\sigma_k,\tau_k])\subseteq\eta_k$, where $\eta_k$ is the $k$-th outermost loop of $\Gamma$ surrounding $z$. Furthermore, we can find a sequence of $z_j\in\cQ$ converging to $\lambda_z(\sigma^z_k)$ such that $\eta_k$ is the limit of $\lambda_{z_j}([\sigma^{z_j}_k,\tau^{z_j}_k])$ for (say) the Hausdorff topology as $j\to\infty$.

If the GFF in (i) has counterclockwise boundary data instead of clockwise boundary data so that $\cL$ has a force point at $(-i)^+$ instead, then the exact same holds, except that the role of clockwise and counterclockwise loops is swapped. Let $\Upsilon_{\SLE\to\CLE}$ be a function such that $\Gamma=\Upsilon_{\SLE\to \CLE}(\cL)$ a.s. 

For a pocket $V$, let $\Gamma|_V:=\{\eta\in\Gamma\,:\,\eta\subseteq\ol V \}$ denote the loops of $\Gamma$ which are contained in $\ol V$. A key property of the coupling between (ii) and (iii) is that if $V$ is a pocket  then $\Gamma|_V=\Upsilon_{\SLE\to \CLE}(\cL|_V)$ almost surely. This can be seen by observing that if $\partial V$ is oriented clockwise (resp.\ counterclockwise) then the outermost loop in $V$ around a point $z\in V\cap\cQ$ corresponds to the first counterclockwise (resp.\ clockwise) turn made by $\lambda_z$ around $z$ after entering $V$. 

For $U\subseteq D$ let $\Gamma^U$ denote the collection of loops $\eta$ for which there exists a pocket $V$ in $U$ such that $\eta\subseteq V$.
\begin{lemma}\label{prop:ii-iii}
	For any fixed domain $U\subseteq\D$ the collection of loops $\Gamma^U$ is a measurable function of $h|_U$. Furthermore, if $\eta,\eta'\in\Gamma|_U$ and $\eta'$ is surrounded by $\eta$ then $\eta'\in\Gamma^U$. 
\end{lemma}

\begin{proof}
	The first assertion is immediate by the last assertion of Lemma \ref{prop:i-ii} and since, as stated above the lemma, $\Gamma|_V=\Upsilon_{\op{SLE}\to \CLE}(\cL|_V)$  for each pocket $V$. We get the second assertion by observing that the complementary component of $\eta$ which contains $\eta'$ is a pocket.
\end{proof}

Note that $\Gamma^U\subseteq\{\eta\in\Gamma\,:\,\eta\in U \}$ and that Lemma \ref{prop:ii-iii} does not hold with the set on the right-hand side instead of $\Gamma^U$, i.e., the set of loops contained in $U$ is not a measurable function of $h|_U$.
 
\begin{proof}[Proof of Theorem \ref{thm:cle-indep} for ${\kappa'}\in(4,8)$]
	It is sufficient to consider the case when $\rho_i<0$ for all $i$ since this implies the general case. Pick $p,q>1$ such that $p\rho_i >-1 + 2/\kappa' + 3\kappa'/32$ and $1/p+1/q=1$. Then in the notation of Lemma \ref{prop:harm-ext}, pick $r\in(0,c(\D,\delta,n(q-1)/2 ))$.

	Consider the coupling between a Dirichlet GFF $h$ with clockwise boundary conditions and a $\CLE_{\kappa'}$ in $\D$ described in (i), (ii) and (iii) above. Define $B'_i:=B_r(z_i)$. Let $\eta_{i,r}$ denote the outermost loop in $\Gamma^{B'_i}$ containing $z_i$ and define $I_i=R(z_i,\eta_{i,r})^{\rho_i}$. Note that $I_i\geq R(z_i,\eta_i^o)^{\rho_i}$ since $\eta_{i,r}$ is either equal to or surrounded by $\eta_i$. Define $B_i:=B_{\delta/3}(z_i)$. 

	Let $\cF$ denote the $\sigma$-algebra generated by $h$ restricted to $\D\setminus \cup_i B_i$, let $\wt h_i$ be a GFF in $B_i$ with clockwise boundary conditions, define $\wt I_i$ just as $I_i$ but with the field $\wt h_i$ instead of $h$, and let $X_i$ denote the Radon-Nikodym derivative of $h|_{B'_i}$ given $\cF$ with respect to $\wt h_i|_{B'_i}$. We have
	\begin{align}
	\begin{split}
		\E\left(\prod_i I_i\right)
		&= \E\left(\E\left(\prod_i I_i \,|\, \cF\right)\right)
		=  \E\left(\prod_i\E(I_i \,|\, \cF)\right)
		=  \E\left( \prod_i\E(\widetilde{I}_i X_i \,|\, \cF)\right) \\
		&\leq  \E\left( \prod_i\E\left(\widetilde{I}_i^{\,p} \,|\, \cF\right)^{1/p} \E(X_i^q \,|\, \cF)^{1/q} \right) 
		= \prod_i \E\left(\widetilde{I}_i^{\,p}\right)^{1/p} 
		\E\left( \prod_i\E(X_i^q \,|\, \cF)^{1/q} \right) \\
		&\leq \prod_i \E\left(\widetilde{I}_i^{\,p}\right)^{1/p} 
		\E\left( \E(X_i^{q} \,|\, \cF)^{n/q}\right)^{1/n}.
	\end{split}
	\label{eq1}
	\end{align}
	By the last assertion of Lemma \ref{prop:ii-iii}, the second outermost loop in $B'_i$ surrounding $z_i$ is either equal to or surrounded by $\eta_{i,r}$. Using this and Lemma \ref{prop:moment-loops-smaller-domain} we get
	\begin{align*}
		\prod_i \E\left(\wt I_i^{\,p}\right)<\infty\;.
	\end{align*}
	For fixed $i$
	let $\frk h$ denote the harmonic extension of $h$ from $\partial B_i$ to $B_i$, minus the harmonic extension of the boundary data of $\wt h_i$ in $B_i$. We have
	\begin{align*}
		X_i = \exp\left( (\wt h_i|_{B'_i},\frk h|_{B'_i})_\nabla-\|\frk h|_{B'_i}\|_\nabla^2/2 \right)
	\end{align*}
	and further
	\begin{align*}
	\begin{split}
		\E\left(\E(X_i^{q}\,|\,\cF)^{n/q}\right)
		&= \E\left( \E(\exp( q(\wt h_i|_{B'_i},\frk h_i|_{B'_i})_\nabla-q\|\frk h|_{B'_i}\|_\nabla^2/2)\,|\,\cF)^{n/q}\right) \\
		&=\E\left( \exp((n/q)\cdot (q^2-q)\|\frk h|_{B'_i}\|_\nabla^2/2)\right)\;.
	\end{split}
	\end{align*}
	This is finite by Lemma \ref{prop:harm-ext}, which implies that the right-hand side of \eqref{eq1} is bounded by a finite constant depending only on $\delta$.
\end{proof}

\section{LQG disks weighted by CLE nesting statistics}
\label{sec:disk-cle-nesting}

\subsection{Definition of the LQG disk weighted by CLE nesting statistics}
\label{sec:disk}

In this section we introduce the LQG surfaces that are studied in the paper. We proceed in two steps: First, we construct a measure on fields on the unit disk (Definition \ref{def:disk}) and then use this to build the weighted LQG surface and weighted generalized LQG surface in the simple and non-simple cases (i.e., without and with pinch points), respectively (Definitions \ref{def:simple-disk} and \ref{def:gen-disk}). Then we derive explicit formulas for the partition functions of the disks in the case of zero or one marked bulk points, and finally we explain the relationship between the partition functions and LCFT correlation functions.

In the following definition we use the notation from Sections \ref{sec:gff-lqg} and \ref{sec:normal-lqg-disks} and Definition \ref{def:main-params}. Let us point out that the measure defined below is not a probability measure.

\begin{defn}
	Suppose that $A,B\subseteq \N$ are finite with $\#B=3$ and that $w_j\in \partial \D$ for $j\in B$ are distinct points. Also, we consider $\kappa\in (8/3,8)\setminus \{4\}$, $\sigma\in \mathfrak{S}_\kappa^A$, $\ell>0$ and $\Lambda\ge 0$. Let $\alpha_i=\alpha^\kappa_{\sigma_i}$ for $i\in A$ and $\gamma=\sqrt{\kappa}\wedge 4/\sqrt{\kappa}$.
	We can now define a measure $\op{M}^{{\sigmab},\kappa}_{\Lambda,\ell}$ on the space $H^{-1}(D)\times \D^A$ by
	\begin{align*}
		\op{M}^{{\sigmab},\kappa}_{\Lambda,\ell}(X)
		& = \frac{\ell^{2/\gamma \sum\alpha_i}}{Z^{(-,-),((\gamma,\gamma,\gamma),\bm{w})}_{0,\ell}} \int_{\D^A}\,d\bm{z}\, \Phi^{{\sigmab},\kappa}_\D(\bm{z}) Z^{(\bm{\alpha},\bm{z}),((\gamma,\gamma,\gamma),\bm{w})}_{\Lambda,\ell}\int P^{(\bm{\alpha},\bm{z}),((\gamma,\gamma,\gamma),\bm{w})}_{\Lambda,\ell}(dk) 1_X(k,\bm{z})
	\end{align*}
	for $X\subseteq H^{-1}(\D)\times \D^A$ measurable.
	We write $\op{M}^{{\sigmab},\kappa}_{\bm{w},\Lambda,\ell}$ when we want to emphasize the dependence on $\bm{w}$ and we write $-$ for the empty tuple.
	\label{def:disk}
\end{defn}

\begin{remark}
	By Definition \ref{def:main-params}, when $\kappa<4$ then the condition $\sigma_i<-\log(\cos(4\pi/\kappa))$ translates to $\alpha^\kappa_{\sigma_i}>2/\gamma+\gamma/4$ and when $\kappa>4$ it translates to $\alpha^\kappa_{\sigma_i}>1/\gamma+\gamma/2$. These constraints appear also in Section \ref{sec:ssw} for the same reason. Indeed, in both cases this threshold arises because we need to assume finiteness of CLE conformal radius moments.
\end{remark}

\begin{remark}
	Note that the constant prefactor in the definition of the measure is chosen such that $|\op{M}^{-,\kappa}_{0,\ell}|=1$ and $|\op{M}^{{\sigmab},\kappa}_{0,\ell}|=\ell^{2/\gamma \sum \alpha_i}|\op{M}^{{\sigmab},\kappa}_{0,1}|$ for $\ell>0$. These conventions are the most convenient ones when stating results such as Proposition \ref{prop:chordal-zipper}.
\end{remark}

The following change of coordinates result implies that the law of the $\gamma$-LQG surface associated to the field in the definition above does not depend on $\w$.

\begin{lemma}
	\label{lem:coordinate-change-field}
	Let us consider the setting of Definition \ref{def:disk} with $B=\{a,b,c\}$ and consider another tuple $(w'_a,w'_b,w'_c)$ of distinct points. Let $\phi\colon \D\to \D$ be the unique conformal or anticonformal transformation mapping $(w_a,w_b,w_c)$ to $(w'_a,w'_b,w'_c)$. Then
	\begin{align*}
		\int \op{M}^{{\sigmab},\kappa}_{\bm{w'},\Lambda,\ell}(dk',d\bm{z'}) f(k'\circ \phi+Q_\gamma \log |\phi'|,\phi^{-1}(\bm{z'})) = \int \op{M}^{{\sigmab},\kappa}_{\bm{w},\Lambda,\ell}(dk,d\bm{z}) f(k,\bm{z})\;.
	\end{align*}
\end{lemma}

\begin{proof}
	By unpacking the definitions and subsequently using the change of coordinates formula \cite[Theorem 3.5]{hrv-disk} (see also Section \ref{sec:gff-lqg} where the relevant Liouville conformal field theory notions were introduced) we get
	\begin{align*}
		&\int \op{M}^{{\sigmab},\kappa}_{\bm{w'},\Lambda,\ell}(dk',d\bm{z'}) f(k'\circ \phi+Q_\gamma \log |\phi'|,\phi^{-1}(\bm{z'})) \\
		&= \frac{\ell^{2/\gamma \sum\alpha_i}}{Z^{(-,-),((\gamma,\gamma,\gamma),\bm{w'})}_{0,\ell}} \int_{\D^A}\,d\bm{z'}\, \Phi^{{\sigmab},\kappa}_\D(\bm{z'}) \\
		&\qquad\qquad\quad \cdot Z^{(\bm{\alpha},\bm{z'}),((\gamma,\gamma,\gamma),\bm{w'})}_{\Lambda,\ell}\int P^{(\bm{\alpha},\bm{z'}),((\gamma,\gamma,\gamma),\bm{w'})}_{\Lambda,\ell}(dk') f(k'\circ \phi+Q_\gamma \log |\phi'|,\phi^{-1}(\bm{z'})) \\
		&= \frac{\ell^{2/\gamma \sum\alpha_i}}{Z^{(-,-),((\gamma,\gamma,\gamma),\bm{w'})}_{0,\ell}} \int_{\D^A}\,d\bm{z'}\, \Phi^{{\sigmab},\kappa}_\D(\bm{z'}) \cdot \prod_{i\in A} |(\phi^{-1})'(z'_i)|^{\alpha^\kappa_{\sigma_i}(Q_\gamma-\alpha^\kappa_{\sigma_i}/2)} \prod_{j\in B} |(\phi^{-1})'(w'_j)| \\
		&\qquad\qquad\quad \cdot Z^{(\bm{\alpha},\phi^{-1}(\bm{z'})),((\gamma,\gamma,\gamma),\bm{w})}_{\Lambda,\ell}\int P^{(\bm{\alpha},\phi^{-1}(\bm{z'})),((\gamma,\gamma,\gamma),\bm{w})}_{\Lambda,\ell}(dk) f(k,\phi^{-1}(\bm{z'}))\;.
	\end{align*}
	Moreover, again by \cite[Theorem 3.5]{hrv-disk},
	\begin{align*}
		Z^{(-,-),((\gamma,\gamma,\gamma),\bm{w'})}_{0,\ell} = Z^{(-,-),((\gamma,\gamma,\gamma),\bm{w})}_{0,\ell} \prod_{j\in B} |(\phi^{-1})'(w'_j)|
	\end{align*}
	and the claim is then immediate after making the change of coordinates $\bm{z'}=\phi(\bm{z})$ and using the behavior of $\Phi^{{\sigmab},\kappa}_\D$ under conformal transformations as stated in Lemma \ref{lem:phi-axioms}.
\end{proof}

The previous lemma implies that we can define a $\gamma$-LQG surface with marked points in the following way.

\begin{defn}
	\label{def:simple-disk}
	Let $\kappa\in (8/3,4)$, $\ell>0$, $\Lambda\ge 0$ and $\sigmab\in \mathfrak{S}^A_\kappa$ for finite $A\subseteq \N$. Then we define a measure $\bar{\op{M}}^{{\sigmab},\kappa}_{\Lambda,\ell}$ on the space of regular $\gamma$-LQG surfaces by
	\begin{align*}
		\bar{\op{M}}^{{\sigmab},\kappa}_{\Lambda,\ell}(X) := \int \op{M}^{{\sigmab},\kappa}_{\bm{w},\Lambda,\ell}(dk,d\bm{z}) 1([(\D,k,\bm{z})]\in X)
	\end{align*}
	for each measurable set $X$. We also let
	\begin{align*}
		W^{{\sigmab},\kappa}_{\Lambda,\ell} := |\op{\bar{M}}^{{\sigmab},\kappa}_{\Lambda,\ell}| = |\op{M}^{{\sigmab},\kappa}_{\Lambda,\ell}|\;.
	\end{align*}
	If $W^{{\sigmab},\kappa}_{\Lambda,\ell}<\infty$ we say that $({\sigmab},\kappa,\Lambda,\ell)$ is admissible and call $\op{\bar{M}}^{{\sigmab},\kappa}_{\Lambda,\ell} / W^{{\sigmab},\kappa}_{\Lambda,\ell}$ the regular $\gamma$-LQG disk with boundary length $\ell$, cosmological constant $\Lambda$ and weights ${\sigmab}$.
\end{defn}

\begin{remark}
	We consider fields with three marked boundary points since this guarantees that \eqref{eq:seiberg} is satisfied, see Remark \ref{rmk:seiberg}. Furthermore, having three marked boundary points gives a natural way to fix the embedding of the disk since we can map the three marked boundary points to three fixed points on the boundary of the disk.
\end{remark}

Next, we will make an analogous definition in the case when $\kappa'\in(4,8)$. In the following definition, recall the definition of a Lévy excursion from Section \ref{sec:normal-lqg-disks} and the definition of generalized disk without marked points from Definition \ref{def:gen-disk-no-points}. In the definition below and throughout the text we write ${\sigmab}|_B = (\sigma_C\colon C\subseteq B)$ whenever $B\subseteq A$.

\begin{defn}
	\label{def:gen-disk}
	Let $\kappa'\in(4,8)$, $A\subseteq\N$ finite, $\sigmab\in\frk S_{\kappa'}^A$, $\Lambda\geq 0$ and $\ell>0$.
	We now define $\bar{\op{M}}^{{\sigmab},\kappa'}_{\Lambda,\ell}$ to be a measure on the space of generalized $\gamma$-LQG surfaces as follows. Let $E$ be the time reversal of a length $\ell$ spectrally positive Lévy excursion with exponent $4/\gamma^2$. Then we set 
	\begin{align*}
		\bar{\op{M}}^{{\sigmab},\kappa'}_{\Lambda,\ell}(X) &:= \E\Biggl( \,\sum_{Q\in \Pi(E,A)} \int \prod_{t<\ell\colon \Delta E_t\neq 0} \op{M}^{{\sigmab}|_{Q_{t-}\setminus Q_t},\kappa'}_{\w,\Lambda,|\Delta E_t|}(dk_t,d\bm{z}_t) \\
		&\qquad\qquad\qquad \cdot 1\left(((E,\{(t,[(\D,k_t,\bm{z}_t,w_a)])\colon t<\ell, \Delta e_t\neq 0\})\in X\right) \Biggr)
	\end{align*}
	for each measurable set $X$ where $\Pi(E,A)$ denotes the set of non-increasing càdlàg functions $Q\colon [0,\ell]\to \mathcal{P}(A)$ such that $Q_0=A$, $Q_{\ell-}=\emptyset$ and $\Delta Q_t:=Q_{t-}\setminus Q_t=\emptyset$ whenever $\Delta E_t=0$. We also let
	\begin{align*}
		W^{{\sigmab},\kappa'}_{\Lambda,\ell} := |\op{\bar{M}}^{{\sigmab},\kappa'}_{\Lambda,\ell}|\;.
	\end{align*}
	If $W^{{\sigmab},\kappa'}_{\Lambda,\ell}<\infty$ we say that $({\sigmab},\kappa',\Lambda,\ell)$ is admissible and call $\op{\bar{M}}^{{\sigmab},\kappa'}_{\Lambda,\ell} / W^{{\sigmab},\kappa'}_{\Lambda,\ell}$ the generalized $\gamma$-LQG disk with boundary length $\ell$, cosmological constant $\Lambda$ and weights ${\sigmab}$.
\end{defn}
We do not include $\w$ in the notation $\bar{\op{M}}^{{\sigmab},\kappa'}_{\Lambda,\ell}$ since this measure does not depend on the choice of $\w$ (by Lemma \ref{lem:coordinate-change-field}).

A key question is when admissibility holds and what the values for the partition functions are. Here, we will answer this question in the case of a single bulk singularity i.e.\ $\#A=1$ or no bulk singularities i.e.\ $\#A=0$ both in the simple and the non-simple case. The case of more than one singularity will be studied in Section \ref{sec:explorations}. By making use of Proposition \ref{prop:resample-lcft-point} the proof will be quite short. For the following results, recall the definition of modified Bessel functions of the second kind (see the beginning of Section \ref{sec:levy-exc}) and define $\bar{K}_\nu$ as in \eqref{eq:Kbar}.

\begin{lemma}
	\label{lem:measure-finite-import}
	Suppose that $\kappa\in (8/3,8)\setminus \{4\}$ and let $\gamma=\sqrt{\kappa}\wedge 4/\sqrt{\kappa}$. Assume that $A=\{i\}$ and let $\alpha=\alpha^\kappa_{\sigma_i}$. Then $|\op{M}^{-,\kappa}_{\Lambda,\ell}|, |\op{M}^{{\sigmab},\kappa}_{\Lambda,\ell}|<\infty$ for all $\ell>0$ and $\Lambda\ge 0$, and
	\begin{align*}
		|\op{M}^{-,\kappa}_{\Lambda,\ell}| = \bar{K}_{4/\gamma^2}\left(\ell\cdot\sqrt{\frac{\Lambda}{\sin(\pi\gamma^2/4)}}\,\right)\;, \quad
		\frac{|\op{M}^{{\sigmab},\kappa}_{\Lambda,\ell}|}{|\op{M}^{{\sigmab},\kappa}_{0,1}|} = \ell^{2\alpha/\gamma}\bar{K}_{2/\gamma\cdot(Q_\gamma-\alpha)}\left(\ell\cdot\sqrt{\frac{\Lambda}{\sin(\pi\gamma^2/4)}}\,\right)\;.
	\end{align*}
\end{lemma}

\begin{proof}
	By Definition \ref{def:disk} (with $\bm{w}$ as in that definition) we get 
	\begin{align*}
		|\op{M}^{-,\kappa}_{\Lambda,\ell}|=\frac{Z^{(-,-),((\gamma,\gamma,\gamma),\bm{w})}_{\Lambda,\ell}}{Z^{(-,-),((\gamma,\gamma,\gamma),\bm{w})}_{0,\ell}} = \int  e^{-\Lambda\mu^\gamma_k(\D)}\,P^{(-,-),((\gamma,\gamma,\gamma),\bm{w})}_{0,\ell}(dk)\;.
	\end{align*}
	The two results \cite[Theorem 1.2]{ag-disk} and \cite[Theorem 1.3]{ars-fzz} show that the law of the area of the unit boundary length $\gamma$-LQG disk as defined in Definition \ref{def:normal-unit-disk} is an inverse Gamma distribution with shape parameter $4/\gamma^2$ and scale parameter $1/(4\sin(\pi\gamma^2/4))$. Thus
	\begin{align*}
		\int  e^{-\Lambda\mu^\gamma_k(\D)}\,P^{(-,-),((\gamma,\gamma,\gamma),\bm{w})}_{0,\ell}(dk) &= \frac{1}{\Gamma(4/\gamma^2)(4\sin(\pi\gamma^2/4))^{4/\gamma^2}}\int_0^\infty e^{-\Lambda\ell^2 t}\, \frac{e^{-1/(4\sin(\pi\gamma^2/4)t)}\,dt}{t^{1+4/\gamma^2}}\\
		&= \frac{1}{\Gamma(4/\gamma^2)}\int_0^\infty e^{-\Lambda\ell^2 x/(4\sin(\pi\gamma^2/4))-1/x}\,\frac{dx}{x^{1+4/\gamma^2}}
	\end{align*}
	and the first claim of the lemma follows from the definition of the modified Bessel functions as given at the beginning of Section \ref{sec:levy-exc}.

	If $\#A=\{i\}$ then $\Phi^{{\sigmab},\kappa}_\D(\bm{z})=R(z_i,\D)^{\rho}$ where $\rho:=\rho^\kappa_{\sigma_i}=\alpha(Q_\gamma-Q_\alpha)$. By Proposition \ref{prop:resample-lcft-point} therefore
	\begin{align*}
		|\op{M}^{{\sigmab},\kappa}_{\Lambda,\ell}| = \frac{\ell^{2\alpha/\gamma}}{Z^{(-,-),((\gamma,\gamma,\gamma),\bm{w})}_{0,\ell}}\, \frac{1}{2}\, e^{(G(w_a,w_b)+G(w_b,w_c)+G(w_c,w_a))/2}Z^{(\alpha,0),(\gamma,1)}_{\Lambda,\ell}<\infty
	\end{align*}
	where the finiteness of the partition functions on the right-hand side is recalled in Definition \ref{def:lcft-field}. Therefore,
	\begin{align*}
		\ell^{-2\alpha/\gamma}\,\frac{|\op{M}^{{\sigmab},\kappa}_{\Lambda,\ell}|}{|\op{M}^{{\sigmab},\kappa}_{0,1}|} = \frac{Z^{(\alpha,0),(\gamma,1)}_{\Lambda,\ell}}{Z^{(\alpha,0),(\gamma,1)}_{0,\ell}} = \frac{Z^{(\alpha,0),(\gamma,1)}_{\Lambda\ell^2,1}}{Z^{(\alpha,0),(\gamma,1)}_{0,1}} = \int e^{-\Lambda\ell^2 \mu^\gamma_k(\D)}\,P^{(\alpha,0),(\gamma,1)}_{0,1}(dk)\;.
	\end{align*}
	As stated in \cite[Theorem 1.2]{ars-fzz}, the law of $\mu^\gamma_k(\D)$ when $k\sim P^{(\alpha,0),(\gamma,1)}_{0,1}$ is an inverse Gamma distribution with shape parameter $2/\gamma\cdot(Q_\gamma-\alpha)$ and scale parameter $1/(4\sin(\pi\gamma^2/4))$; note that in their setup the boundary $\gamma$ singularity is mapped to a point that is uniformly sampled from $\partial\D$ which of course does not affect the law of the total area. The second claim now follows from a computation similar to the one for the first claim.
\end{proof}

We now deduce the corresponding results for the total masses of the measures on the space of (generalized) LQG surfaces and conclude the proof of Theorem \ref{thm:singlepoint-finite}. One thing worth remarking is that in the definition of a Lévy excursion we introduced an arbitrary convention: If $b$ is a unit length Lévy excursion constructed from a Lévy process $B$ (see Section \ref{sec:levy-exc}) and $c>0$ a constant then $c\cdot b$ is the unit length Lévy excursion associated to the Lévy process $c\cdot B$. This arbitrary choice in multiplicative normalization is reflected in the appearance of the $\sin(\pi\gamma^2/4)$ term below; by changing the normalization of the Lévy excursion in the $\kappa>4$ case, we could also replace it by another arbitrary constant depending on $\gamma$ and $\kappa$, e.g.\ $\sin(\pi\kappa/4)$.

\begin{proof}[Proof of Theorem \ref{thm:singlepoint-finite}]
	The case $\kappa\in(8/3,4)$ is immediate from Lemma \ref{lem:measure-finite-import} and Definition \ref{def:simple-disk}. For the $\kappa\in(4,8)$ case let
	\begin{align*}
		c=\sqrt{\frac{\Lambda}{\sin(\pi\gamma^2/4)}}\;,\quad \nu=4/\gamma^2\;,\quad \theta = \alpha\gamma/2\;,\quad c'=c^\nu2^{1-\nu}=2\left(\frac{\Lambda}{4\sin(\pi\gamma^2/4)}\right)^{2/\gamma^2}\;.
	\end{align*}
	We only consider the $\Lambda>0$ case; the $\Lambda=0$ case then follows by monotone convergence. By Definition \ref{def:gen-disk} we have
	\begin{align*}
		W^{-,\kappa}_{\Lambda,\ell} &= \E\left(\prod_{t<\ell\colon \Delta E_t\neq 0} |\M^{-,\kappa}_{\Lambda,|\Delta E_t|}|\right)\;,\\
		W^{{\sigmab},\kappa}_{\Lambda,\ell} &= \E\left(\sum_{t<\ell\colon \Delta E_t\neq 0} |\M^{{\sigmab},\kappa}_{\Lambda,|\Delta E_t|}|\prod_{s\neq t\colon \Delta E_s\neq 0} |\M^{-,\kappa}_{\Lambda,|\Delta E_s|}| \right)
	\end{align*}
	where $E$ is a spectrally positive Lévy excursion of exponent $4/\gamma^2$ and duration $\ell$ (since we are only using the sequence of jump heights, we do not have to consider the time reversal). Thus by Theorem \ref{thm:levy-excursions} and Lemma \ref{lem:measure-finite-import}, we get
	\begin{align*}
		W^{-,\kappa}_{\Lambda,\ell} &= \E\left(\prod_{t<\ell\colon \Delta E_t\neq 0} \bar{K}_\nu(c\Delta E_t)\right) = \bar{K}_{1/\nu}(c'\ell)
	\end{align*}
	from which the first claim follows. Also
	\begin{align*}
		\frac{W^{{\sigmab},\kappa}_{\Lambda,\ell}}{|\op{M}^{{\sigmab},\kappa}_{0,1}|}&=\frac{2(c/2)^{1+\nu-\theta\nu}}{\Gamma(1+\nu-\theta\nu)}\,\E\left(\,\sum_{t<\ell\colon \Delta E_t\neq 0} (\Delta E_t)^{1+\nu}K_{1+\nu-\theta\nu}(c\Delta E_t) \prod_{s\neq t\colon \Delta E_s\neq 0} \bar{K}_\nu(c\Delta E_s) \right) \\
		&= \frac{2(c/2)^{1+\nu-\theta\nu}}{\Gamma(1+\nu-\theta\nu)}\cdot \frac{\sqrt{1/\nu}\,\Gamma(-1/\nu)\sin(\pi(1+1/\nu-\theta))}{\sqrt{\nu}\,\Gamma(-\nu)\,\sin(\pi(1+\nu-\theta\nu))}\,\ell^{1+1/\nu}K_{1+1/\nu-\theta}(c'\ell) \\
		&= c'' \ell^{\alpha\gamma/2} \bar{K}_{\gamma/2\cdot(Q_\gamma-\alpha)}\left(2\ell\left(\frac{\Lambda}{4\sin(\pi\gamma^2/4)}\right)^{2/\gamma^2}\,\right)
	\end{align*}
	for a constant $c''>0$ which does not depend on $\Lambda$ and $\ell$. The claim follows.
\end{proof}

For future reference, let us also record the following corollary which is a direct consequence of properties of modified Bessel functions and will be applied in Section \ref{sec:explorations}.
\begin{cor}
	\label{cor:single-point-properties}
	Suppose that $\kappa\in (8/3,8)\setminus \{4\}$, $A=\{i\}$ and $\Lambda>\Lambda'>0$. Then there exists a constant $C=C(\sigma_i,\kappa,\Lambda,\Lambda')$ such that $W^{{\sigmab},\kappa}_{\Lambda,\ell} \le C\,W^{-,\kappa}_{\Lambda',\ell}$ for all $\ell>0$. Moreover there is $C_\kappa>0$ such that for all $\Lambda\ge 0$ and $\ell>0$,
	\begin{align*}
		1-W^{-,\kappa}_{\Lambda,\ell} \le C_\kappa \Lambda \ell^{2\wedge (8/\kappa)}\;.
	\end{align*}
\end{cor}

\begin{proof}
	Let us first recall the following asymptotic for modified Bessel functions of the second kind (see \cite[Equations (9.6.2), (9.6.10)]{as-handbook72}):
	\begin{align*}
		K_\nu(x)&= \frac{\pi}{2\sin (\pi\nu)}\left( 
		\frac{1}{\Gamma(-\nu+1)} \left( \frac{x}{2} \right)^{-\nu} 
		+ \frac{1}{\Gamma(-\nu+2)} \left( \frac{x}{2} \right)^{-\nu+2}
		+ \frac{1}{\Gamma(\nu+1)} \left( \frac{x}{2} \right)^{\nu}
		\right)\\
		& \qquad +O(x^{-\nu+4}+x^{\nu+2})\quad\text{as $x\to 0$}
	\end{align*}
	for all $\nu>0$. Therefore, by Euler's reflection formula and the definition of $\bar{K}_\nu$,
	\begin{align}
		\label{eq:bessel-mod-0}
		1-\bar{K}_\nu(x)&=  \frac{1}{\nu-1} \left( \frac{x}{2} \right)^{2}
		+ \frac{\pi}{\sin(-\pi\nu)\Gamma(\nu)\Gamma(\nu+1)} \left( \frac{x}{2} \right)^{2\nu}  +O(x^{4}+x^{2\nu+2})\quad\text{as $x\to 0$}\;.
	\end{align}
	In particular we obtain $\bar{K}_\nu(x)\to 1$ as $x\to 0$ and this combined with the following Bessel function asymptotic (see \cite[Equation (9.7.1)]{as-handbook72})
	\begin{align*}
		K_\nu(x) &\sim \sqrt{\frac{\pi}{2x}}\,e^{-x}\quad\text{as $x\to\infty$}
	\end{align*}
	and Theorem \ref{thm:singlepoint-finite} readily yields the first claim. For the second claim, we observe that \eqref{eq:bessel-mod-0} implies that for $\nu\in (0,2)$ there is $C'_\nu>0$ such that
	\begin{align*}
		1-\bar{K}_\nu(x)\le C'_\nu x^{2\wedge (2\nu)}\quad\text{for all $x>0$}\;.
	\end{align*}
	The claim then follows again directly from Theorem \ref{thm:singlepoint-finite}.
\end{proof}

In the $\kappa\in (8/3,4)$ case, by considering the total mass of the measure in Definition \ref{def:disk} we get
\begin{align}
	\label{eq:lcft-lqg-partition}
	W^{{\sigmab},\kappa}_{\Lambda,\ell}
	& = \frac{\ell^{2/\gamma \sum\alpha_i}}{Z^{(-,-),((\gamma,\gamma,\gamma),\bm{w})}_{0,\ell}} \int_{\D^A}\,d\bm{z}\, \Phi^{{\sigmab},\kappa}_\D(\bm{z}) Z^{(\bm{\alpha},\bm{z}),((\gamma,\gamma,\gamma),\bm{w})}_{\Lambda,\ell}\;.
\end{align}
Let us now consider the setting of Definition \ref{def:lcft-field} and define $s:=\sum_i\alpha_i+1/2\sum_j \beta_j-Q_\gamma$.  In the context of Liouville conformal field theory, the constants
\begin{align*}
	Z^{\,\op{LCFT}}_{\Lambda, \ell,  (\bm{\alpha},\z),( \bm{\beta},\w )} := \frac{2\ell^{2s/\gamma-1}}{\gamma}\,e^{-\sum_j \beta^2_j/8} Z^{(\bm{\alpha},\bm{z}),(\bm{\beta},\bm{w})}_{\Lambda,\ell}
\end{align*}
are the fixed boundary length LCFT partition functions (with certain boundary and bulk singularities) -- the multiplicative prefactor is an arbitrary convention which we include to make our presentation consistent with the work \cite{hrv-disk}. We believe the relation \eqref{eq:lcft-lqg-partition} to be possibly useful for the following reason. This work explains how the expressions $W^{{\sigmab},\kappa}_{\Lambda,\ell}$ can in principle be iteratively computed (see Remark \ref{rmk:bootstrap}) so if the fixed boundary length LCFT partition functions could also be determined as a function of the position of the singularities and the cosmological constant, then this could be used to obtain information about the CLE observable $\Phi^{\sigmab,\kappa}_\D$. 

We remark that the constant $Z^{(-,-),((\gamma,\gamma,\gamma),\bm{w})}_{0,\ell}$ appearing in  \eqref{eq:lcft-lqg-partition} can be determined explicitly using results of \cite{remy-zhu-boundary}. We also remark that the fixed boundary length LCFT partition functions are closely related to the LCFT correlation functions; indeed, if $h$ is as in Definition \ref{def:lcft-field} then the LCFT correlation function is given by
\begin{gather*}
	Z^{\,\op{LCFT}}_{\Lambda,  (\bm{\alpha},\z),( \bm{\beta},\w )} = e^{-\sum_j \beta^2_j/8}C^{(\bm{\alpha},\bm{z})}_{(\bm{\beta},\bm{w})}
	\int \E(e^{-\Lambda \mu_{h+c}(\D)}) e^{cs}\,dc\
	=	\int Z^{\,\op{LCFT}}_{\Lambda, \ell,  (\bm{\alpha},\z),( \bm{\beta},\w )} \,d\ell\;,
\end{gather*}
where the second equality in this display is obtained by making the change of coordinates $c=2/\gamma \cdot (\log\ell - \log \nu_{h}(\partial\D))$.
Note however that the LCFT correlation functions (in contrast to the fixed boundary length LCFT partition functions) will be insufficient to extract information about the CLE observable $\Phi^{\sigma,\kappa}_\D$ since the LCFT correlation function is a polynomial in the cosmological constant $\Lambda$; this can be seen by making a change of coordinates $\ell'=\ell\sqrt{\Lambda}$ in the above display concerning LCFT correlation functions.

\subsection{The LQG disk weighted by CLE nesting statistics as a limit}
\label{sec:disk-conv}

This section is devoted to the proof of the following proposition. It says that the field describing a $\gamma$-LQG disk (or a single component of such a disk in the generalized case) with zero interior marked points (see Section \ref{sec:normal-lqg-disks}) converges to the field of a disk with singularities $(\alpha_i \colon i\in A)$ (see Section \ref{sec:disk}), upon a reweighting by the expression
\begin{align*}
	\prod_{i\in A} \left(\epsilon^{\alpha_i^2/2} e^{\alpha_i h_\epsilon(z_i)}\right),
\end{align*}
where the $z_i$ points are sampled independently from the Lebesgue measure and $h_\epsilon(z_i)$ is an appropriate approximation to $h$ evaluated at $z_i$. In other words, we are showing how to obtain the measures in Section \ref{sec:disk} using a limiting procedure which is going to be crucial in the proofs in Section \ref{sec:disk-conv-exploration} in order to transfer results from Section \ref{sec:background-msw}, i.e. the case without marked points, to our setting with marked points. In Section \ref{sec:disk-conv-exploration} we will be in a scenario where the mollifier we use to define $h_\epsilon(z_i)$ is not entirely independent of the field $h$ (because these mollifiers will depend on mapping out functions defined via the field $h$) and for this reason we prove the result below in this generality. The reason for including the space $E$ in the result below is that we want to allow for sampling other geometric objects conditionally on the field $h$, for example we will sample a CPI where its target point is sampled from the LQG boundary measure associated to $h$.

\begin{prop}\label{prop:limit-eps}
	Let $\kappa\in(8/3,8)\setminus\{4 \}$, $\ell>0$, $r\in(0,1/4)$, $A\subseteq\N$ finite, $\Lambda\geq 0$, and define  $K\subseteq\D^A$ by
	\begin{align*}
		K=\{ \z=(z_i\,:\,i\in A)\in B_{1-2r}(0)^A\,:\, B_{2r}(z_i)\cap B_{2r}(z_j)=\emptyset\,\forall i\neq j \}.
	\end{align*}
	Consider some measurable space $E$. Also let $B=\{a,b,c \}$, $(\beta_a,\beta_b,\beta_c)=(\gamma,\gamma,\gamma)$, and let $\w=(w_a,w_b,w_c)$ be three distinct boundary points of $\partial \D$ ordered counterclockwise. We consider a probability kernel $\mu$ from $H^{-1}(\D)$ to $C(\D\times(0,r/2)\times\D)\times E$ (so $\mu_k$ is a probability measure on $C(\D\times(0,r/2)\times\D)\times E$ for each $k\in H^{-1}(\D)$) where we denote elements $\theta$ of $C(\D\times(0,r/2)\times\D)$ by $(z,\epsilon,w)\mapsto \theta^z_\epsilon(w)$.
	Let $G_{B_r(z)}$ denote the Dirichlet Green's function in $B_r(z)$. We assume that for all $k \in H^{-1}(\D)$, $((\theta,x)\mapsto \theta)_* \mu_k$ is supported on functions $\theta$ with the property that for all $z\in B_{1-2r}(0)$ and $\epsilon \in (0,r/2)$ we have
	\begin{align}
		\label{eq:theta}
		\begin{split}
		&\int_\D \theta^z_\epsilon(w)\,dw = 1+o_\epsilon(1)\;,\quad \op{supp}(\theta^z_\epsilon) \subseteq B_{2\epsilon}(z)\;,\\
		&\int_{\D^2} G(w,w')\theta^{z}_\epsilon(w)\theta^{z}_\epsilon(w')\,dw\,dw' = \log\frac{1}{R(z,\D)\epsilon} +o_\epsilon(1)\;,\\
		&\int_{B_r(z)^2} G_{B_r(z)}(w,w')\theta^{z}_\epsilon(w)\theta^{z}_\epsilon(w')\,dw\,dw' = \log\frac{r}{\epsilon} +o_\epsilon(1)
		\end{split}
	\end{align}
	as $\epsilon\to 0$ where the $o_\epsilon(1)$ is uniform in $z\in B_{1-2r}(0)$, in $\theta$ and in $k$.\footnote{\,When $\theta^z_\epsilon(w)\,dw$ is replaced by the probability measure $\sigma^z_{\epsilon}$ with uniform mass on the circle $\partial B_\epsilon(z)$ then \eqref{eq:theta} is also satisfied. One can make sense of the function $(z,\epsilon)\mapsto k(\sigma^z_\epsilon)$ in the setting of the proposition and the random process is called the circle average process associated to $k$. However, for technical reasons, we work with smooth bump functions instead.}
	Lastly, we suppose that $f: H^{-1}(\D)\times\D^A\times E\to[0,\infty)$ is measurable, bounded and satisfies $f(k,\z,x)=0$ for $\z\not\in K$. Moreover, almost every $k$  sampled from $\M^{-,\kappa}_{\w,0,\ell}$ satisfies $\lim_{n\to\infty}f(k+g^\z_n,\z,x)= f(k,\z,x)$ if $(g^\z_n)$ is a sequence of continuous functions with compact support in $\cup_i B_{2r}(z_i)$ which converge uniformly to $0$ away from the points of $\z$ such that $\lim_{n\to\infty}\|g_n^\z\|_{H^{-1}(\D)}=0$. Then
	\begin{align}
	\begin{split}
		&\int f(k,\z,x)\Phi^{{\sigmab},\kappa}_{\D}(\z) \prod_{i\in A} \left(\epsilon^{\alpha_i^2/2} e^{\alpha_i k(\theta_\epsilon^{z_i})} \right)
		e^{-\Lambda \mu^\gamma_{k}(\D)}
		\,\mu_{k^{\perp\z}}(d\theta,dx)\,d\z\,\M^{-,\kappa}_{\w,0,\ell}(dk) \\
		&\to \int f(k,\z,x) \,\mu_{k^{\perp\z}}(d\theta,dx)\,\M^{{\sigmab},\kappa}_{\w,\Lambda,\ell}(dk,d\z)
	\end{split}
	\label{eq16}
	\end{align}
	as $\epsilon\to 0$ where $d\z$ denotes Lebesgue measure on tuples $\z=(z_i\in\D\,:\,i\in A)$ and $k^{\perp\z}$ denotes the harmonic extension of $k$ restricted to the complement of $\cup_i B_r(z_i)$ (see Section \ref{sec:gff-lqg}).
\end{prop}

\begin{proof}
	For each $\z\in K^A$, we decompose $h^0$ (a Neumann GFF in $\D$) as $h^0=h^\z + \mathfrak{h}^\z$ into its projection onto $H^{-1}(\cup_i B_r(z_i))$ (called $h^\z$) and the harmonic extension of $h^0$ restricted to the complement of $\cup_i B_r(z_i)$ (called $\mathfrak{h}^\z$). Crucially, $h^\z$ and $\mathfrak{h}^\z$ are independent and $h^\z$ has the law of an independent Dirichlet GFF on each ball $B_r(z_i)$ for $i\in A$. Let $G^{\z}:\D\times\D\to\R$ (resp.\ $\mathcal{G}^{\z}:\D\times\D\to\R$) denote the function describing the covariance of $h^{\z}$ (resp.\ $\mathfrak{h}^{\z}$), and note that $G=G^{\z}+\mathcal{G}^{\z}$. Recall that (using that $h^0=\mathfrak{h}^\z + h^\z$),
	\begin{align*}
		h = \mathfrak{h}^\z + h^\z + \sum_{j\in B}\frac{\gamma}{2} G(\cdot,w_j)\;, \quad h_{*} = h -\frac{2}{\gamma}\log\nu_{h}^\gamma(\partial\D) + \frac{2}{\gamma}\log \ell\;.
	\end{align*}
	As pointed out in Section \ref{sec:gff-lqg} we have that $\nu_{h}^\gamma(\partial\D)= \nu_{h-h^\z}^\gamma(\partial\D)$ and we see that this random variable
	is measurable with respect to $\mathfrak{h}^\z$. Furthermore,
	\begin{align*}
		h_*^{\perp \z} = \mathfrak{h}^\z + \sum_{j\in B}\frac{\gamma}{2} G(\cdot,w_j) -\frac{2}{\gamma}\log\nu_{h}^\gamma(\partial\D) + \frac{2}{\gamma}\log \ell
	\end{align*}
	which is hence also measurable with respect to $\mathfrak{h}^\z$ (this is the crucial observation which will momentarily allow us to apply Girsanov's theorem). By unpacking the definitions and using the notation (with no marked points in the bulk, $\w$ as the marked points on the boundary and boundary length $\ell$) appearing in Definition \ref{def:lcft-field}, we obtain that the left-hand side of \eqref{eq16} equals $c_0\cdot \int_K A_\epsilon(\z)\,d\z$ where
	\begin{align*}
		A_\epsilon(\z) := \E\left( \int \nu_h^\gamma(\partial\D)^{4/\gamma^2-2} e^{-\Lambda \mu_{h_*}(\D)} f(h_*,\z,x) \Phi^{\sigmab,\kappa}_\D(\z) \prod_{i\in A} \epsilon^{\alpha_i^2/2}e^{\alpha_i h_*(\theta^{z_i}_\epsilon)} \,\mu_{h_*^{\perp\z}}(d\theta,dx)\right)
	\end{align*}
	and $c_0 = \E( \nu_h^\gamma(\partial\D)^{4/\gamma^2-2})^{-1}$. For $\z\in K$ we now condition on $\mathfrak{h}^z$, fix a realization of $\theta$ and apply Girsanov's theorem in the setting obtained after conditioning on the field $\mathfrak{h}^z$ by considering the reweighting
	\begin{align*}
		\prod_{i\in A} \epsilon^{\alpha_i^2/2} e^{\alpha_i h_*(\theta^{z_i}_\epsilon)} &= e^{h^\z(\phi)-\frac 12\int \phi(w)\phi(w')G^\z(w,w')\,dw\,dw'}\\
		&\quad \cdot\prod_{i\in A} \left(\epsilon^{\alpha_i^2/2} e^{\alpha_i^2/2\cdot \int_{B_r(z_i)^2} G^\z(w,w') \theta^{z_i}_\epsilon(w)\theta^{z_i}_\epsilon(w')\,dw\,dw'} \right) \\
		&\quad\cdot\prod_{i\in A} \left( e^{\alpha_i(\mathfrak{h}^\z + \gamma/2 \sum_{j\in B} G(\cdot,w_j))(\theta^{z_i}_\epsilon)} \nu_h(\partial\D)^{-2\alpha_i/\gamma}\ell^{2\alpha_i/\gamma} \right)
	\end{align*}
	where $\phi = \sum_{i\in A} \alpha_i \theta^{z_i}_\epsilon$ and we used that $G^\z(w,w')=0$ if $w\in B_r(z_i)$ and $w'\in B_r(z_j)$ for $i\neq j$. Note that we have written the expression in the display above so that the first line on the right-hand side
	has expectation $1$ and is the term we are reweighting by when applying Girsanov's theorem. Indeed, the Girsanov weighting $\exp(h^\z(\phi)-\frac 12\int \phi(w)\phi(w')G^\z(w,w')\,dw\,dw')$ induces a shift of the field $h^\z$ by
	\begin{align*}
		\sum_{i\in A} \alpha_i \int G^\z(\cdot,y) \theta^{z_i}_\epsilon (y)\,dy\;.
	\end{align*}
	Formally, for $\z \in K$ the field $h_{z,\epsilon}$ has the same law as that of $h$ weighted by $\exp(h^\z(\phi)-\frac 12\int \phi(w)\phi(w')G^\z(w,w')\,dw\,dw')$ where we make the definitions
	\begin{align*}
		G^{\z}_{\epsilon}(z,w) &= \int G^{\z}(z,y) \theta^{w}_{\epsilon}(y) \,dy \;,\\
		h_{\z,\epsilon} &= \mathfrak{h}^\z + h^\z + \sum_{i\in A} \alpha_i G^{\z}_\epsilon(\cdot,z_i) + \sum_{j\in B}\frac{\gamma}{2} G(\cdot,w_j)\;, \\
		h_{*\z,\epsilon} &= h_{\z,\epsilon}-\frac{2}{\gamma}\log\nu_{h_{\z,\epsilon}}^\gamma(\partial\D) + \frac{2}{\gamma}\log \ell \;.
	\end{align*}
	For later reference we also let $(h_{\z,0},h_{*\z,0})$ and $(h_\z,h_{*\z})$ be defined analogously, only with $G^\z$ and $G$, respectively, in place of $G^\z_\epsilon$ (compare the latter definition with Definition \ref{def:lcft-field}). Girsanov's theorem yields
	\begin{align*}
		A_\epsilon(\z) &= \ell^{2/\gamma \sum_{i\in A} \alpha_i}\,\E\Biggl( \int B_\epsilon(\z,\theta,x) \,\mu_{h_{*}^{\perp\z}}(d\theta,dx)\Biggr)\quad\text{where}\\
		B_\epsilon(\z,\theta,x) &:= \nu_{h_{\z,\epsilon}}^\gamma(\partial\D)^{4/\gamma^2-2-2/\gamma \sum_{i\in A} \alpha_i} e^{-\Lambda \mu_{h_{*\z,\epsilon}}(\D)} f(h_{*\z,\epsilon},\z,x) \Phi^{\sigmab,\kappa}_\D(\z) \\
		&\qquad\cdot\prod_{i\in A} \left(\epsilon^{\alpha_i^2/2} e^{\alpha_i^2/2\cdot \int_{B_r(z_i)^2} G^\z(w,w')\theta^{z_i}_\epsilon(w)\theta^{z_i}_\epsilon(w')\,dw\,dw'} e^{\alpha_i(\mathfrak{h}^\z + \gamma/2\sum_{j\in B} G(\cdot,w_j))(\theta^{z_i}_\epsilon)} \right)\;.
	\end{align*}
	By our assumptions (noting that $G^\z=G_{B_r(z_i)}$ on $B_r(z_i)^2$ by definition), we see that
	\begin{align*}
		 B_\epsilon(\z,\theta,x) &\to \nu_{h_{\z,0}}^\gamma(\partial\D)^{4/\gamma^2-2-2/\gamma \sum_{i\in A} \alpha_i} e^{-\Lambda \mu_{h_{*\z,0}}(\D)} f(h_{*\z,0},\z,x) \Phi^{\sigmab,\kappa}_\D(\z) \\
		 &\qquad\cdot\prod_{i\in A} \left(r^{\alpha_i^2/2}  e^{\alpha_i(\mathfrak{h}^\z(z_i) + \gamma/2\sum_{j\in B} G(z_i,w_j))} \right)
	\end{align*}
	as $\epsilon\to 0$. Note that the right-hand side does not depend on $\theta$. Moreover, since $4/\gamma^2-2-2/\gamma \sum_{i\in A} \alpha_i<0$ and using our assumptions, we may bound
	\begin{align*}
		B_\epsilon(\z,\theta,x) \lesssim \nu_{\mathfrak{h}^\z + h^\z}^\gamma(\partial\D)^{4/\gamma^2-2-2/\gamma \sum_{i\in A} \alpha_i}  \prod_{i\in A}e^{\alpha_i\sup_{B_{2r/3}(z_i)}\mathfrak{h}^\z} =: C_\epsilon(\z)
	\end{align*}
	where the implicit constant does not depend on $\z\in K$ and also not on $\theta$; to upper bound $\Phi^{\sigmab,\kappa}_\D(\z)$ by a deterministic constant for $\z\in K$ we used Theorem \ref{thm:loopexpconv}.
	By the mean value property of harmonic functions, we see that $\sup_{B_{2r/3}(z_i)}\mathfrak{h}^\z \le C\mathfrak{h}^\z(z_i)$ uniformly in $\z\in K$ for some constant $C<\infty$. Using Hölder's inequality and the fact that the total mass $\nu^\gamma_{\mathfrak{h}^\z + h^\z}(\partial \D)$ has all negative moments and that these are bounded uniformly in $\z\in K$ (see \cite[Theorem 2.12]{rhodes-vargas-review}) we see that $\sup_{\z\in K} \E(C_\epsilon(\z))<\infty$. Thus by dominated convergence, we obtain
	\begin{align*}
		\int_K A_\epsilon(\z)\,d\z \to\, &\ell^{2/\gamma \sum_{i\in A} \alpha_i} \int \E\Biggl( \nu_{h_{\z,0}}^\gamma(\partial\D)^{4/\gamma^2-2-2/\gamma \sum_{i\in A} \alpha_i} e^{-\Lambda \mu_{h_{*\z,0}}(\D)} f(h_{*\z,0},\z,x)  \\
		&\qquad \Phi^{\sigmab,\kappa}_\D(\z)  \cdot\prod_{i\in A} \left(r^{\alpha_i^2/2}  e^{\alpha_i(\mathfrak{h}^\z(z_i) + \gamma/2\sum_{j\in B} G(z_i,w_j))} \right) \Biggr)\,\mu_{k^{\perp\z}}(d\theta,dx)\,d\z
	\end{align*}
	as $\epsilon\to 0$. The right-hand side can be rewritten as follows. Note that since $\mathcal{G}^\z = G-G^\z$ and by the continuity of $\mathcal{G}^\z$ we have $\mathcal{G}^\z(z_i,z_i)=-\log R(z_i,\D)-\log(r)$ and furthermore $\mathcal{G}^\z(z_i,z_j)=G(z_i,z_j)$ for $i\neq j$. We perform a Girsanov shift by
	\begin{align*}
		e^{\sum_{i\in A} \alpha_i \mathfrak{h}^\z(z_i)} = e^{\sum_{i\in A} \alpha_i \mathfrak{h}^\z(z_i)-1/2\sum_{i,j\in A}\alpha_i\alpha_j \mathcal{G}^\z(z_i,z_j)} \prod_{i\in A} (rR(z_i,\D))^{-\alpha_i^2/2} \prod_{\substack{i,i'\in A\\i<i'}} e^{\alpha_i\alpha_{i'}G(z_i,z_j)}
	\end{align*}
	which shifts the field $\mathfrak{h}^\z$ up by $\sum_{i\in A} \alpha_i \mathcal{G}^\z(\cdot,z_i)$. Rearranging the resulting expression (which now involves the fields $h_\z$ and $h_{*\z}$) yields the claim.
\end{proof}

\subsection{Chordal exploration of LQG disk weighted by CLE nesting statistics}
\label{sec:disk-conv-exploration}

This section is devoted to the proof of Proposition \ref{prop:chordal-zipper} below. The proposition says that if we consider an LQG disk with an independent CLE, weighted by the CLE nesting statistics (as in Definitions \ref{def:simple-disk}, \ref{def:gen-disk}, and \ref{def:cle}), then the boundary length process of the CPI has a law which is reweighted in an explicit way as compared to the case with no marked points, and the loop-decorated LQG surfaces in the complementary components of the CPI have the law of independent CLE decorated LQG disks reweighted by CLE nesting statistics, conditioned on their boundary lengths and marked points.

The idea of the proof is to start with the analogous result for disks without marked points (proved in \cite{msw-simple,msw-non-simple}) and then do a particular reweighting of the measures. The reweighted measure gives us the surface and loop ensemble in Definition \ref{def:simple-disk} (or, in the generalized case, Definition \ref{def:gen-disk}) and Definition \ref{def:cle}, respectively, via an application of the Girsanov theorem. A closely related strategy was used to prove a conformal welding result for a surface with an $\alpha$ singularity on its boundary in \cite{ahs-sle}. However, the case considered here is substantially more technically challenging, e.g.\ since 
there are multiple marked points the locations of which are random and
the considered mollifiers and curves are not always independent of the field (even if we condition on the location of the marked points).

We borrow the following terminology from probability theory in the setting of (finite or infinite) non-probability measures, i.e., measures whose total mass is different from 1. 
Suppose that $E$, $E'$ and $E''$ are  measurable spaces, $\mu$ is measure on $E$, $f\colon E\times E'\to E''$ is measurable, and $\nu$ is a probability kernel from $E$ to $E'$, i.e. $\nu = (\nu_e \colon e \in E)$ is a collection of probability measures on $E'$ such that the map $e \mapsto \nu_e (A)$ from $E$ to $\R$ is measurable for any measurable $A \subseteq E'$. Let $(\mu\otimes\nu)(de,de')=\mu(de)\nu_{e}(de')$ and define $\zeta = f_*(\mu\otimes \nu)$ to be the pushforward of $\mu\otimes \nu$ along $f$. In words we will describe the construction of $\zeta$ as follows: sample $e\sim \mu$, conditionally on $e$ sample $e'$ according to $\nu_e$, and let $\zeta$ denote the law of $f(e,e')$. It is often more convenient to introduce $\zeta$ in words as in the previous sentence since the precise formula for $f$ is rather involved.

In order to state Proposition \ref{prop:chordal-zipper} precisely, we first need to introduce some more notation and definitions.
For $\kappa\in(8/3,8)\setminus\{4 \}$ and a finite and non-empty set $A\subseteq\N$ let ${\sigmab}\in\frk S_\kappa^A$. Also let $\gamma=\sqrt{\kappa}\wedge 4/\sqrt{\kappa} $, $\Lambda\geq 0$, $\ell_L,\ell_R>0$, and $\beta\in[-1,1]$. Let $\ell=\ell_L+\ell_R$.

We define three measures $\bar{\ML}^{{\sigmab},\kappa}_{\Lambda,\ell_L,\ell_R}$, $\bar{\ML}^{{\sigmab},\kappa}_{\Lambda,\ell}$ and $\bar{\MC}^{{\sigmab},\kappa}_{\Lambda,\ell_L,\ell_R}$ on the space of decorated regular $\gamma$-LQG surfaces. We first consider the $\kappa\in (8/3,4)$ case.
\begin{itemize}
	\item Let $(h,\z)\sim \op{M}^{{\sigmab},\kappa}_{\Lambda,\ell_{L}+\ell_{R}}$ (see Definition \ref{def:disk} with $\w$ given by $w_0 = -i$ and $w_{\pm 1}=\pm 1$). Let $w_\infty\in\partial\D$ be the point such that the clockwise (resp.\ counterclockwise) arc from $w_0$ to $w_\infty$ has length $\ell_L$ (resp.\ $\ell_R$) from $w_0$ to $w_\infty$. Let $\Gamma\sim\CLE_\kappa^{\sigmab}(\z)$ and then sample an associated CPI $\lambda$ with parameter $\beta$ (see Definition \ref{def:cle} and Section \ref{sec:cle-explorations}). We obtain a decorated $\gamma$-LQG surface $[(\D,h,\z,\Gamma,\lambda)]$ and we call its law $\bar{\ML}^{{\sigmab},\kappa}_{\Lambda,\ell_L,\ell_R}$.
	\item If $[(\D,h,\z,\Gamma,\lambda)] \sim \bar{\ML}^{{\sigmab},\kappa}_{\Lambda,\ell_L,\ell_R}$ we write $\bar{\ML}^{{\sigmab},\kappa}_{\Lambda,\ell}$ for the law of $[(\D,h,\z,\Gamma)]$.
	\item Let $\Gamma_0$ denote the loops in $\Gamma$ intersecting $\lambda$. If $[(\D,h,\z,\Gamma,\lambda)] \sim \bar{\ML}^{{\sigmab},\kappa}_{\Lambda,\ell_L,\ell_R}$ then $\bar{\MC}^{{\sigmab},\kappa}_{\Lambda,\ell_L,\ell_R}$ denotes the law of $[(h,\z,\Gamma_0,\lambda)]$.
\end{itemize}
As in Section \ref{sec:background-msw}, we can associate to $S=[(\D,h,\z,\Gamma,\lambda)]$ the following objects: the $\gamma$-LQG length $\zeta$ of the CPI, for each $t<\zeta$ the future part of the disk $S_t$ (this is now a loop decorated $\gamma$-LQG surface with marked points) and boundary length processes $L$ and $R$ (writing $X=L+R$). Let $J(X):=\{t<\zeta\colon \Delta X_t\neq 0\}$. Then for each $t\in J(X)$, the cut out $\gamma$-LQG surfaces $\Delta S_t$ is now a loop decorated $\gamma$-LQG surface with marked points, more precisely $\Delta S_t = [(\Delta D_t,h|_{\Delta D_t},\z|_{\mathcal{I}_{\z}(\Delta D_t)},\{\eta\in \Gamma\colon \eta^o\subsetneq \Delta D_t\})]$ where $\mathcal{I}_{\z}(U):=\{i\in A\colon z_i\in U\}$ whenever $U$ is open. Here, $\Delta D_t$ is one of the complementary components that the CPI with attached CLE loops cuts out. We write $P_t=\mathcal{I}_{\z}(D_t)$ for the indices of the points which are in the future of part $S_t$ of the disk. Then $\Delta P_t = P_{t-}\setminus P_{t}$ for $t\in J(X)$ are the indices of the marked points on $\Delta S_t$. The process $P$ is a set-valued càdlàg process and we summarize the processes as $\bm{Y}=(L,R,P)$.

If $S=[(\D,h,\z,\Gamma_0,\lambda)]$, we can again associate to it $\zeta$, $L$, $R$, $X$, $S_t$, $\Delta S_t$, $P$ and $\bm{Y}$, the only difference being that $S_t$ and $\Delta S_t$ are now $\gamma$-LQG disks with marked points but without a loop decoration. As in Section \ref{sec:background-msw}, we can write the cut out $\gamma$-LQG surface as $\Delta S_t=[(\Delta D_t,h|_{\Delta D_t},\z|_{\mathcal{I}_{\z}(\Delta D_t)})]$.

In the case $\kappa\in (4,8)$ we define these measures analogously. We refer the reader to Section \ref{sec:background-msw} for information on the relevant constructions in this case.

In the following, we will leave the dependence of the boundary length processes and cut out quantum surfaces on the underlying decorated $\gamma$-LQG surface implicit. For example, the process $\bm{Y}$ on the left-hand side of the first indented equation below depends on the integration variable $S$ and we leave this dependence implicit in the notation.
\begin{prop}
	In the setting above, for any measurable function $f$ with values in $[0,1]$,
	\begin{align*}
		\begin{split}
			&\int f(\bm{Y}, \{(t,\Delta S_t)\colon t\in J(X)\} )\,\bar{\ML}^{{\sigmab},\kappa}_{\Lambda,\ell_L,\ell_R}(dS)  \\
			&= \E\left( \sum_{Q\in \Pi(X^0,A)} \int f((L^0,R^0,Q),\{(t,S'_t)\colon t\in J(X^0) \}) \prod_{t\in J(X^0)} e^{\sigma_{\Delta Q_t} 1(\Delta X^0_t>0)}\,\bar{\ML}^{{\sigmab}|_{\Delta Q_t},\kappa}_{\Lambda, |\Delta X_t|}(dS'_t) \right)\
		\end{split}
	\end{align*}
	where the expectation on the right-hand side is taken with respect to the random process $(L^0,R^0)$ from Section \ref{sec:background-msw} (corresponding to $A=\emptyset$) with $X^0=L^0+R^0$.
	In particular, the case $f\equiv 1$ yields
	\begin{align*}
		W^{{\sigmab},\kappa}_{\Lambda,\ell}
		= \E\left(\, \sum_{Q\in \Pi(X^0,A)} \int \prod_{t\in J(X^0)} e^{\sigma_{\Delta Q_t} 1(\Delta X^0_t>0)}  W^{{\sigmab}|_{\Delta Q_t},\kappa}_{\Lambda, |\Delta X^0_t|} \right)\;.
	\end{align*}
	\label{prop:chordal-zipper}
\end{prop}

We will also use a version of the proposition in the case where we decorate a $\gamma$-LQG disk with marked points with a (nested) weighted CLE and consider the collection of components which are cut out by the outermost CLE loops. 

\begin{prop}
	\label{prop:full-cle-marked-cut}
	Let $(S_n\colon n\ge 1)$ and $(\ell_n\colon n\ge 1)$ be the $\gamma$-LQG surfaces cut out by the collection of outermost loops on $S$ and the associated boundary lengths (both sequences ordered according to decreasing boundary length). Furthermore, let $(\ell^{\,0}_n\colon n\ge 1)$ be the sequence of boundary lengths of a non-nested $\CLE_\kappa$ decorating a (regular) $\gamma$-LQG disk (also ordered according to decreasing boundary length). Then, for any measurable function $f$ with values in $[0,1]$,
	\begin{align*}
		&\int f(\{(\ell_n,S_n)\colon n\ge 1\}) \,\bar{\ML}^{{\sigmab},\kappa}_{\Lambda,\ell}(dS)  \\
		&= \E\left( \sum_{\cup_n B_n = A} \int f(\{(\ell^{\,0}_n,S'_n)\colon n\ge 1\})\,\prod_{n\ge 1} e^{\sigma_{B_n}} \,\bar{\ML}^{{\sigmab}|_{B_n},\kappa}_{\Lambda, \ell^{\,0}_n}(dS'_n) \right)
	\end{align*}
	where the sum on the right-hand side is over all partitions of $A$ into disjoint sets.
\end{prop}

\begin{proof}
	The proof is omitted since it is identical to that of Proposition \ref{prop:chordal-zipper}. In fact, the proof is slightly easier than that of Proposition \ref{prop:chordal-zipper} since we do not need to consider the CPI.
\end{proof}

The next lemma is a variant of Proposition \ref{prop:chordal-zipper} where we consider surfaces decorated by $(\Gamma_0,\lambda)$ instead of $(\Gamma,\lambda)$. The statement therefore only differs by the fact that we are not considering the decoration by loops inside the cut out surfaces.

\begin{lemma}
	In the setting above, for any measurable function $f$ with values in $[0,1]$,
	\begin{align*}
	\begin{split}
		&\int f(\bm{Y}, \{(t,\Delta S_t)\colon t\in J(X)\} ) \,\bar{\MC}^{{\sigmab},\kappa}_{\Lambda,\ell_L,\ell_R}(dS)  \\
		&= \E\left( \sum_{Q\in \Pi(X^0,A)} \int f( (L^0,R^0,Q),\{(t,S'_t)\colon t\in J(X^0) \}) \prod_{t\in J(X^0)} e^{\sigma_{\Delta Q_t} 1(\Delta X^0_t>0)}\,\bar{\M}^{{\sigmab}|_{\Delta Q_t},\kappa}_{\Lambda, |\Delta X^0_t|}(dS'_t) \right)
	\end{split}
	\end{align*}
	where the expectation on the right-hand side is taken with respect to the random process $(L^0,R^0)$ from Section \ref{sec:background-msw} (corresponding to $A=\emptyset$) with $X^0=L^0+R^0$.
	\label{prop:chordal-zipper0}
\end{lemma}

\begin{proof}[Proof of Proposition \ref{prop:chordal-zipper} given Lemma \ref{prop:chordal-zipper0}]
	We only need to prove the first statement since the second one is immediate. The difference between the proposition and the lemma is that in the latter case we are only considering the CLE loops that are intersecting the CPI while in the former case we are considering a full nested CLE. The result is then a consequence of the definitions and the final part of Lemma \ref{lem:xi-law}.
\end{proof}

It remains to prove Lemma \ref{prop:chordal-zipper0}. We will first do the proof for $\kappa\in(8/3,4)$.

\begin{proof}[Proof of Lemma \ref{prop:chordal-zipper0} for $\kappa\in (8/3,4)$]
We start by introducing some notation. Sample the tuple $(h,\z)\sim \M^{{\sigmab},\kappa}_{0,\ell}$ (Definition \ref{def:disk}), and define $w_\infty$, $\Gamma$, $\lambda$, $\Gamma_0$, $\bm{Y}=(L,R,P)$, $X$, $J(X)$, $\mathcal{I}_{\z}(U)$, and $(\Delta D_t)_{t\in J(X)}$ as above. 

Let $w_t^0$ denote the marked point on $\partial\Delta D_t$ which is equal to the point on $\partial \Delta D_t$ that is hit first by the CPI $\lambda$.
Then let $w_t^1$ and $w_t^{-1}$ be sampled independently according to the (normalized) $\gamma$-LQG boundary length measure with respect to $h$ on $\partial\Delta D_t$ conditioned on $\w_t:=(w_t^0,w_t^1,w_t^{-1})$ being ordered counterclockwise. Let $\psi_t:\D\to \Delta D_t$ be the unique conformal map which is sending $\w_t$ to $(-i,1,-1)$ and define
\begin{align}
\label{eq55}
\begin{split}
\wt h_t = h\circ\psi_t+Q_\gamma\log&|\psi'_t|\;,
\quad
\wt z^{\,t}_{i}=\psi_{t}^{-1}(z_i)\quad\text{ for $t\in J(X)$ and $i\in \Delta P_t$}\;,\\
\quad
&\widetilde{\z}^{\,t}=(\wt z^{\,t}_{i}\colon  t\in J(X)\,,\ i\in \Delta P_t)
\end{split}
\end{align}
whenever $\z\in \D^A$. We also let
\begin{align*}
	 \wt{\M}_{L,R,(\wt h_t)_t}\quad\text{be the conditional law of $(\psi_t)_{t\in J(X)}$ given $L,R,(\wt h_t)_{t\in J(X)}$}
\end{align*}
in the setting of $A=\emptyset$.\footnote{\,In fact, it is stated in \cite[Remark 3.6]{msw-simple} that $(\psi_t)_{\in J(X)}$ is almost determined by $L,R,(\wt h_t)_{t\in J(X)}$ in the sense that $L,R,(\wt h_t)_{t\in J(X)}$ determine a collection of conformal transformations $(\wt\psi_t)_{t\in J(X)}$ such $\wt\psi_t=\phi\circ\psi_t$ for some conformal map $\phi:\D\to\D$. We do not rely on this result.} We denote the law of $(h,\Gamma_0,\lambda,(\w_t)_{t\in J(X)} )$ by $\MC^{-,\kappa}_{0,\ell_{L},\ell_{R}}$. For any non-negative measurable function $f$, by \cite{msw-simple} (see Section \ref{sec:background-msw} for a review),
\begin{align}
\begin{split}
	&\int f(L,R, (\wt h_t,\psi_t)_{t} ) \,d\MC^{-,\kappa}_{0,\ell_{L},\ell_{R}}(h,\Gamma_0,\lambda,(\w_t)_{t\in J(X)} )\\
	&= \E\left(\int f(L^0,R^0, (\wt h_t,\psi_t)_{t} ) \,\wt{\M}_{L^0,R^0,(\wt h_t)_{t\in J(X^0)}}( d(\psi_t)_{t\in J(X^0)} )\, \prod_{t\in J(X^0)}
	\M^{-,\kappa}_{0,|\Delta X^0_t|}(d\wt h_t) \right) \;.
\end{split}
\label{eq31}
\end{align}
Note that we are abusing notation slightly since $\widetilde{h}_t$ and $\psi_t$ are functions of the integration variables on the left-hand side while they are integration variables on the right-hand side.

Given a tuple $\z\in \D^A$ we write $P_t=\{i\in A\colon z_i\in D_t\}$, leaving the dependence on $\z$ implicit. In the following, we will write $\lambda_{U}$ for the Lebesgue measure on $U$ where $U\subseteq \C$ is open. The product measure on $U^B$ for $B\subseteq A$ is denoted by $\lambda_U^{\otimes B}$ with the convention that $\lambda_U^{\otimes\emptyset}$ is a Dirac delta of mass $1$ on the empty tuple.

For any $t\in J(X)$, the pushforward of $\lambda_{\D}^{\otimes B}$ along $\psi_t^{-1}$ is given by $(|\psi_t'\circ \psi_t^{-1}|^{-2}\lambda_{\Delta D_t})^{\otimes B}$. Using \eqref{eq31}, we obtain that for $f$ non-negative and measurable, 
\begin{align}
\begin{split}
	&\int f((L,R,P), (\wt h_t,\widetilde{\z}^{\,t},\psi_t)_{t} ) \,d\MC^{-,\kappa}_{0,\ell_{L},\ell_{R}}(h,\Gamma_0,\lambda,(\w_t)_{t\in J(X)} )\\
	&\qquad\quad \prod_{t\in J(X)}\prod_{i\in \Delta P_t} |\psi_t'(\widetilde{z}^{\,t}_{i})|^{-2}\,\lambda_\D^{\otimes A}(d\z)\\
	&= \E\Biggl( \int \sum_{Q\in \Pi(X,A)} f((L^0,R^0,Q), (\wt h_t,\widetilde{\z}^{\,t},\psi_t)_{t} )\, \wt{\M}_{L^0,R^0,(\wt h_t)_{t}}( d(\psi_t)_{t\in J(X)} )\\
	&\qquad\quad \prod_{t\in J(X^0)}
	\M^{-,\kappa}_{0,|\Delta X^0_t|}(d\wt h_t)\,\lambda_{\D}^{\otimes \Delta Q_t}(d\widetilde{\z}^{\,t}) \Biggr)\;.
\end{split}
\label{eq24}
\end{align}
By the same type of abuse of notation as before $\widetilde{\z}^{\,t}$ has different meanings on the right and left-hand side of the equation, namely, on the left-hand side it is defined by \eqref{eq55} while on the right-hand side it is the integration variable. Let $\theta:\C\to[0,\infty)$ be a smooth radially symmetric bump function satisfying 
\begin{align*}
\int_{\C} \theta = 1,\qquad
\op{supp}(\theta)\subseteq \D\,,\qquad
\iint \theta(z)\theta(z')\log(|z|\vee|z'|)^{-1}\,dz\,dz'=0.
\end{align*}
For $\epsilon\in(0,1)$ and $z\in\D$ set $\wt\theta_\epsilon^z(\wt z)=\epsilon^{-2}\theta( \epsilon^{-1}(\wt z-z) )$ and for a distribution $\wt h$ define $\wt h_\epsilon(z):=\wt h(\wt\theta_\epsilon^z)$. Fix $r\in (0,1/16)$. Consider a function $f=f_\epsilon$ of the form
\begin{align}
\begin{split}
	&f_\epsilon(\bm{Y}, (\wt h_t,\widetilde{\z}^{\,t},\psi_t)_{t})=
	\wt f(\bm{Y}, (\wt h_{t_1},\widetilde{\z}^{\,t_1},\psi_{t_1}),\dots,(\wt h_{t_N},\widetilde{\z}^{\,t_N},\psi_{t_N}))
	\\
	&\qquad\qquad\qquad\qquad\qquad 
	\prod_{t\in J(X)} \Phi^{{\sigmab|_{\Delta P_t}},\kappa}_{\D}(\widetilde{\z}^{\,t})
	e^{\sigma_{\Delta P_t} 1(\Delta X_t>0)}
	 e^{-\Lambda \mu^\gamma_{\,\wt h_t}(\D)} \prod_{i\in \Delta P_t} \epsilon^{\alpha^2_i/2} e^{\alpha_i (\wt h_t)_\epsilon(\wt z^{\,t}_{i} )}
\end{split}
\label{eq23}
\end{align}
where $t_1<\cdots<t_N$ are all the times satisfying $|\Delta X_{t_i}|>r$ and where $\widetilde{f}$ is assumed to be measurable, take values in $[0,1]$ and have the following properties:
\begin{itemize}
	\item The function $\wt f$ is continuous in each of the fields with respect to the topology of $H^{-1}(\D)$.
	\item If $n\le N$ and $\widetilde{\z}^{\,t_n}=(\widetilde{z}^{\,t_n}_i\colon i\in B)$ then $\widetilde{f}$ takes the value $0$ whenever 
	\begin{align*}
		&\min_{i\in B}\ 
		(1-|\wt z^{\,t_n}_{i}|)\wedge 
		(1-|\psi_{t_n}(\wt z^{\,t_n}_{i})|)\wedge |\psi_{t_n}'(\widetilde{z}^{\,t_n}_i)| \\
		&\qquad\wedge
		\min_{\substack{i,j\in B\\i\neq j}} |\wt z^{\,t_n}_{i}-\wt z^{\,t_n}_{j}|
		\wedge |\psi_{t_n}(\wt z^{\,t_n}_{i})-\psi_{t_n}(\wt z^{\,t_n}_{j})| \le 4\sqrt{r}\;.
	\end{align*}
	\item The function $\widetilde{f}$ takes the value $0$ if $\Delta P_{t_1}\cup\cdots \cup \Delta P_{t_N}\neq A$.
\end{itemize}
Inserting \eqref{eq23} into \eqref{eq24}, we will show that the left-hand side of \eqref{eq24} converges, as $\epsilon\to 0$, to
\begin{align}
\begin{split}
	\int \wt f(\bm{Y}, (\wt h_{t_n},\widetilde{\z}^{\,t_n},\psi_{t_n})_{n}) \,d\MC^{{\sigmab},\kappa}_{\Lambda,\ell_{L},\ell_{R}}(h,\z,\Gamma_0,\lambda,(\w_t)_{t\in J(X)} )
\end{split}
\label{eq25:lhs}
\end{align}
and that the right-hand side of \eqref{eq24} converges, as $\epsilon\to 0$, to
\begin{align}
\begin{split}
	&\E\Biggr( \int \sum_{Q\in \Pi(X^0,A)} \wt f((L^0,R^0,Q), (\wt h_{t_n},\widetilde{\z}^{\,t_n},\psi_{t_n})_{n} )  \prod_{t\in J(X^0)} e^{\sigma_{\Delta Q_t}1(\Delta X^0_t>0)} \,\\
	&\qquad\qquad \wt{\M}_{L^0,R^0,(\wt h_t)_{t}}( d(\psi_t)_{t\in J(X^0)})\, \prod_{t\in J(X^0)}
	\M^{{\sigmab}|_{\Delta Q_t},\kappa}_{\Lambda,|\Delta X^0_t|}(d\wt h_t,d\widetilde{\z}^{\,t}) \Biggl)
\end{split}
\label{eq25:rhs}
\end{align}
so \eqref{eq25:lhs} is equal to \eqref{eq25:rhs}. Before proving the convergence results, we claim that Lemma \ref{prop:chordal-zipper0} follows from the fact that \eqref{eq25:lhs} and \eqref{eq25:rhs} are equal. The identification of \eqref{eq25:lhs} and \eqref{eq25:rhs} gives by a monotone class argument and using the fact that the constant $r$ was arbitrary that for any non-negative, bounded, and measurable function $\wt f$,
\begin{align}
\begin{split}
	&\int \wt f(\bm{Y}, (\wt h_t,\widetilde{\z}^{\,t})_{t})  \,d\MC^{{\sigmab},\kappa}_{\Lambda,\ell_{L},\ell_{R}}(h,\z,\Gamma_0,\lambda,(\w_t)_{t\in J(X)} )\\
	&=\E\left(\int \sum_{Q\in \Pi(X^0,A)} \wt f((L^0,R^0,Q), (\wt h_t,\widetilde{\z}^{\,t})_{t} )  \prod_{t\in J(X^0)} e^{\sigma_{\Delta Q_t}1(\Delta X^0_t>0)} \,
	\M^{{\sigmab}|_{\Delta Q_t},\kappa}_{\Lambda,|\Delta X^0_t|}(d\wt h_t,d\widetilde{\z}^{\,t}) \right).
\end{split}
\label{eq-after-weight}
\end{align}
The lemma therefore follows by combining the following two observations (both of which follow directly from the definitions):
\begin{itemize}
	\item If we sample $(h,\z,\Gamma_0,\lambda,(\w_t)_{t\in J(X)} )\sim \MC^{{\sigmab},\kappa}_{\Lambda,\ell_{L},\ell_{R}}$ then the law of $[(\D,h,\z,\Gamma_0,\lambda)]$ is $\bar{\MC}^{{\sigmab},\kappa}_{\Lambda,\ell_{L},\ell_{R}}$.
	\item If $(\wt h_t,\widetilde{\z}^{\,t})\sim \M^{{\sigmab}|_{\Delta Q_t},\kappa}_{\Lambda,|\Delta X_t|}$ then the law of $[(\D,\wt h_t,\widetilde{\z}^{\,t})]$ is $\bar{\M}^{{\sigmab}|_{\Delta Q_t},\kappa}_{\Lambda,|\Delta X_t|}$.
\end{itemize}
It remains to prove that the left- and right-hand sides of \eqref{eq24} converge to \eqref{eq25:lhs} and \eqref{eq25:rhs}, respectively. Convergence of the right-hand side of \eqref{eq24} to \eqref{eq25:rhs} is immediate by Proposition \ref{prop:limit-eps} since we can check that the functions $(\wt\theta^z_\epsilon)_{z,\epsilon}$ satisfy \eqref{eq:theta} with 0 instead of $o_\epsilon(1)$ and $\wt\theta^z_\epsilon$ instead of $\theta^z_\epsilon$. It therefore remains to prove the following:
\begin{align}
\text{The left-hand side of \eqref{eq24} converges to \eqref{eq25:lhs} as $\epsilon\to 0$.}
\label{eq33}
\end{align}

We will now prove \eqref{eq33}. To this end, we suppose that $\widetilde{f}(\bm{Y}, (\wt h_{t_n},\widetilde{\z}^{\,t_n},\psi_{t_n})_{n}) >0$. For each $i\in A$ with $i\in \Delta P_t$, let
\begin{align*}
	\delta_i=\epsilon|\psi'_{t}(\wt z^{\,t}_{i})|\;,\quad
	\theta^{z_i}_{\delta_i}=|\psi_{t}'\circ\psi_{t}^{-1}|^{-2} (\wt\theta^{\,\widetilde{z}^{\,t}_i}_{\epsilon}\circ \psi_{t}^{-1}).
\end{align*}
Varying $\epsilon$ this defines $\theta^{z_i}_{\delta}$ for $\delta \in (0,r/2)$. Extend $\theta^z_\delta$ for $z\in\D\setminus\{z_i\,:\,i\in A \}$ such that they satisfy \eqref{eq:theta} (with $\delta$ instead of $\epsilon$) and set $h_{\delta}(z)=h( \theta^{z}_{\delta})$.

We first argue that $\theta^{z}_{\delta}$ satisfies the condition \eqref{eq:theta} (with $\delta$ instead of $\epsilon$)  of Proposition \ref{prop:limit-eps}. Indeed, this follows from \cite[Theorem 3.20]{lawler-book}, which implies that 
\begin{align*}
	|(\psi_{t}^{-1})'(z)-(\psi_{t}^{-1})'(\wt z^{\,t}_i)|/|(\psi_{t}^{-1})'(\wt z^{\,t}_i)|=o_{\delta_i}(1)\quad\text{for $|z-\wt z^{\,t}_i|<\delta_i$}
\end{align*}
as $\delta_i\to 0$ with the $o_{\delta_i}(1)$ deterministic.
The first identity of Lemma \ref{lem:phi-axioms}, $\wt h_t=h\circ\psi_t+Q_\gamma\log|\psi'_t|$, and the fact that $\log|\psi'_{t}|$ is harmonic yields that
\begin{align*}
	&\prod_{t\in J(X)} \Phi^{{\sigmab}|_{\Delta P_t},\kappa}_{\D}(\widetilde{\z}^{\,t})
	\prod_{i\in \Delta P_t} \epsilon^{\alpha^2_i/2} e^{\alpha_i (\wt h_t)_\epsilon(\wt z^{\,t}_{i} )}|\psi'_{t}(\wt z^{\,t}_{i})|^{-2} \\
	&\quad =\prod_{t\in J(X)} \Phi^{{\sigmab}|_{\Delta P_t},\kappa}_{\Delta D_t}(\z|_{\Delta P_t})
	\prod_{i\in \Delta P_t} \delta_i^{\alpha_i^2/2}
	e^{\alpha_i
		h_{\delta_i}(z_i)} \;.
\end{align*}
Let $\M_h$ denote the law of $(\Gamma_0,\lambda,(\w_t)_{t\in J(X)})$ given $h$ if $(h,\Gamma_0,\lambda,(\w_t)_{t\in J(X)})\sim \MC^{-,\kappa}_{\Lambda,\ell_{L},\ell_{R}}$. Note that $\M_h$ can be made sense of as soon as the $\gamma$-LQG boundary measure on the boundary of the domain and on the boundary of each of the components $\Delta D_t$ is well-defined. We will use Proposition \ref{prop:limit-eps} to argue that as $\epsilon\to 0$,
\begin{align}
\begin{split}
	&\int f_\epsilon(\bm{Y}, (\wt h_{t_n},\widetilde{\z}^{\,t_n},\psi_{t_n})_{n} )
	\,d\M_{h}(\Gamma_0,\lambda, (\w_t)_{t\in J(X)} )
	\,\M^{-,\kappa}_{\Lambda,\ell}(dh )\,d\z\\
	&\to\int \wt f(\bm{Y}, (\wt h_{t_n},\widetilde{\z}^{\,t_n},\psi_{t_n})_{n} )
	\frac{\prod_{t\in J(X)} \Phi^{{\sigmab}|_{\Delta P_t},\kappa}_{D_t}(\z|_{\Delta P_t}) e^{\sigma_{\Delta P_t}1(\Delta X_t>0)}}{\Phi^{{\sigmab},\kappa}_{\D}(\z)}   \\
	&\qquad\qquad\qquad\qquad\,d\M_{h}(\Gamma_0,\lambda, (\w_t)_{t\in J(X)} )
	\,\M^{{\sigmab},\kappa}_{\Lambda,\ell}(dh,d\z),
\end{split}
\label{eq32}
\end{align}
which is sufficient to conclude the proof of \eqref{eq33} since the 	 of \eqref{eq32} is equal to the left-hand side of \eqref{eq24} and the right-hand side of \eqref{eq32} is equal to \eqref{eq25:lhs} by Lemma \ref{lem:xi-law}. To this end, define $F$ by
\begin{align*}
\begin{split}
	&F(h,\z,\Gamma_0,\lambda, (\w_t)_{t\in J(X)}) = \wt f(\bm{Y}, (\wt h_{t_n},\widetilde{\z}^{\,t_n},\psi_{t_n})_{n} ) \frac{\prod_{t\in J(X)} \Phi^{{\sigmab}|_{\Delta P_t},\kappa}_{D_t}(\z|_{\Delta P_t}) e^{\sigma_{\Delta P_t}1(\Delta X_t>0)}}{\Phi^{{\sigmab},\kappa}_{\D}(\z)}\;.
\end{split}
\end{align*}
We now want to apply Proposition \ref{prop:limit-eps} to this functional. In the proof above, we described the probability kernel used to sample $(\Gamma_0,\lambda, (\w_t)_{t\in J(X)})$ and $(\theta^z_\delta)$ in \eqref{eq32} conditionally on $(h,\z)$. The key observation is that this probability kernel only depends on $(h,\z)$ through $\z$ and the restriction of $h$ to the complement of $\cup_i B_r(z_i)$. This is indeed true since the target point of the CPI $w_\infty$ is sampled from the $\gamma$-LQG boundary measure with respect to $h$ which is measurable with respect to the restriction of $h$ to the complement of $\cup_i B_r(z_i)$. Similarly, the points $(\w_t)_{t\in J(X)}$ are sampled from the $\gamma$-LQG boundary measures of the cut out components which are also measurable with respect to the restriction of $h$ to the complement of $\cup_i B_r(z_i)$.

In order to prove \eqref{eq32} it is therefore sufficient to argue that the function $F$ satisfies the assumptions appearing in Proposition \ref{prop:limit-eps}.

Measurability of $F$ is immediate, and boundedness follows from boundedness of $\wt f$ and the upper and lower bounds at the end of  Theorem \ref{thm:loopexpconv}. To argue that $F$ satisfies the required continuity property, we consider a sequence $(g^\z_m)$ as in Proposition \ref{prop:limit-eps}. By the support condition on $(g^\z_m)$, adding the functions to the field does not affect the process $\bm{Y}$ and the boundary lengths of the cut out components. Thus
\begin{align*}
	F(h+g^\z_m,\z,\Gamma_0,\lambda, (\w_t)_{t\in J(X)}) &= \wt f(\bm{Y}, (\wt h_{t_n}+g^\z_m\circ \psi_{t_n},\widetilde{\z}^{\,t_n},\psi_{t_n})_{n} ) \\
	&\qquad\quad\cdot \frac{\prod_{t\in J(X)} \Phi^{{\sigmab}|_{\Delta P_t},\kappa}_{D_t}(\z|_{\Delta P_t}) e^{\sigma_{\Delta P_t}1(\Delta X_t>0)}}{\Phi^{{\sigmab},\kappa}_{\D}(\z)} \\
	&\to F(h,\z,\Gamma_0,\lambda, (\w_t)_{t\in J(X)})
\end{align*}
as $m\to \infty$ as required in Proposition \ref{prop:limit-eps}.
\end{proof}

\begin{proof}[Proof of Lemma \ref{prop:chordal-zipper0} for $\kappa\in(4,8)$] The proof for generalized disks is identical to the case of regular disks, except that a surface is represented by a countable collection of fields rather than a single field due to the difference in topology between the disks. We will therefore just state the counterparts of \eqref{eq31} and \eqref{eq-after-weight} in the setting of generalized disks and leave the details of the proof to the reader. In the following, we will use the index $u$ to denote the time parameter of the excursion that encodes the generalized $\gamma$-LQG disk on which we have a CPI,
we will use $t$ to denote the time parameter of the exploration process, and 
we will use $s$ to denote the time parameter of the excursions encoding the cut-out generalized $\gamma$-LQG disks. We write $\kappa'$ instead of $\kappa$ below.
	
We start by defining a measure $\MC^{{\sigmab},\kappa'}_{\Lambda,\ell_L,\ell_R}$ on tuples $(e,(h_u,\z_u,\Gamma^0_u,\lambda_u),(\bm{w}_s^t)_{t,s} )$.
\begin{itemize}
	\item First sample $(e,(h_u,\z_u)_{u\in J(e)} )$ as in Definition \ref{def:gen-disk}. Note that Definition \ref{def:gen-disk} gives a measure on (generalized) $\gamma$-LQG surfaces rather than fields but it can be immediately adapted to give a measure on fields.
	\item Then sample an independent copy $\Gamma_u$ of $\CLE_{\kappa'}^{\sigmab|_B}(\z_u)$ for each $u\in J(e)$ if $\z_u$ is indexed by the set $B$. Consider the CPI and write $\lambda_u$ for the segment of the CPI if it goes through the component of the generalized $\gamma$-LQG disk corresponding to the index $u$ (and $\lambda_u=-$ otherwise). Let $\Gamma^0_u$ denote the loops in $\Gamma_u$ intersecting $\lambda_u$.
	\item Define $\bm{Y}=(L,R,Q)$ and $X=L+R$ as in Section \ref{sec:background-msw}. For $t\in J(X)$, let $e^t$ be the excursion encoding the cut out generalized $\gamma$-LQG disk, sample boundary points $\w_s^t$ and consider the mapping out functions $\psi^t_s$ as in the simple case whenever $t\in J(X)$ and $s\in J(e^t)$. Define $\wt{h}^t_s$ and $\wt{\z}^{\,t}_s$ analogously to the simple case.
\end{itemize}
In the case $A=\emptyset$ we write $(L^0,R^0,X^0,(E^t))$ instead of $(L,R,X,(e^t))$. The counterpart of \eqref{eq31} in the case of generalized disks, which follows from \cite{msw-non-simple} (see again Section \ref{sec:background-msw}) is then
\begin{align*}
\begin{split}
	&\int f(L,R,(e^t)_{t}, (\wt{h}^t_s,\psi^t_s)_{t,s} )
	 \,d\MC^{-,\kappa'}_{\Lambda,\ell_L,\ell_R}
	(e,(h_u,\Gamma_u,\lambda_u)_u,(\bm{w}_s^t)_{t,s} )\\
	&= \E\left( \int f(L^0,R^0,(E^t)_{t}, (\wt{h}^t_s,\psi^t_s)_{t,s} )
	\,\wt{\M}_{L^0,R^0,(\wt{h}^t_s)_{t,s}}( d(\psi^t_s)_{t,s} )\prod_{t,s}
	\M^{-,\kappa'}_{\Lambda,|\Delta E^t_s|}(d\wt{h}^t_s) \right)
\end{split}
\end{align*}
where, every time we index over $t,s,u$, we mean $t\in J(X),s\in J(e^t),u\in J(e)$ on the left-hand side and $t\in J(X^0), s\in J(E^t)$ on the right-hand side.

Adding marked points via a limiting procedure as in the case of regular disks we get the following, which is the counterpart of \eqref{eq-after-weight} in the generalized case:
\begin{align*}
\begin{split}
	&\int f(\bm{Y},(e^t)_{t}, (\wt{h}^t_s, \wt{\z}^{\,t}_s)) \,d\MC^{{\sigmab},\kappa'}_{\Lambda,\ell_L,\ell_R}
	(e, (h_u,\z_u,\Gamma_u,\lambda_u)_u,(\bm{w}_s^t)_{t,s} )\\
	&\qquad= \E\Biggr( \int \sum_{Q,(Q^t)} f((L^0,R^0,Q),(E^t)_{t}, (\wt{h}^t_s,\wt{\z}^{\,t}_s)_{t,s} )
	\,\wt{\M}_{L^0,R^0,(\wt{h}^t_s)_{t,s}}( d(\psi^t_s)_{t,s} ) \\
	&\qquad\qquad\qquad\qquad\qquad\qquad\qquad \prod_{t,s}\M^{{\sigmab}|_{\Delta Q^t_s},\kappa'}_{\Lambda,|\Delta E^t_s|}(d\wt{h}^t_s,d\wt{\z}^{\,t}_s) \Biggl)
\end{split}
\end{align*}
where the sum is over $Q\in \Pi(X^0,A)$ and $Q^t\in \Pi(E^t,\Delta Q_t)$ for all $t$. By Definition \ref{def:gen-disk} we deduce the result.
\end{proof}

\section{Recursive formula and finiteness of partition functions}
\label{sec:explorations}

In this section, we will put the result together from previous sections and conclude the proofs of Theorems \ref{thm:markov-simple},  \ref{thm:bootstrap}, and \ref{thm:part-fcn-finite}.
Recall the definition of the weight functions in Definitions \ref{def:simple-disk} and \ref{def:gen-disk}, see also the equations \eqref{eq:part-fcn} and \eqref{eq:part-fcn-gen}. One simple property that we will use frequently is the following scaling property of the weight functions for marked disks. For $\kappa\in (8/3,8)\setminus \{4\}$, $\Lambda\ge 0$, $\ell>0$ and $\sigmab\in \mathfrak{S}_\kappa^A$ we have
\begin{align}
	W^{\sigmab,\kappa}_{\Lambda,\ell}
	=\ell^{\frac{2}{\sqrt{\kappa}}\sum_{i\in A}\alpha_i}\, W^{\sigmab,\kappa}_{\Lambda\ell^{2\wedge (8/\kappa)},1} \;.
	\label{eq:W-scaling}
\end{align}
Moreover, throughout the following sections, to simplify notation we will introduce the shorthand notation
\begin{align}
	W^B_\Lambda(\ell) = W^{\sigmab|_B,\kappa}_{\Lambda,\ell}
	\label{eq57}
\end{align}
for $B\subseteq A$ when $\sigmab$ and $\kappa$ are clear from the context.

\subsection{Recursive partition function formula and law of boundary length process}

In this subsection we conclude the proofs of Theorems \ref{thm:markov-simple} and  \ref{thm:bootstrap}. The following is an extended version of Theorem \ref{thm:markov-simple}.
\begin{thm}
	The assertions of Theorem \ref{thm:markov-simple} hold. Additionally, the generator of $\bf Y$ is given by $\mathcal{G}_{\kappa,\beta}^W$ from Theorem \ref{thm:weighted-levy} in the $\kappa<4$ case and Theorem \ref{thm:nonsimple-weighted-levy} in the $\kappa>4$ case, where the weights $(W^B(\ell)\colon B\subseteq A,\ell>0)$ of the latter theorems are given by $(W^{\sigmab,\kappa}_{\Lambda,\ell}\colon B\subseteq A,\ell>0)$.
	\label{thm:markov-full}
\end{thm}
\begin{proof}
	We first treat the case $\kappa\in(8/3,4)$ and start by verifying that the weights $W^B_{\Lambda}(\ell)$ satisfy \eqref{eq:conditions-simple}. Let us first mention that we have a monotonicity property, namely that $\Lambda\mapsto W^B_{\Lambda}(\ell)$ is decreasing. This and \eqref{eq:W-scaling} immediately imply the first assertion of \eqref{eq:conditions-simple} (monotonicity). The  second assertion of \eqref{eq:conditions-simple} (boundedness by $1$) follows from the monotonicity property and since $W^{\emptyset}_0(\ell)=1$. The third assertion of \eqref{eq:conditions-simple} is given as the second part of Corollary \ref{cor:single-point-properties}.
	The final assertion of \eqref{eq:conditions-simple} follows from \eqref{eq:W-scaling} and the monotonicity property
	\begin{align*}
		W^B_{\Lambda}(x)
		&=
		x^{\frac{2}{\sqrt{\kappa}}\sum_{i\in A}\alpha_i}\, W^B_{\Lambda x^2}(1)
		\leq x^{\frac{2}{\sqrt{\kappa}}\sum_{i\in A}\alpha_i}\, W^B_{\Lambda y^2}(1)
		\\
		&\lesssim y^{\frac{2}{\sqrt{\kappa}}\sum_{i\in A}\alpha_i}\, W^B_{\Lambda y^2}(1)
		= W^B_{\Lambda}(y)\quad\text{for $0<y\le x\le 2y$}\;.
	\end{align*}
	By Proposition \ref{prop:chordal-zipper} and Theorem \ref{thm:msw-simple} we know that the process $\bf Y$ in Theorem \ref{thm:markov-simple} is equal to the process ${\bf Y}^W$ in Theorem \ref{thm:weighted-levy} since it satisfies the defining property \eqref{eq:weighted-levy} of ${\bf Y}^W$. It is therefore immediate by the latter theorem and \eqref{eq:conditions-simple} that the generator of $\bf Y$ is given by $\mathcal{G}_{\kappa,\beta}^W$. The assertion about the jump rates of $\bf Y$ in Theorem \ref{thm:weighted-levy} is now immediate from the definition of $\mathcal{G}_{\kappa,\beta}^W$. The final assertion of Theorem \ref{thm:markov-simple} follows from Proposition \ref{prop:chordal-zipper}.
	
	The proof in the case $\kappa\in (4,8)$ is the same, except that we consider \eqref{eq:conditions-nonsimple}, Theorem \ref{thm:nonsimple-weighted-levy} and Theorem \ref{thm:msw-non-simple} instead of \eqref{eq:conditions-simple}, Theorem \ref{thm:weighted-levy} and Theorem \ref{thm:msw-simple}. When verifying \eqref{eq:conditions-nonsimple} we again use \eqref{eq:W-scaling} and the fact that the function $\Lambda\mapsto W^B_{\Lambda}(\ell)$ is decreasing, but the asymptotics of $W^\emptyset_{\Lambda}(\ell)$ as $\ell\to 0$ is different and given as part of Corollary \ref{cor:single-point-properties}.
\end{proof}

We will now prove Theorem \ref{thm:bootstrap}. The intuition behind the result is explained below the statement of the theorem in the introduction. One technical obstacle with turning this intuition into a proof is that one needs to exclude the possibility that the scenario where all the marked points lie in the same complementary component of the CPI in all the infinitely many iteration steps has positive mass. If we however consider a nested weighted CLE and discover the nested loops layer by layer, then it is clear that the points will be split after finitely many iterations. The key of the proof below is relating this CLE exploration back to the CPI exploration.

\begin{proof}[Proof of Theorem \ref{thm:bootstrap}]	
	We prove the result in the case $\kappa \in(8/3,4)$. The non-simple case where $\kappa\in(4,8)$ is exactly the same. Let $(\ell_n\colon n\ge 1)$ denote the $\gamma$-LQG lengths of a non-nested $\CLE_\kappa$ which decorates a (regular) $\gamma$-LQG disk and let $(L,R)$ be the boundary length process associated to a CPI in a $\gamma$-LQG disk. For notational convenience, let us define an operator
	\begin{align*}
		\mathcal{A}_\pm h(\ell) := \E_{(\ell/2,\ell/2)}\left( \,\sum_{t\in (0,\zeta)\colon \pm \Delta X_t>0} h(|\Delta X_t|)\prod_{s\in (0,\zeta)\setminus \{t\}\colon \Delta X_s\neq 0} W^\emptyset_\Lambda(|\Delta X_s|)\right) \;.
	\end{align*}
	By Proposition \ref{prop:chordal-zipper} we have that for any $B\subseteq A$,
	\begin{align}
		\label{eq:bootstrap-bad-cpi}
		W^B_\Lambda(\ell) &= \E_{(\ell/2,\ell/2)} \left( \,\sum_{Q\in \Pi(X,B)} \prod_{t<\zeta} e^{\sigma_{\Delta Q_t}1(\Delta X_t>0)} W^{\Delta Q_t}_\Lambda(|\Delta X_t|) \right)\;.
	\end{align}
	Moreover, by Proposition \ref{prop:full-cle-marked-cut} we get
	\begin{align}
		\label{eq:bootstrap-bad}
		W^B_\Lambda(\ell) &= \E_\ell \left( \sum_{m\ge 1} \sum_{B_1\cup \cdots \cup B_m = B} \sum_{n_1<\cdots < n_m} \prod_{i=1}^m e^{\sigma_{B_i}}W^{B_i}_\Lambda(\ell_{n_i}) \prod_{n\notin \{n_1,\dots, n_m\}} W^\emptyset_\Lambda(\ell_n) \right)
	\end{align}
	where the second sum in the display ranges over all ordered partitions of $B$ into $m$ non-empty sets (this includes the case where $m=1$ and there is just one element in the partition).
	
	We now consider the (possibly infinite) measure $\bar{\op{M}}^{\sigmab,\kappa}_{\Lambda,\ell}$ and decorate it with an independent nested $\CLE_\kappa^{\sigmab}$. Since this measure of $\gamma$-LQG disks with marked points carries zero mass on configurations where not all of the marked points are distinct, we can consider the first nesting level $j+1$ for some $j\ge 0$ in the nested $\CLE_\kappa^{\sigmab}$ where not all of the marked points lie in the same CLE loop. Splitting into all of the possible values for $j$ and then applying Proposition \ref{prop:full-cle-marked-cut} yields
	\begin{align*}
		W^A_\Lambda(\ell) &= \sum_{j\ge 0} g_j(\ell), \\
		g_0(\ell) &= \E_\ell\left(\sum_{m>1}\sum_{B_1\cup\cdots\cup B_m=A} \,\sum_{n_1<\cdots<n_m} \prod_{i=1}^m e^{\sigma_{B_i}}W^{B_i}_\Lambda(\ell_{n_i}) \prod_{n\notin \{n_1,\dots,n_m\}} W^\emptyset_\Lambda(\ell_n) \right),\\
		g_{j}(\ell) &= \E_\ell\left( \sum_{n\ge 0} e^{\sigma_A} g_{j-1}(\ell_n)\prod_{n'\neq n} W^\emptyset_\Lambda(\ell_{n'}) \right)\;.
	\end{align*}
	Note that here the sum over partitions does not include summing over the trivial partition which contains only one element. By Corollary \ref{cor:msw-simple-full-more} we have for each $j\ge 1$ and $\ell>0$ that
	\begin{align*}
		g_j(\ell) = \sum_{k_j\ge 0} e^{\sigma_A} \mathcal{A}_-^{k_j}\mathcal{A}_+ g_{j-1}(\ell)\;.
	\end{align*}
	By using Corollary \ref{cor:msw-simple-full-more} together with \eqref{eq:bootstrap-bad} we obtain
	\begin{align*}
		g_0(\ell) &= \sum_{k_0\ge 0} \mathcal{A}_-^{k_0} \widetilde{g}_0(\ell)\quad\text{where} \\
		\widetilde{g}_0(\ell) &= \E_{(\ell/2,\ell/2)} \left( \,\sum_{\substack{Q\in \Pi(X,A)\\ \Delta Q_t\subsetneq A\; \forall t<\zeta}} \prod_{t<\zeta} e^{\sigma_{\Delta Q_t}1(\Delta X_t>0)} W^{\Delta Q_t}_\Lambda(|\Delta X_t|) \right) \;,
	\end{align*}
	where $k_0$ represents the number of CPI iterations one needs to do before not all marked points lie in the same complementary component.
	Putting everything together yields
	\begin{align}
		\label{eq:recursion-proof-wa}
		\begin{split}
		W^A_\Lambda(\ell) &= \sum_{j\ge 0}\sum_{k_0,\dots,k_j\ge 0} e^{\sigma_A j} \mathcal{A}_-^{k_j}\mathcal{A}_+ \cdots \mathcal{A}_-^{k_1}\mathcal{A}_+ \mathcal{A}_-^{k_0} \widetilde{g}_0(\ell) \\
		&= \sum_{m\ge 0} \sum_{j\ge 0}\sum_{k_0+\cdots+k_j+j = m} e^{\sigma_A j} \mathcal{A}_-^{k_j}\mathcal{A}_+ \cdots \mathcal{A}_-^{k_1}\mathcal{A}_+ \mathcal{A}_-^{k_0} \widetilde{g}_0(\ell) \\
		&= \sum_{m\ge 0} (\mathcal{A}_-+e^{\sigma_A}\mathcal{A}_+)^m \widetilde{g}_0(\ell)\;.
		\end{split}
	\end{align}
	We will now show that $\widetilde{g}_0(\ell)=V^0_{\Lambda,\ell}$ for all $\ell>0$. To see this, we first sum over all possibilities for the first jump time $t$ and the associated jump $\Delta Q_t=B$ and then combine this with the Markov property of $(L,R)$ together with \eqref{eq:bootstrap-bad-cpi}:
	\begin{align*}
		\widetilde{g}_0(\ell) &= \sum_{\emptyset \neq B\subsetneq A} \E_{(\ell/2,\ell/2)} \Biggl(\, \sum_{t<\zeta\colon \Delta X_t\neq 0} \prod_{s<t} W^\emptyset_\Lambda(|\Delta X_t|) e^{\sigma_B 1(\Delta X_t>0)}W^B_\Lambda(|\Delta X_t|) \\
		&\qquad\qquad\qquad\qquad\qquad \cdot \sum_{\substack{Q\in \Pi(X,A\setminus B)\\Q_t=A\setminus B}} \prod_{s\in (t,\zeta)} e^{\sigma_{\Delta Q_s}1(\Delta X_s>0)} W^{\Delta Q_s}_\Lambda(|\Delta X_s|) \Biggr)\\
		&=  \sum_{\emptyset \neq B\subsetneq A} \E_{(\ell/2,\ell/2)} \left(\, \sum_{t<\zeta\colon \Delta X_t\neq 0} \prod_{s<t} W^\emptyset_\Lambda(|\Delta X_t|) e^{\sigma_B 1(\Delta X_t>0)}W^B_\Lambda(|\Delta X_t|)W^{A\setminus B}_\Lambda(X_t) \right)\\
		&= V^0_{\Lambda,\ell}\;.
	\end{align*}
	Let us comment in slightly more detail on
	the second equality above. 
	One first turns the sum over the $t$ values into an integral with the use of the final statement in Lemma \ref{lem:msw-palm-result}, then one applies \eqref{eq:bootstrap-bad-cpi} and finally, one again uses Lemma \ref{lem:msw-palm-result} to turn the integration back into a sum over jump times. At an intuitive level, we are applying the Markov property at the time $t$, however we need to reason slightly differently since the time $t$ appears in a sum within the expectation. We leave the details to the reader.
	
	A similar reasoning as the one in the previous paragraph yields that
	\begin{align*}
		\mathcal{A}_\pm h(\ell) := \E_{(\ell/2,\ell/2)}\left( \,\sum_{t\in (0,\zeta)\colon \pm \Delta X_t>0} h(|\Delta X_t|)\prod_{s<t\colon \Delta X_s\neq 0} W^\emptyset_\Lambda(|\Delta X_s|)\cdot W^\emptyset_\Lambda(X_t) \right) 
	\end{align*}
	and the statement then follows from \eqref{eq:recursion-proof-wa} and the result that $\widetilde{g}_0(\ell)=V^0_{\Lambda,\ell}$ for all $\ell>0$.
\end{proof}

\subsection{Finiteness of the partition function: Proof of Theorem \ref{thm:part-fcn-finite}}

The goal of this subsection is to prove Theorem \ref{thm:part-fcn-finite}, which gives a partial characterization of which parameters are admissible.  The theorem is split up into the Lemmas \ref{prop:part-fcn-finite1} and \ref{prop:part-fcn-finite2} below.

Our proof is based on the recursive expression for the partition function in Theorem \ref{thm:bootstrap}, along with  basic monotonicity, scaling and continuity properties of the partition function. We also use Propositions \ref{prop:power-theta} and \ref{prop:nonsimple-power-theta} in the proof of (ii). The proof of (iv) in Lemma \ref{prop:part-fcn-finite2} will be particularly involved; see the paragraph before it for a brief outline.
For the proof of Lemma \ref{prop:part-fcn-finite1} the following direct consequence of Theorem \ref{thm:bootstrap} will suffice.
\begin{lemma}
	Suppose that $\#A>1$. For any $({\sigmab},\Lambda)\in\mathfrak{S}_\kappa^A$,
	\begin{align}
		W^A_\Lambda(\ell) &= \sum_{\emptyset\neq B\subseteq A} \E_{(\ell/2,\ell/2)}\left( \sum_{t<\zeta} e^{\sigma_B 1(\Delta X_t>0)}\prod_{s<t}
		W^\emptyset_{\Lambda}(|\Delta X_s|)
		W^{B}_{\Lambda}(|\Delta X_t|)
		W^{B\setminus A}_{\Lambda}(X_t) \right)\;.
		\label{eq85}
	\end{align}
	In particular, each term on the right-hand side is strictly smaller than the left-hand side.
	\label{prop:chordal-points-together}
\end{lemma}
\begin{proof}
	This identity follows immediately by using the expression for $W^A_\Lambda(\ell)$ in Theorem \ref{thm:bootstrap}. To be more precise, note that it is sufficient to argue that the term on the right-hand side of \eqref{eq85} corresponding to $B=A$ is equal to $\sum_{m=1}^\infty V^m_\Lambda(|\Delta X_t|)$, and we get this by replacing $W^A_\Lambda(|\Delta X_t|)$ in the former expression by $\sum_{m=0}^\infty V^m_\Lambda(|\Delta X_t|)$ and then using the recursive definition of $V^m_\Lambda(|\Delta X_t|)$ in Theorem \ref{thm:bootstrap}.
\end{proof}

The following lemma gives Theorem \ref{thm:part-fcn-finite} except assertion (iv). 
\begin{lemma}
	Assertions (i), (ii), (iii), (v), (vi) and (vii) of Theorem \ref{thm:part-fcn-finite} hold.
	\label{prop:part-fcn-finite1}
\end{lemma}
\begin{proof}
	Assertion (i) of the proposition is immediate from Theorem \ref{thm:singlepoint-finite} and assertion (vi) is immediate from Definition \ref{def:disk} since $W^{{\sigmab},\kappa}_{\Lambda,\ell}$ is monotone in $\sigma_B$ for $\#B\geq 2$.

	We continue by proving (ii), i.e., if $\#A\geq 2$ and $\Lambda=0$ then $({\sigmab},\Lambda)\not\in\cA_\kappa$. By the final assertion of Lemma \ref{prop:chordal-points-together}, the fact that $W^\emptyset_{0}(\ell)=1$ and 
	\begin{align*}
		W^A_{0}(\ell)=\ell^{\frac{2}{\sqrt{\kappa}} \sum_{i\in A}\alpha_i} W^A_{0}(1)
	\end{align*}
	we have that
	\begin{align}
		W^{A}_{\Lambda}(1) > e^{\sigma_A}
		W^{A}_{\Lambda}(1)
		\E_{(1/2,1/2)}\left( \sum_{t<\zeta\colon \Delta X_t>0}  |\Delta X_t|^{\frac{2}{\sqrt{\kappa}}\sum_i \alpha_i} \right).
		\label{eq35}
	\end{align}
	By Definition \ref{def:main-params}, $\alpha_i\cdot 2/\sqrt{\kappa}>4/\kappa+1/2$, so $(2/\sqrt{\kappa})\cdot \sum_i \alpha_i>1+8/\kappa$. This implies via an application of Proposition \ref{prop:power-theta} (when $\kappa\in(8/3,4)$) and of Proposition \ref{prop:nonsimple-power-theta} (when $\kappa\in(4,8)$) that the right-hand side of \eqref{eq35} is infinite, which gives (ii).
	
	Next we prove (v), i.e., if $({\sigmab},\Lambda)\in\cA_\kappa$ for $\Lambda>0$ then $({\sigmab},\Lambda')\in\cA_\kappa$ for any $\Lambda'>0$. It is immediate from Definition \ref{def:disk} that $W^{A}_{\Lambda}(1)$ is decreasing in $\Lambda$, so if (v) was not true then there would be some $\Lambda_0\in(0,\infty)$ such that $W^{A}_{\Lambda}(1)<\infty$ for all $\Lambda>\Lambda_0$ and $W^{A}_{\Lambda}(1)=\infty$ for all $\Lambda<\Lambda_0$. We suppose such a $\Lambda_0$ exists and want to derive a contradiction. Suppose $\Lambda>\Lambda_0$. Then the right-hand side of \eqref{eq85} is infinite since it holds with probability 1 that there is a $t\in(0,\zeta)$ for which $\Lambda|\Delta X_t|^{2\wedge (8/\kappa)}<\Lambda_0$ and since we have
	\begin{align*}
		W^{A}_\Lambda(|\Delta X_t|)=|\Delta X_t|^{\frac{2}{\sqrt{\kappa}} \sum_i \alpha_i} W^{A}_{\Lambda |\Delta X_t|^{2\wedge (8/\kappa)}}(1)=\infty
	\end{align*}
	for such $t\in(0,\zeta)$. Therefore $W^A_\Lambda(1)=\infty$, which is a contradiction to $\Lambda>\Lambda_0$, and we conclude that (v) holds.
	
	Now we prove (iii), i.e., if $({\sigmab},\Lambda)\in\cA_\kappa$ and  $B\subseteq A$ then $({\sigmab}|_B,\Lambda)\in\cA_\kappa$. We will prove the contrapositive, namely that 
	$({\sigmab}|_B,\Lambda)\not\in\cA_\kappa$
	implies 
	$({\sigmab},\Lambda)\not\in\cA_\kappa$.
	By (v) and the scaling property we have that for each $\ell>0$,
	\begin{align*}
		W^{B}_\Lambda(\ell)<\infty\quad\text{if and only if}\quad ({\sigmab}|_B,\Lambda)\in\cA_\kappa \;.
	\end{align*}
	Therefore the right-hand side of 
	\eqref{eq85} is infinite if
	$({\sigmab}|_B,\Lambda)\not\in\cA_\kappa$, so
	\eqref{eq85} gives that $({\sigmab},\Lambda)\not\in\cA_\kappa$. This concludes the proof of (iii).
	
	Finally, we prove (vii), i.e., if ${\sigmab}'$ and $B'\subseteq A$ are such that $\#B'\geq 2$ and $\sigma'_B=\sigma_B$ for all $B\subseteq A,B\neq B'$ then by choosing $\sigma'_{B'}$ sufficiently large we have $({\sigmab}',\Lambda)\not\in\cA_\kappa$. In the following we will not use the simplified notation for the weight functions from \eqref{eq57} since the dependence on $\sigmab$ is important. By (iii) it is sufficient to consider the case $B'=A$. The case $({\sigmab},\Lambda)\not\in\cA_\kappa$ is immediate since $W^{\sigmab',\kappa}_{\Lambda,\ell}$ is increasing in $\sigma'_B$, so suppose $({\sigmab},\Lambda)\in\cA_\kappa$. Since $W^{{\sigmab'},\kappa}_{\Lambda,1}$ increases as $\Lambda$ decreases, 
	for all $\ell\in(1/2,1)$, 
	\begin{align*}
	W^{{\sigmab'},\kappa}_{\Lambda,\ell}
	= \ell^{\frac{2}{\sqrt{\kappa}}\sum_i \alpha_i} W^{{\sigmab'},\kappa}_{\Lambda \ell^{2\wedge (8/\kappa)},1} > 2^{-\frac{2}{\sqrt{\kappa}}\sum_i \alpha_i} W^{{\sigmab'},\kappa}_{\Lambda,1} \;.
	\end{align*}
	Applying Lemma \ref{prop:chordal-points-together} and this estimate we get 
	\begin{align*}
		W^{{\sigmab}',\kappa}_{\Lambda,1}
		&>
		e^{\sigma'_A} \E_{(1/2,1/2)}\left( \sum_{t<\zeta\colon\Delta X_t\in(1/2,1)} 
		\prod_{s<t} W^{-,\kappa}_{\Lambda ,|\Delta X_s|} W^{{\sigmab}',\kappa}_{\Lambda,|\Delta X_t|}
		W^{-,\kappa}_{\Lambda, X_t} \right)\\
		&> 2^{-\frac{2}{\sqrt{\kappa}}\sum_i \alpha_i} W^{{\sigmab}',\kappa}_{\Lambda,1} e^{\sigma'_A} \E_{(1/2,1/2)}\left( \sum_{t<\zeta\colon\Delta X_t\in(1/2,1)} 
		\prod_{s<t} W^{-,\kappa}_{\Lambda ,|\Delta X_s|} 
		W^{-,\kappa}_{\Lambda,X_t} \right).
	\end{align*}
	It follows that $({\sigmab}',\Lambda)\not\in\cA_\kappa$ if 
	\begin{align*}
	2^{-\frac{2}{\sqrt{\kappa}}\sum_i \alpha_i} e^{\sigma'_A} \E_{(1/2,1/2)}\left( \sum_{t<\zeta\colon \Delta X_t\in(1/2,1)} 
	\prod_{s<t} W^{-,\kappa}_{\Lambda ,|\Delta X_s|} 
	W^{-,\kappa}_{\Lambda,X_t} \right)>1 \;,
	\end{align*}
	which we can obtain by choosing $\sigma'_A$ sufficiently large since the expectation on the left-hand side is positive and does not depend on $\sigma'_A$. This implies (vii).
\end{proof}

It remains to prove assertion (iv) of Theorem \ref{thm:part-fcn-finite}, which will be done in Lemma \ref{prop:part-fcn-finite2}. The idea of the proof is to split the $\gamma$-LQG disk with marked points into components containing just one marked point each (by decorating it with a weighted CLE) using Theorem \ref{thm:bootstrap} and then decoupling the individual disks at the cost of decreasing the cosmological constant $\Lambda$.

\begin{lemma}
	Assertion (iv) of Theorem \ref{thm:part-fcn-finite} holds, i.e., if $\sigma_B<\sigma_i$ whenever $i\in B\subseteq A$ and $\#B>1$ then $({\sigmab},\Lambda)\in\cA_\kappa$ for all $\Lambda>0$.
	\label{prop:part-fcn-finite2}
\end{lemma}
\begin{proof}
	The proof will proceed via induction on $\#B$. We will show that there are constants $\rho_{B,\Lambda}<\infty$ and $0<\Lambda_{B,\Lambda}<\Lambda$ such that for all $\ell>0$, we have
	\begin{align}
		W^B_{\Lambda}(\ell)\le \rho_{B,\Lambda} \sum_{i\in B} W^{\{i\}}_{\Lambda_{B,\Lambda}}(\ell)\quad\text{and}\quad W^B_{\Lambda}(\ell)\le \rho_{B,\Lambda} W^\emptyset_{\Lambda_{B,\Lambda}}(\ell).
		\label{eq39}
	\end{align}
	This implies the lemma since the right-hand sides are finite by assertion (i) of Theorem \ref{thm:part-fcn-finite}.
	
	We start by verifying the induction hypothesis $\#B=1$. The first identity of \eqref{eq39} is trivial in this case and the second one appears as the first statement of Corollary \ref{cor:single-point-properties}.
			
	For the induction step, we use the induction hypothesis to get the following for $\Lambda_{B,\Lambda} = \min\{\Lambda_{C,\Lambda}\colon \emptyset\neq C\subsetneq B\}$ and $V^B_{\Lambda,m}(\ell)$ just as $V^m_{\Lambda,\ell}$ in Theorem \ref{thm:bootstrap} but with $A$ replaced by $B\subseteq A$,
	\begin{align*}
		V^B_{\Lambda,0}(\ell) &\le e^{\sigma^*} \sum_{\emptyset\neq C\subsetneq B} \E_{(\ell/2,\ell/2)}\left( \sum_{t<\zeta} \prod_{s<t} W^\emptyset_\Lambda (|\Delta X_s|)W^{C}_\Lambda(|\Delta X_t|)W^{B\setminus C}_\Lambda(X_t) \right) \\
		&\le e^{\sigma^*}\sum_{\emptyset\neq C\subsetneq B} \rho_{B\setminus C,\Lambda}\E_{(\ell/2,\ell/2)}\left( \sum_{t<\zeta} \prod_{s<t} W^\emptyset_\Lambda (|\Delta X_s|)W^{C}_\Lambda(|\Delta X_t|)W^\emptyset_{\Lambda_{B\setminus C,\Lambda}}(X_t) \right) \\
		&\le e^{\sigma^*}\sum_{\emptyset\neq C\subsetneq B} \rho_{B\setminus C,\Lambda} e^{\sigma^*}W^C_{\Lambda_{B,\Lambda}}(\ell)
	\end{align*}
	where $\sigma^*=\max\{|\sigma_C|\colon C\subseteq A\}$.
	By the induction hypothesis \eqref{eq39}, we therefore get that there exists $\rho'_{B,\Lambda}<\infty$ such that
	\begin{align}
		V^B_{\Lambda,0}(\ell) \le \rho'_{B,\Lambda}  \sum_{i\in B} W^{\{i\}}_{\Lambda_{B,\Lambda}}(\ell).
		\label{eq54}
	\end{align}
	For $i\in B$ we now define
	\begin{align*}
		U_0^{B,i}(\ell) := W_0^{\{i\}}(\ell)\;,\qquad
		U_{n+1}^{B,i}(\ell) := \E_{(\ell/2,\ell/2)}\left(\sum_{t<\zeta} e^{\sigma_B 1(\Delta X_t>0)}U_n^{B,i}(|\Delta X_t|)\right).
	\end{align*}
	Since $\sigma_B<\sigma_i$ for $i\in B$ with $\#B>1$ we get
	\begin{align*}
		U^{B,i}_{1}(1)
		&< \E_{(1/2,1/2)}\left(\sum_{t<\zeta} e^{\sigma_i 1(\Delta X_t>0)} W^{\{i\}}_0(|\Delta X_t|) \right) = W^{\{i\}}_0(1)\;.
	\end{align*}
	We can therefore define $\alpha_{B,i} := W^{\{i\}}_0(1)/U^{B,i}_{1}(1) \in (0,1)$. By scaling, for $\ell>0$,
	\begin{align*}
		\E_{(\ell/2,\ell/2)}\left(\sum_{t<\zeta} e^{\sigma_B 1(\Delta X_t>0)}W^{\{i\}}_0(|\Delta X_t|)\right) 
		= 
		U^{B,i}_{1}(\ell)
		=
		\alpha_{B,i} W^{\{i\}}_0(\ell)\;.
	\end{align*}
	To proceed, we will argue two statements, each of which is established using induction on $n$. Note that the induction on $\#B$ above is not complete, and the induction step of that proof requires the further inductions on $n$ that we do right below.

	First we argue by induction on $n$ that 
	\begin{equation}
	U_{n}^{B,i}(\ell)= (\alpha_{B,i})^n W_0^{\{i\}}(\ell)
	\label{eq52}
    \end{equation}
	 for all $\ell>0$ and $n\ge 0$. The cases $n=0,1$ are immediate, and assuming the assertion is true for $n-1$ (with $n\geq 2$) we get that it also holds for $n$ since by first applying the definition of $U_{n}^{B,i}(\ell)$, then using the induction hypothesis, and then using \eqref{eq51}
	\begin{align}
	U_{n}^{B,i}(\ell) &= \E_{(\ell/2,\ell/2)}\left(\sum_{t<\zeta} e^{\sigma_B 1(\Delta X_t>0)} U^{B,i}_{n-1}(|\Delta X_t|) \right)\\
	&= (\alpha_{B,i})^{n-1}\E_{(\ell/2,\ell/2)}\left(\sum_{t<\zeta} e^{\sigma_i 1(\Delta X_t>0)} W_0^{\{i\}}(|\Delta X_t|) \right) = (\alpha_{B,i})^{n} W_0^{\{i\}}(\ell),
	\label{eq51}
	\end{align}
	which concludes the proof of \eqref{eq52} by induction.

	Next we argue via yet another induction on $n$ that 
	\begin{equation}
	V_{\Lambda,n}^B(\ell)\leq \rho'_{B,\Lambda}\sum_{i\in B} U_n^{B,i}(\ell)
	\label{eq53}
    \end{equation}
	 for $n\geq 0,\ell> 0$. The case $n=0$ is \eqref{eq54}. Assuming the induction hypothesis has been proved for $n$ we see that it also holds for $n+1$ since  
	\begin{align*}
		V^B_{\Lambda,n+1}(\ell) &= \E_{(\ell/2,\ell/2)}\left( \sum_{t<\zeta} \prod_{s\neq t} W^\emptyset_\Lambda(|\Delta X_s|)e^{\sigma_B 1(\Delta X_t>0)} V_{\Lambda,n}^B(|\Delta X_t|) \right)\\
		&\leq \sum_{i\in B} \rho'_{B,\Lambda} \E_{(\ell/2,\ell/2)}\left( \sum_{t<\zeta} e^{\sigma_B 1(\Delta X_t>0)} U_n^{B,i}(|\Delta X_t|) \right) = \rho_{B,\Lambda}'\sum_{i\in B} U_{n+1}^{B,i}(\ell),
	\end{align*}
	where we used $W^\emptyset_\Lambda(|\Delta X_s|)\leq 1$. This completes the proof of \eqref{eq53} by induction. 
	
	Combining the above, we get
	\begin{align*}
		W^{{\sigmab}|_B}_\Lambda(\ell) 
		= \sum_{n\geq 0} V_{\Lambda,n}^B(\ell)
		\le \rho'_{B,\Lambda} \sum_{i\in B} \sum_{n\ge 0} U_n^{B,i}(\ell) 
		=\rho'_{B,\Lambda} \sum_{i\in B} \sum_{n\ge 0} (\alpha_{B,i})^n W_0^{\{i\}}(\ell)< \infty
\end{align*}
and the induction step of the main induction (where we prove \eqref{eq39}) is done by setting $\rho_{B,\Lambda} = \rho'_{B,\Lambda} \sum_{i\in B} 1/(1-\alpha_{B,i})$.
\end{proof}

\appendix

\section{A result on renewal processes}
\label{app:renewal}

This section establishes an asymptotic result on exponentials of renewal processes, namely Proposition \ref{prop:renewalresult} right below. While we think that this addresses a very natural question which has likely appeared earlier in other contexts, we were not able to locate similar statements in existing literature. For the statement below, recall the definition $\mathcal{L}f(\lambda) = \int_0^\infty e^{-\lambda t}f(t)\,dt$ of the Laplace transform of a function $f:\R_+\to\R$.

\begin{prop}
	\label{prop:renewalresult}
	Let $(X_i)$ be an i.i.d. sequence of random variables in $(0,\infty)$ with density $f$ and cumulative distribution function $F$ and suppose that
	\begin{align*}
		\sup_{t\ge 0} \frac{f(t)}{1-F(t)}<\infty \quad\text{and}\quad \inf_{t\ge 0} \frac{F(t+c) - F(t)}{1-F(t)}>0\quad\text{for some $c>0$}\;.
	\end{align*}
	Suppose that $\mathcal{L}f(\rho_0)$ exists for some $\rho_0\in\R$. Moreover, assume $\rho>\rho_0$ with $\rho\neq 0$ is such that $\mathcal{L}f(\rho+iy)\neq \mathcal{L}f(\rho)$ for all $y\in \R\setminus\{0\}$. Let $\theta = -\log \mathcal{L}f(\rho)\in  \R$ and define $N$ to be the renewal process associated to $X$, i.e.\ $N_t = \sup\{n\ge 0\colon X_1+\cdots+X_n\le t\}$ for $t\ge 0$. Then
	\begin{align*}
		\beta(t) :=\E\left( e^{\theta N_t - \rho t}\right) \to \frac{e^{-\theta}(1-e^{-\theta})/\rho}{-(\mathcal{L}f)'(\rho)} \quad\text{as $t\to\infty$}
	\end{align*}
	and $\beta(t)<\infty$ for all $t\ge 0$.
\end{prop}

The proof of the proposition uses as an essential ingredient a Tauberian result which relates the asymptotic behavior of a function to the asymptotic behavior of its Laplace transform. A proof of the theorem below is given in \cite[Chapter III, Theorem 7.1]{tauberian-theory}.

\begin{thm}[\cite{tauberian-theory}]
	\label{thm:tauberian}
	Let $\alpha\colon [0,\infty)\to \R$ be measurable and bounded so that the Laplace transform $\mathcal{L}\alpha$ exists on $\{z\in \D\colon \Re z > 0\}$ and suppose that $\mathcal{L}\alpha$ has an analytic extension to a neighborhood of each point on the imaginary axis which we denote by $\mathcal{L}\alpha$ too. Then
	\begin{align*}
		\int_0^t \alpha(s)\,ds \to \mathcal{L}\alpha(0)\quad\text{as}\quad t\to \infty\;.
	\end{align*}
\end{thm}

\begin{proof}[Proof of Proposition \ref{prop:renewalresult}]
	The idea of the proof will be to observe the process $N$ at a random independent exponentially distributed time $T_\lambda\sim \text{Exp}(\lambda)$ for $\lambda>0$ and deduce the result using the Tauberian theorem stated in Theorem \ref{thm:tauberian}. We begin by observing that for $\lambda+\rho\neq 0$ we have
	\begin{align}
		\label{eq:momentasymptotic}
		\begin{split}
		\E\left( e^{\theta N_{T_\lambda} - \rho T_\lambda}\right)
		&= \sum_{n\ge 0} e^{\theta n}\, \E(e^{-\rho T_\lambda}; T_\lambda \in [X_1+\cdots + X_n, X_1+\cdots + X_{n+1})) \\
		&= \frac{\lambda}{\lambda + \rho} \sum_{n\ge 0} e^{\theta n}\,\E\left( e^{-(\lambda + \rho)(X_1 + \cdots + X_n)} - e^{-(\lambda + \rho)(X_1 + \cdots + X_{n+1})} \right)\\
		&= \frac{\lambda}{\lambda + \rho}\sum_{n\ge 0} (1-\mathcal{L}f(\lambda+\rho)) (e^\theta \mathcal{L}f(\lambda+\rho))^n \\
		&= \frac{\lambda}{\lambda + \rho}\,\frac{1-\mathcal{L}f(\lambda+\rho)}{1-e^\theta \mathcal{L}f(\lambda+\rho)}
		\to \frac{e^{-\theta}(1-e^{-\theta})/\rho}{-(\mathcal{L}f)'(\rho)}\quad\text{as $\lambda\to 0$}
		\end{split}
	\end{align}
	noting that $\mathcal{L}f(\lambda+\rho) < \mathcal{L}f(\rho)=e^{-\theta}$. If $\lambda=-\rho$ we get by an analogous computation that
	\begin{align*}
		\E\left( e^{\theta N_{T_\lambda} - \rho T_\lambda}\right) = \frac{-\lambda (\mathcal{L}f)'(0)}{1-e^\theta}\;.
	\end{align*}

	Our goal is to show that the limit in \eqref{eq:momentasymptotic} (as $\lambda\to 0$)
	is the same as the limit of $\beta(t)$ as $t\to\infty$; this is intuitive since $T_\lambda\to\infty$ in distribution as $\lambda\to 0$. In the last step, we will apply Theorem \ref{thm:tauberian} with $\alpha=\beta'$ to conclude the proof.

	The technical bulk of the proof will be establishing the boundedness of $\beta$ and its derivative $\beta'$. Fix $t\ge 0$. By the memoryless property of $T_\lambda$, we can bound
    \begin{align}
        \label{eq:firstrenewalbound}
        \begin{split}
            \E\left( e^{\theta N_{T_\lambda} - \rho T_\lambda}\right) &\ge \E\left( e^{\theta N_{T_\lambda} - \rho T_\lambda}; T_\lambda \ge t+c \right) \\
            &= \E\left( e^{\theta N_{T_\lambda+(t+c)} - \rho (T_\lambda+(t+c))}\right)\,\P(T_\lambda \ge t+c)\;.
        \end{split}
    \end{align}
    To simplify notation, let us write $J_n=X_1+\cdots+X_n$ for the jump times. The key observation which follows directly from the definition is that the conditional law of $J_{N_t+1}$ given $(N_t,J_{N_t})$ has density $f(\cdot - J_{N_t})1_{(t,\infty)}/(1-F(t-J_{N_t}))$, so by the renewal property of $N$ at time $J_{N_t+1}$, for any $\tau\ge 0$,
    \begin{align*}
      \E\left(e^{\theta N_{\tau + (t+c)}}\right) &\ge \E\left( \frac{e^{\theta N_t}}{1-F(t-J_{N_t})}\, \int_{t}^{\tau + t+c} f(s-J_{N_t})\, e^\theta\, \E(e^{\theta N_{\tau + t + c - s}})\,ds \right)\\
      &\ge \E\left( \frac{e^{\theta N_t}}{1-F(t-J_{N_t})}\, \int_{t}^{t+c} f(s-J_{N_t})\, e^\theta\, \E(e^{\theta N_{\tau + t + c - s}})\,ds \right)\\
      &\ge e^\theta\,\E(e^{\theta N_\tau})\,\E\left( e^{\theta N_t}\,\frac{F(t+c-J_{N_t})-F(t-J_{N_t})}{1-F(t-J_{N_t})}\right)\\
      &\ge C\cdot \E(e^{\theta N_\tau})\,\E( e^{\theta N_t})
    \end{align*}
	using the lower bound from the assumptions where $C$ is a constant depending only on $\theta$, $f$ and $c$. Therefore \eqref{eq:firstrenewalbound} and the above imply that
	\begin{align*}
		\E\left( e^{\theta N_{T_\lambda} - \rho T_\lambda}\right) &\ge \E\left( \left.\E\left(e^{\theta N_{\tau + (t+c)}}\right)\right|_{\tau=T_\lambda}\, e^{-\rho(T_\lambda + (t+c))}\right)\,\P(T_\lambda \ge t+c) \\
		&\ge Ce^{-\rho c}\, \E\left( e^{\theta N_{T_\lambda} - \rho T_\lambda}\right)\,\E\left( e^{\theta N_t - \rho t}\right)\,\P(T_\lambda \ge t+c) \;.
	\end{align*}
	Since $\P(T_\lambda \ge t+c)=e^{-1}$ for $\lambda = 1/(t+c)$ and the left-hand side of the above estimate is finite, we can cancel the two terms to obtain the required uniform bound on $\beta$. To establish the existence and find a uniform bound for $\beta'$, let us first write
	\begin{align*}
		\E\left( e^{\theta N_t} \right) = \sum_{n\ge 0} e^{\theta n}\,\P(t\in [J_n, J_{n+1})) = \sum_{n\ge 1} e^{\theta n} \int_t^\infty (1-F(t-s))\,f^{*n}(s)\,ds+ 1 - F(t)
	\end{align*}
	where $f^{*n}$ is the density of $J_n$. Note that each of the summands is differentiable, indeed a straightforward computation shows that
	\begin{align*}
		\frac{d}{dt}\,\int_t^\infty (1-F(t-s))\,f^{*n}(s)\,ds = f^{*(n+1)}(t)-f^{*n}(t)\;.
	\end{align*}
	Moreover, by assumption there is $C'<\infty$ such that
	\begin{align*}
		\sum_{n\ge 0} e^{\theta n}\,f^{*(n+1)}(t) &= \sum_{n\ge 0} e^{\theta n}\,\E\left( f(t-J_n);J_n\ge t\right)
		= \sum_{n\ge 0} e^{\theta n}\, \E\left( \frac{f(t-J_n)}{1-F(t-J_n)};t\in [J_n,J_{n+1})\right)\\
		&\le C'\,\sum_{n\ge 0} e^{\theta n}\, \P(t\in [J_n,J_{n+1}))
		= C'\,\E(e^{\theta N_t})\;.
	\end{align*}
	From this, it follows directly that $\beta$ is differentiable and $|\beta'|$ is bounded.

	We therefore infer that for any $\lambda>0$,
	\begin{align*}
		\mathcal{L}(\beta')(\lambda)&=\int_0^\infty e^{-\lambda t} \beta'(t)\,dt = \int_0^\infty \lambda e^{-\lambda t}\,\beta(t)\,dt-1 = \E\left( e^{\theta N_{T_\lambda} -\rho T_\lambda}\right)-1 \\
		&=  -1+\frac{\lambda}{1-e^\theta \mathcal{L}f(\lambda+\rho)} \cdot \begin{cases}
			\frac{1-\mathcal{L}f(\lambda+\rho)}{\lambda+\rho} &\colon \lambda+\rho\neq 0\;,\\
			-(\mathcal{L}f)'(0) &\colon \lambda+\rho = 0\;.
		\end{cases}
	\end{align*}
	We will show that the right-hand side extends to an analytic function on a neighborhood of $\{\lambda\in \C\colon \Re \lambda \ge 0\}$ by defining the extension at $0$ to be
	\begin{align*}
		\mathcal{L}(\beta')(0) = -1 + \frac{e^{-\theta}(1-e^{-\theta})/\rho}{-(\mathcal{L}f)'(\rho)}\;.
	\end{align*}
	Our strategy is to show that the right-hand side of the identity above defines an analytic function on a set of the form $U\setminus \{0\}$ where $U$ is an open neighborhood of $\{\lambda\in \C\colon \Re \lambda \ge 0\}$. This will yield the claim since by \eqref{eq:momentasymptotic} the function has a removable singularity at $0$.

	To this end, note that since $\rho>\rho_0$, the function $\mathcal{L}f(\cdot+\rho)$ is an analytic function on a neighborhood of $\{\lambda\in \C\colon \Re \lambda \ge 0\}$. In order to conclude, it is now sufficient to argue that $e^\theta \mathcal{L}f(\lambda+\rho)\neq 1$ for all $\lambda\in U$ for some set $U$ as above. This property is satisfied since for all $\lambda\in \C$ with $\Re \lambda>0$ we have $|\mathcal{L}f(\lambda+\rho)|<\mathcal{L}f(\rho)=e^{-\theta}$ as $f$ is a probability density function from $(0,\infty)$ to $[0,\infty)$ and moreover we may use the assumption that $\mathcal{L}f(\rho+\lambda)\neq \mathcal{L}f(\rho)$ for each $\lambda\neq 0$ with $\Re \lambda =0$.

	The result now follows since
	\begin{align*}
		\beta(t)=\int_0^t \beta'(s)\,ds + 1 \to \mathcal{L}(\beta')(0)+1 = \frac{e^{-\theta}(1-e^{-\theta})/\rho}{-(\mathcal{L}f)'(\rho)} \quad\text{as $t\to\infty$}
	\end{align*}
	where the convergence part follows from Theorem \ref{thm:tauberian} applied with $\alpha=\beta'$ and where $ \mathcal{L}(\beta')(0)$ denotes the analytic continuation of $ \mathcal{L}(\beta')$ to $0$.
\end{proof}

\section{Integral computations}

This paper, and in particular the study of Lévy processes here, relies on the explicit computations of various complicated integrals. These computations could not be located by us in the literature and have therefore been included in this appendix. Lemma \ref{lem:residue-bessel-computation} is the key computation that will be used in Section \ref{sec:levy-exc} to explicitly compute functionals of Lévy excursions and Lemmas \ref{lem:simple-integrals-comp} and \ref{lem:nonsimple-integrals-comp} are used in Sections \ref{subsec:simple-fragmentation} and \ref{subsec:nonsimple-fragmentations}, respectively, to evaluate functionals associated to the spatially inhomogeneous Markov processes discussed in these sections. 

The integrals are rather involved and it is not clear if there exist explicit closed-form antiderivatives of the integrands; the computations rely on the fact that we consider indefinite integrals. Several computations involve the integral of sums of terms which are not integrable individually. A key trick we will use multiple times is to evaluate such integrals via integration by parts. In Lemma \ref{lem:residue-bessel-computation} we further rely on Cauchy's residue theorem.

We start with the computation of certain Bessel integrals. Recall the definition of the modified Bessel functions of the second kind:
\begin{align}
	\label{eq:mod-bessel-def}
	K_\nu(x)=\int_0^\infty e^{-x \cosh(t)}\,\cosh(\nu t)\,dt = \frac{1}{2(x/2)^\nu} \int_0^\infty e^{-(x/2)^2 y-1/y}\,\frac{dy}{y^{1+\nu}}\;.
\end{align}
It is easy to see that $2(x/2)^\nu K_\nu(x)\to \Gamma(\nu)$ as $x\downarrow 0$. Also, the derivative of $(x/2)^\nu K_\nu(x)$ (a smooth function) with respect to $x$ is $-(x/2)^\nu K_{\nu-1}(x)$.

\begin{lemma}
	\label{lem:residue-bessel-computation}
	Suppose that $\alpha\in (0,1)$, $\beta \in (0,2)\setminus \{1\}$ and $\rho>-1$. We write $\rho=\cos(\theta)$ for $\theta \in [0,\pi)$ when $\rho\in (-1,1]$ and $\rho=\cosh(\theta)$ for $\theta>0$ when $\rho>1$. Then
	\begin{align*}
		\int_0^\infty \frac{\cosh(\alpha t)\,dt}{\cosh(t)+\rho} &= \frac{\pi}{\sin(\pi\alpha)}\cdot \begin{cases}
			\sin(\alpha\theta)/\sin(\theta) &\colon \rho\in (-1,1)\;, \\
			\alpha  &\colon \rho = 1\;, \\
			\sinh(\alpha\theta)/\sinh(\theta) &\colon \rho>1\;,
		\end{cases} \\
		\int_0^\infty \frac{\cosh(\beta t)\,dt}{(\cosh(t)+\rho)^2} &= \frac{\pi}{\sin(\pi\beta)} \begin{cases}
			(\beta\cos(\beta\theta)-\cot(\theta)\sin(\beta\theta))/\sin^2(\theta) &\colon \rho \in (-1,1) \;,\\
			\beta(1-\beta^2)/3 &\colon \rho = 1\;,\\
			(\coth(\theta)\sinh(\beta\theta)-\beta\cosh(\beta\theta))/\sinh^2(\theta) &\colon \rho>1\;.
		\end{cases}
	\end{align*}
	Moreover for $\nu\in (1,2)$, $\nu'\in (0,1)$ we have
	\begin{align*}
		\int_0^\infty \frac{dh}{h^{1+\nu}}\left(e^{-\rho h}\frac{2(h/2)^\nu}{\Gamma(\nu)}\,K_\nu(h) - 1 + \rho h\right) &= \frac{2^{1-\nu}\pi/\sin(-\pi\nu)}{\Gamma(1+\nu)}\cdot \begin{cases}
			\cos(\nu\theta) &\colon \rho \in (-1,1]\;, \\
			\cosh(\nu\theta) &\colon \rho>1\;,
		\end{cases} \\
		\int_0^\infty \frac{dh}{h^{1+\nu'}}\left(e^{-\rho h}\frac{2(h/2)^{\nu'}}{\Gamma(\nu')}\,K_{\nu'}(h) - 1 \right) &= \frac{-2^{1-\nu'}\pi/\sin(\pi\nu')}{\Gamma(1+\nu')}\cdot \begin{cases}
			\cos(\nu'\theta) &\colon \rho \in (-1,1]\;, \\
			\cosh(\nu'\theta) &\colon \rho>1\;,
		\end{cases} \\
		\int_0^\infty dh\, e^{-\rho h} K_{\nu'}(h) &= \int_0^\infty \frac{\cosh(\nu' t)\,dt}{\cosh(t)+\rho}\;.
	\end{align*}
\end{lemma}

\begin{proof}
	The first two integrals are applications of Cauchy's residue theorem. In particular, we carefully choose a rectangular contour such that the asymptotic contribution of the left and right segments vanish and the limit of the top and bottom segments can be related to the integral in the lemma. We only explain the first computation, the second one is analogous. Let us consider the contour integral of $u(z)= e^{\alpha z}(\cosh(z)+\rho)^{-1}$ along the boundary of $(-R,R)+i(0,2\pi)$ (oriented counterclockwise) and send $R\to \infty$ to obtain the identity
	\begin{align*}
		2\pi i \sum_{z\in \R+i(0,2\pi)} \text{Res}(u,z) &= \int_\R \frac{e^{\alpha t}}{\cosh(t)+\rho}\,dt - \int_\R \frac{e^{\alpha (t+2\pi i)}}{\cosh(t+2\pi i)+\rho}\,dt \;.\\
		&= 2(1-e^{2\pi i\alpha})\int_0^\infty \frac{\cosh(\alpha t)}{\cosh(t)+\rho}\,dt\;.
	\end{align*}
	Computing the residues yields the result. Let us consider the integrals involving Bessel functions. For the third integral in the statement, we write $h^{-\nu-1}$ as the second derivative of $h^{1-\nu}/(\nu(\nu-1))$ and integrate by parts twice to get
	\begin{align*}
		&\int_0^\infty \frac{dh}{h^{1+\nu}}\left(e^{-\rho h}\frac{2(h/2)^\nu}{\Gamma(\nu)}\,K_\nu(h) - 1 + \rho h\right) \\
		&\quad = \frac{2^{1-\nu}}{\nu(\nu -1)\Gamma(\nu)}\int_0^\infty e^{-\rho h}\left( (1+\rho^2)h K_{\nu-2}(h) + (2\rho h + (2\rho^2(\nu-1)-1))K_{\nu-1}(h) \right) dh \\
		&\quad = \frac{2^{1-\nu}}{\nu(\nu -1)\Gamma(\nu)}\int_0^\infty \frac{(1+\rho^2)\cosh((2-\nu)t)+2\rho\cosh((\nu-1)t)}{(\cosh(t)+\rho)^2}\,dt \\
		&\qquad\quad+ \frac{2^{1-\nu}(2\rho^2(\nu-1)-1)}{\nu(\nu -1)\Gamma(\nu)} \int_0^\infty\frac{\cosh((\nu-1)t)}{\cosh(t)+\rho}\,dt
	\end{align*}
	by the definition \eqref{eq:mod-bessel-def}. Using the first part of the lemma and applying trigonometric identities yields the result. The fourth integral can be computed similarly, except that only one integration by parts is needed. For the final integral, equation \eqref{eq:mod-bessel-def} can be used directly.
\end{proof}

In the proof of Lemmas \ref{lem:simple-integrals-comp} and \ref{lem:nonsimple-integrals-comp} we will need the Beta function and the incomplete Beta function, which are defined by the following respective formulas
	\begin{gather*}
		B(a,b)=\int_0^1 t^{a-1}(1-t)^{b-1}\,dt = \int_0^\infty t^{a-1}(1+t)^{-a-b}\,dt = \frac{\Gamma(a)\Gamma(b)}{\Gamma(a+b)}\text{\,\,for\,\,}a,b>0\;,\\
		B_x(a,b) = \int_0^x t^{a-1} (1-t)^{b-1}\,dt\text{\,\,for\,\,}a>0,b\in\R,x\in[0,1)\;.
	\end{gather*}
The next lemma recalls some basic properties of these functions.
\begin{lemma}
	\label{lem:beta-results}
	For $a,b>0$ and $x\in(0,1)$ we have $B_x(a,b)+B_{1-x}(b,a)=B(a,b)$ and, for $a>0,b\in\R,x\in[0,1)$,
	\begin{align*}
		B_x(a,b)=\frac{a+b}{b}B_x(a,b+1) - \frac{x^a(1-x)^b}{b}.
	\end{align*}
	Moreover, $\Gamma(a)\Gamma(1-a)=\pi/\sin(\pi a)$ for $a\notin \Z$.
\end{lemma}

\begin{proof}
	The first claim is trivial. The second one follows since it holds at $x=0$ and the derivatives of both sides agree for all $x>0$. The third claim is Euler's reflection formula.
\end{proof}

\begin{lemma}
    \label{lem:simple-integrals-comp}
  Let $\kappa\in (8/3,4)$. For $\theta<4/\kappa$ and $u\in (0,1)$ define 
    \begin{align*}
        &A^\theta_u=-\cos(4\pi/\kappa)\int_0^\infty \frac{dh}{h^{1+4/\kappa}}\left( (1+h)^\theta - 1 - \theta h \right) + \int_0^\infty \frac{dh}{2h^{1+4/\kappa}}\left( (1-h)^\theta 1(h<u) - 1 + \theta h\right) \\
	&\qquad\qquad + \int_0^\infty \frac{dh}{2h^{1+4/\kappa}}\left( (1-h)^\theta 1(h<1-u) - 1 + \theta h\right).
    \end{align*}
    Then $A^{-1-4/\kappa}_u = 0$ for all $u\in (0,1)$ and, for all $\theta\in (-1,4/\kappa)$ and $u\in (0,1)$,
    \begin{align*}
        A^\theta_u + \int_0^u \frac{h^\theta\,dh}{2(1-h)^{1+4/\kappa}} + \int_0^{1-u} \frac{h^\theta\,dh}{2(1-h)^{1+4/\kappa}} = -\cos(\pi\theta) B(\theta+1,4/\kappa-\theta)\;.
    \end{align*}
\end{lemma}

\begin{proof}
    We write
    \begin{align*}
        \frac{1}{h^{1+4/\kappa}} = \frac{d^2}{dh^2}\left( \frac{h^{1-4/\kappa}}{4/\kappa \cdot (4/\kappa -1)} \right)\quad\text{and}\quad \frac{1}{h^{4/\kappa}} = \frac{d}{dh} \left( \frac{h^{1-4/\kappa}}{1-4/\kappa} \right)\;.
    \end{align*}
    Applying integration by parts multiple times thus allows us to establish the following integral identities. For $u\in (0,1)$ and $\theta <4/\kappa$ we have
    \begin{align*}
        &\int_0^\infty \frac{dh}{h^{1+4/\kappa}}\left( (1+h)^\theta - 1 - \theta h \right) = \frac{\theta(\theta-1)}{4/\kappa\cdot (4/\kappa - 1)}\, B(2-4/\kappa,4/\kappa-\theta)\;,\\
        &\int_0^\infty \frac{dh}{h^{1+4/\kappa}}\left( (1-h)^\theta 1(h<u) - 1 + \theta h\right)\\
        &\qquad =  \frac{\theta(\theta-1)}{4/\kappa\cdot(4/\kappa-1)}\,B_u(2-4/\kappa,\theta-1)+\frac{\theta u^{1-4/\kappa}(1-u)^{\theta-1}}{4/\kappa\cdot (4/\kappa-1)} -\frac{u^{-4/\kappa}(1-u)^\theta}{4/\kappa}\;.
    \end{align*}
    For the $\theta = -1-4/\kappa$ case, applying the indented equation in Lemma \ref{lem:beta-results} four times allows us to express $B_u(2-4/\kappa,-2-4/\kappa)$ in terms of an expression involving $B_u(2-4/\kappa,2-4/\kappa)$. One then verifies that indeed $A^\theta_u=0$ for $\theta=-1-4/\kappa$. 
    
    Suppose now that $\theta\in (-1,4/\kappa)$. Then
    \begin{align*}
        &\int_0^\infty \frac{dh}{h^{1+4/\kappa}}\left( (1-h)^\theta 1(h<u) - 1 + \theta h\right) + \int_0^{1-u} \frac{h^\theta\,dh}{(1-h)^{1+4/\kappa}}\\
        &\quad =\int_0^\infty \frac{dh}{h^{1+4/\kappa}}\left( (1-h)^\theta 1(h<1-u) - 1 + \theta h\right) + \int_0^{u} \frac{h^\theta\,dh}{(1-h)^{1+4/\kappa}}\\
        &\quad = \int_0^\infty \frac{dh}{h^{1+4/\kappa}}\left( (1-h)^\theta 1(h<1) - 1 + \theta h\right) \\
        &\quad = \frac{(1-4/\kappa+\theta)(2-4/\kappa+\theta)}{4/\kappa\cdot(4/\kappa-1)}\, B(2-4/\kappa,\theta+1)\;,
    \end{align*}
where the second equality is obtained by writing it as the $u\uparrow 1$ limit in the second equality of the second display above, and then applying the indented equation of Lemma \ref{lem:beta-results} twice to determine the limit.
Replacing $u$ by $1-u$ we also get
 \begin{align*}
	&\int_0^\infty \frac{dh}{h^{1+4/\kappa}}\left( (1-h)^\theta 1(h<1-u) - 1 + \theta h\right) + \int_0^{u} \frac{h^\theta\,dh}{(1-h)^{1+4/\kappa}}\\ 
	&\qquad= \frac{(1-4/\kappa+\theta)(2-4/\kappa+\theta)}{4/\kappa\cdot(4/\kappa-1)}\, B(2-4/\kappa,\theta+1)\;,
\end{align*}
  Combining the previous displays,
    \begin{align*}
        &A^\theta_u + \int_0^u \frac{h^\theta\,dh}{2(1-h)^{1+4/\kappa}} + \int_0^{1-u} \frac{h^\theta\,dh}{2(1-h)^{1+4/\kappa}} \\
        &\quad = -\cos(4\pi/\kappa)\frac{\theta(\theta-1)}{4/\kappa\cdot (4/\kappa - 1)}\, B(2-4/\kappa,4/\kappa-\theta) \\
        &\quad\qquad + \frac{(1-4/\kappa+\theta)(2-4/\kappa+\theta)}{4/\kappa\cdot(4/\kappa-1)}\, B(2-4/\kappa,\theta+1)\;.\\
        &\quad = -\cos(\pi\theta) B(\theta+1,4/\kappa-\theta)
    \end{align*}
    where the final equality follows by writing the Beta functions in terms of Gamma functions and using Euler's reflection formula.
\end{proof}

We now establish a similar result which will appear in the context of Lévy processes without compensation. Note that the actual formula is formally the same as in the previous lemma.
\begin{lemma}
    \label{lem:nonsimple-integrals-comp}
    Suppose that $\kappa'\in (4,8)$ and define $A^\theta_u$ to be
    \begin{align*}
        &-\cos(4\pi/\kappa')\int_0^\infty \frac{dh}{h^{1+4/\kappa'}}\left( (1+h)^\theta - 1 \right) + \int_0^\infty \frac{dh}{2h^{1+4/\kappa'}}\left( (1-h)^\theta 1(h<u) - 1\right) \\
        &\qquad + \int_0^\infty \frac{dh}{2h^{1+4/\kappa'}}\left( (1-h)^\theta 1(h<1-u) - 1\right)
    \end{align*}
    whenever $\theta<4/\kappa'$ and $u\in (0,1)$. Then $A^{-1-4/\kappa'}_u=0$ for all $u\in (0,1)$ and
    \begin{align*}
        A^\theta_u + \int_0^u \frac{h^\theta\,dh}{2(1-h)^{1+4/\kappa'}} + \int_0^{1-u} \frac{h^\theta\,dh}{2(1-h)^{1+4/\kappa'}} = -\cos(\pi\theta) B(\theta+1,4/\kappa'-\theta)\;.
    \end{align*}
    for all $\theta\in (-1,4/\kappa')$ and $u\in (0,1)$.
\end{lemma}

\begin{proof}
    The proof of this result is completely analogous to the proof of Lemma \ref{lem:simple-integrals-comp}. We observe that by integration by parts,
    \begin{align*}
        \int_0^\infty \frac{dh}{h^{1+4/\kappa'}} \left( (1+h)^\theta - 1\right) &= \frac{\theta}{4/\kappa'}\,B(1-4/\kappa', 4/\kappa' - \theta)\;, \\
        \int_0^\infty \frac{dh}{h^{1+4/\kappa'}} \left( (1-h)^\theta 1(h<u) -1 \right) &= -\frac{\theta}{4/\kappa'}\,B_u(1-4/\kappa',\theta) - \frac{u^{-4/\kappa'}(1-u)^\theta}{4/\kappa'}\;.
    \end{align*}
    The result that $A^{-1-4/\kappa'}_u=0$ for $u\in (0,1)$ then follows from a lengthy calculation -- see Lemma \ref{lem:simple-integrals-comp} for some more detail on this. Suppose now that $\theta \in (-1,4/\kappa')$. Then
    \begin{align*}
        &\int_0^\infty \frac{dh}{h^{1+4/\kappa'}}\left( (1-h)^\theta 1(h<u) - 1\right) + \int_0^{1-u} \frac{h^\theta\,dh}{(1-h)^{1+4/\kappa}}\\
        &\quad = \int_0^\infty \frac{dh}{h^{1+4/\kappa'}}\left( (1-h)^\theta 1(h<1) - 1 \right) \\
        &\quad = - \frac{1-4/\kappa'+\theta}{4/\kappa'}\,B(1-4/\kappa',\theta+1)\;.
    \end{align*}
    The results now follows by putting everything together and simplifying the expressions, making use of Euler's reflection formula.
\end{proof}

\bibliographystyle{hmralphaabbrv}
\bibliography{lqg-cle-weighted}

\end{document}

%% file: img/cle-lqg-msw.pdf_tex
%% Creator: Inkscape 1.0.1 (3bc2e813f5, 2020-09-07), www.inkscape.org
%% Accompanies image file 'cle-lqg-msw.pdf' (pdf, eps, ps)
\begingroup%
  \makeatletter%
  \providecommand\color[2][]{%
    \errmessage{(Inkscape) Color is used for the text in Inkscape, but the package 'color.sty' is not loaded}%
    \renewcommand\color[2][]{}%
  }%
  \providecommand\transparent[1]{%
    \errmessage{(Inkscape) Transparency is used (non-zero) for the text in Inkscape, but the package 'transparent.sty' is not loaded}%
    \renewcommand\transparent[1]{}%
  }%
  \providecommand\rotatebox[2]{#2}%
  \newcommand*\fsize{\dimexpr\f@size pt\relax}%
  \newcommand*\lineheight[1]{\fontsize{\fsize}{#1\fsize}\selectfont}%
  \ifx\svgwidth\undefined%
    \setlength{\unitlength}{369.59981772bp}%
    \ifx\svgscale\undefined%
      \relax%
    \else%
      \setlength{\unitlength}{\unitlength * \real{\svgscale}}%
    \fi%
  \else%
    \setlength{\unitlength}{\svgwidth}%
  \fi%
  \global\let\svgwidth\undefined%
  \global\let\svgscale\undefined%
  \makeatother%
  \begin{picture}(1,0.31207227)%
    \lineheight{1}%
    \setlength\tabcolsep{0pt}%
    \put(0,0){\includegraphics[width=\unitlength,page=1]{cle-lqg-msw.pdf}}%
  \end{picture}%
\endgroup%

%% file: img/cpi-lqg-intro.pdf_tex
%% Creator: Inkscape 1.1.2 (0a00cf5339, 2022-02-04, custom), www.inkscape.org
%% Accompanies image file 'cpi-lqg-intro.pdf' (pdf, eps, ps)
\begingroup%
  \makeatletter%
  \providecommand\color[2][]{%
    \errmessage{(Inkscape) Color is used for the text in Inkscape, but the package 'color.sty' is not loaded}%
    \renewcommand\color[2][]{}%
  }%
  \providecommand\transparent[1]{%
    \errmessage{(Inkscape) Transparency is used (non-zero) for the text in Inkscape, but the package 'transparent.sty' is not loaded}%
    \renewcommand\transparent[1]{}%
  }%
  \providecommand\rotatebox[2]{#2}%
  \newcommand*\fsize{\dimexpr\f@size pt\relax}%
  \newcommand*\lineheight[1]{\fontsize{\fsize}{#1\fsize}\selectfont}%
  \ifx\svgwidth\undefined%
    \setlength{\unitlength}{418.6945305bp}%
    \ifx\svgscale\undefined%
      \relax%
    \else%
      \setlength{\unitlength}{\unitlength * \real{\svgscale}}%
    \fi%
  \else%
    \setlength{\unitlength}{\svgwidth}%
  \fi%
  \global\let\svgwidth\undefined%
  \global\let\svgscale\undefined%
  \makeatother%
  \begin{picture}(1,0.34897177)%
    \lineheight{1}%
    \setlength\tabcolsep{0pt}%
    \put(0,0){\includegraphics[width=\unitlength,page=1]{cpi-lqg-intro.pdf}}%
    \put(0.81672645,0.00286605){\color[rgb]{0,0,0}\makebox(0,0)[lt]{\lineheight{1.25}\smash{\begin{tabular}[t]{l}\textcolor{myviolet}{$w_0$}\end{tabular}}}}%
    \put(0.47952826,0.00464815){\color[rgb]{0,0,0}\makebox(0,0)[lt]{\lineheight{1.25}\smash{\begin{tabular}[t]{l}\textcolor{myviolet}{$w_0$}\end{tabular}}}}%
    \put(0.15049866,0.00638914){\color[rgb]{0,0,0}\makebox(0,0)[lt]{\lineheight{1.25}\smash{\begin{tabular}[t]{l}\textcolor{myviolet}{$w_0$}\end{tabular}}}}%
    \put(0,0){\includegraphics[width=\unitlength,page=2]{cpi-lqg-intro.pdf}}%
    \put(0.22964693,0.07045276){\color[rgb]{0,0,0}\makebox(0,0)[lt]{\lineheight{1.25}\smash{\begin{tabular}[t]{l}\textcolor{blue}{$\lambda$}\end{tabular}}}}%
    \put(0.54893224,0.06758016){\color[rgb]{0,0,0}\makebox(0,0)[lt]{\lineheight{1.25}\smash{\begin{tabular}[t]{l}\textcolor{blue}{$\lambda$}\end{tabular}}}}%
    \put(0.89219086,0.07076136){\color[rgb]{0,0,0}\makebox(0,0)[lt]{\lineheight{1.25}\smash{\begin{tabular}[t]{l}\textcolor{blue}{$\lambda$}\end{tabular}}}}%
    \put(0.40798128,0.14154841){\color[rgb]{0,0,0}\makebox(0,0)[lt]{\lineheight{1.25}\smash{\begin{tabular}[t]{l}\textcolor{myviolet}{$\lambda(t)$}\end{tabular}}}}%
    \put(0.85741859,0.14957183){\color[rgb]{0,0,0}\makebox(0,0)[lt]{\lineheight{1.25}\smash{\begin{tabular}[t]{l}\textcolor{myviolet}{$\lambda(t)$}\end{tabular}}}}%
    \put(0.10568645,0.18173996){\color[rgb]{0,0,0}\makebox(0,0)[lt]{\lineheight{1.25}\smash{\begin{tabular}[t]{l}\textcolor{myviolet}{$\lambda(t)$}\end{tabular}}}}%
    \put(0.0975038,0.25317122){\color[rgb]{0,0,0}\makebox(0,0)[lt]{\lineheight{1.25}\smash{\begin{tabular}[t]{l}$D_t$\end{tabular}}}}%
    \put(0.92903657,0.32302056){\color[rgb]{0,0,0}\makebox(0,0)[lt]{\lineheight{1.25}\smash{\begin{tabular}[t]{l}\textcolor{myviolet}{$w_\infty$}\end{tabular}}}}%
    \put(0.597474,0.32302056){\color[rgb]{0,0,0}\makebox(0,0)[lt]{\lineheight{1.25}\smash{\begin{tabular}[t]{l}\textcolor{myviolet}{$w_\infty$}\end{tabular}}}}%
    \put(0.26917816,0.32401991){\color[rgb]{0,0,0}\makebox(0,0)[lt]{\lineheight{1.25}\smash{\begin{tabular}[t]{l}\textcolor{myviolet}{$w_\infty$}\end{tabular}}}}%
    \put(0,0){\includegraphics[width=\unitlength,page=3]{cpi-lqg-intro.pdf}}%
  \end{picture}%
\endgroup%

%% file: img/on-model.pdf_tex
%% Creator: Inkscape 1.1.2 (0a00cf5339, 2022-02-04, custom), www.inkscape.org
%% Accompanies image file 'on-model.pdf' (pdf, eps, ps)
\begingroup%
  \makeatletter%
  \providecommand\color[2][]{%
    \errmessage{(Inkscape) Color is used for the text in Inkscape, but the package 'color.sty' is not loaded}%
    \renewcommand\color[2][]{}%
  }%
  \providecommand\transparent[1]{%
    \errmessage{(Inkscape) Transparency is used (non-zero) for the text in Inkscape, but the package 'transparent.sty' is not loaded}%
    \renewcommand\transparent[1]{}%
  }%
  \providecommand\rotatebox[2]{#2}%
  \newcommand*\fsize{\dimexpr\f@size pt\relax}%
  \newcommand*\lineheight[1]{\fontsize{\fsize}{#1\fsize}\selectfont}%
  \ifx\svgwidth\undefined%
    \setlength{\unitlength}{397.62151574bp}%
    \ifx\svgscale\undefined%
      \relax%
    \else%
      \setlength{\unitlength}{\unitlength * \real{\svgscale}}%
    \fi%
  \else%
    \setlength{\unitlength}{\svgwidth}%
  \fi%
  \global\let\svgwidth\undefined%
  \global\let\svgscale\undefined%
  \makeatother%
  \begin{picture}(1,0.73579679)%
    \lineheight{1}%
    \setlength\tabcolsep{0pt}%
    \put(0,0){\includegraphics[width=\unitlength,page=1]{on-model.pdf}}%
    \put(0.24813444,0.23843061){\color[rgb]{0,0,0}\makebox(0,0)[lt]{\lineheight{1.25}\smash{\begin{tabular}[t]{l}\textcolor{blue}{$\CLE_{\kappa}^\sigma$ on $\bar{\op{M}}_{\Lambda,\ell}^{\sigma,\kappa}$ (dense)}\end{tabular}}}}%
    \put(0.24813444,0.19316137){\color[rgb]{0,0,0}\makebox(0,0)[lt]{\lineheight{1.25}\smash{\begin{tabular}[t]{l}\textcolor{blue}{$\kappa=4/(1-\arccos(n/2)/\pi)\;,\ \gamma= 4/\sqrt{\kappa}$}\end{tabular}}}}%
    \put(0.41326107,0.13732747){\color[rgb]{0,0,0}\makebox(0,0)[lt]{\lineheight{1.25}\smash{\begin{tabular}[t]{l}\textcolor{red}{$\CLE_{\kappa}^\sigma$ on $\bar{\op{M}}_{\Lambda,\ell}^{\sigma,\kappa}$ (dilute)}\end{tabular}}}}%
    \put(0.41326107,0.09205829){\color[rgb]{0,0,0}\makebox(0,0)[lt]{\lineheight{1.25}\smash{\begin{tabular}[t]{l}\textcolor{red}{$\kappa=4/(1+\arccos(n/2)/\pi)\;,\ \gamma= \sqrt{\kappa}$}\end{tabular}}}}%
    \put(0.51112377,0.00301795){\color[rgb]{0,0,0}\makebox(0,0)[lt]{\lineheight{1.25}\smash{\begin{tabular}[t]{l}$g$\end{tabular}}}}%
    \put(0.1252311,0.24343258){\color[rgb]{0,0,0}\makebox(0,0)[lt]{\lineheight{1.25}\smash{\begin{tabular}[t]{l}$h$\end{tabular}}}}%
    \put(0,0){\includegraphics[width=\unitlength,page=2]{on-model.pdf}}%
  \end{picture}%
\endgroup%

%% file: img/loop-trunk.pdf_tex
%% Creator: Inkscape 1.0.1 (3bc2e813f5, 2020-09-07, custom), www.inkscape.org
%% Accompanies image file 'loop-trunk.pdf' (pdf, eps, ps)
\begingroup%
  \makeatletter%
  \providecommand\color[2][]{%
    \errmessage{(Inkscape) Color is used for the text in Inkscape, but the package 'color.sty' is not loaded}%
    \renewcommand\color[2][]{}%
  }%
  \providecommand\transparent[1]{%
    \errmessage{(Inkscape) Transparency is used (non-zero) for the text in Inkscape, but the package 'transparent.sty' is not loaded}%
    \renewcommand\transparent[1]{}%
  }%
  \providecommand\rotatebox[2]{#2}%
  \newcommand*\fsize{\dimexpr\f@size pt\relax}%
  \newcommand*\lineheight[1]{\fontsize{\fsize}{#1\fsize}\selectfont}%
  \ifx\svgwidth\undefined%
    \setlength{\unitlength}{269.22171653bp}%
    \ifx\svgscale\undefined%
      \relax%
    \else%
      \setlength{\unitlength}{\unitlength * \real{\svgscale}}%
    \fi%
  \else%
    \setlength{\unitlength}{\svgwidth}%
  \fi%
  \global\let\svgwidth\undefined%
  \global\let\svgscale\undefined%
  \makeatother%
  \begin{picture}(1,0.27933575)%
    \lineheight{1}%
    \setlength\tabcolsep{0pt}%
    \put(0,0){\includegraphics[width=\unitlength,page=1]{loop-trunk.pdf}}%
    \put(0.34622414,0.24847061){\color[rgb]{0,0,0}\makebox(0,0)[lt]{\lineheight{1.25}\smash{\begin{tabular}[t]{l}$\SLE_\kappa^\beta(\kappa-6)$\end{tabular}}}}%
    \put(0.34622414,0.20946928){\color[rgb]{0,0,0}\makebox(0,0)[lt]{\lineheight{1.25}\smash{\begin{tabular}[t]{l}\textcolor{red}{$\SLE_{\kappa'}(\rho',\kappa'-6-\rho')$}\end{tabular}}}}%
    \put(0.83083655,0.21106049){\color[rgb]{0,0,0}\makebox(0,0)[lt]{\lineheight{1.25}\smash{\begin{tabular}[t]{l}Full $\SLE_{\kappa'}^\beta(\kappa'-6)$\end{tabular}}}}%
    \put(0.83083655,0.17205916){\color[rgb]{0,0,0}\makebox(0,0)[lt]{\lineheight{1.25}\smash{\begin{tabular}[t]{l}\textcolor{red}{$\SLE_{\kappa}(\rho,\kappa-6-\rho)$}\end{tabular}}}}%
    \put(0,0){\includegraphics[width=\unitlength,page=2]{loop-trunk.pdf}}%
  \end{picture}%
\endgroup%

%% file: img/loop-tree-def.pdf_tex
%% Creator: Inkscape 1.1 (c4e8f9ed74, 2021-05-24), www.inkscape.org
%% Accompanies image file 'loop-tree-def.pdf' (pdf, eps, ps)
\begingroup%
  \makeatletter%
  \providecommand\color[2][]{%
    \errmessage{(Inkscape) Color is used for the text in Inkscape, but the package 'color.sty' is not loaded}%
    \renewcommand\color[2][]{}%
  }%
  \providecommand\transparent[1]{%
    \errmessage{(Inkscape) Transparency is used (non-zero) for the text in Inkscape, but the package 'transparent.sty' is not loaded}%
    \renewcommand\transparent[1]{}%
  }%
  \providecommand\rotatebox[2]{#2}%
  \newcommand*\fsize{\dimexpr\f@size pt\relax}%
  \newcommand*\lineheight[1]{\fontsize{\fsize}{#1\fsize}\selectfont}%
  \ifx\svgwidth\undefined%
    \setlength{\unitlength}{248.43375821bp}%
    \ifx\svgscale\undefined%
      \relax%
    \else%
      \setlength{\unitlength}{\unitlength * \real{\svgscale}}%
    \fi%
  \else%
    \setlength{\unitlength}{\svgwidth}%
  \fi%
  \global\let\svgwidth\undefined%
  \global\let\svgscale\undefined%
  \makeatother%
  \begin{picture}(1,0.28920079)%
    \lineheight{1}%
    \setlength\tabcolsep{0pt}%
    \put(0,0){\includegraphics[width=\unitlength,page=1]{loop-tree-def.pdf}}%
    \put(0.96816971,0.00295799){\color[rgb]{0,0,0}\makebox(0,0)[lt]{\lineheight{1.25}\smash{\begin{tabular}[t]{l}$t$\end{tabular}}}}%
    \put(0.36122394,0.27390967){\color[rgb]{0,0,0}\makebox(0,0)[lt]{\lineheight{1.25}\smash{\begin{tabular}[t]{l}$e_t$\end{tabular}}}}%
    \put(0,0){\includegraphics[width=\unitlength,page=2]{loop-tree-def.pdf}}%
  \end{picture}%
\endgroup%

%% file: img/msw-review.pdf_tex
%% Creator: Inkscape 1.1 (c4e8f9ed74, 2021-05-24), www.inkscape.org
%% Accompanies image file 'msw-review.pdf' (pdf, eps, ps)
\begingroup%
  \makeatletter%
  \providecommand\color[2][]{%
    \errmessage{(Inkscape) Color is used for the text in Inkscape, but the package 'color.sty' is not loaded}%
    \renewcommand\color[2][]{}%
  }%
  \providecommand\transparent[1]{%
    \errmessage{(Inkscape) Transparency is used (non-zero) for the text in Inkscape, but the package 'transparent.sty' is not loaded}%
    \renewcommand\transparent[1]{}%
  }%
  \providecommand\rotatebox[2]{#2}%
  \newcommand*\fsize{\dimexpr\f@size pt\relax}%
  \newcommand*\lineheight[1]{\fontsize{\fsize}{#1\fsize}\selectfont}%
  \ifx\svgwidth\undefined%
    \setlength{\unitlength}{272.36705619bp}%
    \ifx\svgscale\undefined%
      \relax%
    \else%
      \setlength{\unitlength}{\unitlength * \real{\svgscale}}%
    \fi%
  \else%
    \setlength{\unitlength}{\svgwidth}%
  \fi%
  \global\let\svgwidth\undefined%
  \global\let\svgscale\undefined%
  \makeatother%
  \begin{picture}(1,0.480149)%
    \lineheight{1}%
    \setlength\tabcolsep{0pt}%
    \put(0,0){\includegraphics[width=\unitlength,page=1]{msw-review.pdf}}%
    \put(0.01155359,0.38232325){\color[rgb]{0,0,0}\makebox(0,0)[lt]{\lineheight{1.25}\smash{\begin{tabular}[t]{l}$l$\end{tabular}}}}%
    \put(0.39339941,0.0337841){\color[rgb]{0,0,0}\makebox(0,0)[lt]{\lineheight{1.25}\smash{\begin{tabular}[t]{l}$r$\end{tabular}}}}%
    \put(0.83840569,0.03673929){\color[rgb]{0,0,0}\makebox(0,0)[lt]{\lineheight{1.25}\smash{\begin{tabular}[t]{l}$r$\end{tabular}}}}%
    \put(0.56673009,0.33791575){\color[rgb]{0,0,0}\makebox(0,0)[lt]{\lineheight{1.25}\smash{\begin{tabular}[t]{l}$l$\end{tabular}}}}%
    \put(0,0){\includegraphics[width=\unitlength,page=2]{msw-review.pdf}}%
  \end{picture}%
\endgroup%

%% file: img/loop-soup.pdf_tex
%% Creator: Inkscape 1.1.1 (3bf5ae0d25, 2021-09-20), www.inkscape.org
%% Accompanies image file 'loop-soup.pdf' (pdf, eps, ps)
\begingroup%
  \makeatletter%
  \providecommand\color[2][]{%
    \errmessage{(Inkscape) Color is used for the text in Inkscape, but the package 'color.sty' is not loaded}%
    \renewcommand\color[2][]{}%
  }%
  \providecommand\transparent[1]{%
    \errmessage{(Inkscape) Transparency is used (non-zero) for the text in Inkscape, but the package 'transparent.sty' is not loaded}%
    \renewcommand\transparent[1]{}%
  }%
  \providecommand\rotatebox[2]{#2}%
  \newcommand*\fsize{\dimexpr\f@size pt\relax}%
  \newcommand*\lineheight[1]{\fontsize{\fsize}{#1\fsize}\selectfont}%
  \ifx\svgwidth\undefined%
    \setlength{\unitlength}{298.02017941bp}%
    \ifx\svgscale\undefined%
      \relax%
    \else%
      \setlength{\unitlength}{\unitlength * \real{\svgscale}}%
    \fi%
  \else%
    \setlength{\unitlength}{\svgwidth}%
  \fi%
  \global\let\svgwidth\undefined%
  \global\let\svgscale\undefined%
  \makeatother%
  \begin{picture}(1,0.42492784)%
    \lineheight{1}%
    \setlength\tabcolsep{0pt}%
    \put(0,0){\includegraphics[width=\unitlength,page=1]{loop-soup.pdf}}%
    \put(0.37056993,0.37925312){\color[rgb]{0,0,0}\makebox(0,0)[lt]{\lineheight{1.25}\smash{\begin{tabular}[t]{l}$\partial B_{\delta/2}(z_i)$\end{tabular}}}}%
    \put(0.30982934,0.00447891){\color[rgb]{0,0,0}\makebox(0,0)[lt]{\lineheight{1.25}\smash{\begin{tabular}[t]{l}$(\eta^-_{i,k}\colon k\ge 1)$\end{tabular}}}}%
    \put(0.73078004,0.00395678){\color[rgb]{0,0,0}\makebox(0,0)[lt]{\lineheight{1.25}\smash{\begin{tabular}[t]{l}$(\eta^+_{i,k}\colon k\ge 1)$\end{tabular}}}}%
  \end{picture}%
\endgroup%

%% file: img/cle-nonsimple-flowline.pdf_tex
%% Creator: Inkscape 1.2.1 (9c6d41e410, 2022-07-14, custom), www.inkscape.org
%% Accompanies image file 'cle-nonsimple-flowline.pdf' (pdf, eps, ps)
\begingroup%
  \makeatletter%
  \providecommand\color[2][]{%
    \errmessage{(Inkscape) Color is used for the text in Inkscape, but the package 'color.sty' is not loaded}%
    \renewcommand\color[2][]{}%
  }%
  \providecommand\transparent[1]{%
    \errmessage{(Inkscape) Transparency is used (non-zero) for the text in Inkscape, but the package 'transparent.sty' is not loaded}%
    \renewcommand\transparent[1]{}%
  }%
  \providecommand\rotatebox[2]{#2}%
  \newcommand*\fsize{\dimexpr\f@size pt\relax}%
  \newcommand*\lineheight[1]{\fontsize{\fsize}{#1\fsize}\selectfont}%
  \ifx\svgwidth\undefined%
    \setlength{\unitlength}{418.91143041bp}%
    \ifx\svgscale\undefined%
      \relax%
    \else%
      \setlength{\unitlength}{\unitlength * \real{\svgscale}}%
    \fi%
  \else%
    \setlength{\unitlength}{\svgwidth}%
  \fi%
  \global\let\svgwidth\undefined%
  \global\let\svgscale\undefined%
  \makeatother%
  \begin{picture}(1,0.31322045)%
    \lineheight{1}%
    \setlength\tabcolsep{0pt}%
    \put(0,0){\includegraphics[width=\unitlength,page=1]{cle-nonsimple-flowline.pdf}}%
    \put(0.04921167,0.09079719){\color[rgb]{0,0,0}\makebox(0,0)[lt]{\lineheight{1.25}\smash{\begin{tabular}[t]{l}$x_1$\end{tabular}}}}%
    \put(0.0169935,0.01371812){\color[rgb]{0,0,0}\makebox(0,0)[lt]{\lineheight{1.25}\smash{\begin{tabular}[t]{l}\textcolor{blue}{$\xi^{\op{W}}_{x_1}$}\end{tabular}}}}%
    \put(0.03625738,0.15737518){\color[rgb]{0,0,0}\makebox(0,0)[lt]{\lineheight{1.25}\smash{\begin{tabular}[t]{l}\textcolor{red}{$\xi^{\op{E}}_{x_1}$}\end{tabular}}}}%
    \put(0.51050341,0.19754275){\color[rgb]{0,0,0}\makebox(0,0)[lt]{\lineheight{1.25}\smash{\begin{tabular}[t]{l}$z$\end{tabular}}}}%
    \put(0.89530583,0.13840414){\color[rgb]{0,0,0}\makebox(0,0)[lt]{\lineheight{1.25}\smash{\begin{tabular}[t]{l}$z_j$\end{tabular}}}}%
    \put(0.75941397,0.132857){\color[rgb]{0,0,0}\makebox(0,0)[lt]{\lineheight{1.25}\smash{\begin{tabular}[t]{l}$z_i$\end{tabular}}}}%
    \put(0.79993381,0.04111665){\color[rgb]{0,0,0}\makebox(0,0)[lt]{\lineheight{1.25}\smash{\begin{tabular}[t]{l}\textcolor{blue}{$B_i'$}\end{tabular}}}}%
    \put(0.75073841,0.05353407){\color[rgb]{0,0,0}\makebox(0,0)[lt]{\lineheight{1.25}\smash{\begin{tabular}[t]{l}\textcolor{myviolet}{$B_i$}\end{tabular}}}}%
    \put(0,0){\includegraphics[width=\unitlength,page=2]{cle-nonsimple-flowline.pdf}}%
    \put(0.80865722,0.23636011){\color[rgb]{0,0,0}\makebox(0,0)[lt]{\lineheight{1.25}\smash{\begin{tabular}[t]{l}\textcolor{red}{$\eta_{i,r}$}\end{tabular}}}}%
    \put(0,0){\includegraphics[width=\unitlength,page=3]{cle-nonsimple-flowline.pdf}}%
    \put(0.88824601,0.24451931){\color[rgb]{0,0,0}\makebox(0,0)[lt]{\lineheight{1.25}\smash{\begin{tabular}[t]{l}\textcolor{red}{$\eta_{j,r}$}\end{tabular}}}}%
    \put(0.89741299,0.0400878){\color[rgb]{0,0,0}\makebox(0,0)[lt]{\lineheight{1.25}\smash{\begin{tabular}[t]{l}\textcolor{myviolet}{$B_j$}\end{tabular}}}}%
    \put(0.84560394,0.05033531){\color[rgb]{0,0,0}\makebox(0,0)[lt]{\lineheight{1.25}\smash{\begin{tabular}[t]{l}\textcolor{blue}{$B_j'$}\end{tabular}}}}%
    \put(0,0){\includegraphics[width=\unitlength,page=4]{cle-nonsimple-flowline.pdf}}%
  \end{picture}%
\endgroup%

%% file: img/cle-nonsimple-ig-proof.pdf_tex
%% Creator: Inkscape 1.1.1 (3bf5ae0d25, 2021-09-20), www.inkscape.org
%% Accompanies image file 'cle-nonsimple-ig-proof.pdf' (pdf, eps, ps)
\begingroup%
  \makeatletter%
  \providecommand\color[2][]{%
    \errmessage{(Inkscape) Color is used for the text in Inkscape, but the package 'color.sty' is not loaded}%
    \renewcommand\color[2][]{}%
  }%
  \providecommand\transparent[1]{%
    \errmessage{(Inkscape) Transparency is used (non-zero) for the text in Inkscape, but the package 'transparent.sty' is not loaded}%
    \renewcommand\transparent[1]{}%
  }%
  \providecommand\rotatebox[2]{#2}%
  \newcommand*\fsize{\dimexpr\f@size pt\relax}%
  \newcommand*\lineheight[1]{\fontsize{\fsize}{#1\fsize}\selectfont}%
  \ifx\svgwidth\undefined%
    \setlength{\unitlength}{420.02365019bp}%
    \ifx\svgscale\undefined%
      \relax%
    \else%
      \setlength{\unitlength}{\unitlength * \real{\svgscale}}%
    \fi%
  \else%
    \setlength{\unitlength}{\svgwidth}%
  \fi%
  \global\let\svgwidth\undefined%
  \global\let\svgscale\undefined%
  \makeatother%
  \begin{picture}(1,0.31306348)%
    \lineheight{1}%
    \setlength\tabcolsep{0pt}%
    \put(0,0){\includegraphics[width=\unitlength,page=1]{cle-nonsimple-ig-proof.pdf}}%
    \put(0.18434842,0.19948746){\color[rgb]{0,0,0}\makebox(0,0)[lt]{\lineheight{1.25}\smash{\begin{tabular}[t]{l}$w_1$\end{tabular}}}}%
    \put(0.14646811,0.10286845){\color[rgb]{0,0,0}\makebox(0,0)[lt]{\lineheight{1.25}\smash{\begin{tabular}[t]{l}$w$\end{tabular}}}}%
    \put(0.0396654,0.07882964){\color[rgb]{0,0,0}\makebox(0,0)[lt]{\lineheight{1.25}\smash{\begin{tabular}[t]{l}$x_1$\end{tabular}}}}%
    \put(0.54150312,0.18998224){\color[rgb]{0,0,0}\makebox(0,0)[lt]{\lineheight{1.25}\smash{\begin{tabular}[t]{l}$w_1$\end{tabular}}}}%
    \put(0.48838056,0.21421954){\color[rgb]{0,0,0}\makebox(0,0)[lt]{\lineheight{1.25}\smash{\begin{tabular}[t]{l}\textcolor{blue}{$\xi^{\op{W}}_{w_1}$}\end{tabular}}}}%
    \put(0.37531379,0.18845295){\color[rgb]{0,0,0}\makebox(0,0)[lt]{\lineheight{1.25}\smash{\begin{tabular}[t]{l}\textcolor{red}{$\xi^{\op{E}}_{w_1}$}\end{tabular}}}}%
    \put(0,0){\includegraphics[width=\unitlength,page=2]{cle-nonsimple-ig-proof.pdf}}%
    \put(0.84945914,0.17851421){\color[rgb]{0,0,0}\makebox(0,0)[lt]{\lineheight{1.25}\smash{\begin{tabular}[t]{l}\textcolor{cyan}{$V$}\end{tabular}}}}%
    \put(0.7045214,0.15126136){\color[rgb]{0,0,0}\makebox(0,0)[lt]{\lineheight{1.25}\smash{\begin{tabular}[t]{l}\textcolor{ForestGreen}{$\partial U$}\end{tabular}}}}%
    \put(0.72829911,0.08917964){\color[rgb]{0,0,0}\makebox(0,0)[lt]{\lineheight{1.25}\smash{\begin{tabular}[t]{l}$x_1$\end{tabular}}}}%
    \put(0.71178033,0.21470099){\color[rgb]{0,0,0}\makebox(0,0)[lt]{\lineheight{1.25}\smash{\begin{tabular}[t]{l}\textcolor{red}{$\xi^{\op{E}}_{x_1}$}\end{tabular}}}}%
    \put(0.69351113,0.01790406){\color[rgb]{0,0,0}\makebox(0,0)[lt]{\lineheight{1.25}\smash{\begin{tabular}[t]{l}\textcolor{blue}{$\xi^{\op{W}}_{x_1}$}\end{tabular}}}}%
  \end{picture}%
\endgroup%